\documentclass[12pt, a4paper, openright, twoside]{report}

\usepackage[left=30mm, right=20mm, top=30mm, bottom=22mm]{geometry}
\usepackage{emptypage}
\usepackage{amssymb, amsmath, amsthm, amsfonts, amscd}
\usepackage{mathrsfs,mathtools}
\usepackage[shortlabels]{enumitem}
\usepackage{varwidth}
\usepackage{xcolor}
\usepackage{accents}
\usepackage{romannum}
\usepackage{yfonts, euscript}
\usepackage{csquotes, dirtytalk}
\usepackage[pagewise]{lineno}%\linenumbers
\usepackage{tabularx}
\usepackage{array, nicematrix}
\usepackage{booktabs}
\usepackage{diagbox}
\usepackage{longtable}
\usepackage{makecell}
\usepackage{colortbl}
\usepackage{multirow, multicol}

\usepackage{graphicx}
\usepackage{tikz, tikz-cd} 
\usepackage[cmtip, all]{xy}
\usepackage{pst-node}
\usepackage[new]{old-arrows}
\usepackage{subfig}
\usepackage{pdfpages}

\usetikzlibrary{arrows, arrows.meta}
\usetikzlibrary{positioning}

\usepackage{imakeidx} 
\makeindex[intoc] %% Must use this for {imakeidx} to work

\usepackage[refpage,intoc]{nomencl}
\makeatletter
    
    \def\@@@nomenclature[#1]#2#3#4{
        \def\@tempa{#2} \def\@tempb{#3} \def\@tempc{#4}
        \protected@write\@nomenclaturefile{}{
            \string\nomenclatureentry{#1\nom@verb\@tempa @[{\nom@verb\@tempa}]
            \begingroup\nom@verb\@tempb\protect\nomeqref{\theequation}
            |nomlabelref}{\@tempc}} 
        \endgroup
        \@esphack}
\makeatother
\makenomenclature

\renewcommand{\nompreamble}{The following is a list of key symbols frequently used throughout this thesis:}

\usepackage[colorlinks=true, allcolors=blue]{hyperref}

\usepackage{fancyhdr}  
\usepackage{skt}

\newtheorem{thm}{{\bf Theorem}}[section]
\newtheorem{lemma}[thm]{{\bf Lemma}}
\newtheorem{prop}[thm]{{\bf Proposition}}
\newtheorem{cor}[thm]{{\bf Corollary}}
\newtheorem*{defn}{{\bf Definition}}

\newtheorem*{rmk}{{\bf Remark}}
\newtheorem*{conv}{{\bf Convention}}

\newtheorem*{facts}{{\bf Facts}}

\newtheorem*{ex}{{\bf Example}}

\newcommand{\m}{\mathfrak m}

\begin{document}

%%%%%%%%%%%%%%%%%%%%%%%%%%%%%%%%%%%%%%%%%%%%%%%%%% 

\pagenumbering{gobble} %Remove page numbers (and reset to 1)

%%%%%%%%%%%%%%%%%%%%%%%%%%%%%%%%%%%%%%%%%%%%%%%%%% 

\thispagestyle{empty} 

%%%%%%%%%%%%%%%%%%%%%%%%%%%%%%%%%%%%%

\begin{center}

{\Large \bf{Some Properties of Twisted Chevalley Groups}}\\

\vspace{1.5cm}

A thesis submitted \\ 
in partial fulfillment of the requirements \\
for the degree of \\

\vspace{1cm}

{\bf {\large Doctor of Philosophy}} \\

\vspace{1cm}

by \\

\vspace{1cm}

{\large 
{\bf Makadiya Deepkumar Hasamukhbhai} \\
% {\bf Makadiya Deepkumar Hasamukhbhai} \\
\vspace{0.4cm}
(Roll No. 194090008) \\
}

\vspace{1cm}

under the guidance of \\

\vspace{1cm}

{\large {\bf Professor Shripad M. Garge}} \\

\vspace {3cm}

\includegraphics[height=40mm,width=40mm]{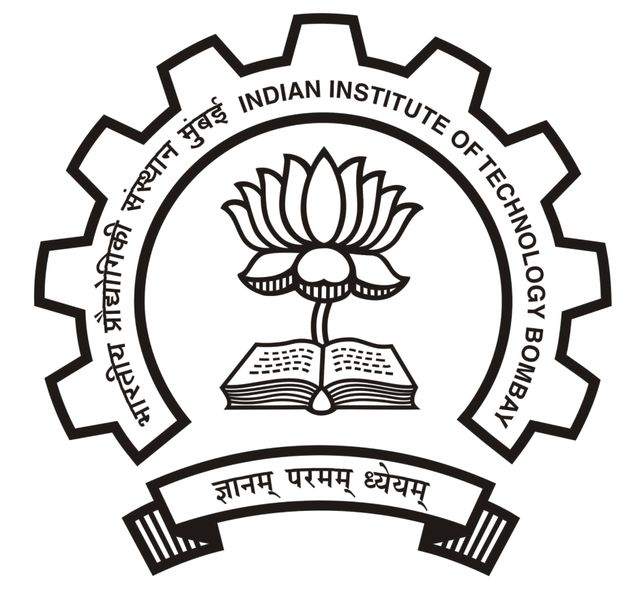} \\

\vspace {1cm}
{\large {\bf
Department of Mathematics \\
\vspace{0.5cm}
INDIAN INSTITUTE OF TECHNOLOGY BOMBAY \\
\vspace{0.6cm}
2025
}}

\end{center}

%%%%%%%%%%%%%%%%%%%%%%%%%%%%%%%%%%%%%

\newpage
\cleardoublepage

%%%%%%%%%%%%%%%%%%%%%%%%%%%%%%%%%%%%%%%%%%%%%%%%%% 

%\include{approval} 

% \thispagestyle{empty}
% \includepdf[pages=1]{signed_approval.pdf}
% \newpage
% \cleardoublepage

\thispagestyle{empty} 

%%%%%%%%%%%%%%%%%%%%%%%%%%%%%%%%%%%%%

% \addcontentsline{toc}{chapter}{Declaration}

\begin{center}
    \textbf{\Large {Declaration}}
\end{center}

\vspace{1.5cm}

I declare that this written submission represents my ideas in my own words. Where others' ideas and words have been included, I have adequately cited and referenced the original source. I declare that I have adhered to all principles of academic honesty and integrity and have not misrepresented or fabricated, or falsified any idea/data/fact/source in my submission. I understand that any violation of the above will cause disciplinary action by the Institute and can also evoke penal action from the sources which have thus not been properly cited or from whom proper permission has not been taken when needed.

%%%%%%%%%%%%%%%%%%%%%%%%%%%%%%%%%%%%%

\vspace{4cm}

\begin{center}
        \begin{tikzpicture}[scale=1]
            \node[align=left] at (0,0) {\textbf{Date:} May 28, 2025};
            \node[align=left] at (-0.05,-0.7) {\textbf{Place:} IIT Bombay};
            \draw [thick, dash pattern={on 7pt off 2pt}] (7,0.5) -- (13,0.5);
            \node at (10,0) {Makadiya Deepkumar Hasamukhbhai};
            \node at (10,-0.7) {(Roll No. 194090008)};
        \end{tikzpicture}
\end{center}

%%%%%%%%%%%%%%%%%%%%%%%%%%%%%%%%%%%%%

% \vspace{3cm}

% \begin{flushleft}
%     \begin{tikzpicture}[scale=1]
%         \begin{scope}[shift={(0.5,0)}]
%             \node[align=left] at (0,0) {\textbf{Date:} December 12, 2024};
%             \node[align=left] at (-0.55,-0.7) {\textbf{Palce:} IIT Bombay};
%         \end{scope}

%         \begin{scope}[shift={(7,0.3)}]
%             \draw [thick, dash pattern={on 7pt off 2pt}] (0,0) -- (6,0);
%             \node at (3,-0.5) {Makadiya Deepkumar Hasamukhbhai};
%             \node at (3,-1.1) {(Roll No. 194090008)};
%         \end{scope}
%     \end{tikzpicture}
% \end{flushleft}

%%%%%%%%%%%%%%%%%%%%%%%%%%%%%%%%%%%%%

\newpage
\cleardoublepage

%\thispagestyle{empty}
%\includepdf[pages=1]{signed_declaration.pdf}
%\newpage
%\cleardoublepage

%%%%%%%%%%%%%%%%%%%%%%%%%%%%%%%%%%%%%%%%%%%%%%%%%% 

\setcounter{secnumdepth}{-2} 

\pagenumbering{roman} 

\thispagestyle{empty} 

%%%%%%%%%%%%%%%%%%%%%%%%%%%%%%%%%%%%%

\vspace*{\fill}
\begin{center}
    \textit{\Large Dedicated to \\
    \vspace{2mm}
    my beloved parents}
\end{center} 
\vspace*{\fill}

%%%%%%%%%%%%%%%%%%%%%%%%%%%%%%%%%%%%%

\newpage
\cleardoublepage

\thispagestyle{empty} 

%%%%%%%%%%%%%%%%%%%%%%%%%%%%%%%%%%%%%

\chapter{Abstract}

This thesis investigates certain structural properties of twisted Chevalley groups over commutative rings, focusing on three key problems. 

\medskip

Let $R$ be a commutative ring satisfying mild conditions. Let $G_{\pi,\sigma} (\Phi, R)$ denote a twisted Chevalley group over $R$, and let $E'_{\pi, \sigma} (\Phi, R)$ denote its elementary subgroup. 
The first problem concerns the normality of $E'_{\pi, \sigma} (\Phi, R, J)$, the relative elementary subgroups at level $J$, in the group $G_{\pi, \sigma} (\Phi, R)$. 
The second problem addresses the classification of the subgroups of $G_{\pi, \sigma}(\Phi, R)$ that are normalized by $E'_{\pi, \sigma}(\Phi, R)$. This classification provides a comprehensive characterization of the normal subgroups of $E'_{\pi, \sigma}(\Phi, R)$.
Lastly, the third problem investigates the normalizers of $E'_{\pi, \sigma}(\Phi, R)$ and $G_{\pi, \sigma}(\Phi, R)$ in the bigger group $G_{\pi, \sigma}(\Phi, S)$, where $S$ is a ring extension of $R$. We prove that these normalizers coincide. Moreover, for groups of adjoint type, we show that they are precisely equal to $G_{\pi, \sigma}(\Phi, R)$.

%%%%%%%%%%%%%%%%%%%%%%%%%%%%%%%%%%%%%

\newpage
\cleardoublepage 

\thispagestyle{empty} 

%%%%%%%%%%%%%%%%%%%%%%%%%%%%%%%%%%%%%

\chapter{Acknowledgements}

I am extremely grateful for the support and encouragement I have received from numerous individuals throughout the process of completing this thesis. 
Although I cannot mention everyone by name, I would like to express my heartfelt thanks to all who have contributed in one way or another.

First and foremost, I would like to thank my advisor, Prof. Shripad M. Garge, for his constant support and guidance. His insightful mentorship, constructive feedback, and consistent encouragement have played a key role in shaping my research trajectory and helping me reach this significant milestone. From my early days as a student to my development into an independent researcher, Prof. Garge has been a strong pillar in both my academic and personal growth. His expertise has not only deepened my understanding of the subject but also fostered a thoughtful approach to selecting research problems and significantly improved my scientific writing. I am truly grateful for the time, patience, and care he has invested in guiding me. I could not have asked for a better mentor.

I am sincerely grateful to Prof. Dipendra Prasad for his constant encouragement and support. He was always generous with his time, willing to engage in discussions, and offered many insightful ideas. I thank both Prof. Prasad and Prof. Garge for organizing seminars on algebraic groups and representation theory, which significantly enriched my understanding of these topics. I would like to extend my appreciation to my colleagues and friends—Akash, Chayan, Dibyendu, and Saad—who were regular participants and speakers in these seminars. Their invaluable support and thoughtful discussions have been truly invaluable. 

I would also like to thank my Research Progress Committee members, Prof. Saurav Bhaumik and Prof. Sudharsan Gurjar, for their valuable feedback and suggestions, which helped improve the quality of this thesis.

I am deeply grateful to all my teachers at IIT Bombay, where I pursued my MSc and PhD, for their unwavering support and invaluable guidance. I am especially thankful to Prof. Ananthnarayan Hariharan for his encouragement and thoughtful advice. I also extend my sincere gratitude to my teachers at The Maharaja Sayajirao University of Baroda, where I completed my BSc; their guidance helped shape my academic journey and love for mathematics. I am indebted to Prof. Yogesh Patel, whose mentorship during my IIT JAM preparation marked a turning point for me. I offer my heartfelt thanks to all my schoolteachers who first inspired my interest in mathematics, particularly Mr. Dinesh K. Kansagara, whose lasting influence continues to inspire me. To all who have taught and supported me along the way—thank you.

I warmly thank all my friends who stood by me through both joyful and challenging times. I am especially grateful to Swagat, Payal, and Rati for their emotional support and treasured friendship. This journey would have been much harder without you. I am also deeply thankful to Nitin, Shubham, Anshuman, and Rajesh for their steady support throughout my PhD days. My thanks also go to Rahul, Sakshi, Sai, Parvez, Umesh, Swapan, Pallab, Raunak, Ankita, Aratrika, Priyanka, Unnati, Renu, Omkar, Samrendra, Krishan, Saumyajit, Sayan, and Chaman for their constant encouragement and for making this journey more enjoyable. I would also like to acknowledge all my other friends from my MSc, BSc, and school days—though I cannot name everyone here, please know that your presence and support have meant a great deal to me.

Most importantly, I owe my deepest gratitude to my family—especially my parents, my brother Krunal, my sister-in-law Mansi, and my grandmother—for their endless encouragement and support. Their faith in me gave me the strength to follow my dreams. Without them, I would not be where I am today. I also warmly thank my cousins—Raj, Dhara, Komal, Bhavik, Vivek, Vidhi and Sagar—for their encouragement and guidance. To all my other family members, I remain truly grateful for your love, patience, and support, which have been a constant source of strength. Finally, I offer my heartfelt thanks to the Almighty for guiding and blessing me through every step of this journey.

And lastly, to chai and coffee—thank you for being the silent companions who understood me when no one else could.

%%%%%%%%%%%%%%%%%%%%%%%%%%%%%%%%%%%%%

\newpage
\cleardoublepage 

\setcounter{secnumdepth}{2} 

%%%%%%%%%%%%%%%%%%%%%%%%%%%%%%%%%%%%%%%%%%%%%%%%%% 

% \nomenclature[A,001]{$\mathbb{N}$ or $\mathbb{Z}^{+}$}{Set of natural numbers}{nomencl:N}
% \nomenclature[A,002]{$\mathbb{Z}$}{Set of all integers}{nomencl:Z}
% \nomenclature[A,003]{$\mathbb{Q}$}{Set of all rational numbers}{nomencl:Q}
% \nomenclature[A,004]{$\mathbb{R}$}{Set of all real numbers}{nomencl:RR}
% \nomenclature[A,005]{$\mathbb{C}$}{Set of all complex numbers}{nomencl:C}

\nomenclature[B,001]{$R$}{A commutative ring with unity}{nomencl:R}
\nomenclature[B,002]{$\Phi$}{Root system}{nomencl:Phi}
\nomenclature[B,003]{$\Delta$}{Simple system}{nomencl:Delta}
\nomenclature[B,004]{$W = W(\Phi)$}{Weyl group of $\Phi$}{nomencl:W}
\nomenclature[B,005]{$\mathcal{L} = \mathcal{L}(\Phi, \mathbb{C})$}{Lie algebra over $\mathbb{C}$ with root system $\Phi$}{nomencl:L}
\nomenclature[B,006]{$\Lambda_r$ or $\Lambda_{\text{ad}}$}{The root lattice corresponding to $\Phi$}{nomencl:rootlattice}
\nomenclature[B,007]{$\Lambda_{sc}$}{The (fundamental) weight lattice corresponding to $\Phi$}{nomencl:weightlattice}
\nomenclature[B,008]{$\Lambda_{\pi}$}{A sublattice of $\Lambda_{sc}$ that contains $\Lambda_r$}{nomencl:lattice}

\nomenclature[D,001]{$G_{\pi}(\Phi, R)$ or $G(R)$}{Chevalley group of type $\Phi$ over $R$}{nomencl:G(R)}
\nomenclature[D,002]{$E_{\pi}(\Phi, R)$ or $E(R)$}{Elementary Chevalley group of type $\Phi$ over $R$}{nomencl:E(R)}
\nomenclature[D,003]{$G_{\pi}(\Phi, J)$ or $G(J)$}{Principal congruence subgroup of $G_{\pi}(\Phi, R)$ of level $J$}{nomencl:G(J)}
\nomenclature[D,004]{$G_{\pi}(\Phi, R, J)$ or $G(R, J)$}{Full congruence subgroup of $G_{\pi}(\Phi, R)$ of level $J$}{nomencl:G(R,J)}
\nomenclature[D,005]{$E_{\pi}(\Phi, R, J)$ or $E(R, J)$}{Relative elementary subgroup of $G_{\pi}(\Phi, R)$ of level $J$}{nomencl:E(R,J)}

\nomenclature[E,001]{$\theta$}{A ring automorphism of $G_{\pi}(\Phi, R)$}{nomencl:theta}
\nomenclature[E,002]{$\rho$}{A graph automorphism of $G_{\pi}(\Phi, R)$}{nomencl:rho}
\nomenclature[E,003]{$\sigma$}{An automorphism of $G_{\pi}(\Phi, R)$ given by $\sigma = \theta \circ \rho$}{nomencl:sigma}
\nomenclature[E,004]{$\Phi_\rho$}{Twisted root system}{nomencl:Phi_rho}

\nomenclature[F,001]{$G_{\pi, \sigma}(\Phi, R)$ or $G_\sigma(R)$}{Twisted Chevalley group of type $\Phi$ over $R$}{nomencl:G_sigma(R)}
\nomenclature[F,002]{$E'_{\pi, \sigma}(\Phi, R)$ or $E'_\sigma(R)$}{Elementary Twisted Chevalley group of type $\Phi$ over $R$}{nomencl:E'_sigma(R)}
\nomenclature[F,003]{$G_{\pi, \sigma}(\Phi, J)$ or $G_\sigma(J)$}{Principal congruence subgroup of $G_{\pi, \sigma}(\Phi, R)$ of level $J$}{nomencl:G_sigma(J)}
\nomenclature[F,004]{$G_{\pi,\sigma}(\Phi, R, J)$ or $G_\sigma(R, J)$}{Full congruence subgroup of $G_{\pi, \sigma}(\Phi, R)$ of level $J$}{nomencl:G_sigma(R,J)}
\nomenclature[F,005]{$E'_{\pi, \sigma}(\Phi, R, J)$ or $E'_\sigma(R, J)$}{Relative elementary subgroup of $G_{\pi, \sigma}(\Phi, R)$ of level $J$}{nomencl:E'_sigma(R,J)}

\printnomenclature

%%%%%%%%%%%%%%%%%%%%%%%%%%%%%%%%%%%%%%%%%%%%%%%%%% 

\thispagestyle{empty} 

\tableofcontents 

\newpage 

\cleardoublepage 

%%%%%%%%%%%%%%%%%%%%%%%%%%%%%%%%%%%%%%%%%%%%%%%%%% 

\pagenumbering{arabic} 

\chapter{Introduction}

%%%%%%%%%%%%%%%%%%%%%%%%%%%%%%%%%%%%%%%%%%%%%%
%%%%%%%%%%%%%%%%%%%%%%%%%%%%%%%%%%%%%%%%%%%%%%

This thesis aims to study certain structural properties of the twisted Chevalley groups. 
In this introductory chapter, we present the main results, followed by a historical overview and the motivations behind this study.
To provide context, we first introduce some basic definitions, which will be explored in greater detail in Chapters \ref{chapter:chevalley groups} and \ref{chapter:TCG}. 

%%%%%%%%%%%%%%%%%%%%%%%%%%%%%%%%%%%%%%%%%%%%%%

\section{Basic Definitions}

It is well known that semisimple algebraic groups over $\mathbb{C}$ (or any algebraically closed field of characteristic zero) are classified by their root systems, which determine their type, and by the lattice between the root and (fundamental) weight lattices, which determines their isogeny class.

\smallskip

Let $\Phi$ be a (reduced) root system, and let $\Lambda_{\text{ad}}$ and $\Lambda_{\text{sc}}$ denote the corresponding root lattice and (fundamental) weight lattice, respectively. Consider a lattice $\Lambda_{\pi}$ such that $\Lambda_{\text{ad}} \subseteq \Lambda_{\pi} \subseteq \Lambda_{\text{sc}}$. The unique (up to isomorphism) semisimple linear algebraic group over $\mathbb{C}$ associated with the pair $(\Phi, \Lambda_{\pi})$ is denoted by $G_{\pi}(\Phi, \mathbb{C})$.

Moreover, this group is defined over $\mathbb{Z}$, which allows us to define the $R$-split group $G_{\pi}(\Phi, R)$ for any commutative ring $R$ with unity.\label{nomencl:R} This group is known as the \textit{Chevalley group} of type $\Phi$ over $R$.

For each root $\alpha \in \Phi$, there is an associated unipotent subgroup $U_{\alpha} = \{ x_{\alpha}(t) \mid t \in R \} \subseteq G_{\pi}(\Phi, R),$ which is isomorphic to the additive group $(R, +)$. The group generated by all elements $x_{\alpha}(t)$, where $\alpha \in \Phi$ and $t \in R$, is called the \textit{elementary Chevalley group}, and is denoted by $E_{\pi}(\Phi, R)$.

For an ideal $J$ of $R$, the natural projection map $R \longrightarrow R/J$ induces a group homomorphism
\[
\phi: G_\pi (\Phi, R) \longrightarrow G_\pi (\Phi, R/J).
\]  
Define  
\[
G_\pi (\Phi, J) \coloneq \ker(\phi) \quad \text{and} \quad G_\pi (\Phi, R, J) \coloneq \phi^{-1}(Z(G_\pi (\Phi, R/J))),
\]  
where $Z(G_\pi (\Phi, R/J))$ denotes the center of $G_\pi (\Phi, R/J)$.  
The subgroup $G_\pi(\Phi, J)$ of $G_\pi (\Phi, R)$ is referred to as the \emph{principal congruence subgroup} of level $J$, while $G_\pi(\Phi, R, J)$ is called the \emph{full congruence subgroup} of level $J$.
Let $E_\pi (\Phi, J)$ denote the subgroup of $E_\pi (\Phi, R)$ generated by $x_\alpha(t)$ for all $\alpha \in \Phi$ and $t \in J$.  
Additionally, define $E_\pi (\Phi, R, J)$ as the normal subgroup of $E_\pi (\Phi, R)$ generated by $E_\pi (\Phi, J)$. The subgroup $E_\pi(\Phi, R, J)$ of $G_\pi (\Phi, R)$ is referred to as the \emph{relative elementary subgroup} of level $J$. 

\smallskip

Now, let $\sigma$ be the composition of a graph automorphism $\rho$ and a ring automorphism $\theta$ of $G_\pi (\Phi, R)$, where $\rho$ and $\theta$ have the same order. The twisted Chevalley group $G_{\pi, \sigma} (\Phi, R)$ is defined to be the set of all elements in $G_\pi (\Phi, R)$ that are fixed by the automorphism $\sigma$. 
Let $E'_{\pi, \sigma} (\Phi, R)$ be the subgroup of $G_{\pi, \sigma} (\Phi, R)$ generated by $x_{[\alpha]}(t)$ where $[\alpha] \in \Phi_\rho$ and $t \in R_{[\alpha]}$ (see Section~\ref{sec:E(R)} for the notation). 

Consider an ideal $J$ of $R$ that is invariant under $\theta$ (i.e., $\theta (J) \subset J$). 
The natural projection map $R \longrightarrow R/J$ induces a group homomorphism  
\[
\phi: G_{\pi, \sigma} (\Phi, R) \longrightarrow G_{\pi, \sigma} (\Phi, R/J).
\]  
Define the subgroups 
\[
G_{\pi, \sigma} (\Phi, J) \coloneq \ker(\phi) \quad \text{and} \quad G_{\pi, \sigma} (\Phi, R, J) \coloneq \phi^{-1} (Z(G_{\pi, \sigma} (\Phi, R/J))),
\]  
where $Z(G_{\pi, \sigma} (\Phi, R/J))$ is the center of the group $G_{\pi, \sigma} (\Phi, R/J)$.
The subgroup $G_{\pi, \sigma} (\Phi, J)$ of $G_{\pi, \sigma} (\Phi, R)$ is called a \emph{principal congruence subgroup} of level $J$, while $G_{\pi, \sigma} (\Phi, R, J)$ is referred to as a \emph{full congruence subgroup} of level $J$. 
Let $E'_{\pi, \sigma} (\Phi, J)$ be the subgroup of $E'_{\pi, \sigma} (\Phi, R)$ generated by elements $x_{[\alpha]}(t)$ for all $[\alpha] \in \Phi_\rho$ and $t \in J_{[\alpha]}$. Additionally, define $E'_{\pi, \sigma} (\Phi, R, J)$, called the \emph{relative elementary subgroup} of level $J$, as the normal subgroup of $E'_{\pi, \sigma} (\Phi, R)$ generated by $E'_{\pi, \sigma} (\Phi, J)$. 

%%%%%%%%%%%%%%%%%%%%%%%%%%%%%%%%%%%%%%%%%%%%%%

\section{Certain Commutator Formulas}

We begin by presenting one of the main structure theorems, which asserts that the elementary subgroups are normal in the Chevalley group $G_\pi (\Phi, R)$. 

\begin{thm}[{L. N. Vaserstein \cite{LV}}]\label{thm:Intro1}
    Let $\Phi$ be an irreducible root system of rank $\geq 2$ and \(R\) a commutative ring with unity. If $J$ is an ideal of $R$, then the following commutator relations hold: 
    \[
        [E_\pi(\Phi, R, J), G_\pi(\Phi, R)] \subset E_\pi(\Phi, R, J) \quad \text{and} \quad [E_\pi(\Phi, R), G_\pi(\Phi, R, J)] \subset E_\pi(\Phi, R, J).
    \]
    Except in the cases where $\Phi = B_2$ or $G_2$ and $R$ has a residue field with two elements, these inclusions are equalities.
\end{thm}

\begin{rmk}
    The first commutator relation in the above theorem is equivalent to saying that $E_\pi (\Phi, R, J)$ is a normal subgroup of $G_\pi (\Phi, R)$. In particular, $E_\pi (\Phi, R)$ is normal in $G_\pi (\Phi, R)$.
\end{rmk}

Several variants of this theorem for different classical and exceptional groups over various rings are available in the literature. 
These results have several important applications which includes describing normal subgroups (see the next section). Indeed, such results were first observed in the context of the general linear group over the ring \(\mathbb{Z}\) in the works of H. Bass~\cite{HB1, HB2}. 
It was subsequently extended to other classical groups by A. Bak~\cite{ABak1, ABak2}. 
A significant breakthrough came with the work of A. A. Suslin~\cite{AS1}, who proved that the elementary subgroup is always normal in \(GL(n, R)\) when \(R\) is commutative and \(n \geq 3\).
This result was later generalized by A. A. Suslin and V. I. Kopeiko~\cite{AS&VK, VK} to the even-dimensional split orthogonal group and the symplectic group. The symplectic case was independently rediscovered by G. Taddei~\cite{GT2}, while the odd-dimensional split orthogonal group was studied by N. A. Vavilov~\cite{NV4}.
For Chevalley groups, these types of results were primarily studied by M. R. Stein, G. Taddei, L. N. Vaserstein, Z. I. Borevich, E. B. Plotkin and N. A. Vavilov (see \cite{MS, GT, LV, LV2, LV3, LV4, NV5, NV6, NV7, ZB&NV, EP1, EP2}). For a more comprehensive historical perspective, we refer the reader to N. A. Vavilov~\cite{NV1}.

In a joint study with S. M. Garge~\cite{SG&DM1}, the author established similar results for twisted Chevalley groups. The statement of the result is as follows:

\begin{thm}[Main Theorem 1]\label{mainthm1}
    Let $\Phi_\rho$ be one of the following types: ${}^2 A_n \ (n \geq 3), {}^2 D_n \ (n \geq 4), {}^2 E_6$, or ${}^3 D_4$. Assume that $1/2 \in R$ and in addition, $1/3 \in R$ if $\Phi_\rho \sim {}^3 D_4$.
    Let $J$ be a $\theta$-invariant ideal of $R$. Then 
    \begin{align*}
        E'_{\pi,\sigma} (\Phi, R, J) &= [E'_{\pi, \sigma} (\Phi, R), E'_{\pi,\sigma}(\Phi, J)] \\
        &= [E'_{\pi,\sigma} (\Phi, R), G_{\pi,\sigma} (\Phi, R, J)] \\
        &= [G_{\pi,\sigma} (\Phi, R), E'_{\pi,\sigma} (\Phi, R, J)].
    \end{align*}
\end{thm}

\begin{rmk}
    The proof of the above theorem can be found in Chapter~\ref{chapter:Normal_subgroups} (see Theorem~\ref{Ch5_mainthm1}).
\end{rmk}

%%%%%%%%%%%%%%%%%%%%%%%%%%%%%%%%%%%%%%%%%%%%%%
	
\section{Classification of Normal Subgroups}

The second main structure theorem focuses on the description of normal subgroups. However, the typical approach involves describing subgroups that are normalized by the elementary subgroups (whether they are normal in the entire Chevalley group is a separate matter).

\begin{thm}[{L. N. Vaserstein \cite{LV}}, E. Abe \cite{EA3}]\label{thm:Intro2}
    Let $\Phi$ be an irreducible root system of rank $\geq 2$ and $R$ a commutative ring with unity. Assume $1/ 2 \in R$ if $\Phi \sim B_\ell, C_\ell, F_4$, and that $1/3 \in R$ and $R$ has no factor ring with two elements if $\Phi \sim G_2$. If $H$ is a subgroup of $G_\pi(\Phi, R)$ normalized by $E_\pi(\Phi, R)$, then there exists a unique ideal $J$ of $R$ such that 
    \[
        E_\pi(\Phi, R, J) \subset H \subset G_\pi (\Phi, R, J).
    \]
\end{thm}

\begin{rmk}
    Both L. N. Vaserstein and E. Abe provide slightly more general versions of this result than the one presented here. However, for the sake of simplicity, we presented a weaker version.
\end{rmk}

Results similar to Theorem \ref{thm:Intro2} for the classical groups have a long history extending back to the 19th century. Since then numerous mathematicians have addressed the problem for various classical and exceptional groups over different rings. For a more detailed historical account, we recommend referring to N. A. Vavilov~\cite{NV1}.

In the context of Chevalley groups, the problem of describing its normal subgroups was first addressed by E. Abe~\cite{EA2}, who considered the case of local rings. Later, in 1976, E. Abe and K. Suzuki~\cite{EA&KS} proved results that closely approximated the result of Theorem~\ref{thm:Intro2}. In 1986, L. N. Vaserstein~\cite{LV} proved the result presented in Theorem~\ref{thm:Intro2}. In 1989, E. Abe~\cite{EA3} further generalized this result and provided an alternative proof to that of Vaserstein. 

In 1977, K. Suzuki~\cite{KS1} established a similar result for twisted Chevalley groups over local rings. More recently, the author, in collaboration with S. M. Garge~\cite{SG&DM1}, extended this work to twisted Chevalley groups over commutative rings.

\begin{thm}[Main Theorem 2]\label{mainthm}
    \normalfont
    Let $\Phi_\rho$ be one of the following types: ${}^2 A_n \ (n \geq 3), {}^2 D_n \ (n \geq 4), {}^2 E_6$, or ${}^3 D_4$. Assume that $1/2 \in R$, and in addition, $1/3 \in R$ if $\Phi_\rho \sim {}^3 D_4$. Then $H$ is a subgroup of $G_{\pi, \sigma} (\Phi, R)$ normalized by $E'_{\pi, \sigma} (\Phi, R)$ if and only if there exists a unique $\theta$-invariant ideal $J$ of $R$ such that 
    \[
        E'_{\pi, \sigma} (\Phi, R, J) \subseteq H \subseteq G_{\pi, \sigma} (\Phi, R, J).
    \]
\end{thm}

\begin{rmk}
        The proof of the above theorem is provided in Chapter~\ref{chapter:Normal_subgroups} (see Theorem~\ref{Ch5_mainthm} and Corollary~\ref{cor:converse}). Our approach to the proof is inspired by the work of L. N. Vaserstein~\cite{LV} and E. Abe~\cite{EA3}, although the calculations involved here are significantly more intricate.
\end{rmk}

As an immediate consequence of the above theorem, we obtain a classification of all normal subgroups of elementary twisted Chevalley groups:

\begin{cor}
    \normalfont
    Let $R$ and $\Phi_\rho$ be as in Theorem~\ref{mainthm}. A subgroup $H$ of $E'_{\pi, \sigma}(\Phi, R)$ is normal in $E'_{\pi, \sigma}(\Phi, R)$ if and only if there exists a $\theta$-invariant ideal $J$ of $R$ such that
    \[
        E'_{\pi, \sigma}(\Phi, R, J) \subseteq H \subseteq G_{\pi, \sigma}(\Phi, R, J) \cap E'_{\pi, \sigma}(\Phi, R).
    \]
\end{cor}

%%%%%%%%%%%%%%%%%%%%%%%%%%%%%%%%%%%%%%%%%%%%%%
	
\section{Normalizer of \texorpdfstring{$G_{\pi, \sigma} (\Phi, R)$}{G(R)} and \texorpdfstring{$E'_{\pi, \sigma} (\Phi, R)$}{E(R)} in the larger group \texorpdfstring{$G_{\pi, \sigma} (\Phi, S)$}{G(S)}}

Let \( R \) be a commutative ring with unity. Let \( G_{\pi, \sigma} (\Phi, R) \) denote a twisted Chevalley group over the ring \( R \) and let \( E'_{\pi, \sigma} (\Phi, R) \) be its elementary subgroup. Suppose \( S \) is a ring extension of \( R \); that is, \( S \) is a ring with unity such that \( R \) is a subring of \( S \) and \( 1_S = 1_R \). 
Assume further that there exists a ring automorphism of \( S \) whose restriction to \( R \) coincides with the ring automorphism \( \theta \) of \( R \), and the order of this automorphism of \( S \) is equal to the order of \( \theta \). Therefore, the group $G_{\pi, \sigma} (\Phi, S)$ make sense.

We aim to study the normalizer of \( G_{\pi, \sigma} (\Phi, R) \) and \( E'_{\pi, \sigma} (\Phi, R) \) in the larger group \( G_{\pi, \sigma} (\Phi, S) \). The motivation for this investigation comes from the problem of characterizing all automorphisms of twisted Chevalley groups. A version of this theorem for Chevalley groups was proven by E. I. Bunina~\cite{EB12:main}.

\begin{thm}[Main Theorem 3]\label{mainthm3}
    \normalfont
    Suppose that $\Phi_\rho$ is of type ${}^2 A_n \ (n \geq 3), {}^2 D_n \ (n \geq 4)$ or ${}^2 E_6$.  
    Assume that \( 1/2 \in R \), and in the case where \( \Phi_\rho \sim {}^2 A_{2n} \), also assume that \( 1/3 \in R \). Furthermore, assume that there exists an invertible element \( a \in R \) such that \( \theta(a) = -a \).  
    Under these assumptions, we have the following equality:
    \[
        N_{G_{\pi,\sigma}(\Phi, S)}(G_{\pi, \sigma}(\Phi, R)) = N_{G_{\pi,\sigma} (\Phi, S)} (E'_{\pi, \sigma} (\Phi, R)).
    \]
    Moreover, if $G$ is of adjoint type, we have
    \[
        N_{G_{\text{ad},\sigma}(\Phi, S)}(G_{\text{ad}, \sigma}(\Phi, R)) = N_{G_{\text{ad},\sigma} (\Phi, S)} (E'_{\text{ad}, \sigma} (\Phi, R)) = G_{\text{ad}, \sigma} (\Phi, R).
    \]
\end{thm}

\begin{rmk}
    The proof of the above theorem is provided in Chapter~\ref{chapter:Normalizer_of_TCG} (see Theorem~\ref{thm:normalizer_of_G_and_E}).
\end{rmk}

%%%%%%%%%%%%%%%%%%%%%%%%%%%%%%%%%%%%%%%%%%%%%%
	
\section{Outline of the Thesis} 

In this section, we provide an overview of the organizational structure and content of the thesis.

Chapter~\ref{chapter:preliminaries} introduces essential concepts and foundational material necessary for understanding the subsequent chapters. The chapter begins by defining the concepts of roots and weights. It then proceeds with introducing the concept of Lie algebras and an exploration of the structural theory of semisimple Lie algebras. The chapter concludes with a discussion of their representation theory.

Chapter~\ref{chapter:chevalley groups} focuses on the study of Chevalley algebras and Chevalley groups. It introduces their fundamental properties and provides a survey of several significant known results in this area.

Chapter~\ref{chapter:TCG} delves into twisted root systems, twisted Chevalley algebras, and twisted Chevalley groups. The chapter examines their fundamental properties, including the Chevalley commutator relations and Steinberg relations, and concludes with a review of additional significant results in this area.

Chapter~\ref{chapter:Normal_subgroups} is dedicated to proving Theorems~\ref{mainthm1} and \ref{mainthm}, which were introduced earlier in this chapter. Additionally, the chapter presents several other important results related to the structure of these groups.

The final chapter, Chapter~\ref{chapter:Normalizer_of_TCG}, provides a proof of Theorem~\ref{mainthm3}. Furthermore, the chapter examines the concept of the tangent space of twisted Chevalley groups of adjoint type.

%%%%%%%%%%%%%%%%%%%%%%%%%%%%%%%%%%%%%%%%%%%%%% 

%%%%%%%%%%%%%%%%%%%%%%%%%%%%%%%%%%%%%%%%%%%%%%%%%%%%%%%%%%%%%%
%%%%%%%%%%%%%%%%%%%%%%%%%%%%%%%%%%%%%%%%%%%%%%%%%%%%%%%%%%%%%%

\chapter{Preliminaries}\label{chapter:preliminaries}

%%%%%%%%%%%%%%%%%%%%%%%%%%%%%%%%%%%%%%%%%%%%%%%%%%%%%%%%%%%%%%
%%%%%%%%%%%%%%%%%%%%%%%%%%%%%%%%%%%%%%%%%%%%%%%%%%%%%%%%%%%%%%

This chapter covers key preliminaries, starting with the axiomatic definition of roots and weights.
We then define the concept of Lie algebras and proceed to classify semisimple Lie algebras, examining how root systems are intricately related to these algebras.
Finally, we provide a broad overview of the representation theory of semisimple Lie algebras.

%%%%%%%%%%%%%%%%%%%%%%%%%%%%%%%%%%%%%%%%%%%%%%%%%%%%%%%%%%%%%%
%Section - Roots and Weights
%%%%%%%%%%%%%%%%%%%%%%%%%%%%%%%%%%%%%%%%%%%%%%%%%%%%%%%%%%%%%%

\section{Roots and Weights}\label{roots and weights}

Let $E$ be a real Euclidean space endowed with a positive definite symmetric bilinear form $(\cdot, \cdot)$. 
For a non-zero $\alpha$ in $E$, the subspace $P_\alpha = \{ \lambda \in E \mid (\alpha, \lambda) = 0 \}$ is a hyperplane in $E$, orthogonal to $\alpha$. 
A \textbf{reflection} $s_\alpha$ is a linear map on $E$ to itself which sends $\alpha$ to $-\alpha$ and fixes all the elements of $P_\alpha$. 
Note that 
\[
s_\alpha (\lambda) = \lambda - \displaystyle\frac{2(\lambda, \alpha)}{(\alpha, \alpha)}\alpha, \text{ for all } \lambda \in E.
\]
It is clear that $s_\alpha$ is an element of order $2$ in the group $O(E)$ of all orthogonal transformations of $E$. 
If $t \in O(E)$, then we can easily check that $t s_\alpha t^{-1} = s_{t(\alpha)}$.

%%%%%%%%%%%%%%%%%%%%%%%%%%%%%%%%%%%%%%%%%%%%%%%%%%%%%%%%%%%%%%
%%%%%%%%%%%%%%%%%%%%%%%%%%%%%%%%%%%%%%%%%%%%%%%%%%%%%%%%%%%%%%

\subsection{Root Systems and Weyl Groups}\label{subsec:Root Systems and Weyl Groups}

\begin{defn}\label{def:rootsystem}
    \normalfont
    A nonempty finite set $\Phi \subset E$\label{nomencl:Phi} is called a \textbf{(reduced) root system} \index{root system} if it satisfies the following axioms:
    \begin{enumerate}[(a)]
        \item $0 \notin \Phi$ and span$(\Phi) = E$.
        \item $\Phi \cap \mathbb{R} \alpha = \{ \alpha, -\alpha \}$, for all $\alpha \in \Phi$.
        \item $s_\alpha (\Phi) \subset \Phi$, for all $\alpha \in \Phi$.
        \item For all $\alpha, \beta \in \Phi,$ $$\langle \beta, \alpha \rangle := \displaystyle\frac{2(\beta, \alpha)}{(\alpha, \alpha)} \in \mathbb{Z}.$$ 
    \end{enumerate}
\end{defn}

A \textbf{non-reduced root system} is a subset $\Phi \subset E$ that satisfies only axioms $(a), (c), (d)$. 
A \textbf{non-crystallographic root system} is a subset $\Phi \subset E$ that satisfies only axioms $(a), (b), (c)$. 
In the context of the latter case, we sometimes call our (reduced) root system to be a \textbf{crystallographic root system}. 
The elements of $\Phi$ are called \textbf{roots}. The \textbf{rank} of the root system $\Phi$ is $\ell =$ dim $(E)$. 

The subgroup $W$\label{nomencl:W} of $O(E)$, generated by the reflections $s_\alpha$ for $\alpha \in \Phi$, is called the \textbf{Weyl group} \index{Weyl group} of $\Phi$. 
By axiom (c), $W$ permutes the set $\Phi$, which is finite and spans $E$. 
This allows us to identify $W$ with a subgroup of the symmetric group on $\Phi$; in particular, $W$ is finite.

In light of axiom (a), it is often convenient to use the notation $(\Phi, E)$ rather than simply $\Phi$. 
Two pairs $(\Phi, E)$ and $(\Phi', E')$ are said to be \textbf{isomorphic} if there exists a vector space isomorphism $\phi: E \to E'$ that maps $\Phi$ onto $\Phi'$ and satisfies $\langle \phi(\beta), \phi(\alpha) \rangle = \langle \beta, \alpha \rangle$ for all roots $\alpha, \beta \in \Phi$. 
If two root systems are isomorphic (via the isomorphism $\phi$), then the corresponding Weyl groups are also isomorphic under the natural map $s \mapsto \phi \circ s \circ \phi^{-1}$.
An isomorphism from $(\Phi, E)$ to itself is called an \textbf{automorphism}, and the set of all automorphisms of $(\Phi, E)$ is denoted by $\text{Aut}(\Phi)$.

If $t \in \text{GL}(E)$ leaves $\Phi$ invariant, then $t s_{\alpha} t^{-1} = s_{t(\alpha)}$ for all $\alpha \in \Phi$, and $\langle \beta, \alpha \rangle = \langle t(\beta), t(\alpha) \rangle$ for all $\alpha, \beta \in \Phi$. 
In particular, this holds for any element $t \in \text{Aut}(\Phi)$. 
As a result, we can regard $W$ as a normal subgroup of $\text{Aut}(\Phi)$.

Let $\alpha, \beta \in \Phi$, and let $\theta$ denote the angle between them. 
Then $(\alpha, \beta) = \lVert \alpha \rVert \lVert \beta \rVert \cos(\theta)$. 
Therefore, $\langle \alpha, \beta \rangle \langle \beta, \alpha \rangle = 4 \cos^2(\theta)$. 
This leads to the following possibilities for the values of $\langle \alpha, \beta \rangle$ and $\langle \beta, \alpha \rangle$, under the assumption that $\alpha \neq \pm \beta$ and $\lVert \beta \rVert \geq \lVert \alpha \rVert$:

\begin{center}
    \begin{tabular}{cccc}
        $\langle \alpha, \beta \rangle$ \hspace{2mm} & \hspace{2mm} $\langle \beta, \alpha \rangle$ \hspace{2mm} & \hspace{5mm} $\theta$ \hspace{5mm} & \hspace{2mm} $\lVert \beta \rVert^2 / \lVert \alpha \rVert^2$ \\
        \hline
        $0$ & $0$ & $\pi/2$ & undermined \\
        $1$ & $1$ & $\pi/3$ & $1$ \\
        $-1$ & $-1$ & $2\pi/3$ & $1$ \\
        $1$ & $2$ & $\pi/4$ & $2$ \\
        $-1$ & $-2$ & $3\pi/4$ & $2$ \\
        $1$ & $3$ & $\pi/6$ & $3$ \\
        $-1$ & $-3$ & $5\pi/6$ & $3$ 
    \end{tabular}
\end{center}

Let $\alpha, \beta \in \Phi$ with $\beta \neq \pm \alpha$. 
If $(\alpha, \beta) > 0$ (respectively, $(\alpha, \beta) < 0$), then $\alpha - \beta$ (respectively, $\alpha + \beta$) is a root. 
Let $r$ and $q$ be the largest integers such that $\beta - r \alpha$ and $\beta + q \alpha$ are roots, respectively. 
Then $\beta + i\alpha \in \Phi$ is a root for all $-r \leq i \leq q$. 
In other words, the roots $\beta + i \alpha$ form a string, known as the \textbf{$\alpha$-string through $\beta$}, given by $\beta - r \alpha, \dots, \beta, \dots, \beta + q \alpha$. 
Observe that at most two root lengths occur in this string. 
Moreover, we have $r - q = \langle \beta, \alpha \rangle$, and $r + q \leq 3$, implying that root strings have a length at most 4. 
If $q \geq 1$, then $$ \frac{r+1}{q} = \frac{\lVert \alpha + \beta \rVert^2}{\lVert \beta \rVert^2}.$$

%%%%%%%%%%%%%%%%%%%%%%%%%%%%%%%%%%%%%%%%%%%%%%%%%%%%%%%%%%%%%%

\subsubsection*{Simple and Positive Systems}

A subset $\Phi^+$ of a root system $\Phi$ is called a \textbf{positive system} \index{root system!positive} if it consists of the roots that are positive relative to some total ordering of $E$. 
Set $\Phi^- = -\Phi^+$, then $\Phi = \Phi^+ \cup \Phi^-$. 
A subset $\Delta$ of root system $\Phi$ is called a \textbf{simple system}\index{root system!simple}\label{nomencl:Delta} if: 
\begin{enumerate}[(a)]
    \item $\Delta$ is a basis of $E$, and
    \item each root $\alpha$ in $\Phi$ is an integral linear combination of the elements of $\Delta$ with coefficients all of the same sign (all nonnegative or all nonpositive).
\end{enumerate}
Note that each simple system is contained in a unique positive system, and each positive system contains a unique simple system. 
In particular, simple systems exist. 

Let $\Delta$ be a simple system and $\Phi^+$ be the positive system containing $\Delta$. 
Then $(\alpha, \beta) \leq 0$ for $\alpha \neq \beta$ in $\Delta$, and $\alpha - \beta$ is not a root. 
If $\alpha \in \Delta$ is a simple root, then $s_\alpha$ permutes the positive roots other than $\alpha$. 
For any positive root $\alpha \in \Phi^+$, there exists a simple root $\beta \in \Delta$ such that $(\alpha, \beta) > 0$.
If $\alpha$ is positive but not simple, then $\alpha - \beta$ is a root (necessarily positive) for some $\beta \in \Delta$. 
Consequently, each $\beta \in \Phi^+$ can be written in the from $\alpha_1 + \dots + \alpha_k \ (\alpha_i \in \Delta, \text{ not necessarily distinct})$ in such a way that each partial sum $\alpha_1 + \dots + \alpha_i$ is a root.

Any two simple (respectively, positive) systems in $\Phi$ are conjugate under $W$. In other words, $W$ acts transitively on simple (respectively, positive) systems. Furthermore, each root belongs to at least one base. Therefore, if $\Delta$ is a simple system and $\alpha$ is any root, there exists $w \in W$ such that $w(\alpha) \in \Delta$.

%%%%%%%%%%%%%%%%%%%%%%%%%%%%%%%%%%%%%%%%%%%%%%%%%%%%%%%%%%%%%%

\subsubsection*{Structure of the Weyl Groups}

Fix a simple system $\Delta$ and the corresponding positive system $\Phi^+$.  
The Weyl group $W$ is generated by the simple reflections $s_\alpha$ for $\alpha \in \Delta$. 
Let $w = s_1 \cdots s_r$ (where $s_i = s_{\alpha_i}$ for some $\alpha_i \in \Delta$) be an expression of $w \in W$ as a product of simple reflections  (allowing for possible repetitions).
Define the \textbf{length} $\ell(w)$ of $w$ (relative to $\Delta$) to be the smallest $r$ for which such an expression exists; we call this expression \textbf{reduced}. 
By convention, we set $\ell(1) = 0$. 
Notice that $\ell(w) = 1$ if and only if $w = s_\alpha$ for some $\alpha \in \Delta$. 
Furthermore, it is clear that $\ell(w) = \ell(w^{-1})$.

Let $w = s_1 \cdots s_t$ be an expression of $w \in W$ as a product of simple reflections. 
If $w(\alpha_t) \in \Phi^+$, then there exists an index $1 \leq r < t$ such that $w$ can be rewritten as $$w = s_1 \cdots s_{r-1} s_{r+1} \cdots s_{t-1}.$$ 
In particular, if the expression $w = s_1 \cdots s_t$ is reduced, it must be the case that $w(\alpha_t) \in \Phi^-$. 
This leads to the \textbf{Deletion Condition}: given an expression $w = s_1 \cdots s_t$ that is not reduced, there must exist indices $1 \leq i < j \leq t$ such that $w = s_1 \cdots \hat{s_i} \cdots \hat{s_j} \cdots s_t$ (where the hat denotes omission).

Define $n(w):= \text{Card}(\Phi^+ \cap w^{-1}(\Phi^-))$, which is the number of positive roots that are sent to negative roots by $w$. 
Then we have the equality $n(w) = \ell(w)$.
The following conditions are equivalent for an element $w \in W$:
\[
    w(\Phi^+) = \Phi^+ \iff w(\Delta) = \Delta \iff \ell(w) = n(w) = 0 \iff w = 1.
\]
As a result, the action of $W$ on the set of simple (respectively, positive) systems is simply transitive. 

We now provide a precise description of the set of all positive roots that are mapped to negative roots by $w$. 
Let $\Phi^+(w) := \Phi^+ \cap w^{-1} (\Phi^{-})$. 
Suppose $w = s_1 \cdots s_t$ is the reduced expression of $w$. 
Define 
\[
    \beta_i := s_t s_{t-1} \cdots s_{i+1} (\alpha_i), \quad \text{with} \ \beta_t := \alpha_t.
\]
Then we have $\Phi^+(w) = \{ \beta_1, \dots, \beta_t \}$.

Observe that if $\Phi^+$ is a positive system, then $\Phi^{-}$ is also a positive system (by definition). 
Thus, there exists a unique element $w_0 \in W$ that maps $\Phi^+$ to $\Phi^{-}$ (uniqueness follows from the simply transitive action of $W$). 
This element is called the \textbf{longest} element in the Weyl group $W$ (with respect to the simple system $\Delta$). 
Moreover, $\ell(w_0) = n(w_0) = \text{Card}(\Phi^{+})$ is the largest possible length, and no other element of $W$ has a greater length (hence the name “longest element”). In particular, we have $w_0^{-1} = w_0$.
The element $w_0$ is the longest element if and only if $\ell(ws_\alpha) < \ell(w)$ for all $\alpha \in \Delta$. 
Consequently, for any $w \in W$, there exists $w' \in W$ such that $w_0 = ww'$ with $\ell(w_0) = \ell(w) + \ell(w')$. 
This can also be reformulated as 
$$
\ell(w_0w) = \ell(w_0) - \ell(w) \quad \text{for all } w \in W.
$$ 
Note that in any reduced expression for $w_0$, every simple reflection must appear at least once.

Finally, we present the Weyl group $W$ in terms of generators and relations. 
Let $m(\alpha, \beta)$ denote the order of the product $s_\alpha s_\beta$ in $W$ for any simple roots $\alpha, \beta \in \Delta$. 
Then the Weyl group $W$ is generated by the set 
\[
    S := \{ s_\alpha \mid \alpha \in \Delta \},
\]
subject to the following relations:
\[
    (s_\alpha s_\beta)^{m(\alpha, \beta)} = 1 \quad \text{for all} \ \alpha, \beta \in \Delta.
\]

%%%%%%%%%%%%%%%%%%%%%%%%%%%%%%%%%%%%%%%%%%%%%%%%%%%%%%%%%%%%%%

\subsubsection*{The Weyl Chambers}

The connected components of $E \setminus \displaystyle\cup_{\alpha \in \Phi} P_\alpha$ are called the \textbf{Weyl chambers} \index{Weyl chambers} of $E$. 
For a simple system $\Delta$, the \textbf{fundamental Weyl chamber relative to $\Delta$} is defined as
$$
\mathfrak{C}(\Delta) = \{ \lambda \in E \mid (\lambda, \alpha) > 0 \ \text{for all} \ \alpha \in \Delta \}.
$$
There is a natural one-to-one correspondence between Weyl chambers and simple systems. 
The Weyl group $W$ acts simply transitively on the Weyl chambers (and thus on the simple systems). 
The closure $\overline{\mathfrak{C}(\Delta)}$ of the fundamental Weyl chamber $\mathfrak{C}(\Delta)$ relative to $\Delta$ is a \textbf{fundamental domain} for the action of $W$ on $E$, that is, each vector in $E$ is $W$-conjugate to precisely one point in this set.

%%%%%%%%%%%%%%%%%%%%%%%%%%%%%%%%%%%%%%%%%%%%%%%%%%%%%%%%%%%%%%

\subsubsection*{Irreducible Root Systems}

A root system $\Phi$ is called \textbf{irreducible} \index{root system!irreducible} if it cannot be partitioned into the union of two proper subsets such that each root in one set is orthogonal to each root in the other. Observe that $\Phi$ is irreducible if and only if $\Delta$ cannot be partitioned in this manner. 

Any root system $\Phi$ can be uniquely decomposed as the disjoint union of irreducible root systems $\Phi_1, \dots, \Phi_n$. 
If $\Delta_i$ is a simple system for $\Phi_i$ and $W_i$ is the Weyl group of $\Phi_i$, then
$$
\Delta = \Delta_1 \cup \dots \cup \Delta_n \quad \text{and} \quad W = W_1 \times \dots \times W_n
$$
are a simple system and the Weyl group of $\Phi$, respectively.

Let $\Phi$ be an irreducible root system. The Weyl group $W$ acts irreducibly on $E$. In particular, the $W$-orbit of any root $\alpha$ spans $E$. There are at most two distinct root lengths that can occur in $\Phi$. When we have roots of two distinct lengths, we refer to them as \textbf{long} roots or \textbf{short} roots, and the ratio of their squared lengths can only be $2$ or $3$. If there is only one root length, all roots are referred to as long. Additionally, all roots of the same length are conjugate under the action of $W$. 

Fix a simple system $\Delta$ in a root system $\Phi$ (not necessarily irreducible). The \textbf{height} of a root $\beta = \sum_{\alpha \in \Delta} k_\alpha \alpha$ (relative to $\Delta$) is defined as $ht(\beta) = \sum_{\alpha \in \Delta} k_\alpha$.
Clearly, $ht(\beta) > 0$ if and only if $\beta \in \Phi^+$.
Next, we define a partial order `$\leq$’ on $E$ (relative to $\Delta$) as follows: for vectors $\mu, \lambda \in E$, we write $\mu \leq \lambda$ if $\lambda - \mu$ is a sum of positive roots (equivalently, a sum of simple roots), or if $\mu = \lambda$. Furthermore, for two roots $\alpha, \beta \in \Phi$, if $\alpha \leq \beta$, then $ht(\alpha) \leq ht(\beta)$. 

Now, let $\Phi$ be an irreducible root system. There exists a unique \textbf{maximal root} (or \textbf{highest root}) \index{highest root} $\beta$ with respect to the partial order `$\leq$'. When $\beta$ is expressed as $\beta = \sum_{\alpha \in \Delta} k_\alpha \alpha$, all coefficients $k_\alpha$ are strictly positive.
Furthermore, for any root $\gamma = \sum_{\alpha \in \Delta} p_\alpha \alpha$ in $\Phi$ with $\gamma \neq \beta$, the following conditions hold:
$p_\alpha \leq k_\alpha$ for all $\alpha \in \Delta$, $\text{ht}(\gamma) < \text{ht}(\beta)$, and $(\beta, \gamma) \geq 0$ (assuming $\gamma > 0$).
As a consequence, the maximal root $\beta$ lies in the fundamental Weyl chamber $\mathfrak{C}(\Delta)$. 
Observe that the maximal root $\beta$ is always a long root. 

We now present a key property of an irreducible root system $\Phi$, which will be utilized in Chapter~\ref{chapter:Normal_subgroups}. For any pair of roots $\alpha$ and $\beta$ in $\Phi$, there exists a sequence of roots $\gamma_1, \dots, \gamma_n$ such that $\gamma_1 = \alpha$, $\gamma_n = \beta$, and each consecutive pair $(\gamma_i, \gamma_{i+1})$ (for $i = 1, \dots, n-1$) forms a rank-2 subsystem of type $A_2$, $B_2$, or $G_2$ (see the diagram below).

%%%%%%%%%%%%%%%%%%%%%%%%%%%%%%%%%%%%%%%%%%%%%%%%%%%%%%%%%%%%%%

\subsubsection*{Root Systems of Rank 2}

%Before concluding this subsection, we illustrate all possible rank 2 root systems:

We now illustrate all possible rank 2 root systems:

\begin{center}
\begin{tikzpicture}[scale=1.8]
% A_1 x A_1
        \draw[->] (0,0) -- (1,0) node[right] {$\alpha$};
        \draw[->] (0,0) -- (0,1) node[above] {$\beta$};
        \draw[->] (0,0) -- (-1,0) node[left] {$-\alpha$};
        \draw[->] (0,0) -- (0,-1) node[below] {$-\beta$};
        \node at (0,-1.5) {$A_1 \times A_1$};
% A_2
    \begin{scope}[shift={(5,0)}]
        \draw[->] (0,0) -- (1,0) node[right] {$\alpha$};
        \draw[->] (0,0) -- (0.5,0.866) node[above right] {$\alpha + \beta$};
        \draw[->] (0,0) -- (-0.5,0.866) node[above left] {$\beta$};
        \draw[->] (0,0) -- (-1,0) node[left] {$-\alpha$};
        \draw[->] (0,0) -- (-0.5,-0.866) node[below left] {$-\alpha - \beta$};
        \draw[->] (0,0) -- (0.5,-0.866) node[below right] {$-\beta$};
        \node at (0,-1.5) {$A_2$};
    \end{scope}
% B_2
    \begin{scope}[shift={(0,-3.5)}]
        \draw[->] (0,0) -- (1,0) node[right] {$\alpha$};
        \draw[->] (0,0) -- (1,1) node[above right] {$2 \alpha + \beta$};
        \draw[->] (0,0) -- (0,1) node[above] {$\alpha + \beta$};
        \draw[->] (0,0) -- (-1,1) node[above left] {$\beta$};
        \draw[->] (0,0) -- (-1,0) node[left] {$-\alpha$};
        \draw[->] (0,0) -- (-1,-1) node[below left] {$-2 \alpha - \beta$};
        \draw[->] (0,0) -- (0,-1) node[below] {$-\alpha - \beta$};
        \draw[->] (0,0) -- (1,-1) node[below right] {$-\beta$};
        \node at (0,-1.5) {$B_2$};
    \end{scope}
% G_2
    \begin{scope}[shift={(5,-3.5)}]
        \draw[->] (0,0) -- (0.58,0) node[right] {$\alpha$};
        \draw[->] (0,0) -- (0.87,0.5) node[above right] {$3 \alpha + \beta$};
        \draw[->] (0,0) -- (0.29,0.5) node[above] {\tiny $2 \alpha + \beta$};
        \draw[->] (0,0) -- (0,1) node[above] {$3 \alpha + 2 \beta$};
        \draw[->] (0,0) -- (-0.29,0.5) node[above] {\tiny $\alpha + \beta$};
        \draw[->] (0,0) -- (-0.87,0.5) node[above left] {$\beta$};
        \draw[->] (0,0) -- (-0.58,0) node[left] {$-\alpha$};
        \draw[->] (0,0) -- (-0.87,-0.5) node[below left] {$-3 \alpha - \beta$};
        \draw[->] (0,0) -- (-0.29,-0.5) node[below] {\tiny $-2 \alpha - \beta$};
        \draw[->] (0,0) -- (0,-1) node[below] {$-3 \alpha - 2 \beta$};
        \draw[->] (0,0) -- (0.29,-0.5) node[below] {\tiny $-\alpha - \beta$};
        \draw[->] (0,0) -- (0.87,-0.5) node[below right] {$-\beta$};
        \node at (0,-1.5) {$G_2$};
    \end{scope}
\end{tikzpicture}
\end{center}

%%%%%%%%%%%%%%%%%%%%%%%%%%%%%%%%%%%%%%%%%%%%%%%%%%%%%%%%%%%%%%

\subsubsection*{Closed Root Subsystems}

A subset $\Phi'$ of $\Phi$ is called a \textbf{root subsystem} \index{root system!closed subsystem} if $\Phi'$ itself is a root system. A root subsystem $\Phi'$ of $\Phi$ is called \textbf{closed} if $\alpha, \beta \in \Phi'$ and $\alpha + \beta \in \Phi$ imply $\alpha + \beta \in \Phi'$.

For example, the long roots $\{ \pm \beta, \pm (2\alpha + \beta) \}$ form a closed root subsystem of $B_2$. However, the short roots $\{ \pm \alpha, \pm (\alpha + \beta) \}$ form a root subsystem of $B_2$ that is not closed.

Let $S$ be a subset of $\Phi$. The smallest closed root subsystem of $\Phi$ containing $S$ is called the \textbf{(closed) root subsystem of $\Phi$ generated by $S$}.

%%%%%%%%%%%%%%%%%%%%%%%%%%%%%%%%%%%%%%%%%%%%%%%%%%%%%%%%%%%%%%

\subsubsection*{Closed Subsets}\label{subsubsec:closed subsets}

A subset $S$ of the root system $\Phi$ is called \textbf{closed} \index{root system!closed subset} if for any $\alpha, \beta \in S$ such that $\alpha + \beta \in \Phi$, one has $\alpha + \beta \in S$. Examples of closed subsets of $\Phi$ include: 

\begin{enumerate}[(1)]
    \item[(0)] $\emptyset$ and $\Phi$, 
    \item $\Phi^+$, the set of all positive roots,
    \item $\Phi^+ \setminus \{\alpha\}$, where $\alpha$ is a simple root,
    \item $\{\alpha, -\alpha\}$, where $\alpha$ is any root,
    \item $\Phi_r = \{\alpha \in \Phi \mid \mathrm{ht}(\alpha) \geq r\}$, where $r \in \mathbb{Z}^+$.
\end{enumerate}

A closed subset $S$ of $\Phi$ is called 
\begin{enumerate}[(a)]
    \item \textbf{symmetric} (or \textbf{reductive}) if $\alpha \in S$ implies $-\alpha \in S$, 
    \item \textbf{special} (or \textbf{unipotent}) if $\alpha \in S$ implies $-\alpha \notin S$, and 
    \item \textbf{parabolic} if $\Phi = S \cup -S$. 
\end{enumerate}

Any closed set $S$ can be uniquely decomposed into the disjoint union of its symmetric part $S^r = \{\alpha \in S \mid -\alpha \in S\}$ and its special part $S^u = \{\alpha \in S \mid -\alpha \notin S\}$.
Both $S^r$ and $S^u$ are also closed subsets of $\Phi$.

A subset $I$ of a closed set $S$ is called an \textbf{ideal} if $\alpha \in I$, $\beta \in S$, and $\alpha + \beta \in S$ imply $\alpha + \beta \in I$. Examples (1), (2), and (4) above are ideals in $\Phi^+$. Note that $S^u$ is always an ideal in $S$.

Closed subsets $S_1$ and $S_2$ of a root system $\Phi$ are conjugate if there exists an element $w \in W$, the Weyl group, such that $w (S_1) = S_2$.
A special closed set is conjugate to a subset of $\Phi^+$.

%%%%%%%%%%%%%%%%%%%%%%%%%%%%%%%%%%%%%%%%%%%%%%%%%%%%%%%%%%%%%%
%%%%%%%%%%%%%%%%%%%%%%%%%%%%%%%%%%%%%%%%%%%%%%%%%%%%%%%%%%%%%%

\subsection{Classification of Irreducible Root Systems}

Let $\Phi$ be a root system of rank $\ell$ in a Euclidean space $E$. Let $W$ be the Weyl group of $\Phi$. Let $\Delta = \{ \alpha_1, \dots, \alpha_\ell \}$ be a simple system for $\Phi$. 

%%%%%%%%%%%%%%%%%%%%%%%%%%%%%%%%%%%%%%%%%%%%%%%%%%%%%%%%%%%%%%

\subsubsection*{Cartan Matrix}

Fix an ordering $(\alpha_1, \dots, \alpha_\ell)$ of the simple roots. The matrix defined by $(\langle \alpha_i, \alpha_j \rangle)$ is called the \textbf{Cartan matrix} \index{Cartan matrix} of $\Phi$, and its entries are referred to as \textbf{Cartan integers}. \index{Cartan integers} Notably, the diagonal entries are all equal to $2$, while the off-diagonal entries are non-positive integers. The Cartan matrix is nonsingular because the set of simple roots forms a basis for the ambient space.
For example, for the systems of rank 2, the matrices are
$$
A_1 \times A_1
\left(\begin{array}{ll}
    2 & 0 \\
    0 & 2
\end{array}\right); \hspace{2mm}
A_2 
\left(\begin{array}{rr}
    2 & -1 \\
    -1 & 2
    \end{array}\right); \hspace{2mm} 
B_2
\left(\begin{array}{rr}
    2 & -2 \\
    -1 & 2
\end{array}\right); \hspace{2mm} 
G_2
\left(\begin{array}{rr}
    2 & -1 \\
    -3 & 2
\end{array}\right).
$$

Observe that, the specific form of the Cartan matrix depends on the ordering of the simple roots, but this dependence is not significant since it is limited to a permutation of the index set. Moreover, since the Weyl group $W$ acts transitively on the set of bases, the Cartan matrix of $\Phi$ does not depend on the choice of the basis $\Delta$.
On the other hand, the Cartan matrix completely determines the root system $\Phi$ up to isomorphism. 

If $\Phi$ is a reducible root system, there exists an ordering of the simple roots such that the Cartan matrix becomes block diagonal, with each block corresponding to an irreducible component in the decomposition of $\Phi$.

%%%%%%%%%%%%%%%%%%%%%%%%%%%%%%%%%%%%%%%%%%%%%%%%%%%%%%%%%%%%%%

\subsubsection*{Dynkin Diagrams}

Recall that if $\alpha$ and $\beta$ are non-proportional roots, then $\langle \alpha, \beta \rangle \langle \beta, \alpha \rangle$ can take the values $0, 1, 2,$ or $3$. Furthermore, the value $\langle \alpha, \beta \rangle \langle \beta, \alpha \rangle$ equals $2$ or $3$ only if $\alpha$ and $\beta$ have different lengths.

The \textbf{Dynkin diagram} \index{Dynkin diagram} of a root system $\Phi$ is a graph with $\ell$ vertices, each corresponding to a simple root, and the $i$-th vertex is joined to the $j$-th vertex $(i \neq j)$ by $\langle \alpha_i, \alpha_j \rangle \langle \alpha_j, \alpha_i \rangle$ edges. If multiple edges (double or triple) occur, an arrow is added, pointing toward the vertex representing the shorter root. 

For example, for the systems of rank 2, the Dynkin diagrams are
\begin{center}
\begin{tikzpicture}[scale=1]
% A_1 x A_1
     \begin{scope}[shift={(0,0)}]
        \node at (-1.5,-0.15){$A_1 \times A_1:$};
        \node[draw,circle,inner sep=0.6mm] at (0,0) {};
        \node at (0,-0.4) {$\alpha$};
        \node[draw,circle,inner sep=0.6mm] at (1.5,0) {};
        \node at (1.5,-0.4) {$\beta$};
    \end{scope}
% A_2
    \begin{scope}[shift={(5,0)}]
        \node at (-1,-0.15){$A_2:$};
        \node[draw,circle,inner sep=0.6mm] at (0,0) {};
        \node at (0,-0.4) {$\alpha$};
        \draw (0.08,0)--(1.42,0);
        \node[draw,circle,inner sep=0.6mm] at (1.5,0) {};
        \node at (1.5,-0.4) {$\beta$};
    \end{scope}
% B_2
    \begin{scope}[shift={(0,-1.5)}]
        \node at (-1,-0.15){$B_2:$};
        \node[draw,circle,inner sep=0.6mm] at (0,0) {};
        \node at (0,-0.4) {$\alpha$};
        \draw (0.08,0.05)--(1.42,0.05);
        \node at (0.75,0) {$<$};
        \draw (0.08,-0.05)--(1.42,-0.05);
        \node[draw,circle,inner sep=0.6mm] at (1.5,0) {};
        \node at (1.5,-0.4) {$\beta$};
    \end{scope}
% G_2
    \begin{scope}[shift={(5,-1.5)}]
        \node at (-1,-0.15){$G_2:$};
        \node[draw,circle,inner sep=0.6mm] at (0,0) {};
        \node at (0,-0.4) {$\alpha$};
        \draw (0.08,0.06)--(1.42,0.06);
        \draw (0.08,0)--(1.42,0);
        \node at (0.75,0) {$<$};
        \draw (0.08,-0.06)--(1.42,-0.06);
        \node[draw,circle,inner sep=0.6mm] at (1.5,0) {};
        \node at (1.5,-0.4) {$\beta$};
    \end{scope}
\end{tikzpicture}
\end{center}

Note that $\Phi$ is irreducible if and only if its Dynkin diagram is connected (in the usual sense). Additionally, the Dynkin diagram completely determines the Cartan matrix (after fixing an ordering of the simple roots) and the root system (up to isomorphism), and vice versa.

%%%%%%%%%%%%%%%%%%%%%%%%%%%%%%%%%%%%%%%%%%%%%%%%%%%%%%%%%%%%%%

\subsubsection*{Automorphism of $\Phi$}

Recall that, $W$ is a normal subgroup of Aut $\Phi$ (cf. Section~\ref{subsec:Root Systems and Weyl Groups}). Let $\Gamma = \{\sigma \in \text{Aut} \, \Phi \mid \sigma(\Delta) = \Delta\}$, where $\Delta$ is a fixed simple system of $\Phi$. Clearly, $\Gamma$ is a subgroup of Aut $\Phi$. If $\sigma \in \Gamma \cap W$, then $\sigma = 1$ due to the simple transitivity of $W$. For any $\sigma \in \text{Aut} \, \Phi$, since $\sigma(\Delta)$ is another simple system, there exists $w \in W$ such that $w \sigma(\Delta) = \Delta$, implying $\sigma \in \Gamma {W}$. Hence, Aut $\Phi$ is the semidirect product of $\Gamma$ and $W$.

For all $\sigma \in \text{Aut} \, \Phi$ and $\alpha, \beta \in \Phi$, we have $\langle \alpha, \beta \rangle = \langle \sigma(\alpha), \sigma(\beta) \rangle$. Therefore, each $\sigma \in \Gamma$ induces an automorphism of the Dynkin diagram of $\Phi$. If $\sigma$ acts trivially on the diagram, then $\sigma = 1$ (since $\Delta$ spans $E$). Conversely, every automorphism of the Dynkin diagram corresponds to an automorphism of $\Phi$. Therefore, $\Gamma$ can be identified with the group of \textbf{diagram automorphisms}. A complete description of these automorphisms for an irreducible root system is given in Table~\ref{tab:classification}.

%%%%%%%%%%%%%%%%%%%%%%%%%%%%%%%%%%%%%%%%%%%%%%%%%%%%%%%%%%%%%%

\subsubsection*{Classification}

If $\Phi$ is an irreducible root system of rank $\ell$, its Dynkin diagram must be one of the following type: $A_\ell \, (\ell \geq 1)$, $B_\ell \, (\ell \geq 2)$, $C_\ell \, (\ell \geq 3)$, $D_\ell \, (\ell \geq 4)$, $E_6$, $E_7$, $E_8$, $F_4$, or $G_2$ (see Table~\ref{tab:classification}). 
Conversely, for each Dynkin diagram of type $A-G$ (as described above), there exists a unique (up to isomorphism) irreducible root system corresponding to that diagram (see Table~\ref{tab:classification}).

The table below offers a brief summary of the irreducible root systems for different types:

{
\small
\begin{longtable}{|c|c|c|}
    \caption{Properties of irreducible root systems by type}
    \label{tab:classification} \\
    \hline
    \textbf{Type} & \textbf{Aspects} & \textbf{Details} \\
    \hline
    \endfirsthead
    
    \hline
    \textbf{Type} & \textbf{Aspects} & \textbf{Details} \\
    \hline
    \endhead
    
    \hline
    \endfoot

    \hline
    \endlastfoot

%Type A_l
    \multirow{8}{*}{$A_\ell \ (\ell \geq 1)$} 
    & Euclidean Space $E$ & \makecell{the $\ell$-dimensional subspace of $\mathbb{R}^{\ell+1}$ \\ orthogonal to $e_1 + e_2 + \cdots + e_{\ell+1}$} \\ 
    \cline{2-3}
    & Root System $\Phi$ & $\{\pm(e_i - e_j) \mid 1 \leq i < j \leq \ell+1 \}$, cardinality $\ell(\ell+1)$ \\ 
    \cline{2-3}
    & Simple System $\Delta$ & $\{ \alpha_i = e_i - e_{i+1} \mid 1 \leq i \leq \ell \}$ \\ 
    \cline{2-3}
    & Dynkin Diagram & {\begin{tikzpicture}[scale=1]
            \node[draw,circle,inner sep=0.6mm] at (0,0) {};
            \node at (0,-0.4) {$\alpha_1$};
            \draw (0.08,0)--(1.42,0);
            \node[draw,circle,inner sep=0.6mm] at (1.5,0) {};
            \node at (1.5,-0.4) {$\alpha_2$};
            \draw (1.58,0)--(2.6,0);
            \node at (3,0) {$\cdots$};
            \draw (3.4,0)--(4.42,0);
            \node[draw,circle,inner sep=0.6mm] at (4.5,0) {};
            \node at (4.5,-0.4) {$\alpha_{\ell-1}$};
            \draw (4.58,0)--(5.92,0);
            \node[draw,circle,inner sep=0.6mm] at (6,0) {};
            \node at (6,-0.4) {$\alpha_\ell$};
        \end{tikzpicture}} \\ 
    \cline{2-3}
    & Cartan Matrix & {$\begin{pmatrix} 
        2 & -1 & 0 & 0 & \cdots & 0 & 0 \\ 
        -1 & 2 & -1 & 0 & \cdots & 0 & 0 \\ 
        0 & -1 & 2 & -1 & \cdots & 0 & 0 \\ 
        0 & 0 & -1 & 2 & \cdots & 0 & 0 \\
        \vdots & \vdots & \vdots & \vdots & \ddots & \vdots & \vdots \\
        0 & 0 & 0 & 0 & \cdots & 2 & -1 \\
        0 & 0 & 0 & 0 & \cdots & -1 & 2 \\
    \end{pmatrix}$} \\ 
    \cline{2-3}
    & Weyl Group $W$ & the symmetric group $\mathscr{S}_{\ell+1}$ of order $(\ell+1)!$ \\ 
    \cline{2-3} 
    & $\Gamma$ & $\mathbb{Z}/2\mathbb{Z}$ \\
    \cline{2-3} 
    & Highest Root & $\alpha_1 + \alpha_2 + \cdots + \alpha_\ell$ \\ 
    \hline

%Type B_l
    \multirow{8}{*}{$B_\ell \ (\ell \geq 2)$} 
    & Euclidean Space $E$ & $\mathbb{R}^\ell$ \\ 
    \cline{2-3}
    & Root System $\Phi$ & \makecell{$\{ \pm (e_i \pm e_j) \mid 1 \leq i < j \leq \ell \} \cup \{\pm e_i \mid 1 \leq i \leq \ell \}$, \\ cardinality $2 \ell^2$} \\ 
    \cline{2-3}
    & Simple System $\Delta$ & $\{ \alpha_i = e_i - e_{i+1} \mid 1 \leq i \leq \ell-1 \} \cup \{ \alpha_\ell = e_\ell \}$ \\ 
    \cline{2-3}
    & Dynkin Diagram & {\begin{tikzpicture}[scale=1]
            \node[draw,circle,inner sep=0.6mm] at (0,0) {};
            \node at (0,-0.4) {$\alpha_1$};
            \draw (0.08,0)--(1.42,0);
            \node[draw,circle,inner sep=0.6mm] at (1.5,0) {};
            \node at (1.5,-0.4) {$\alpha_2$};
            \draw (1.58,0)--(2.6,0);
            \node at (3,0) {$\cdots$};
            \draw (3.4,0)--(4.42,0);
            \node[draw,circle,inner sep=0.6mm] at (4.5,0) {};
            \node at (4.5,-0.4) {$\alpha_{\ell-1}$};
            \draw (4.58,0.05)--(5.92,0.05);
            \node at (5.25,0) {$>$};
            \draw (4.58,-0.05)--(5.92,-0.05);
            \node[draw,circle,inner sep=0.6mm] at (6,0) {};
            \node at (6,-0.4) {$\alpha_\ell$};
        \end{tikzpicture}} \\ 
    \cline{2-3}
    & Cartan Matrix & {$\begin{pmatrix} 
        2 & -1 & 0 & 0 & \cdots & 0 & 0 \\ 
        -1 & 2 & -1 & 0 & \cdots & 0 & 0 \\ 
        0 & -1 & 2 & -1 & \cdots & 0 & 0 \\ 
        0 & 0 & -1 & 2 & \cdots & 0 & 0 \\
        \vdots & \vdots & \vdots & \vdots & \ddots & \vdots & \vdots \\
        0 & 0 & 0 & 0 & \cdots & 2 & -2 \\
        0 & 0 & 0 & 0 & \cdots & -1 & 2 \\
    \end{pmatrix}$} \\ 
    \cline{2-3}
    & Weyl Group $W$ & the hyperoctahedral group $(\mathbb{Z}/2\mathbb{Z})^\ell \rtimes \mathscr{S}_\ell$ of order $2^\ell \ell!$ \\ 
    \cline{2-3}
    & $\Gamma$ & 1 \\
    \cline{2-3}
    & Highest Root & \makecell{long: $\alpha_1 + 2\alpha_2 + \cdots + 2\alpha_\ell$ \\
    short: $\alpha_1 + \alpha_2 + \cdots + \alpha_\ell$} \\ 
    \hline
    
%Type C_l
    \multirow{8}{*}{$C_\ell \ (\ell \geq 3)$} 
    & Euclidean Space $E$ & $\mathbb{R}^\ell$ \\ 
    \cline{2-3}
    & Root System $\Phi$ & \makecell{$\{ \pm (e_i \pm e_j) \mid 1 \leq i < j \leq \ell \} \cup \{\pm 2 e_i \mid 1 \leq i \leq \ell \}$, \\ cardinality $2 \ell^2$} \\ 
    \cline{2-3}
    & Simple System $\Delta$ & $\{ \alpha_i = e_i - e_{i+1} \mid 1 \leq i \leq \ell-1 \} \cup \{ \alpha_\ell = 2 e_\ell \}$ \\ 
    \cline{2-3}
    & Dynkin Diagram & {\begin{tikzpicture}[scale=1]
            \node[draw,circle,inner sep=0.6mm] at (0,0) {};
            \node at (0,-0.4) {$\alpha_1$};
            \draw (0.08,0)--(1.42,0);
            \node[draw,circle,inner sep=0.6mm] at (1.5,0) {};
            \node at (1.5,-0.4) {$\alpha_2$};
            \draw (1.58,0)--(2.6,0);
            \node at (3,0) {$\cdots$};
            \draw (3.4,0)--(4.42,0);
            \node[draw,circle,inner sep=0.6mm] at (4.5,0) {};
            \node at (4.5,-0.4) {$\alpha_{\ell-1}$};
            \draw (4.58,0.05)--(5.92,0.05);
            \node at (5.25,0) {$<$};
            \draw (4.58,-0.05)--(5.92,-0.05);
            \node[draw,circle,inner sep=0.6mm] at (6,0) {};
            \node at (6,-0.4) {$\alpha_\ell$};
        \end{tikzpicture}} \\ 
    \cline{2-3}
    & Cartan Matrix & {$\begin{pmatrix} 
        2 & -1 & 0 & 0 & \cdots & 0 & 0 \\ 
        -1 & 2 & -1 & 0 & \cdots & 0 & 0 \\ 
        0 & -1 & 2 & -1 & \cdots & 0 & 0 \\ 
        0 & 0 & -1 & 2 & \cdots & 0 & 0 \\
        \vdots & \vdots & \vdots & \vdots & \ddots & \vdots & \vdots \\
        0 & 0 & 0 & 0 & \cdots & 2 & -1 \\
        0 & 0 & 0 & 0 & \cdots & -2 & 2 \\
    \end{pmatrix}$} \\ 
    \cline{2-3}
    & Weyl Group $W$ & the hyperoctahedral group $(\mathbb{Z}/2\mathbb{Z})^\ell \rtimes \mathscr{S}_\ell$ of order $2^\ell \ell!$ \\ 
    \cline{2-3}
    & $\Gamma$ & 1 \\
    \cline{2-3}
    & Highest Root & \makecell{long: $2 \alpha_1 + 2\alpha_2 + \cdots + 2\alpha_{\ell-1} + \alpha_\ell$ \\
    short: $\alpha_1 + 2\alpha_2 + \cdots + 2\alpha_{\ell-1} + \alpha_\ell$} \\ 
    \hline

%Type D_l
    \multirow{8}{*}{$D_\ell \ (\ell \geq 4)$} 
    & Euclidean Space $E$ & $\mathbb{R}^\ell$ \\ 
    \cline{2-3}
    & Root System $\Phi$ & $\{\pm (e_i \pm e_j) \mid 1 \leq i < j \leq \ell \}$, cardinality $2 (\ell^2 - \ell)$ \\ 
    \cline{2-3}
    & Simple System $\Delta$ & $\{ \alpha_i = e_i - e_{i+1} \mid 1 \leq i \leq \ell-1 \} \cup \{ \alpha_\ell = e_{\ell-1} + e_\ell \}$ \\ 
    \cline{2-3}
    & Dynkin Diagram & {\begin{tikzpicture}[scale=1]
            \node[draw,circle,inner sep=0.6mm] at (0,0) {};
            \node at (0,-0.4) {$\alpha_1$};
            \draw (0.08,0)--(1.42,0);
            \node[draw,circle,inner sep=0.6mm] at (1.5,0) {};
            \node at (1.5,-0.4) {$\alpha_2$};
            \draw (1.58,0)--(2.6,0);
            \node at (3,0) {$\cdots$};
            \draw (3.4,0)--(4.42,0);
            \node[draw,circle,inner sep=0.6mm] at (4.5,0) {};
            \node at (4.5,-0.4) {$\alpha_{\ell-2}$};
            \draw (4.58,0)--(5.575,0.72);
            \node[draw,circle,inner sep=0.6mm] at (5.64,0.75) {};
            \node at (6.2,0.75) {$\alpha_{\ell-1}$};
            \draw (4.58,0)--(5.575,-0.72);
            \node[draw,circle,inner sep=0.6mm] at (5.64,-0.75) {};
            \node at (6.1,-0.75) {$\alpha_\ell$};
        \end{tikzpicture}} \\ 
    \cline{2-3}
    & Cartan Matrix & {$\begin{pmatrix} 
    2 & -1 & \cdots & 0 & 0 & 0 & 0 \\ 
    -1 & 2 & \cdots & 0 & 0 & 0 & 0 \\ 
    \vdots & \vdots & \ddots & \vdots & \vdots & \vdots & \vdots \\ 
    0 & 0 & \cdots & 2 & -1 & 0 & 0 \\ 
    0 & 0 & \cdots & -1 & 2 & -1 & -1 \\
    0 & 0 & \cdots & 0 & -1 & 2 & 0 \\
    0 & 0 & \cdots & 0 & -1 & 0 & 2
    \end{pmatrix}$} \\ 
    \cline{2-3}
    & Weyl Group $W$ & the group $(\mathbb{Z}/2\mathbb{Z})^{\ell-1} \rtimes \mathscr{S}_\ell$ of order $2^{\ell-1} \ell!$ \\ 
    \cline{2-3}
    & $\Gamma$ & $\mathscr{S}_3$, if $\ell = 4$; and $\mathbb{Z}/2\mathbb{Z}$, if $\ell > 4$ \\
    \cline{2-3}
    & Highest Root & $\alpha_1 + 2\alpha_2 + \cdots + 2\alpha_{\ell-2} + \alpha_{\ell-1} + \alpha_{\ell}$ \\ 
    \hline

%Type E_6
    \multirow{8}{*}{$E_6$} 
    & Euclidean Space $E$ & \makecell{the $6$-dimensional subspace of $\mathbb{R}^8$ defined by \\ the solutions to the system $x_6 = x_7 = -x_8$} \\ 
    \cline{2-3}
    & Root System $\Phi$ & \makecell{$\{ \pm (e_i \pm e_j) \mid 1 \leq i < j \leq 5 \} \cup \{\pm \frac{1}{2} \Big(e_8 - e_7 - e_6 + $ \\ $ \sum_{i=1}^5 (-1)^{k(i)}e_i\Big) \mid k(i) = 0 \text{ or } 1 \text{ such that }$ \\ $\sum_{i=1}^5 k(i) \text{ is even} \}$, cardinality $72$}  \\ 
    \cline{2-3}
    & Simple System $\Delta$ & \makecell{$\alpha_1 = \frac{1}{2}(e_1 + e_8 - (e_2 + \cdots + e_7))$, \\ $\alpha_2 = e_1 + e_2$, $\alpha_3 = e_2 - e_1$, $\alpha_4 = e_3 - e_2$, \\ $\alpha_5 = e_4 - e_3$, $\alpha_6 = e_5 - e_4$ } \\ 
    \cline{2-3}
    & Dynkin Diagram & {\begin{tikzpicture}[scale=1]
        \node[draw,circle,inner sep=0.6mm] at (0,0) {};
        \node at (0,-0.4) {$\alpha_1$};
        \draw (0.08,0)--(1.42,0);
        \node[draw,circle,inner sep=0.6mm] at (1.5,0) {};
        \node at (1.5,-0.4) {$\alpha_3$};
        \draw (1.58,0)--(2.92,0);
        \node[draw,circle,inner sep=0.6mm] at (3,0) {};
        \node at (3,-0.4) {$\alpha_4$};
        \draw (3.08,0)--(4.42,0);
        \node[draw,circle,inner sep=0.6mm] at (4.5,0) {};
        \node at (4.5,-0.4) {$\alpha_5$};
        \draw (4.58,0)--(5.92,0);
        \node[draw,circle,inner sep=0.6mm] at (6,0) {};
        \node at (6,-0.4) {$\alpha_6$};
        \draw (3,0.08)--(3,1.02);
        \node[draw,circle,inner sep=0.6mm] at (3,1.1) {};
        \node at (3.4,1.1) {$\alpha_2$};
    \end{tikzpicture}} \\ 
    \cline{2-3}
    & Cartan Matrix & $\begin{pmatrix} 
        2 & 0 & -1 & 0 & 0 & 0 \\ 
        0 & 2 & 0 & -1 & 0 & 0 \\ 
        -1 & 0 & 2 & -1 & 0 & 0 \\ 
        0 & -1 & -1 & 2 & -1 & 0 \\ 
        0 & 0 & 0 & -1 & 2 & -1 \\ 
        0 & 0 & 0 & 0 & -1 & 2 
    \end{pmatrix}$ \\ 
    \cline{2-3}
    & Weyl Group $W$ & it is of order $51840 = 2^7 \times 3^4 \times 5$ \\ 
    \cline{2-3}
    & $\Gamma$ & $\mathbb{Z}/ 2 \mathbb{Z}$ \\
    \cline{2-3}
    & Highest Root & $\alpha_1 + 2\alpha_2 + 2\alpha_3 + 3\alpha_4 + 2\alpha_5 + \alpha_6$ \\ 
    \hline

%Type E_7
    \multirow{8}{*}{$E_7$} 
    & Euclidean Space $E$ & \makecell{the $7$-dimensional subspace of $\mathbb{R}^8$ \\ orthogonal to $e_7 + e_8$} \\ 
    \cline{2-3}
    & Root System $\Phi$ & \makecell{$\{\pm (e_i \pm e_j) \mid 1 \leq i < j \leq 6 \} \cup \{ \pm(e_7 - e_8) \} \cup $ \\ $ \{\pm \frac{1}{2} \Big(e_7 - e_8 + \sum_{i=1}^6 (-1)^{k(i)} e_i \Big) \mid k(i) = 0 \text{ or } 1$ \\ $\text{ such that } \sum_{i=1}^6 k(i) \text{ is odd} \}$, cardinality $126$} \\ 
    \cline{2-3}
    & Simple System $\Delta$ & \makecell{$\alpha_1 = \frac{1}{2}(e_1 + e_8 - (e_2 + \cdots + e_7))$, \\ $\alpha_2 = e_1 + e_2$, $\alpha_3 = e_2 - e_1$, $\alpha_4 = e_3 - e_2$, \\ $\alpha_5 = e_4 - e_3$, $\alpha_6 = e_5 - e_4$, $\alpha_7 = e_6 - e_5$} \\ 
    \cline{2-3}
    & Dynkin Diagram & {\begin{tikzpicture}[scale=1]
        \node[draw,circle,inner sep=0.6mm] at (0,0) {};
        \node at (0,-0.4) {$\alpha_1$};
        \draw (0.08,0)--(1.42,0);
        \node[draw,circle,inner sep=0.6mm] at (1.5,0) {};
        \node at (1.5,-0.4) {$\alpha_3$};
        \draw (1.58,0)--(2.92,0);
        \node[draw,circle,inner sep=0.6mm] at (3,0) {};
        \node at (3,-0.4) {$\alpha_4$};
        \draw (3.08,0)--(4.42,0);
        \node[draw,circle,inner sep=0.6mm] at (4.5,0) {};
        \node at (4.5,-0.4) {$\alpha_5$};
        \draw (4.58,0)--(5.92,0);
        \node[draw,circle,inner sep=0.6mm] at (6,0) {};
        \node at (6,-0.4) {$\alpha_6$};
        \draw (6.08,0)--(7.42,0);
        \node[draw,circle,inner sep=0.6mm] at (7.5,0) {};
        \node at (7.5,-0.4) {$\alpha_7$};
        \draw (3,0.08)--(3,1.02);
        \node[draw,circle,inner sep=0.6mm] at (3,1.1) {};
        \node at (3.4,1.1) {$\alpha_2$};
    \end{tikzpicture}} \\ 
    \cline{2-3}
    & Cartan Matrix & $\begin{pmatrix} 
        2 & 0 & -1 & 0 & 0 & 0 & 0 \\ 
        0 & 2 & 0 & -1 & 0 & 0 & 0 \\ 
        -1 & 0 & 2 & -1 & 0 & 0 & 0 \\ 
        0 & -1 & -1 & 2 & -1 & 0 & 0 \\ 
        0 & 0 & 0 & -1 & 2 & -1 & 0 \\ 
        0 & 0 & 0 & 0 & -1 & 2 & -1 \\
        0 & 0 & 0 & 0 & 0 & -1 & 2
    \end{pmatrix}$ \\ 
    \cline{2-3}
    & Weyl Group $W$ & it is of order $2903040 = 2^{10} \times 3^4 \times 5 \times 7$ \\ 
    \cline{2-3}
    & $\Gamma$ & $1$ \\
    \cline{2-3}
    & Highest Root & $2\alpha_1 + 2\alpha_2 + 3\alpha_3 + 4\alpha_4 + 3\alpha_5 + 2\alpha_6 + \alpha_7$ \\ 
    \hline

%Type E_8
    \multirow{8}{*}{$E_8$} 
    & Euclidean Space $E$ & $\mathbb{R}^8$ \\ 
    \cline{2-3}
    & Root System $\Phi$ & \makecell{ $\{ \pm (e_i \pm e_j) \mid 1 \leq i < j \leq 8 \} \cup $ \\
    $\{ \frac{1}{2} \Big( \sum_{i=1}^8 (-1)^{k(i)} e_i \Big) \mid k(i) = 0 \text{ or } 1 \text{ such that }$ \\ 
    $ \sum_{i=1}^8 k(i) \text{ is even} \}$, cardinality $240$}  \\ 
    \cline{2-3}
    & Simple System $\Delta$ & \makecell{$\alpha_1 = \frac{1}{2}(e_1 + e_8 - (e_2 + \cdots + e_7))$, $\alpha_2 = e_1 + e_2$, \\ $\alpha_3 = e_2 - e_1$, $\alpha_4 = e_3 - e_2$, $\alpha_5 = e_4 - e_3$, \\ $\alpha_6 = e_5 - e_4$, $\alpha_7 = e_6 - e_5$, $\alpha_8 = e_7 - e_6$} \\ 
    \cline{2-3}
    & Dynkin Diagram & {\begin{tikzpicture}[scale=0.8]
        \node[draw,circle,inner sep=0.6mm] at (0,0) {};
        \node at (0,-0.4) {$\alpha_1$};
        \draw (0.08,0)--(1.42,0);
        \node[draw,circle,inner sep=0.6mm] at (1.5,0) {};
        \node at (1.5,-0.4) {$\alpha_3$};
        \draw (1.58,0)--(2.92,0);
        \node[draw,circle,inner sep=0.6mm] at (3,0) {};
        \node at (3,-0.4) {$\alpha_4$};
        \draw (3.08,0)--(4.42,0);
        \node[draw,circle,inner sep=0.6mm] at (4.5,0) {};
        \node at (4.5,-0.4) {$\alpha_5$};
        \draw (4.58,0)--(5.92,0);
        \node[draw,circle,inner sep=0.6mm] at (6,0) {};
        \node at (6,-0.4) {$\alpha_6$};
        \draw (6.08,0)--(7.42,0);
        \node[draw,circle,inner sep=0.6mm] at (7.5,0) {};
        \node at (7.5,-0.4) {$\alpha_7$};
        \draw (7.58,0)--(8.92,0);
        \node[draw,circle,inner sep=0.6mm] at (9,0) {};
        \node at (9,-0.4) {$\alpha_8$};
        \draw (3,0.08)--(3,1.02);
        \node[draw,circle,inner sep=0.6mm] at (3,1.1) {};
        \node at (3.4,1.1) {$\alpha_2$};
    \end{tikzpicture}} \\ 
    \cline{2-3}
    & Cartan Matrix & $\begin{pmatrix} 
        2 & 0 & -1 & 0 & 0 & 0 & 0 & 0 \\ 
        0 & 2 & 0 & -1 & 0 & 0 & 0 & 0 \\ 
        -1 & 0 & 2 & -1 & 0 & 0 & 0 & 0 \\ 
        0 & -1 & -1 & 2 & -1 & 0 & 0 & 0 \\ 
        0 & 0 & 0 & -1 & 2 & -1 & 0 & 0 \\ 
        0 & 0 & 0 & 0 & -1 & 2 & -1 & 0 \\
        0 & 0 & 0 & 0 & 0 & -1 & 2 & -1 \\
        0 & 0 & 0 & 0 & 0 & 0 & -1 & 2
    \end{pmatrix}$ \\ 
    \cline{2-3}
    & Weyl Group $W$ & it is of order $696729600 = 2^{14} \times 3^5 \times 5^2 \times 7$ \\ 
    \cline{2-3}
    & $\Gamma$ & $1$ \\
    \cline{2-3}
    & Highest Root & $2\alpha_1 + 3\alpha_2 + 4\alpha_3 + 6\alpha_4 + 5\alpha_5 + 4\alpha_6 + 3\alpha_7 + 2\alpha_8$ \\ 
    \hline

%Type F_4
    \multirow{7}{*}{$F_4$} 
    & Euclidean Space $E$ & $\mathbb{R}^4$ \\ 
    \cline{2-3}
    & Root System $\Phi$ & \makecell{$\{\pm e_i \mid 1 \leq i \leq 4 \} \cup \{ \pm (e_i \pm e_j) \mid 1 \leq i < j \leq 4 \}$ \\ $\cup \{\frac{1}{2} (\pm e_1 \pm e_2 \pm e_3 \pm e_4) \}$, cardinality $48$} \\ 
    \cline{2-3}
    & Simple System $\Delta$ & \makecell{$\{ \alpha_1 = e_2 - e_3, \alpha_2 = e_3 - e_4, \alpha_3 = e_4,$ \\ $\alpha_4 = \frac{1}{2}(e_1 - e_2 - e_3 - e_4) \}$} \\ 
    \cline{2-3}
    & Dynkin Diagram & {\begin{tikzpicture}[scale=1]
            \node[draw,circle,inner sep=0.6mm] at (0,0) {};
            \node at (0,-0.4) {$\alpha_1$};
            \draw (0.08,0)--(1.42,0);
            \node[draw,circle,inner sep=0.6mm] at (1.5,0) {};
            \node at (1.5,-0.4) {$\alpha_2$};
            \draw (1.58,0.05)--(2.92,0.05);
            \node at (2.25,0) {$>$};
            \draw (1.58,-0.05)--(2.92,-0.05);
            \node[draw,circle,inner sep=0.6mm] at (3,0) {};            
            \node at (3,-0.4) {$\alpha_{3}$};
            \draw (3.08,0)--(4.42,0);
            \node[draw,circle,inner sep=0.6mm] at (4.5,0) {};
            \node at (4.5,-0.4) {$\alpha_{4}$};
        \end{tikzpicture}} \\ 
    \cline{2-3}
    & Cartan Matrix & $\begin{pmatrix} 2 & -1 & 0 & 0 \\ 
        -1 & 2 & -2 & 0 \\ 
        0 & -1 & 2 & -1 \\ 
        0 & 0 & -1 & 2  \\ 
    \end{pmatrix}$ \\ 
    \cline{2-3}
    & Weyl Group $W$ & it is of order $1152 = 2^7 3^2$ \\ 
    \cline{2-3}
    & $\Gamma$ & $1$ \\
    \cline{2-3}
    & Highest Root & \makecell{long: $2\alpha_1 + 3\alpha_2 + 4\alpha_3 + 2\alpha_4$ \\ short: $\alpha_1 + 2\alpha_2 + 3\alpha_3 + 2\alpha_4$} \\ 
    \hline

%Type G_2
    \multirow{7}{*}{$G_2$} 
    & Euclidean Space $E$ & \makecell{the $2$-dimensional subspace of $\mathbb{R}^{3}$ \\ orthogonal to $e_1 + e_2 + e_3$} \\ 
    \cline{2-3}
    & Root System $\Phi$ & \makecell{$\{\pm(e_1 - e_2), \pm(e_2 - e_3), \pm(e_1 - e_3), \pm(2e_1 - e_2 - e_3),$ \\ $\pm(2e_2 - e_1 - e_3), \pm(2e_3 - e_1 - e_2) \}$, cardinality $12$ } \\ 
    \cline{2-3}
    & Simple System $\Delta$ & $\{ e_1 - e_2, -2e_1 + e_2 + e_3 \}$ \\ 
    \cline{2-3}
    & Dynkin Diagram & {\begin{tikzpicture}[scale=1]
            \node[draw,circle,inner sep=0.6mm] at (0,0) {};
            \node at (0,-0.4) {$\alpha_1$};
            \draw (0.08,0.06)--(1.42,0.06);
            \draw (0.08,0)--(1.42,0);
            \node at (0.75,0) {$<$};
            \draw (0.08,-0.06)--(1.42,-0.06);
            \node[draw,circle,inner sep=0.6mm] at (1.5,0) {};
            \node at (1.5,-0.4) {$\alpha_2$};
        \end{tikzpicture}} \\ 
    \cline{2-3}
    & Cartan Matrix & $\begin{pmatrix} 2 & -1 \\ -3 & 2 \end{pmatrix}$ \\ 
    \cline{2-3}
    & Weyl Group $W$ & the dihedral group $\mathscr{D}_6$ of order $12$ \\ 
    \cline{2-3}
    & $\Gamma$ & $1$ \\
    \cline{2-3}
    & Highest Root & \makecell{long: $3\alpha_1 + 2\alpha_2$ \\ short: $2\alpha_1 + \alpha_2$} \\ 
\end{longtable}
}

%%%%%%%%%%%%%%%%%%%%%%%%%%%%%%%%%%%%%%%%%%%%%%%%%%%%%%%%%%%%%%
\subsection{Non-reduced and Non-crystallographic Root Systems}
%%%%%%%%%%%%%%%%%%%%%%%%%%%%%%%%%%%%%%%%%%%%%%%%%%%%%%%%%%%%%%

We now provide a brief discussion on the classification of irreducible non-reduced and non-crystallographic root systems (for definitions, see page \pageref{def:rootsystem}). Note that the previous definition of “irreducibility” continues to hold in these cases.

%%%%%%%%%%%%%%%%%%%%%%%%%%%%%%%%%%%%%%%%%%%%%%%%%%%%%%%%%%%%%%

\subsubsection*{Classification of Irreducible Non-reduced Root Systems} \index{root system!Non-reduced}

For an irreducible non-reduced root system, the classification is relatively simple. Up to isomorphism, there is exactly one family of non-reduced root systems: $BC_\ell \ (\ell \geq 2)$. This family is a hybrid structure that combines elements from both the $B_\ell$ and $C_\ell$ root systems.

The precise description of $BC_\ell$ can be given as follows: The ambient space is $E = \mathbb{R}^\ell$, and the root system is given by
\[
    \Phi = \{ \pm (e_i \pm e_j) \mid 1 \leq i \neq j \leq \ell \} \cup \{ \pm e_i \mid 1 \leq i \leq \ell \} \cup \{ \pm 2 e_i \mid 1 \leq i \leq \ell \}.
\]

As an example, for $\ell = 2$, the root system is of type $BC_2$, as illustrated in the following figure:

\begin{center}
    \begin{tikzpicture}[scale=1.5]
        \draw[->] (0,0) -- (1,0) node[below] {\small $\alpha$};
        \draw[->] (1,0) -- (2,0) node[below] {\small $2 \alpha$};
        \draw[->] (0,0) -- (1,1);
        \draw[->] (0,0) -- (0,1);
        \draw[->] (0,1) -- (0,2);
        \draw[->] (0,0) -- (-1,1) node[above left] {\small $\beta$};
        \draw[->] (0,0) -- (-1,0);
        \draw[->] (-1,0) -- (-2,0);
        \draw[->] (0,0) -- (-1,-1);
        \draw[->] (0,0) -- (0,-1);
        \draw[->] (0,-1) -- (0,-2);
        \draw[->] (0,0) -- (1,-1);
    \end{tikzpicture}
\end{center}

We can derive the corresponding reduced root systems from the non-reduced system by focusing on either the long or short roots when a positive multiple of a given root exists. For instance, within the $BC_n$ family, we can recover the $B_n$ root system by selecting the short roots, and the $C_n$ root system by selecting the long roots.

%%%%%%%%%%%%%%%%%%%%%%%%%%%%%%%%%%%%%%%%%%%%%%%%%%%%%%%%%%%%%%

\subsubsection*{Classification of Irreducible Non-crystallographic Root Systems} \index{root system!Non-crystallographic}

To classify irreducible non-crystallographic root systems, we begin by introducing the concept of Coxeter graphs, which can be associated with both crystallographic and non-crystallographic root systems.

Let $\Phi$ be a root system in a Euclidean space $E$, either crystallographic or non-crystallographic. Let $\Delta = \{\alpha_1, \dots, \alpha_\ell\}$ be a simple system of $\Phi$. (For non-crystallographic root systems, note that the simple system does not necessarily satisfy the integrality conditions usually associated with crystallographic systems.) For any $\alpha, \beta \in \Phi$, define $m(\alpha, \beta)$ as the order of the product $(s_\alpha s_\beta)$, where $s_\alpha$ and $s_\beta$ are the reflections corresponding to the roots $\alpha$ and $\beta$.

The \textbf{Coxeter graph} \index{Coxeter graph} of $\Phi$ is an undirected graph with $\ell$ vertices, each corresponding to a simple root $\alpha_i$. An edge is drawn between the vertices corresponding to $\alpha_i$ and $\alpha_j$ if $m(\alpha_i, \alpha_j) \geq 3$, and such an edge is labeled with the value of $m(\alpha_i, \alpha_j)$. Since the label $3$ occurs frequently, it is customary to omit it when drawing the graph. 
(It is worth noting that Coxeter graphs are typically defined in the broader context of Coxeter groups. However, the definition provided here is specifically tailored to root systems.)

If $\Phi$ is an irreducible non-crystallographic root system, then its Coxeter graph must be one of the following type:
\begin{center}
\begin{tikzpicture}[scale=1]
% H_3
    \begin{scope}[shift={(0,0)}]
        \node at (-1,-0.15){$H_3:$};
        \node[draw,circle,inner sep=0.6mm] at (0,0) {};
        \node at (0.75,0.2){\small $5$};
        \node at (0,-0.4) {$\alpha_1$};
        \draw (0.08,0)--(1.42,0);
        \node[draw,circle,inner sep=0.6mm] at (1.5,0) {};
        \node at (1.5,-0.4) {$\alpha_2$};
        \draw (1.58,0)--(2.92,0);
        \node[draw,circle,inner sep=0.6mm] at (3,0) {};
        \node at (3,-0.4) {$\alpha_3$};
    \end{scope}
% H_4
    \begin{scope}[shift={(0,-1.5)}]
        \node at (-1,-0.15){$H_4:$};
        \node[draw,circle,inner sep=0.6mm] at (0,0) {};
        \node at (0.75,0.2){\small $5$};
        \node at (0,-0.4) {$\alpha_1$};
        \draw (0.08,0)--(1.42,0);
        \node[draw,circle,inner sep=0.6mm] at (1.5,0) {};
        \node at (1.5,-0.4) {$\alpha_2$};
        \draw (1.58,0)--(2.92,0);
        \node[draw,circle,inner sep=0.6mm] at (3,0) {};
        \node at (3,-0.4) {$\alpha_3$};
        \draw (3.08,0)--(4.42,0);
        \node[draw,circle,inner sep=0.6mm] at (4.5,0) {};
        \node at (4.5,-0.4) {$\alpha_4$};
    \end{scope}
% I_2(m)
    \begin{scope}[shift={(0,-3)}]
        \node at (-1.2,-0.15){$I_2 (m):$};
        \node[draw,circle,inner sep=0.6mm] at (0,0) {};
        \node at (0.75,0.2){\small $m$};
        \node at (0,-0.4) {$\alpha_1$};
        \draw (0.08,0)--(1.42,0);
        \node[draw,circle,inner sep=0.6mm] at (1.5,0) {};
        \node at (1.5,-0.4) {$\alpha_2$};
    \end{scope}
\end{tikzpicture}
\end{center}
Conversely, for each Coxeter graph of type $H_3, H_4$ or $I_2(m) \ (m = 5 \text{ or } m \geq 7)$, there exists a unique (up to isomorphism) irreducible non-crystallographic root system corresponding to that graph (see Humphreys \cite{JH_RC}).

%%%%%%%%%%%%%%%%%%%%%%%%%%%%%%%%%%%%%%%%%%%%%%%%%%%%%%%%%%%%%%
\subsection{Abstract Theory of Weights}\label{subsec:weights}
%%%%%%%%%%%%%%%%%%%%%%%%%%%%%%%%%%%%%%%%%%%%%%%%%%%%%%%%%%%%%%

An element $\lambda \in E$ is called a \textbf{weight} \index{weight} if $\langle \lambda, \alpha \rangle \in \mathbb{Z}$ for all $\alpha \in \Phi$.
The set of all weights $\Lambda_{sc} := \{ \lambda \in E \mid \langle \lambda, \alpha \rangle \in \mathbb{Z}, \forall \alpha \in \Phi \}$ forms a lattice, called the \textbf{(fundamental) weight lattice}\index{weight lattice}\label{nomencl:weightlattice}.
Clearly, all the roots are contained in $\Lambda_{sc}$. A sublattice $\Lambda_r$ (or $\Lambda_{\text{ad}}$) of $\Lambda_{sc}$ generated by all the roots is called the \textbf{root lattice}\index{root lattice}\label{nomencl:rootlattice}.
Let $\Delta = \{ \alpha_1, \dots, \alpha_{\ell} \}$ be a simple system. There exist elements $\lambda_1, \dots, \lambda_\ell \in \Lambda_{sc}$ satisfying $\langle \lambda_i, \alpha_j \rangle = \delta_{ij}$; these are called the \textbf{fundamental dominant weights (relative to $\Delta$)}. They form a $\mathbb{Z}$-basis of $\Lambda_{sc}$. 

The quotient $\Lambda_{sc}/\Lambda_r$ is a finite group, whose order is $|\det M|$, where $M$ is the Cartan matrix corresponding to some fixed ordering of the simple roots. This group is called the \textbf{fundamental group} of $\Phi$. The following table provides an explicit description of the fundamental groups for all irreducible root systems.
\begin{center}
    \begin{tabular}{c|c|c|c|c|c|c|c|c|c|c}
        Type of $\Phi$ & $A_{\ell}$ & $B_{\ell}$ & $C_{\ell}$ & $D_{2{\ell}+1}$ & $D_{2{\ell}}$ & $E_6$ & $E_7$ & $E_8$ & $F_4$ & $G_2$ \\
        \hline
        $\Lambda_{sc} / \Lambda_{r}$ & $\mathbb{Z}_{{\ell}+1}$ & $\mathbb{Z}_{2}$ & $\mathbb{Z}_{2}$ & $\mathbb{Z}_{4}$  & $\mathbb{Z}_{2} \times \mathbb{Z}_{2}$ & $\mathbb{Z}_{3}$ & $\mathbb{Z}_{2}$ & $0$ & $0$ & $0$
    \end{tabular}
\end{center}

Fix a simple system $\Delta \subset \Phi$. Define $\lambda \in \Lambda_{sc}$ to be \textbf{dominant} \index{weight!dominant} if all the integers $\langle \lambda, \alpha \rangle \ (\alpha \in \Delta)$ are nonnegative, and \textbf{strongly dominant} \index{weight!strongly dominant} if all these integers are strictly positive. Let $\Lambda_{sc}^+$ denote the set of all dominant weights. An element $\lambda = \sum_{i=1}^{\ell} m_i \lambda_i \in \Lambda_{sc}^+$ if and only if $m_i \geq 0$ for all $i$. 

Since the Weyl group $W$ preserves the inner product on $E$, it leaves $\Lambda_{sc}$ invariant. Each weight in $\Lambda_{sc}$ is conjugate under $W$ to one and only one dominant weight. 
As a subset of $E$, the lattice $\Lambda_{sc}$ is partially ordered by the relation: $\lambda \geq \mu$ if and only if $\lambda - \mu$ is a sum of positive roots or $\lambda = \mu$.
If $\lambda$ is dominant, then $w(\lambda) \leq \lambda$ for all $w \in W$. Furthermore, if $\lambda$ is strongly dominant, then $w(\lambda) = \lambda$ only when $w = 1$. Finally, for $\lambda \in \Lambda_{sc}^+$, the number of dominant weights $\mu \leq \lambda$ is finite.

Set $\delta = \frac{1}{2} \sum_{\beta \in \Phi^+} \beta$. Then, for all $\alpha \in \Delta$, we have $s_\alpha(\delta) = \delta - \alpha$.
Moreover, $\delta$ can be expressed as $\delta = \sum_{j=1}^{\ell} \lambda_j$, which shows that $\delta$ is a strongly dominant weight. 

%%%%%%%%%%%%%%%%%%%%%%%%%%%%%%%%%%%%%%%%%%%%%%%%%%%%%%%%%%%%%%

\subsubsection*{Saturated sets of weights}

A subset $\Pi \subset \Lambda_{sc}$ is called \textbf{saturated} if for all $\lambda \in \Pi$, $\alpha \in \Phi$, and integers $i$ between $0$ and $\langle \lambda, \alpha \rangle$, the weight $\lambda - i\alpha$ also lies in $\Pi$. 
It is evident that the Weyl group $W$ leaves $\Pi$ invariant.
A saturated set $\Pi$ is said to have a \textbf{highest weight} $\lambda$ if $\lambda \in \Pi \cap \Lambda_{sc}^+$ and $\mu \leq \lambda$ for all $\mu \in \Pi$.

\begin{ex}
    \normalfont
    \begin{enumerate}[(a)]
        \item The set consisting of $0$ alone is saturated, with highest weight $0$.
        \item The set $\Phi$ of all roots, together with $0$, forms a saturated set. When $\Phi$ is irreducible, the set $\Pi$ has a highest weight, which is the unique highest root of $\Phi$ (relative to fixed simple system $\Delta$).
    \end{enumerate}
\end{ex}

\begin{facts}
    \normalfont
    Let $\Pi$ be a saturated set with the highest weight $\lambda$. Then:
    \begin{enumerate}[label=(\alph*)]
        \item $\Pi$ is finite.
        \item If $\mu \in \Lambda_{sc}^+$ and $\mu \leq \lambda$, then $\mu \in \Pi$.
        \item If $\mu \in \Pi$, then $(\mu + \delta, \mu + \delta) \leq (\lambda + \delta, \lambda + \delta)$, with equality if and only if $\mu = \lambda$.
    \end{enumerate}
\end{facts}

%%%%%%%%%%%%%%%%%%%%%%%%%%%%%%%%%%%%%%%%%%%%%%%%%%%%%%%%%%%%%%
%Section - Semisimple Lie Algebras
%%%%%%%%%%%%%%%%%%%%%%%%%%%%%%%%%%%%%%%%%%%%%%%%%%%%%%%%%%%%%%

\section{Semisimple Lie Algebras}\label{semi Lie algebra}

We assume that the reader is already familiar with the concept of semisimple Lie algebras and their classifications. However, we will briefly review these concepts to ensure that the reader becomes familiar with the notation used throughout the text. The material in this chapter is primarily derived from J. Humphreys~\cite{JH} and R. Steinberg~\cite{RS}.

%%%%%%%%%%%%%%%%%%%%%%%%%%%%%%%%%%%%%%%%%%%%%%%%%%%%%%%%%%%%%%
%%%%%%%%%%%%%%%%%%%%%%%%%%%%%%%%%%%%%%%%%%%%%%%%%%%%%%%%%%%%%%

\subsection{Introduction to Lie Algebras}

A \textbf{Lie algebra}\index{Lie algebra}\label{nomencl:L} $\mathcal{L}$ over a field $F$ is an $F$-vector space equipped with a bilinear operation, known as the \textbf{Lie bracket} (or \textbf{commutator}), 
\[
    [\cdot, \cdot]: \mathcal{L} \times \mathcal{L} \longrightarrow \mathcal{L}, \quad (X, Y) \longmapsto [X, Y],
\]
that satisfies the following axioms:
\begin{enumerate}[(a)]
    \item \textbf{(Alternativity)} $[X,X]=0$ for all $X \in \mathcal{L}$.
    \item \textbf{(Jacobi identity)} $[X,[Y,Z]]+[Y,[Z,X]]+[Z,[X,Y]] = 0$ for all $X,Y,Z \in \mathcal{L}$.
\end{enumerate}

A \textbf{Lie subalgebra} of $\mathcal{L}$ is a subspace $\mathcal{M}$ of $\mathcal{L}$ that is closed under the Lie bracket, that is, $[X, Y] \in \mathcal{M}$ for all $X, Y \in \mathcal{M}$.
An \textbf{ideal} of $\mathcal{L}$ is a subspace $\mathcal{I}$ satisfying the condition that for all $X \in \mathcal{I}$ and $Y \in \mathcal{L}$, we have $[X, Y] \in \mathcal{I}$.
Here are some examples of ideals of $\mathcal{L}$:
\begin{enumerate}[(a)]
    \item The subspaces $\{0\}$ and $\mathcal{L}$ are trivially ideals of $\mathcal{L}$.
    \item Let $V$ be a finite-dimensional vector space over a field $F$, and let $\text{End}(V)$ denote the set of all $F$-linear endomorphisms of $V$. Clearly, $\text{End}(V)$ forms an associative $F$-algebra. It can also be regarded as a Lie algebra with the bracket operation defined by $[X, Y] = XY - YX$ for all $X, Y \in \text{End}(V)$. When $\text{End}(V)$ is viewed as a Lie algebra under this operation, it is customary to denote it by $\mathfrak{gl}(V)$. The subalgebra $\mathfrak{sl}(V)$, consisting of all elements in $\mathfrak{gl}(V)$ with trace zero, is an ideal of $\mathfrak{gl}(V)$.
    \item If $\mathcal{I}$ and $\mathcal{J}$ are two ideals of $\mathcal{L}$, then
    \[
    \mathcal{I} + \mathcal{J} = \{ X + Y \mid X \in \mathcal{I}, Y \in \mathcal{J} \} \quad \text{and} \quad [\mathcal{I}, \mathcal{J}] = \Big\{ \sum [X_i, Y_i] \mid X_i \in \mathcal{I}, Y_i \in \mathcal{J} \Big\}
    \]
    are also ideals of $\mathcal{L}$. In particular, the \textbf{derived subalgebra} $[\mathcal{L}, \mathcal{L}]$ is an ideal of $\mathcal{L}$.
    \item The \textbf{center} of $\mathcal{L}$, defined as $Z(\mathcal{L}) = \{ Z \in \mathcal{L} \mid [X, Z] = 0 \text{ for all } X \in \mathcal{L} \}$, is an ideal of $\mathcal{L}$.
\end{enumerate}

A linear map $\phi: \mathcal{L}_1 \longrightarrow \mathcal{L}_2$ between two Lie algebras $\mathcal{L}_1$ and $\mathcal{L}_2$ is called a \textbf{Lie algebra homomorphism} if it preserves the Lie bracket, that is, $\phi([X,Y]) = [\phi (X), \phi (Y)], $ for all $X,Y \in \mathcal{L}_1.$ 
The notions of \textbf{isomorphism} and \textbf{automorphism} are defined as usual.

The \textbf{normalizer} of a subalgebra (or just subspace) $\mathcal{K}$ of $\mathcal{L}$ is defined as 
$$
    N_{\mathcal{L}}(\mathcal{K}) = \{ X \in \mathcal{L} \mid [X, \mathcal{K}] \subset \mathcal{K} \}.
$$ 
The normalizer $N_{\mathcal{L}}(\mathcal{K})$ is the largest subalgebra of $\mathcal{L}$ containing $\mathcal{K}$ as an ideal.
If $\mathcal{K} = N_{\mathcal{L}}(\mathcal{K})$, we say that $\mathcal{K}$ is \textbf{self-normalizing}.
The \textbf{centralizer} of a subset $\mathcal{S} \subset \mathcal{L}$ is given by
\[
    C_{\mathcal{L}}(\mathcal{S}) = \{ X \in \mathcal{L} \mid [X, \mathcal{S}] = 0 \}.
\]
The centralizer $C_{\mathcal{L}}(\mathcal{S})$ is always a subalgebra of $\mathcal{L}$.

%%%%%%%%%%%%%%%%%%%%%%%%%%%%%%%%%%%%%%%%%%%%%%%%%%%%%%%%%%%%%%

\subsubsection*{Representations of Lie algebras}

A \textbf{representation} of a Lie algebra $\mathcal{L}$ is a homomorphism $\pi: \mathcal{L} \longrightarrow \mathfrak{gl} (V)$ for a finite-dimensional vector space $V$. 
One of the most significant example for us is the \textbf{adjoint representation} $\text{ad}: \mathcal{L} \longrightarrow \mathfrak{gl}(\mathcal{L})$ given by $(\text{ad} \ X) (Y)=[X, Y]$. 

An \textbf{$\mathcal{L}$-module} is a vector space $V$ endowed with an operation $\mathcal{L} \times V \longrightarrow V$, denoted $(X,v) \longmapsto X \cdot v$ (or simply $Xv$), satisfying the following axioms:
\begin{enumerate}[(a)]
    \item $(aX + bY) \cdot v = a(X \cdot v) + b(Y \cdot v),$
    \item $X \cdot (av + bw) = a(X \cdot v) + b(Y \cdot w),$
    \item $[X,Y] \cdot v = X \cdot Y \cdot v - Y \cdot X \cdot v$, 
\end{enumerate}
where $X, Y \in \mathcal{L}$; $v, w \in V$; and $a, b \in \mathbb{C}$. 

Note that, any representation $\pi: \mathcal{L} \longrightarrow \mathfrak{gl} (V)$ can be viewed as $\mathcal{L}$-module $V$ via the action $X \cdot v = \pi(X)(v)$. Conversely, given an $\mathcal{L}$-module $V$, this action defines a representation $\pi: \mathcal{L} \longrightarrow \mathfrak{gl} (V)$. 
An $\mathcal{L}$-module $V$ is called \textbf{irreducible} if it has precisely two $\mathcal{L}$-submodules (itself and $0$). It is said to be \textbf{completely reducible} if $V$ is a direct sum of irreducible $\mathcal{L}$-submodules. A representation is irreducible (respectively, completely reducible) if its corresponding $\mathcal{L}$-module is. 

%%%%%%%%%%%%%%%%%%%%%%%%%%%%%%%%%%%%%%%%%%%%%%%%%%%%%%%%%%%%%%

\subsubsection*{Solvable and Nilpotent Lie algebras}

Let $\mathcal{L}$ be a Lie algebra over a field $F$. 
\begin{enumerate}[(a)]
    \item The \textbf{derived series} of $\mathcal{L}$ is the sequence of ideals of $\mathcal{L}$ given by 
    \[
        \mathcal{L}^{(0)} = \mathcal{L}, \quad \mathcal{L}^{(k+1)} = [\mathcal{L}^{(k)}, \mathcal{L}^{(k)}] \quad (\text{for all } k \geq 0).
    \]
    \item The \textbf{lower central series} (or the \textbf{descending central series}) of $\mathcal{L}$ is the sequence of ideals of $\mathcal{L}$ given by 
    \[
        \mathcal{L}^{0} = \mathcal{L}, \quad \mathcal{L}^{k+1} = [\mathcal{L}, \mathcal{L}^{k}] \quad (\text{for all } k \geq 0).
    \]
\end{enumerate}

\begin{defn}
    \normalfont
    A Lie algebra $\mathcal{L}$ over a field $F$ is called
    \begin{enumerate}[(a)]
        \item \textbf{solvable}, \index{Lie algebra!solvable} if its derived series terminates, i.e., $\mathcal{L}^{(m)} = 0$ for some $m \geq 0$. 
        \item \textbf{nilpotent}, \index{Lie algebra!nilpotent} if its lower central series terminates, i.e., $\mathcal{L}^{m} = 0$ for some $m \geq 0$. 
        \item \textbf{abelian}, if $[\mathcal{L}, \mathcal{L}] = 0$. 
    \end{enumerate}
\end{defn}

\begin{facts}
    \normalfont
    Let $\mathcal{L}$ be a Lie algebra over a field $F$. 
    \begin{enumerate}[(a)]
        \item $\mathcal{L}$ is abelian if and only if $Z(\mathcal{L}) = \mathcal{L}$. 
        \item If $\mathcal{L}$ is abelian, then $\mathcal{L}$ is nilpotent.
        \item If $\mathcal{L}$ is nilpotent, then $\mathcal{L}$ is solvable.
    \end{enumerate}
\end{facts}

%%%%%%%%%%%%%%%%%%%%%%%%%%%%%%%%%%%%%%%%%%%%%%%%%%%%%%%%%%%%%%
\subsection{Semisimple Lie Algebras}
%%%%%%%%%%%%%%%%%%%%%%%%%%%%%%%%%%%%%%%%%%%%%%%%%%%%%%%%%%%%%%

A \textbf{simple} \index{Lie algebra!simple} Lie algebra is a non-abelian Lie algebra that has no nonzero proper ideals. 
A Lie algebra $\mathcal{L}$ is called \textbf{semisimple} \index{Lie algebra!semisimple} if it can be written as a direct sum of simple Lie algebras. 

\begin{facts}
    \normalfont
    Let $\mathcal{L}$ be a finite-dimensional Lie algebra over a field $F$ of characteristic $0$. The following statements are equivalent:
    \begin{enumerate}[(a)]
        \item $\mathcal{L}$ is semisimple. 
        \item $\mathcal{L}$ has no non-zero abelian ideals.
        \item $\mathcal{L}$ has no non-zero solvable ideals.
        \item the radical $rad(\mathcal{L})$ (maximal solvable ideal) of $\mathcal{L}$ is zero.
        \item The Killing form $\kappa$ on $\mathcal{L}$ is non-degenerate, where the Killing form is defined as $\kappa(X, Y) = \operatorname{Tr}(\operatorname{ad} X \circ \operatorname{ad} Y)$ for all $X, Y \in \mathcal{L}$.
    \end{enumerate}
\end{facts}

%%%%%%%%%%%%%%%%%%%%%%%%%%%%%%%%%%%%%%%%%%%%%%%%%%%%%%%%%%%%%%

\subsubsection*{Complete Reducibility of Representations}

Let $\mathcal{L}$ be a finite-dimensional semisimple Lie algebra over a field $F$ of characteristic zero.

\begin{thm}[Weyl]
    \normalfont
    Any finite-dimensional representation $\phi: \mathcal{L} \longrightarrow \mathfrak{gl}(V)$ of a semisimple Lie algebra $\mathcal{L}$ is completely reducible.
\end{thm}

%%%%%%%%%%%%%%%%%%%%%%%%%%%%%%%%%%%%%%%%%%%%%%%%%%%%%%%%%%%%%%

\subsubsection*{Jordan-Chevalley Decomposition}

Let $V$ be a finite-dimensional vector space over a field $F$. An element $x \in \operatorname{End}(V)$ is said to be \textbf{semisimple} if the roots of its minimal polynomial over $F$ are all distinct. 
If $F$ is an algebraically closed field, then $x$ is semisimple if and only if $x$ is diagonalizable.

\begin{prop}
    \normalfont
    Let $F$ be an algebraically closed field, and let $x \in \operatorname{End}(V)$. 
    \begin{enumerate}[(a)]
        \item There exist unique elements $x_s, x_n \in \operatorname{End}(V)$ such that:
        \[
        x = x_s + x_n, \quad x_s x_n = x_n x_s, \quad x_s \text{ is semisimple} \quad \text{and} \quad x_n \text{ is nilpotent}.
        \]
        \item There are polynomials $p(T)$ and $q(T)$ in a single variable $T$, both without constant terms, such that 
        \[
        x_s = p(x) \quad \text{and} \quad x_n = q(x).
        \]
        As a consequence, $x_s$ and $x_n$ commute with any endomorphism that commutes with $x$.
    \end{enumerate}
    The decomposition $x = x_s + x_n$ is called the (additive) \textbf{Jordan–Chevalley decomposition} of $x$. The components $x_s$ and $x_n$ are referred to as the \textbf{semisimple part} and the \textbf{nilpotent part} of $x$, respectively.
\end{prop}

\begin{lemma}
    \normalfont
    Let $x \in \operatorname{End}(V)$ and let $x = x_s + x_n$ be its Jordan decomposition. Then 
    \[
    \operatorname{ad}(x) = \operatorname{ad}(x_s) + \operatorname{ad}(x_n)
    \]
    is the Jordan decomposition of $\operatorname{ad}(x)$ (in $\operatorname{End}(\operatorname{End}(V))$).
\end{lemma}

Now, let $\mathcal{L}$ be a semisimple Lie algebra over an algebraically closed field. 
Consider the adjoint representation $\operatorname{ad}: \mathcal{L} \longrightarrow \mathfrak{gl}(\mathcal{L})$ of $\mathcal{L}$. 
For a given $x \in \mathcal{L}$, the Jordan–Chevalley decomposition implies that there exist elements $y_s, y_n \in \mathfrak{gl}(\mathcal{L})$ such that $\operatorname{ad}(x) = y_s + y_n,$ where $y_s$ is semisimple and $y_n$ is nilpotent. 
Moreover, it is known that there exist $x_s, x_n \in \mathcal{L}$ such that 
\[
    x = x_s + x_n, \quad [x_s, x_n] = 0, \quad y_s = \operatorname{ad}(x_s) \quad \text{and} \quad y_n = \operatorname{ad}(x_n).
\]
We refer to this as the \textbf{abstract Jordan–Chevalley decomposition} of $x$.
Clearly, $x_s$ is \textbf{ad-semisimple}, and $x_n$ is \textbf{ad-nilpotent}. By a slight abuse of terminology, we refer to $x_s$ and $x_n$ as the \textbf{semisimple} and \textbf{nilpotent} parts of $x$, respectively. The following theorem justifies this terminology.

\begin{thm}
    \normalfont
    Suppose $\mathcal{L} \subset \mathfrak{gl}(V)$ is a semisimple linear Lie algebra. Then, for every element of $\mathcal{L}$, its semisimple and nilpotent parts (as elements of $\mathfrak{gl}(V)$) also belong to $\mathcal{L}$. In particular, the abstract Jordan–Chevalley decomposition in $\mathcal{L}$ agrees with the usual Jordan decomposition in $\mathfrak{gl}(V)$.
\end{thm}

\begin{cor}
    \normalfont
    Let $\mathcal{L}$ be a semisimple Lie algebra and let $\phi: \mathcal{L} \to \mathfrak{gl}(V)$ be a (finite-dimensional) representation of $\mathcal{L}$. 
    If $x = x_s + x_n$ represents the abstract Jordan–Chevalley decomposition of $x \in \mathcal{L}$, then $\phi(x) = \phi(s) + \phi(n)$ gives the usual Jordan decomposition of \(\phi(x)\) in \(\mathfrak{gl}(V)\).
\end{cor}

%%%%%%%%%%%%%%%%%%%%%%%%%%%%%%%%%%%%%%%%%%%%%%%%%%%%%%%%%%%%%%

\subsubsection*{Structure of Semisimple Lie algebras}

Let $\mathcal{L}$ be a semisimple Lie algebra over $\mathbb{C}$ (or any algebraically closed field with characteristic zero). 
A nonzero subalgebra $\mathcal{T}$ of $\mathcal{L}$ is called \textbf{toral} if it only consists of semisimple elements. 
Every toral subalgebra of $\mathcal{L}$ is abelian. 
Let $\mathcal{H}$ be a maximal toral subalgebra of $\mathcal{L}$. Then $\mathcal{H}$ is abelian and every element of $\operatorname{ad}(\mathcal{H})$ is semisimple (i.e., diagonalizable). 
Consequently, $\operatorname{ad}(\mathcal{H})$ forms a commuting family of linear transformations on $\mathcal{L}$, and they can be simultaneously diagonalized.
For $\alpha \in \mathcal{H}^*$, define $\mathcal{L}_\alpha = \{ X \in \mathcal{L} \mid [H,X] = \alpha (H) X$ for all $H \in \mathcal{H} \}$. 
Then $\mathcal{L}_0 = C_{\mathcal{L}}(\mathcal{H}) = \mathcal{H}$. 
Let $\Phi$ be the set of all nonzero $\alpha \in \mathcal{H}^*$ such that $\mathcal{L}_\alpha \neq 0$. 
The elements of $\Phi$ are called the \textbf{roots} of $\mathcal{L}$ relative to $\mathcal{H}$. 
This leads to the \textbf{root space decomposition} of $\mathcal{L}$: 
\[
    \mathcal{L} = \mathcal{H} \oplus \coprod_{\alpha \in \Phi} \mathcal{L}_\alpha.
\]

For all $\alpha, \beta \in \Phi$, $[\mathcal{L}_{\alpha}, \mathcal{L}_\beta] = \mathcal{L}_{\alpha + \beta}.$ 
Each root space $\mathcal{L}_\alpha \hspace{2mm} (\alpha \in \Phi)$ is one-dimensional. 
If $\alpha \in \Phi$, then $-\alpha \in \Phi$. Moreover, the only scalar multiples of a root $\alpha$ in $\Phi$ are $\alpha$ and $-\alpha$. 
Note that $\Phi$ spans $\mathcal{H}^*$. Since the restriction of the Killing form $\kappa$ to $\mathcal{H}$ is nondegenerate, for given $\alpha \in \mathcal{H}^*$ there exists a unique element $H'_\alpha \in \mathcal{H}$ satisfying $\phi(H) = \kappa (H'_{\alpha},H)$ for all $H \in \mathcal{H}$. 
Let $\alpha \in \Phi, X \in \mathcal{L}_\alpha, Y \in \mathcal{L}_{-\alpha}$ then $[X, Y] = \kappa (X, Y) H'_\alpha$. 
For $\alpha \in \Phi$, define $H_\alpha = 2 H'_\alpha / \kappa (H'_\alpha, H'_\alpha)$. Then $H_\alpha = - H_{-\alpha}$. If $\alpha \in \Phi$ and $X_\alpha$ is any nonzero element of $\mathcal{L}_\alpha$, then there exists a unique $X_{-\alpha} \in \mathcal{L}_{-\alpha}$ such that $[X_\alpha, X_{-\alpha}] = H_\alpha$. 

Let $\alpha, \beta \in \Phi$ with $\beta \neq \pm \alpha$. Let $r, q$ be the largest integers for which $\beta - r \alpha$ and $\beta + q \alpha$ are roots, respectively. Then all the roots $\beta + i\alpha \in \Phi \ (-r \leq i \leq q)$, and $\beta (H_\alpha)= r - q \in \mathbb{Z}$. In other words, the roots $\beta + i \alpha$ form a string (the \textbf{$\alpha-$string through $\beta$}) $\beta - r \alpha, \dots, \beta, \dots, \beta + q \alpha$. Notice that $\beta - \beta (H_\alpha)\alpha \in \Phi$. The number $\beta(H_\alpha)$ are called \textbf{Cartan integers}.

The rank of $\mathcal{L}$ is defined by $l:=$ dim$_{\mathbb{C}} (\mathcal{H}) = \dim_{\mathbb{C}} (\mathcal{H}^*) = \dim_{\mathbb{C}} (\text{Span }(\Phi))$. Hence $$\text{dim}_{\mathbb{C}} (\mathcal{L}) = \text{rank} (\mathcal{L}) + |\Phi|.$$ 
Let \(E_\mathbb{Q}\) denote \(\mathcal{H}^*_\mathbb{Q}\), the vector space over \(\mathbb{Q}\) generated by the roots. Then \(\dim_{\mathbb{Q}} (E_\mathbb{Q}) = l\). Define the bilinear form \((\alpha, \beta) = \kappa(H'_\alpha, H'_\beta)\) for all \(\alpha, \beta \in E_\mathbb{Q}\). This defines a symmetric, nondegenerate, positive definite bilinear form on \(E_\mathbb{Q}\).
Let \(E = \mathbb{R} \otimes_\mathbb{Q} E_\mathbb{Q}\). The bilinear form on \(E_\mathbb{Q}\) extends canonically to \(E\) and remains positive definite, thus making \(E\) a Euclidean space. Note that \(\Phi\) contains a basis for \(E\), and \(\dim_{\mathbb{R}} (E) = l\). 

Finally, we observe that \(\Phi\) can be interpreted as the root system in the sense of Section~\ref{subsec:Root Systems and Weyl Groups}. In this context, we sometimes use the term crystallographic root system to refer to the axiomatic definition of a root system.

%%%%%%%%%%%%%%%%%%%%%%%%%%%%%%%%%%%%%%%%%%%%%%%%%%%%%%%%%%%%%%

\subsubsection*{Existence and Isomorphism Theorems}

Let $\mathcal{L}$ be a semisimple Lie algebra over a field $\mathbb{C}$. Let $\mathcal{H}$ and $\Phi$ be as defined above. Then $\Phi$ is irreducible if and only if $\mathcal{L}$ is simple. 

\begin{thm}
    \normalfont
    \begin{enumerate}[(a)]
        \item Let $\Phi$ be a crystallographic root system. Then there exists a semisimple Lie algebra whose root system is $\Phi$.
        \item Let \(\mathcal{L}, \mathcal{L}'\) be semisimple Lie algebras over \(\mathbb{C}\), with maximal toral subalgebras \(\mathcal{H}, \mathcal{H}'\) and root systems \(\Phi, \Phi'\), respectively. Let an isomorphism \(\Phi \to \Phi'\) be given, sending a given simple system \(\Delta\) to \(\Delta'\), and denote by \(\pi: \mathcal{H} \to \mathcal{H}'\) the associated isomorphism. Choose nonzero elements \(X_\alpha \in \mathcal{L}_\alpha\) (\(X'_{\alpha'} \in \mathcal{L}'_{\alpha'}\)) for each \(\alpha \in \Delta\) (\(\alpha' \in \Delta'\)). Then, there exists a unique isomorphism \(\pi: \mathcal{L} \to \mathcal{L}'\) extending \(\pi\) on \(\mathcal{H}\) and mapping \(X_\alpha\) to \(X'_{\alpha'}\) for all \(\alpha \in \Delta\).
    \end{enumerate}
\end{thm}

%%%%%%%%%%%%%%%%%%%%%%%%%%%%%%%%%%%%%%%%%%%%%%%%%%%%%%%%%%%%%%
\subsection{Universal Enveloping Algebra}
%%%%%%%%%%%%%%%%%%%%%%%%%%%%%%%%%%%%%%%%%%%%%%%%%%%%%%%%%%%%%%

Let $\mathcal{L}$ be a Lie algebra over a field $F$, and $\mathcal{U}$ an associative algebra over $F$. Since $\mathcal{U}$ is an associative algebra, it naturally inherits a Lie algebra structure by defining the Lie bracket as $[X, Y] = XY-YX$ for all $X, Y \in \mathcal{U}$.

\begin{defn}\label{uni_env_algebra}
    \normalfont
    A {\bf universal enveloping algebra} \index{universal enveloping algebra} of a Lie algebra $\mathcal{L}$ is a pair $(\mathcal{U},\phi)$ satisfying:
    \begin{enumerate}[(a)]
        \item $\mathcal{U}$ is an associative algebra with $1$.
        \item $\phi: \mathcal{L} \to \mathcal{U}$ is a Lie algebra homomorphism, i.e., $\phi$ is $F$-linear and satisfies $$\phi([X, Y]) = [\phi(X), \phi(Y)] \quad \text{for all } X, Y \in \mathcal{L}.$$
        \item If $(\mathcal{A}, \psi)$ is any other such pair, there exists a unique algebra homomorphism $\theta: \mathcal{U} \longrightarrow \mathcal{A}$ such that $\theta \circ \phi = \psi$ and $\theta(1)=1$. 
    \end{enumerate}
\end{defn}

We denote the universal enveloping algebra of $\mathcal{L}$ by $\mathcal{U}(\mathcal{L})$ when specifying the Lie algebra explicitly.

\begin{rmk}[Existence and Uniqueness]
    \normalfont
    The universal enveloping algebra can be constructed as $\mathcal{U} = \mathcal{T}(\mathcal{L}) / J$, where $\mathcal{T}(\mathcal{L})$ is the tensor algebra of $\mathcal{L}$, and $J$ is the two-sided ideal generated by 
    \[
    \langle X \otimes Y - Y \otimes X - [X, Y] \mid X, Y \in \mathcal{L} \rangle.
    \]
    The map $\phi$ is defined as follows:
    \[
    \xymatrix{
    \mathcal{L} \ar[rd]_-{\phi=\pi \circ \textit{i}} \ar[r]^-{\textit{i}} & \mathcal{T}(\mathcal{L}) \ar[d]^-{\pi} \\
    & \mathcal{T}(\mathcal{L}) / J
    }
    \]
    This construction ensures that $(\mathcal{U}, \phi)$ satisfies the definition. Uniqueness follows from the universal property.
\end{rmk}

\begin{thm}[PBW Theorem]
    \normalfont
    Let $\mathcal{L}$ be a Lie algebra over a field $F$, and $(\mathcal{U}, \phi)$ its universal enveloping algebra. Then:
    \begin{enumerate}[(a)]
        \item The homomorphism $\phi$ is injective.
        \item Identifying $\mathcal{L}$ with its image under $\phi$, if $X_1, \dots, X_r$ is a basis of $\mathcal{L}$, then the monomials $X_1^{k_1} X_2^{k_2} \dots X_r^{k_r}$ (where $k_i \geq 0$) form a basis of $\mathcal{U}$.
    \end{enumerate}
\end{thm}

\begin{rmk}
    PBW stands for Poincaré–Birkhoff–Witt.
\end{rmk}

\begin{rmk}
    Since $\phi$ is injective, elements of $\mathcal{L}$ can be viewed as elements of $\mathcal{U}$.
\end{rmk}

%%%%%%%%%%%%%%%%%%%%%%%%%%%%%%%%%%%%%%%%%%%%%%%%%%%%%%%%%%%%%%
\subsection{Basic Representation Theory}
%%%%%%%%%%%%%%%%%%%%%%%%%%%%%%%%%%%%%%%%%%%%%%%%%%%%%%%%%%%%%%

Let $\pi: \mathcal{L} \longrightarrow \mathfrak{gl}(V)$ be a finite-dimensional representation of $\mathcal{L}$, that is, $V$ is a finite-dimensional $\mathcal{L}$-module. For $\mu \in \mathcal{H}^*$, define $V_\mu = \{ v \in V \mid H \cdot v = \mu (H) v \text{ for all } H \in \mathcal{H} \}$. 
If $V_\mu \neq 0$, we call it a {\bf weight space} and $\mu$ a {\bf weight} of $V$. Any nonzero vector in $V_\mu$ is called a {\bf weight vector} corresponding to the weight $\mu$. In the case of the adjoint representation, the nonzero weights are precisely the roots. 

Let $\Omega_\pi$ denote the set of all weights of the representation $\pi$. Then, the module $V$ decomposes as 
\[
    V = \coprod_{\mu \in \Omega_\pi} V_\mu,
\]
where $V_\mu$ is the weight space corresponding to the weight $\mu$. 
For $\mu \in \Omega_\pi$ and $\alpha \in \Phi$, if $v \in V_\mu$, then $X_\alpha \cdot v \in V_{\mu + \alpha}$, i.e., the action of $\mathcal{L}_\alpha$ maps $V_\mu$ into $V_{\mu + \alpha}$.

\begin{thm}\label{thm:highestweight}
    \normalfont
    Let $\mathcal{L}$ be a finite-dimensional semisimple Lie algebra over $\mathbb{C}$ and let $\mathcal{H}$ be a maximal toral subalgebra of $\mathcal{L}$. Let $\pi: \mathcal{L} \to \mathfrak{gl}(V)$ be a finite-dimensional irreducible representation of $\mathcal{L}$.  
    \begin{enumerate}[(a)]
        \item\label{(a)} There exists a weight $\lambda \in \Omega_\pi$ of the representation $\pi$ and a corresponding nonzero weight vector $v^+ \in V_\lambda$ such that $X_\alpha v^+ = 0$ for all $\alpha > 0$. The weight $\lambda$ is uniquely determined and is called the {\bf highest weight} of the representation. Similarly, the line spanned by $v^+$ is uniquely determined, and any vector on this line is referred to as a {\bf highest weight vector} (or {\bf maximal vector}).
        
        \item If $\lambda$ is the highest weight as described in part~\ref{(a)}, then $\dim(V_\lambda) = 1$. Moreover, every weight $\mu \in \Omega_\pi$ can be expressed in the form $\mu = \lambda - \sum_{\alpha} \alpha$, where the $\alpha$'s are positive roots (repetition allowed).
    \end{enumerate}
\end{thm}

\begin{prop}
    \normalfont
    Let $\mathcal{L}$ and $\pi$ be as in Theorem~\ref{thm:highestweight}.
    \begin{enumerate}[(a)]
        \item If $\mu \in \Omega_\pi$ then $\mu (H_\alpha) = \langle \mu, \alpha \rangle \in \mathbb{Z}$ for all $\alpha \in \Phi$. That is, $\mu$ satisfies the definition of a weight in the sense of Section~\ref{subsec:weights}. 
        \item Let $\lambda$ be the highest weight of $\pi$. Then $\lambda (H_\alpha) \in \mathbb{Z}_{\geq 0}$ for all $\alpha > 0$. Since $\lambda (H_\alpha) = \langle \lambda, \alpha \rangle$, it follows that $\lambda$ is a dominant weight in the sense of Section~\ref{subsec:weights}
    \end{enumerate} 
\end{prop}

\begin{thm}
    \normalfont
    Let $\lambda \in \mathcal{H}^*$ be a dominant weight as defined in Section~\ref{subsec:weights}. Then, there exists a unique finite-dimensional $\mathcal{L}$-module $V$ such that $\lambda$ is its highest weight.
\end{thm}

Before concluding this subsection, we introduce some useful notations:
Let $\pi: \mathcal{L} \to \mathfrak{gl}(V)$ be a finite-dimensional representation of $\mathcal{L}$. The {\bf weight lattice}\label{nomencl:lattice} $\Lambda_\pi$ of the representation $\pi$ is defined as the $\mathbb{Z}$-module generated by all the weights of $\pi$. 

If \(\pi\) is the adjoint representation, then \(\Lambda_\pi = \Lambda_r\), the root lattice. If \(\pi\) is a representation in which all the dominant fundamental weights lie in \(\Lambda_\pi\), then \(\Lambda_\pi = \Lambda_{sc}\), the fundamental weight lattice (cf. Section~\ref{subsec:weights}). 
For a general representation \(\pi\), we have 
\[
    \Lambda_r \subset \Lambda_\pi \subset \Lambda_{sc}.
\]
Since both \(\Lambda_r\) and \(\Lambda_{sc}\) are \(\mathbb{Z}\)-modules of rank \(\ell\), the lattice \(\Lambda_\pi\) also has rank \(\ell\), where \(\ell\) denotes the rank of the root system \(\Phi\). Furthermore, the quotient \(\Lambda_\pi / \Lambda_r\) is a subgroup of the finite group \(\Lambda_{sc} / \Lambda_r\).
It is known that \(\Lambda_{sc} / \Lambda_r\) is finite, and consequently, \(\Lambda_\pi / \Lambda_r\) is also finite. For a detailed list of \(\Lambda_\pi / \Lambda_r\), please refer to Section~\ref{subsec:weights}.

%%%%%%%%%%%%%%%%%%%%%%%%%%%%%%%%%%%%%%%%%%%%%%%%%%%%%%%%%%%%%%
%%%%%%%%%%%%%%%%%%%%%%%%%%%%%%%%%%%%%%%%%%%%%%%%%%%%%%%%%%%%%%

%%%%%%%%%%%%%%%%%%%%%%%%%%%%%%%%%%%%%%%%%%%%%%%%%%%%%%%%%%%%%%
%%%%%%%%%%%%%%%%%%%%%%%%%%%%%%%%%%%%%%%%%%%%%%%%%%%%%%%%%%%%%%

\chapter{Chevalley Groups}\label{chapter:chevalley groups}

%%%%%%%%%%%%%%%%%%%%%%%%%%%%%%%%%%%%%%%%%%%%%%%%%%%%%%%%%%%%%%
%%%%%%%%%%%%%%%%%%%%%%%%%%%%%%%%%%%%%%%%%%%%%%%%%%%%%%%%%%%%%%

In this chapter, we provide a comprehensive exposition of Chevalley groups and their key properties.
We begin by introducing the Chevalley basis and the corresponding Chevalley algebra.
Next, we discuss the concept of admissible lattices.
Utilizing these concepts, we define Chevalley groups and explore their properties. 
Furthermore, we introduce minimal representations, which play a crucial role in facilitating calculations for Chevalley groups.

The content of this chapter is primarily based on R. Steinberg~\cite{RS}, R. Carter~\cite{RC}, J. E. Humphreys~\cite{JH}, N. A. Vavilov~\cite{NV1}, E. Plotkin and N. A. Vavilov~\cite{EP&NV}, and E. I. Bunina~\cite{EB12:main, EB24:final}.

Throughout this chapter, we continue to use the same notations and conventions introduced in the preceding chapters.

%%%%%%%%%%%%%%%%%%%%%%%%%%%%%%%%%%%%%%%%%%%%%%%%%%%%%%%%%%%%%%
%Section - Chevalley Basis and Chevalley Algebra
%%%%%%%%%%%%%%%%%%%%%%%%%%%%%%%%%%%%%%%%%%%%%%%%%%%%%%%%%%%%%%

\section{Chevalley Basis and Chevalley Algebra}

%%%%%%%%%%%%%%%%%%%%%%%%%%%%%%%%%%%%%%%%%%%%%%%%%%%%%%%%%%%%%%

Let $\mathcal{L} = \mathcal{L}(\Phi, \mathbb{C})$ be a complex semisimple Lie algebra with root system $\Phi$. Let $\mathcal{H}$ be a Cartan subalgebra of $\mathcal{L}$ and consider the root decomposition: $\mathcal{L} = \mathcal{H} \oplus \coprod_{\alpha \in \Phi} \mathcal{L}_\alpha$, where
$$
    \mathcal{L}_\alpha = \{ X \in \mathcal{L} \mid [H,X] = \alpha (H) X, \ \forall H \in \mathcal{H} \}.
$$
Fix a simple system $\Delta = \{ \alpha_1, \dots, \alpha_\ell \}$. As before, for any $\alpha \in \Phi$, we define $H'_\alpha \in \mathcal{H}$ such that $\phi(H) = \kappa (H'_{\alpha},H)$ for all $H \in \mathcal{H}$. Define $H_\alpha = 2 H'_\alpha / \kappa (H'_\alpha, H'_\alpha) \ (\alpha \in \Phi)$ and $H_i = H_{\alpha_i} \ (\alpha_i \in \Delta)$.

\begin{lemma}[{\cite[Lemma~1]{RS}}]
    \normalfont
    For each root $\alpha$, the element $H_\alpha$ is an integral linear combination of the $H_i \ (i=1, \dots, l)$.
\end{lemma}

For each root $\alpha \in \Phi$, choose a nonzero element $X_\alpha \in \mathcal{L}_\alpha$. If $\alpha, \beta \in \Phi$ satisfy $\alpha + \beta \in \Phi$, define $N_{\alpha, \beta}$ by the relation $[X_\alpha, X_\beta] = N_{\alpha, \beta} X_{\alpha + \beta}$.
The constants $N_{\alpha, \beta}$ (for $\alpha, \beta \in \Phi$ with $\alpha + \beta \in \Phi$) are called the \textbf{structure constants}\index{structure constants} of $\mathcal{L}$. They clearly depend upon the choice of root vectors $X_\alpha \ (\alpha \in \Phi)$.
Before proceeding further, we outline some important properties of the structure constants (see, for example, \cite[Theorem 4.1.2]{RC}):
\begin{enumerate}
    \item $N_{\alpha, \beta} = -N_{\beta, \alpha}$ for all $\alpha, \beta \in \Phi$ such that $\alpha + \beta \in \Phi$.
    \item Let $\beta - r\alpha, \dots, \beta + q\alpha$ be the $\alpha$-string through $\beta$. If $q \geq 1$, then 
    \[
    N_{\alpha, \beta} N_{-\alpha, -\beta} = -(r+1)^2.
    \]
    \item If $\alpha, \beta, \gamma \in \Phi$ satisfy $\alpha + \beta + \gamma = 0$, then 
    \[
    \frac{N_{\alpha, \beta}}{\lVert \gamma \rVert^2} = \frac{N_{\beta, \gamma}}{\lVert \alpha \rVert^2} = \frac{N_{\gamma, \alpha}}{\lVert \beta \rVert^2}.
    \]
    \item If $\alpha, \beta, \gamma, \delta \in \Phi$ satisfy $\alpha + \beta + \gamma + \delta = 0$ and if no pairs are opposite, then 
    \[
    \frac{N_{\alpha, \beta} N_{\gamma, \delta}}{\lVert \alpha + \beta \rVert^2} + \frac{N_{\beta, \gamma} N_{\alpha, \delta}}{\lVert \beta + \gamma \rVert^2} + \frac{N_{\gamma, \alpha} N_{\beta, \delta}}{\lVert \gamma + \alpha \rVert^2} = 0.
    \]
\end{enumerate}

%%%%%%%%%%%%%%%%%%%%%%%%%%%%%%%%%%%%%%%%%%%%%%%%%%%%%%%%%%%%%%
\subsection{Chevalley Basis}
%%%%%%%%%%%%%%%%%%%%%%%%%%%%%%%%%%%%%%%%%%%%%%%%%%%%%%%%%%%%%%

We now present Chevalley’s basis theorem for a semisimple Lie algebra $\mathcal{L}$.

\begin{thm}[Chevalley]\label{chevalleybasis}
    \normalfont
    We can choose the root vectors $X_\alpha \in \mathcal{L}_\alpha$ for each $\alpha \in \Phi$ such that the set $\{ X_\alpha, H_i \mid \alpha \in \Phi, i=1, \dots, l \}$ forms a basis of the Lie algebra $\mathcal{L}$ that satisfies the following (integrality) conditions:
    \begin{enumerate}[(a)]
        \item $[H_i, H_j] = 0$.
        \item $[H_i, X_\alpha] = \langle \alpha, \alpha_i \rangle X_\alpha = \alpha (H_i) X_\alpha$.
        \item $[X_\alpha, X_{-\alpha}] = H_\alpha = $ an integral linear combination of the $H_i$.
        \item\label{part d} For $\alpha, \beta \in \Phi$, let $r$ be the largest nonnegative integer such that $\alpha - r \beta \in \Phi$. Then,
        \[
            [X_\alpha, X_\beta] = \begin{cases}
                \pm (r+1) X_{\alpha + \beta} & \text{if } \alpha + \beta \in \Phi; \\
                0 & \text{if } \alpha + \beta \not\in \Phi \cup \{0\}.
            \end{cases}
        \]
        That is, $N_{\alpha, \beta} = - N_{-\alpha, -\beta} = - N_{\beta, \alpha} = \pm (r+1)$.
    \end{enumerate}
    Such a basis is called a \textbf{Chevalley basis}\index{Chevalley basis}.
\end{thm} 

\begin{proof}
    See Theorem~1 of Steinberg \cite{RS} or Theorem~25.2 of Humphreys \cite{JH}.
\end{proof}

One can naturally wonder about the uniqueness of a Chevalley basis. 
Let us delve into this question. 
Once the simple system $\Delta$ is fixed, the elements $H_i$ are uniquely determined. 
However, there is some flexibility in the choice of the root vectors $X_\alpha$. 
Consider another Chevalley basis $\{ H_i, X'_\alpha \mid 1 \leq i \leq l, \alpha \in \Phi \}$, and let $N'_{\alpha, \beta}$ be such that $[X'_\alpha, X'_\beta] = N'_{\alpha, \beta} X'_{\alpha + \beta}$. 
Since $\dim \mathcal{L}_\alpha = 1$, there exists a scalar $c(\alpha) \in k$ such that $X'_\alpha = c(\alpha) X_\alpha$ for all $\alpha \in \Phi$. 
These scalars satisfy the following properties:
\begin{enumerate}[(a)]
    \item $c(\alpha) c(-\alpha) = 1$ for all $\alpha \in \Phi$;
    \item $c(\alpha) c(\beta) = \pm c(\alpha + \beta)$ whenever $\alpha + \beta \in \Phi$.
\end{enumerate}
Conversely, let $\{ H_i, X_\alpha \mid 1 \leq i \leq l, \alpha \in \Phi \}$ be a Chevalley basis of the Lie algebra $\mathcal{L}$, and suppose there exists a function $c: \Phi \to k$ satisfying the above two conditions. Then, the set $\{ H_i, c(\alpha) X_\alpha \mid 1 \leq i \leq l, \alpha \in \Phi \}$
also forms a Chevalley basis of $\mathcal{L}$.

Note that, as mentioned in part \ref{part d} above, we have $[X_\alpha, X_\beta] = \pm (r+1) X_{\alpha + \beta}$ where $\alpha, \beta \in \Phi$ are such that $\alpha + \beta \in \Phi$. However, the argument used to establish this equation left the choice of plus or minus undecided. This is not a coincidence; for instance, by substituting $\begin{pmatrix}
    0 & 0 & -1\\
    0 & 0 & 0 \\
    0 & 0 & 0 
\end{pmatrix}$ for $\begin{pmatrix}
    0 & 0 & 1 \\
    0 & 0 & 0 \\
    0 & 0 & 0 
\end{pmatrix}$ as part of a Chevalley basis for $\mathfrak{sl}(3, \mathbb{C})$ we can see there is no inherent reason to favor one choice over the other. Nevertheless, there does exist an algorithm for making a consistent choice of signs based solely on the knowledge of $\Phi$.

%%%%%%%%%%%%%%%%%%%%%%%%%%%%%%%%%%%%%%%%%%%%%%%%%%%%%%%%%%%%%%
\subsection{Chevalley Algebras}\label{sec:CheAlg}
%%%%%%%%%%%%%%%%%%%%%%%%%%%%%%%%%%%%%%%%%%%%%%%%%%%%%%%%%%%%%%

The integrality property of the Chevalley basis allows us to define the structure of the Lie algebra over an arbitrary (commutative) ring in a natural manner.

The $\mathbb{Z}$-span $\mathcal{L}(\mathbb{Z})$ of a Chevalley basis $\{X_\alpha, H_i\}$ forms a lattice in $\mathcal{L}$, irrespective of the choice of $\Delta$. Moreover, it even becomes a Lie algebra over $\mathbb{Z}$ (in the natural sense) under the bracket operation that is inherited from $\mathcal{L}$.

Let $R$ be a commutative ring with unity. The the tensor product $\mathcal{L}(\Phi, R) = \mathcal{L}(\mathbb{Z}) \otimes_{\mathbb{Z}} R$ is well-defined and is a free $R$-module with basis $\{ X_\alpha \otimes 1, H_i \otimes 1 \}$. Moreover, the bracket operation in $\mathcal{L}(\mathbb{Z})$ induces a natural Lie algebra structure on $\mathcal{L}(\Phi, R)$. We refer to $\mathcal{L}(\Phi, R)$ as a {\bf Chevalley  algebra}\index{Chevalley  algebra} over a ring $R$ (of type $\Phi$).

%%%%%%%%%%%%%%%%%%%%%%%%%%%%%%%%%%%%%%%%%%%%%%%%%%%%%%%%%%%%%%
\subsection{Elementary Chevalley Groups of Adjoint Type}
%%%%%%%%%%%%%%%%%%%%%%%%%%%%%%%%%%%%%%%%%%%%%%%%%%%%%%%%%%%%%%

We now introduce the concept of the adjoint elementary Chevalley group over a commutative ring with unity. We provide a more detailed definition later in Section~\ref{sec:Chevalley groups}.

Let $\text{ad}:\mathcal{L} \longrightarrow gl(\mathcal{L})$ be the adjoint representation defined by $X \mapsto \text{ad}(X)$, where $\text{ad}(X)$ is a linear map from $\mathcal{L}$ to itself given by $(\text{ad}(X)) (Y) = [X, Y]$ for every $Y \in \mathcal{L}$. 
Suppose $X \in \mathcal{L}$ is such that $\text{ad}(X)$ is nilpotent.
Then we can define the exponential map $\text{exp (ad} (X))$ as follows:
\[
    \text{exp (ad} (X)) = \sum_{n=0}^{\infty} \frac{(\text{ad}(X))^n}{n!}.
\]
It follows that $\text{exp (ad}(X))$ is an automorphism (as a vector space) of $\mathcal{L}$.

Now, let us fix a Chevalley basis $\{X_\alpha, H_i\}$ of $\mathcal{L}$. Since, for each $\alpha \in \Phi$, the map $\text{ad}(X_{\alpha})$ is a nilpotent endomorphism of $\mathcal{L}$, it follows that $\text{exp (ad}(X_\alpha))$ is well-defined and forms an automorphism of $\mathcal{L}$. The following proposition tells us that we can also consider $\text{exp (ad}(X_\alpha))$ as an automorphism of the $\mathbb{Z}$-algebra $\mathcal{L}(\mathbb{Z})$.

\begin{prop}[{\cite[Proposition 25.5]{JH}}]
    \normalfont
    Let $\alpha \in \Phi, \, m \in \mathbb{Z}^+$. Then $(\text{ad } X_\alpha)^m / m!$ leaves $\mathcal{L}(\mathbb{Z})$ invariant. In particular, $\text{exp} (t \text{ ad} (X_\alpha)) \ (t \in \mathbb{Z})$ also leaves $\mathcal{L}(\mathbb{Z})$ invariant.
\end{prop}

Now, let $R$ be a commutative ring with unity.
In light of the aforementioned proposition, we can consider $\text{exp} (t \, \text{ad} (X_\alpha)) \ (t \in R)$ as an automorphism of $\mathcal{L}(\Phi, R)$.
Our objective is to study these types of automorphisms of $\mathcal{L}(\Phi, R)$.

The subgroup $E_{\text{ad}} (\Phi, R)$ of $\text{Aut}_R(\mathcal{L}(\Phi, R))$ generated by all the automorphisms of $\mathcal{L}(\Phi, R)$ of the form $\text{exp} (t \text{ ad} (X_\alpha)) \ (t \in R, \, \alpha \in \Phi)$ is called an {\bf adjoint elementry Chevalley group}\index{elementary Chevalley group!adjoint} of type $\Phi$ over a ring $R$. 
We will discover in the future that the group $E_{\text{ad}} (\Phi,R)$ is independent of the choice of a Chevalley basis.

%%%%%%%%%%%%%%%%%%%%%%%%%%%%%%%%%%%%%%%%%%%%%%%%%%%%%%%%%%%%%%

\subsubsection*{Action of $E_{\text{ad}} (\Phi,R)$ on $\mathcal{L}(\Phi, R)$}

\begin{lemma}[{\cite[Lemma 4.5.1]{RC}}]\label{lemma:actionofexpadx}
    \normalfont
    Let $\mathcal{L}(\Phi, R)$ be a standard semisimple Lie algebra over a ring $R$ (i.e., viewed as a subalgebra of $\mathfrak{gl}_n(R)$). Suppose $X \in \mathcal{L}(\Phi, R)$ be such that ad $(X)$ is nilpotent. Then $$ \text{exp } (\text{ad }X) \cdot Y = \text{exp } (X) \hspace{1mm} Y \hspace{1mm} (\text{exp } (X))^{-1}$$ for all $Y \in \mathcal{L} (\Phi, R)$. 
    %Note that the Lie bracket of $\mathcal{L}$ is given by $[X,Y]=XY - YX$ and the product on the right-hand side is the usual product of matrices.
\end{lemma}

\begin{ex}
    Let $k$ be a field. We aim to show that $E_{\text{ad}}(A_1, k)$ is isomorphic to $PSL_2(k)$. 
    Consider a Chevalley basis for $\mathcal{L} = \mathfrak{sl}_2(\mathbb{C})$, given by $\{ X_\alpha, H_\alpha, X_{-\alpha} \}$, where 
    \[
        X_\alpha = \begin{pmatrix} 0 & 1 \\ 0 & 0 \end{pmatrix}, \quad 
        H_\alpha = \begin{pmatrix} 1 & 0 \\ 0 & -1 \end{pmatrix}, \quad 
        X_{-\alpha} = \begin{pmatrix} 0 & 0 \\ 1 & 0 \end{pmatrix}.
    \]
    By definition, the group $E_{\text{ad}}(A_1, k)$ is generated by the automorphisms $\exp(\text{ad }X_\alpha)$ and $\exp(\text{ad }X_{-\alpha})$ of $\mathcal{L}(A_1, k)$, whose action on $\mathcal{L}(A_1, k)$ is as described in Lemma \ref{lemma:actionofexpadx}. 
    Observe that: 
    \[
        \exp(tX_\alpha) = \begin{pmatrix} 1 & t \\ 0 & 1 \end{pmatrix}, \quad 
        \exp(tX_{-\alpha}) = \begin{pmatrix} 1 & 0 \\ t & 1 \end{pmatrix}.
    \]
    The matrices 
    \[
        \begin{pmatrix} 1 & t \\ 0 & 1 \end{pmatrix} \quad \text{and} \quad 
        \begin{pmatrix} 1 & 0 \\ t & 1 \end{pmatrix}, 
    \]
    as $t$ varies over $k$, generate $SL_2(k)$. Thus, there is a surjective homomorphism 
    \[
        SL_2(k) \longrightarrow E_{\text{ad}}(A_1, k),
    \]
    where the image of $A \in SL_2(k)$ is the automorphism $X \mapsto AXA^{-1}$ of $\mathcal{L}(A_1, k)$. Under this homomorphism, we have:
    \[
        \begin{pmatrix} 1 & t \\ 0 & 1 \end{pmatrix} \mapsto \exp(t \, \text{ad}(X_\alpha)), \quad 
        \begin{pmatrix} 1 & 0 \\ t & 1 \end{pmatrix} \mapsto \exp(t \, \text{ad}(X_{-\alpha})).
    \]
    The kernel of this homomorphism consists of all $A \in SL_2(k)$ that commute with every $X \in \mathcal{L}(k)$. This is easily seen to be $\{\pm I_2\}$. Therefore, $E_{\text{ad}}(A_1, k)$ is isomorphic to $PSL_2(k)$.
\end{ex}

%%%%%%%%%%%%%%%%%%%%%%%%%%%%%%%%%%%%%%%%%%%%%%%%%%%%%%%%%%%%%%
%Section - Admissible Lattice
%%%%%%%%%%%%%%%%%%%%%%%%%%%%%%%%%%%%%%%%%%%%%%%%%%%%%%%%%%%%%%

\section{Admissible Lattice}

We wish to generalize the concept of elementary Chevalley groups of adjoint type introduced in the previous section. 
Specifically, we intend to replace $\exp(t \, \text{ad}(X_\alpha))$ with $\exp(t \, \pi(X_\alpha))$, where $\pi$ is an arbitrary representation of the corresponding Lie algebra $\mathcal{L}$. 
In this section, we establish the essential framework to ensure that this generalization is well-defined.

%%%%%%%%%%%%%%%%%%%%%%%%%%%%%%%%%%%%%%%%%%%%%%%%%%%%%%%%%%%%%%
\subsection{Kostant's Theorem}
%%%%%%%%%%%%%%%%%%%%%%%%%%%%%%%%%%%%%%%%%%%%%%%%%%%%%%%%%%%%%%

Let $\mathcal{L}$ be a Lie algebra over a field $\mathbb{C}$ and let $\mathcal{U} := \mathcal{U}(\mathcal{L})$ denote its universal enveloping algebra (cf. page~\pageref{uni_env_algebra}). 

\begin{defn}
    The $\mathbb{Z}$-algebra $\mathcal{U}_{\mathbb{Z}} := \mathcal{U}_{\mathbb{Z}}(\mathcal{L})$ generated by all $X_\alpha^m / m! \ (\alpha \in \Phi, \, m \in \mathbb{Z}_{\geq 0})$ is called \textbf{Kostant's $\mathbb{Z}$-form}\index{Kostant's $\mathbb{Z}$-form}.
\end{defn}

\begin{thm}[Kostant]\label{thm:kostant}
    \normalfont
    Let $\{ H_i, X_\alpha \}$ be a Chevalley basis of a Lie algebra $\mathcal{L}$ over $\mathbb{C}$. Then the collection 
    \[
        \left\{ 
            \prod_{\alpha \in \Phi^{+}} \frac{X_\alpha^{m(\alpha)}}{m(\alpha)!} \ 
            \prod_{i=1}^{\ell} \binom{H_i}{n_i} \ 
            \prod_{\alpha \in \Phi^{-}} \frac{X_\alpha^{p(\alpha)}}{p(\alpha)!} 
            \ \bigg| \ m(\alpha), n_i, p(\alpha) \in \mathbb{Z}_{\geq 0} 
        \right\}
    \]
    forms a $\mathbb{Z}$-basis for $\mathcal{U}_\mathbb{Z}$. Here, 
    \[
        \binom{H_i}{n_i} = \frac{H_i (H_i - 1) \cdots (H_i - (n_i - 1))}{(n_i)!},
    \]
    and the products above are taken in some fixed order of roots.
\end{thm}

\begin{proof}
    See Theorem~2 of Steinberg~\cite{RS} or Theorem~26.4 of Humphreys~\cite{JH}.
\end{proof}

\begin{rmk}
    \normalfont
    If $\mathcal{U}_{\mathbb{Z}}^{+}$, $\mathcal{U}_{\mathbb{Z}}^{-}$, and $\mathcal{U}_{\mathbb{Z}}^0$ are the $\mathbb{Z}$-algebras generated by $\frac{X_\alpha^m}{m!} \ (\alpha \in \Phi^{+})$, $\frac{X_\alpha^m}{m!} \ (\alpha \in \Phi^{-})$, and $\binom{H_i}{n_i}$, respectively, then 
    \[
        \mathcal{U}_\mathbb{Z} = \mathcal{U}_\mathbb{Z}^- \mathcal{U}_\mathbb{Z}^0 \mathcal{U}_\mathbb{Z}^+.
    \]
\end{rmk}

%%%%%%%%%%%%%%%%%%%%%%%%%%%%%%%%%%%%%%%%%%%%%%%%%%%%%%%%%%%%%%
\subsection{Admissible Lattice}
%%%%%%%%%%%%%%%%%%%%%%%%%%%%%%%%%%%%%%%%%%%%%%%%%%%%%%%%%%%%%%

Let $V$ be a (finite-dimensional) vector space over $\mathbb{C}$. A subgroup $M$ of the additive group $V$ is said to be a \textbf{lattice} in $V$ if there exists a $\mathbb{C}$-basis of $V$ which is also a $\mathbb{Z}$-basis of $M$.

Let $\mathcal{L}$ be a finite-dimensional Lie algebra over the field $\mathbb{C}$. Suppose $\pi: \mathcal{L} \to \mathfrak{gl}(V)$ is a finite-dimensional representation of $\mathcal{L}$, that is, $V$ is a finite-dimensional $\mathcal{L}$-module. This representation extends naturally to the universal enveloping algebra $\mathcal{U}$, endowing $V$ with the structure of a $\mathcal{U}$-module. Consequently, the Kostant $\mathbb{Z}$-form $\mathcal{U}_{\mathbb{Z}}$ also acts on $V$, making $V$ a $\mathcal{U}_{\mathbb{Z}}$-module.

\begin{prop}[{\cite[Theorem~27.1]{JH}}]\label{prop:admissible lattice}
    \normalfont
    Let $V$ be a finite-dimensional $\mathcal{L}$-module. Then
    \begin{enumerate}[(a)]
            \item Any subgroup of $V$ invariant under $\mathcal{U}_{\mathbb{Z}}$ is the direct sum of its weight components.
            \item $V$ contains a lattice $M$ invariant under all $X_\alpha^m/m! \hspace{2mm} (\alpha \in \Phi, m \in \mathbb{Z}^+);$ i.e., $M$ is invariant under $\mathcal{U}_\mathbb{Z}$.
        \end{enumerate}
\end{prop}

\begin{rmk}
    \normalfont
    A lattice $M$ in a finite-dimensional $\mathcal{L}$-module $V$ that is invariant under the action of $\mathcal{U}_\mathbb{Z}$ is called \textbf{admissible}\index{admissible lattice}. 
    Part (b) of the theorem mentioned above ensures the existence of such a lattice.
\end{rmk}

\begin{ex}
    Note that $\mathcal{L} (\mathbb{Z})$, the $\mathbb{Z}$-span of a Chevalley basis, is an admissible lattice in the $\mathcal{L}$-module $\mathcal{L}$. 
\end{ex}

\begin{prop}[{\cite[Proposition~27.2]{JH}}]
    \normalfont
    Let $\pi: \mathcal{L} \longrightarrow \mathfrak{gl}(V)$ be a finite-dimensional faithful representation of $\mathcal{L}$, and let $M$ be an admissible lattice in $V$. Define $\mathcal{L}_\pi(\mathbb{Z})$ as the stabilizer of $M$ in $\mathcal{L}$. Then:
    \begin{enumerate}[(a)]
        \item $\mathcal{L}_\pi(\mathbb{Z})$ is an admissible lattice in the $\mathcal{L}$-module $\mathcal{L}$.
        \item Moreover, 
        \[
            \mathcal{L}_\pi(\mathbb{Z}) = \mathcal{H}_\pi(\mathbb{Z}) \oplus \coprod_{\alpha \in \Phi} \mathbb{Z} X_\alpha,
        \]
        where 
        \[
            \mathcal{H}_\pi(\mathbb{Z}) = \mathcal{H} \cap \mathcal{L}_\pi(\mathbb{Z}) = \{ H \in \mathcal{H} \mid \mu(H) \in \mathbb{Z} \text{ for all } \mu \in \Omega_\pi \}.
        \]
        \item In particular, $\mathcal{L}(\mathbb{Z}) \subset \mathcal{L}_\pi (\mathbb{Z})$ and $\mathcal{L}_\pi(\mathbb{Z})$ is independent of the choice of the admissible lattice \( M \). (However, \( \mathcal{L}_\pi(\mathbb{Z}) \) does depend on the representation \( \pi \).)
    \end{enumerate}
\end{prop}

\begin{ex}
    Let $\mathcal{L} = \mathfrak{sl}(2,\mathbb{C})$, with the standard basis $(X,Y,H)$. 
    \begin{enumerate}[(a)]
        \item Let $\pi: \mathcal{L} \longrightarrow \mathfrak{gl}(\mathcal{L})$ be the adjoint representation of $\mathcal{L}$. Define 
        \(
            M = \mathcal{L} (\mathbb{Z}) = \mathbb{Z}X + \mathbb{Z}Y + \mathbb{Z}H.
        \) 
        Then $M$ is an admissible lattice, and the corresponding stabilizer is given by
        \[
            \mathcal{L}_\pi (\mathbb{Z}) = \mathbb{Z}X + \mathbb{Z}Y + \mathbb{Z}(H/2).
        \]
        \item Let $V = \mathbb{C}^2$, and let $\pi: \mathcal{L} \longrightarrow \mathfrak{gl}(V)$ be the usual representation of $\mathcal{L}$. 
        Then $M = \mathbb{Z}^2$ is an admissible lattice, and the corresponding stabilizer is given by
        \[
            \mathcal{L}_\pi (\mathbb{Z}) = \mathbb{Z}X + \mathbb{Z}Y + \mathbb{Z}H.
        \]
    \end{enumerate}
\end{ex}

%%%%%%%%%%%%%%%%%%%%%%%%%%%%%%%%%%%%%%%%%%%%%%%%%%%%%%%%%%%%%%
\subsection{Passage to an Arbitrary Ring}
%%%%%%%%%%%%%%%%%%%%%%%%%%%%%%%%%%%%%%%%%%%%%%%%%%%%%%%%%%%%%%

Let $\mathcal{L}$ be a semisimple Lie algebra over $\mathbb{C}$ with the root system $\Phi$ and let $\pi: \mathcal{L} \longrightarrow \text{ GL} (V)$ be a finite-dimensional faithful representation. 
Consider the weight space decomposition 
\[
    V = \coprod_{\mu \in \Omega_\pi} V_{\mu},
\]
where $\Omega_\pi$ is the set of weights of the representation $\pi$. 

Let $M$ be an admissible lattice in $V$, and define $M_\mu = M \cap V_\mu$ (where $\mu \in \Omega_\pi$). 
Additionally, let $\mathcal{L}(\mathbb{Z})$, $\mathcal{L}_\pi (\mathbb{Z})$, and $\mathcal{H}_\pi (\mathbb{Z})$ be defined as before. 
To enhance clarity and avoid ambiguity, we sometimes explicitly include the root system $\Phi$ in the notation. 
For instance, we write $\mathcal{L}(\Phi, \mathbb{Z})$, $\mathcal{L}_\pi(\Phi, \mathbb{Z})$, and $\mathcal{H}_\pi(\Phi, \mathbb{Z})$, where $\Phi$ signifies the underlying root system associated with the Lie algebra.

Let $R$ be a commutative ring with unity. Define $V_\pi (\Phi, R) = M \otimes_{\mathbb{Z}} R,$ $V_{\pi, \mu}(\Phi, R) = M_\mu \otimes_{\mathbb{Z}} R,$ $\mathcal{L} (\Phi, R) = \mathcal{L} (\Phi, \mathbb{Z}) \otimes_{\mathbb{Z}} R,$ $\mathcal{L}_\pi (\Phi, R) = \mathcal{L}_\pi (\Phi, \mathbb{Z}) \otimes_{\mathbb{Z}} R,$ $\mathcal{H}_\pi (\Phi, R) = \mathcal{H}_\pi (\Phi, \mathbb{Z}) \otimes_{\mathbb{Z}} R$ and $R X_{\alpha} = \mathbb{Z} X_{\alpha} \otimes_{\mathbb{Z}} R$. 

\begin{prop}[{\cite[Corollary 3 on page 17]{RS}}]
    \normalfont
    Granting the above notations, the following hold: 
    \begin{enumerate}[(a)]
        \item $V_\pi (\Phi, R) = \coprod_{\mu \in \Omega_\pi} V_{\pi, \mu} (\Phi, R)$ and $\dim_{R} V_{\pi, \mu} (\Phi, R) = \dim_{\mathbb{C}} V_{\mu}$.
        \item $\mathcal{L}_\pi (\Phi, R) = \mathcal{H}_\pi (\Phi, R) \oplus \coprod_{\alpha \in \Phi} R X_{\alpha},$ $\dim_{R} (R X_{\alpha}) = 1,$ $\dim_{R} \mathcal{H}_\pi (\Phi, R) = \dim_{\mathbb{C}} (\mathcal{H}),$ and $\dim_{R} \mathcal{L}_\pi (\Phi, R) = \dim_{\mathbb{C}} (\mathcal{L})$.
    \end{enumerate}
\end{prop}

If $V$ is an irreducible $\mathcal{L}$-module with the highest weight $\lambda$, then $V_\pi(\Phi, R)$ is called the \textbf{Weyl module}\index{Weyl module} of the Chevalley algebra $\mathcal{L}(\Phi, R)$ with the highest weight $\lambda$

%%%%%%%%%%%%%%%%%%%%%%%%%%%%%%%%%%%%%%%%%%%%%%%%%%%%%%%%%%%%%%
%Section - Chevalley Groups 
%%%%%%%%%%%%%%%%%%%%%%%%%%%%%%%%%%%%%%%%%%%%%%%%%%%%%%%%%%%%%%

\section{Chevalley Groups}\label{sec:Chevalley groups}

We adopt the notation established in the preceding section. With these notations in place, we are now ready to study the automorphisms of $V_\pi (\Phi, R)$ of the form 
\[
    x_\alpha(t) := \exp(t \pi(X_\alpha)) \quad (t \in R, \alpha \in \Phi),
\]
where the exponential function is given by
\[
    \exp(t \pi(X_\alpha)) = \sum_{n=0}^{\infty} \frac{t^n \pi(X_\alpha)^n}{n!}.
\]

We now demonstrate that the action of $x_\alpha(t)$ on $V_\pi(\Phi, R)$ is well-defined. 
To do this, we denote the action of $\pi(X_\alpha)$ on $v \in V_\pi(\Phi, R)$ as $X_\alpha \cdot v$. Thus, we can express $x_{\alpha}(t)$ as
\[
    x_{\alpha}(t) = \exp(t X_\alpha) = \sum_{n=0}^{\infty} \frac{t^n X_\alpha^n}{n!}.
\]
Since ${X_\alpha^n}/{n!} \in \mathcal{U}_\mathbb{Z}$, it acts on $M$ (see Proposition~\ref{prop:admissible lattice}). This induces an action of $\lambda^n X^n_{\alpha}/n!$ on $M \otimes_{\mathbb{Z}} \mathbb{Z}[\lambda]$, where $\lambda$ is an indeterminate. 
Since $X_\alpha^n$ acts as zero for sufficiently large $n$, we observe that $\sum_{n=0}^{\infty} {\lambda^n X_\alpha^n}/{n!}$ acts on $M \otimes_{\mathbb{Z}} \mathbb{Z}[\lambda]$ and hence on \( M \otimes_{\mathbb{Z}} \mathbb{Z}[\lambda] \otimes_{\mathbb{Z}} R \).
Following this last action by the homomorphism from $M \otimes_\mathbb{Z} \mathbb{Z}[\lambda] \otimes_\mathbb{Z} R$ to $V_\pi(\Phi, R) = M \otimes_{\mathbb{Z}} R$ given by $\lambda \mapsto t$ we obtain an action of $\sum_{n=0}^{\infty} t^n X^n_{\alpha}/n!$ on $V_\pi(\Phi, R)$.

%%%%%%%%%%%%%%%%%%%%%%%%%%%%%%%%%%%%%%%%%%%%%%%%%%%%%%%%%%%%%%
\subsection{Elementary Chevalley Groups}
%%%%%%%%%%%%%%%%%%%%%%%%%%%%%%%%%%%%%%%%%%%%%%%%%%%%%%%%%%%%%%

The subgroup of $\text{Aut}_R(V_\pi(\Phi, R))$ generated by all $x_\alpha(t) \ (t \in R, \alpha \in \Phi)$ is called the \textbf{elementary Chevalley group}\index{elementary Chevalley group}\label{nomencl:E(R)} of type $\Phi$ over the ring $R$, denoted by $E_\pi(\Phi, R)$. 
For a representation $\pi$, let $\Lambda_\pi$ denote the weight lattice of $\pi$, i.e., the lattice generated by all weights of $\pi$. 
If $\pi$ and $\pi'$ are representations of $\mathcal{L}$ such that $\Lambda_\pi = \Lambda_{\pi'}$, then $E_\pi (\Phi, R) \cong E_{\pi'} (\Phi, R)$.
Let $\Lambda_r$ be the lattice generated by roots, and let $\Lambda_{\text{sc}}$ be the lattice generated by fundamental weights. 
If $\pi$ is such that $\Lambda_\pi = \Lambda_{r}$ (respectively, $\Lambda_\pi = \Lambda_{\text{sc}}$), then $E_\pi(\Phi, R) = E_{\text{ad}}(\Phi, R)$ (respectively, $E_\pi(\Phi, R) = E_{\text{sc}}(\Phi, R)$) is called an  \textbf{adjoint elementary Chevalley group}\index{elementary Chevalley group!adjoint} (respectively, \textbf{universal} (or \textbf{simply connected}) \textbf{elementary Chevalley group}\index{elementary Chevalley group!universal or simply connected}). 

%%%%%%%%%%%%%%%%%%%%%%%%%%%%%%%%%%%%%%%%%%%%%%%%%%%%%%%%%%%%%%

\subsubsection*{The Root Subgroup $\mathfrak{X}_\alpha$}

Fix a root $\alpha \in \Phi$. Define $\mathfrak{X}_\alpha$ as the subgroup of $E_\pi(\Phi, R)$ generated by the elements $x_\alpha(t) \ (t \in R)$, referred to as the \textbf{root subgroup}\index{root subgroup} corresponding to the root $\alpha$.
It is straightforward to verify that 
\[
    x_\alpha(t)x_\alpha(u) = x_\alpha(t+u),
\]
for all $t, u \in R$.
As a result, the natural map $R \longrightarrow \mathfrak{X}_\alpha$, defined by $t \mapsto x_\alpha(t)$, is an isomorphism of the additive group of $R$ onto the root subgroup $\mathfrak{X}_\alpha$. In particular, we have $\mathfrak{X}_\alpha = \{x_\alpha(t) \mid t \in R\}$.
Furthermore, if $\beta$ is any other root different from $\alpha$, then $\mathfrak{X}_\alpha \neq \mathfrak{X}_\beta$.

%%%%%%%%%%%%%%%%%%%%%%%%%%%%%%%%%%%%%%%%%%%%%%%%%%%%%%%%%%%%%%

\subsubsection*{Chevalley Commutator Formula} \index{Chevalley commutator formulas}

We now present a key relation in the study of Chevalley groups. For any group $G$ and elements $x, y \in G$, we define the commutator by $[x, y] = xyx^{-1}y^{-1}$.

\begin{prop}[{\cite[Lemma 15]{RS}}]\label{prop:chevalley commutator formula}
    \normalfont
    Let $\alpha, \beta$ be roots such that $\alpha + \beta \neq 0$. Then the following identity holds:
    \[
        [x_\alpha(t), x_\beta(u)] = \prod_{i,j > 0} x_{i\alpha + j\beta}\left(c_{\alpha, \beta; i, j} t^i u^j\right),
    \]
    where the product on the right-hand side is taken over all roots of the form $i\alpha + j\beta$ with $i, j \in \mathbb{Z}^+$, arranged in some fixed (but arbitrary) order. The constants $c_{\alpha, \beta; i, j}$ are unique integers depending on $\alpha$, $\beta$, and the chosen ordering, but are independent of $t$ and $u$.
\end{prop}

The integers $c_{\alpha, \beta; i, j}$ are referred to as the \textbf{structure constants} of the Chevalley group. 
It is known that they belong to the set $\{\pm 1, \pm 2, \pm 3\}$. 
When the roots $i\alpha + j\beta$ are ordered according to increasing values of $i + j$, the structure constants satisfy the following relations (see, for example, Section 9 of \cite{EP&NV}):
\[
    c_{\alpha, \beta; i, 1} = M_{\alpha, \beta; i}, \quad 
    c_{\alpha, \beta; 1, j} = -M_{\beta, \alpha; j}, \quad
    c_{\alpha, \beta; 2, 3} = \frac{1}{3} M_{\alpha + \beta, \beta; 2}, \quad
    c_{\alpha, \beta; 3, 2} = \frac{2}{3} M_{\alpha + \beta, \alpha; 2},
\]
where
\[
    M_{\gamma, \delta; k} = \frac{1}{k!} \prod_{i=1}^{k} N_{\gamma, \delta + (i-1)\gamma} = \pm \binom{r+k}{k},
\]
and $r$ is the largest nonnegative integer such that $\delta - r\gamma$ is a root.

\begin{rmk}
    The following are the immediate consequences of the above:
    \begin{enumerate}
        \item If $\alpha, \beta \in \Phi$ such that $\alpha + \beta \not\in \Phi \cup \{0\}$, then $x_{\alpha}(t)$ and $x_{\beta}(u)$ commutes for all $t, u \in R$.
        \item If $\alpha, \beta \in \Phi$ such that $\alpha + \beta \in \Phi$ but $i \alpha + j \beta \not\in \Phi$ for all $i,j$ such that $i+j > 2$, then 
        \[
            [x_\alpha(t), x_\beta(u)] = x_{\alpha + \beta} \left(N_{\alpha, \beta} t u \right).
        \]
    \end{enumerate}
    In particular, if $\Phi$ is a root system with all roots having the same length, the Chevalley commutator formula simplifies to:
    \[
        [x_\alpha(t), x_\beta(u)] = \begin{cases}
            x_{\alpha + \beta} (N_{\alpha, \beta} t u) & \text{if } \alpha, \beta \in \Phi \text{ such that } \alpha + \beta \in \Phi, \\
            1 & \text{if } \alpha, \beta \in \Phi \text{ such that } \alpha + \beta \not\in \Phi \cup \{0\}.
        \end{cases} 
    \]
\end{rmk}

%%%%%%%%%%%%%%%%%%%%%%%%%%%%%%%%%%%%%%%%%%%%%%%%%%%%%%%%%%%%%%

\subsubsection*{The subgroups $U$ and $U^{-}$}\label{U and U-}

Let $S$ be a closed subset of $\Phi$ (cf. page~\pageref{subsubsec:closed subsets}). Define the subgroup $\mathfrak{X}_S$ of $E_\pi (\Phi, R)$ generated by all $x_{\alpha} (t) \ (t \in R, \, \alpha \in S)$. We sometimes use the notation $E_\pi(S, R)$ instead of $\mathfrak{X}_S$ to align with the notation used in \cite{EP&NV}.

\begin{lemma}[{\cite[Lemma 17]{RS}}]
    \normalfont
    Let $S$ be a special closed subset of $\Phi$. 
    \begin{enumerate}[(a)]
        \item Every element of $\mathfrak{X}_S$ can be written uniquely as $\prod_{\alpha \in S} x_{\alpha} (t_\alpha)$ where $t_\alpha \in R$ and the product is taken in any fixed order.
        \item If $I$ be an ideal in $S$, then $\mathfrak{X}_I$ is a normal subgroup of $\mathfrak{X}_S$.
        \item If $S = P \cup Q$ with $P$ and $Q$ closed sets such that $P \cap Q = \emptyset$, then $$\mathfrak{X}_S = \mathfrak{X}_P \, \mathfrak{X}_Q.$$
    \end{enumerate}
\end{lemma}

Now, set
\[
    U = U_{\pi} (\Phi, R) := \mathfrak{X}_{\Phi^+} = E_\pi (\Phi^{+}, R) \text{ and } U^{-} = U^{-}_{\pi} (\Phi, R) := \mathfrak{X}_{\Phi^-} = E_\pi (\Phi^{-}, R).
\]

\begin{cor}
    \normalfont
    Every element of $U_\pi (\Phi, R)$ (respectively, $U^{-}_\pi (\Phi, R)$) can be written uniquely as $\prod_{\alpha \in \Phi^{+}} x_{\alpha} (t_\alpha)$ (respectively, $\prod_{\alpha \in \Phi^{-}} x_{\alpha} (t_\alpha)$) where $t_\alpha \in R$ and the product is taken in any fixed order.
\end{cor}

\begin{cor}
    \normalfont
    For $r \in \mathbb{Z}^+$, define $U_r := \mathfrak{X}_{\Phi_r} = E_\pi (\Phi_r, R)$ (for definition of $\Phi_r$ see page~\pageref{subsubsec:closed subsets}). Then
    \begin{enumerate}[(a)]
        \item $U_r$ is normal subgroup of $U$.
        \item $[U_r, U_s] \subset U_{r+s}$.
        \item $U$ is nilpotent.
    \end{enumerate}
\end{cor}

%%%%%%%%%%%%%%%%%%%%%%%%%%%%%%%%%%%%%%%%%%%%%%%%%%%%%%%%%%%%%%

\subsubsection*{The subgroup $\langle \mathfrak{X}_\alpha, \mathfrak{X}_{-\alpha} \rangle$}

Let $E_2(R)$ denote the subgroup of $SL_2(R)$ generated by the matrices 
\[
\begin{pmatrix}
    1 & t \\
    0 & 1
\end{pmatrix}
\quad \text{and} \quad
\begin{pmatrix}
    1 & 0 \\
    t & 1
\end{pmatrix},
\]
where $t \in R$. 
It is a known result that if $R$ is a field, then $E_2(R) = SL_2(R)$. 

Now fix a root $\alpha \in \Phi$, and let $E_\alpha(R)$ denote the subgroup of $E_\pi(\Phi, R)$ generated by $x_{\alpha}(t)$ and $x_{-\alpha}(t)$ for all $t \in R$. That is,
\[
E_\alpha(R) = \langle \mathfrak{X}_\alpha, \mathfrak{X}_{-\alpha} \rangle,
\]
where $\mathfrak{X}_\alpha = \{x_{\alpha}(t) \mid t \in R\}$ and $\mathfrak{X}_{-\alpha} = \{x_{-\alpha}(t) \mid t \in R\}$.

\begin{prop}\label{prop:E_2 in E_alpha}
    \normalfont
    There exists an onto homomorphism $\phi_\alpha : E_2 (R) \to E_\alpha (R)$ such that 
    \[
        \phi_\alpha \begin{pmatrix}
            1 & t \\
            0 & 1
        \end{pmatrix} = x_{\alpha} (t) 
        \quad \text{and} \quad
        \phi_\alpha \begin{pmatrix}
            1 & 0 \\
            t & 1
        \end{pmatrix} = x_{-\alpha} (t).
    \]
\end{prop}

\begin{proof}
    The proof immediately follows from the facts stated in Section 12 of \cite{EP&NV}.
\end{proof}

Observe that if $t \in R^*$ then 
\[
    \begin{pmatrix}
        0 & t \\
        -t^{-1} & 0
    \end{pmatrix} \in E_2(R)
    \quad \text{and} \quad
    \begin{pmatrix}
        t & 0 \\
        0 & t^{-1} 
    \end{pmatrix} \in E_2(R).
\]
For $t \in R^*$, we define 
\[
    w_{\alpha} (t) := \phi_\alpha \begin{pmatrix}
        0 & t \\
        -t^{-1} & 0
    \end{pmatrix}
    \quad \text{and} \quad
    h_{\alpha} (t) := \phi_\alpha \begin{pmatrix}
        t & 0 \\
        0 & t^{-1} 
    \end{pmatrix}.
\]

\begin{prop}
    \normalfont
    For all $t,u \in R^*$, we have 
    \begin{enumerate}[(a)]
        \item $w_{\alpha} (t) = x_{\alpha}(t) x_{-\alpha}(-t^{-1}) x_{\alpha}(t)$ and $h_{\alpha} (t) = w_{\alpha}(t) w_{\alpha}(1)^{-1}$.
        \item $h_\alpha (t) h_\alpha (u) = h_\alpha (t u)$ and $w_\alpha (t) w_\alpha (u) = h_\alpha (- t u^{-1})$. 
        \item $h_\alpha (t)^{-1} = h_{-\alpha} (t)$ and $w_\alpha (t)^{-1} = w_{\alpha} (-t) = w_{-\alpha}(t^{-1})$.
    \end{enumerate}
\end{prop}

\begin{proof}
    Using the definitions and Proposition~\ref{prop:E_2 in E_alpha}, the proof follows from straightforward calculations in $E_2 (R)$.
\end{proof}

%%%%%%%%%%%%%%%%%%%%%%%%%%%%%%%%%%%%%%%%%%%%%%%%%%%%%%%%%%%%%%

\subsubsection*{The action of $w_\alpha (t)$ and $h_\alpha (t)$ on $V_\pi (\Phi, R)$}

Recall that, we have the weight space decomposition of a Weyl module $$V_\pi (\Phi, R) = \coprod_{\mu \in \Omega_\pi} V_{\pi, \mu} (\Phi, R).$$

\begin{prop}[{\cite[Lemma 19, (b) and (c)]{RS}}]
    \normalfont
    Let $v \in V_{\pi, \mu} (\Phi, R)$. Then for any $t \in R^{*}$,
    \begin{enumerate}[(a)]
        \item there exists $v' \in V_{\pi, s_{\alpha}(\mu)} (\Phi, R)$ independent of $t$ such that $w_\alpha (t) v = t^{-\langle \mu, \alpha \rangle} v'$.
        \item $h_\alpha (t) v = t^{\langle \mu, \alpha \rangle} v$, i.e., $h_\alpha (t)$ acts ``diagonally" on $V_{\pi, \mu} (\Phi, R)$ as multiplication by $t^{\langle \mu, \alpha \rangle}$.
    \end{enumerate}
\end{prop}

\begin{cor}[{\cite[Lemma 19, (a)]{RS}}]\label{cor:c(alpha,beta)}
    For $\alpha, \beta \in \Phi$ and $t \in R^*$, we have
    \[
        w_\alpha (t) X_\beta w_\alpha (t)^{-1} = c \, t^{-\langle \beta, \alpha \rangle} X_{s_\alpha (\beta)},
    \]
    where $c = c(\alpha, \beta) = \pm 1$ is independent of $t$ or $R$ and the representation chosen.
\end{cor}

%%%%%%%%%%%%%%%%%%%%%%%%%%%%%%%%%%%%%%%%%%%%%%%%%%%%%%%%%%%%%%

\subsubsection*{Steinberg Relations}

Consider the generators $x_{\alpha}(t) \ (t \in R, \alpha \in \Phi)$ of $E_\pi (\Phi, R)$. Recall that for $t \in R^*$ and $\alpha \in \Phi$, we have $w_\alpha (t) = x_\alpha (t) x_{-\alpha}(-t^{-1}) x_{\alpha}(t)$ and $h_\alpha (t) = w_\alpha(t) w_\alpha(1)^{-1}$. Then the following relations $(R)$, known as the \textbf{Steinberg relations}, hold:

\begin{enumerate}\label{st-relations}
    \item[(R1)] $x_\alpha(t) x_\alpha (u) = x_\alpha (t + u)$ for all $\alpha \in \Phi$ and $t, u \in R$.
    \item[(R2)] $[x_\alpha (t), x_\beta (u)] = \prod x_{i \alpha + j \beta} (c_{\alpha, \beta; i, j} t^{i} u^{j})$ for all $\alpha, \beta \in \Phi$ such that $\alpha + \beta \neq 0$ and $t,u \in R$; where the structure constants $c_{\alpha, \beta; i, j}$ are as in Proposition~\ref{prop:chevalley commutator formula}.
    \item[(R3)] $w_\alpha (t) x_\beta(u) w_\alpha(t)^{-1} = x_{s_\alpha (\beta)} (c(\alpha,\beta) t^{-\langle \beta, \alpha \rangle} u)$ for all $\alpha, \beta \in \Phi, \ t \in R^{*}$ and $u \in R$; where the constants $c(\alpha, \beta)$ are as in Corollary~\ref{cor:c(alpha,beta)}
    \item[(R4)] $w_\alpha (t) w_\beta(u) w_\alpha(t)^{-1} = w_{s_\alpha (\beta)} (c(\alpha,\beta) t^{-\langle \beta, \alpha \rangle} u)$ for all $\alpha, \beta \in \Phi$ and $t,u \in R^{*}$.
    \item[(R5)] $w_\alpha (t) h_\beta(u) w_\alpha(t)^{-1} = h_{s_\alpha (\beta)} (u)$ for all $\alpha, \beta \in \Phi$ and $t,u \in R^{*}$.
    \item[(R6)] $h_\alpha (t) x_\beta(u) h_\alpha(t)^{-1} = x_{\beta} (t^{\langle \beta, \alpha \rangle} u)$ for all $\alpha, \beta \in \Phi, t \in R^{*}$ and $u \in R$.
    \item[(R7)] $h_\alpha (t) w_\beta(u) h_\alpha(t)^{-1} = w_{\beta} (t^{\langle \beta, \alpha \rangle} u)$ for all $\alpha, \beta \in \Phi$ and $t,u \in R^{*}$.
    \item[(R8)] $h_\alpha (t) h_\beta(u) h_\alpha(t)^{-1} = h_{\beta} (u)$ for all $\alpha, \beta \in \Phi$ and $t, u \in R^{*}$.
\end{enumerate}

Before we proceed further we will see the properties of the constants $c(\alpha, \beta)$. The $c(\alpha, \beta)$ satisfy the following properties:
\begin{gather*}
    c(\alpha, \beta) = c(\alpha, -\beta); \quad c(-\alpha, -\beta) = c(-\alpha, \beta) = c(\alpha, s_\alpha(\beta)); \\
    c(\alpha, \alpha) = c(\alpha, -\alpha) = -1; \quad c(\alpha,\beta) c(\alpha, s_\alpha (\beta)) = (-1)^{\langle \beta, \alpha \rangle}; \\
    c(\alpha, \beta) = 1, \quad \text{if } \alpha \pm \beta \not\in \Phi \cup \{ 0 \}; \\
    c(\alpha, \beta) = -1, \quad \text{if } \langle \alpha, \beta \rangle = 0, \alpha \pm \beta \in \Phi; \\
    c(\alpha, \beta) = N_{\alpha, \beta}, \quad \text{if } \alpha, \beta \in \Phi \text{ such that } \alpha + \beta \in \Phi \text{ but } \alpha - \beta \not \in \Phi. 
\end{gather*}

Since the simple reflections $s_\alpha \ (\alpha \in \Delta)$ generate the Weyl group $W$ of $\Phi$, as a consequence of the Steinberg relations, we can deduce the following corollary.

\begin{cor}
    \normalfont
    The group $E_\pi (\Phi, R)$ is generated by $x_{\alpha}(t) \ (t \in R, \ \pm \alpha \in \Delta)$.
\end{cor}

%%%%%%%%%%%%%%%%%%%%%%%%%%%%%%%%%%%%%%%%%%%%%%%%%%%%%%%%%%%%%%

\subsubsection*{The subgroups $H, B$ and $N$}

Let $H_{\pi}(\Phi, R)$ be the subgroup generated by all $h_\alpha (t) \ (t \in R^*, \alpha \in \Phi)$. 
Let $B_{\pi}(\Phi, R)$ be the subgroup generated by $U_{\pi}(\Phi, R)$ and $H_{\pi}(\Phi, R)$, called a \textbf{standard Borel subgroup}\index{subgroups of $G(R)$!standard Borel subgroup} of $E_\pi (\Phi, R)$. 
Then $U_{\pi}(\Phi, R) \cap H_{\pi}(\Phi, R) = 1$, $U_{\pi}(\Phi, R)$ is normal in $B_{\pi}(\Phi, R)$ and $B_{\pi}(\Phi, R) = U_{\pi}(\Phi, R) H_{\pi}(\Phi, R)$.

Let $N_\pi(\Phi, R)$ be the subgroup generated by all $w_\alpha (t) \ (t \in R^*, \alpha \in \Phi)$ and let $W$ be the Weyl group of $\Phi$. 
Then $H_\pi(\Phi, R)$ is normal in $N_\pi(\Phi, R)$ and $W \cong N_\pi(\Phi, R)/H_\pi(\Phi, R)$ with the map $$s_\alpha \mapsto w_\alpha(1) H_\pi(\Phi, R), \ \forall \alpha \in \Phi.$$

For abusive use of notations, we sometimes write $H(\Phi, R), H(R)$ or $H$ instead of $H_\pi(\Phi, R)$, similar for $B_\pi (\Phi, R)$ and $N_\pi (\Phi, R)$.

%%%%%%%%%%%%%%%%%%%%%%%%%%%%%%%%%%%%%%%%%%%%%%%%%%%%%%%%%%%%%%

\subsubsection*{Relation with Algebraic Groups}

Let $k$ be an algebraically closed field. 
The semisimple linear algebraic groups over $k$ are precisely the elementary Chevalley groups $E_\pi(\Phi, k)$ (see \cite[Chapter 5]{RS}). 
These groups can be realized as subgroups of $GL_n(k)$, defined as the common set of zeros of a family of polynomials in the matrix entries $x_{ij}$ with integer coefficients.
Moreover, note that both the multiplication map and the inverse map are also defined by polynomials with integer coefficients.

%%%%%%%%%%%%%%%%%%%%%%%%%%%%%%%%%%%%%%%%%%%%%%%%%%%%%%%%%%%%%%
%%%%%%%%%%%%%%%%%%%%%%%%%%%%%%%%%%%%%%%%%%%%%%%%%%%%%%%%%%%%%%

\subsection{Chevalley Groups} 

Let $E_\pi(\Phi, \mathbb{C})$ be an elementary Chevalley group, viewed as a subgroup of $GL_n(\mathbb{C})$, defined by the zero locus of polynomials $p_1(x_{ij}), \dots, p_m(x_{ij})$.  
As previously noted, these polynomials can be chosen to have integer coefficients.  
Let $R$ be a commutative ring with unity. Define the group  
\[
G(R) = \{ (a_{ij}) \in GL_n(R) \mid \widetilde{p}_1(a_{ij}) = 0, \dots, \widetilde{p}_m(a_{ij}) = 0 \},
\]
where $\widetilde{p}_1(x_{ij}), \dots, \widetilde{p}_m(x_{ij})$ are the same polynomials as $p_1(x_{ij}), \dots, p_m(x_{ij})$, but considered over the ring $R$ (this can be done by tensoring $\mathbb{Z}$ with $R$).  
This group is called the \textbf{Chevalley group}\index{Chevalley group}\label{nomencl:G(R)} $G_\pi(\Phi, R)$ of type $\Phi$ over the ring $R$.  
If $\pi$ is a representation such that $\Lambda_\pi = \Lambda_r$, then $G_\pi(\Phi, R) = G_{ad}(\Phi, R)$ is referred to as the \textbf{adjoint Chevalley group}\index{Chevalley group!adjoint}. 
If $\pi$ is a representation such that $\Lambda_\pi = \Lambda_{sc}$, then $G_\pi(\Phi, R) = G_{sc}(\Phi, R)$ is called the \textbf{universal}\index{Chevalley group!universal or simply connected} (or \textbf{simply connected}) \textbf{Chevalley group}.

%%%%%%%%%%%%%%%%%%%%%%%%%%%%%%%%%%%%%%%%%%%%%%%%%%%%%%%%%%%%%%

\subsubsection*{Definition via Group Scheme}

Let $G = G_{\mathbb{C}}$ be a connected complex semisimple Lie group with Lie algebra $\mathcal{L} = \mathcal{L}(\Phi, \mathbb{C})$ and weight lattice $\Lambda$. Consider a finite-dimensional faithful representation $\pi$ of $G$ on a complex vector space $V$, such that $\Lambda_\pi = \Lambda$. Let the differential of $\pi$, also denoted by $\pi$, be $\pi: \mathcal{L} \longrightarrow \operatorname{End}(V)$.

Let $\mathbb{C}[G]$ denote the affine algebra of $G$, i.e., the algebra of all regular complex-valued functions on $G$, viewed as a \emph{Hopf algebra}\index{Hopf algebra} (see \cite{EA5} for the definition). 
Choosing a basis $\{v_\lambda\}_{\lambda \in \Omega_\pi}$ of $V$ consisting of weight vectors, we can identify $V$ with $\mathbb{C}^n$, where $n = \dim_{\mathbb{C}}(V)$. 
This identification introduces coordinate functions $x_{\lambda, \mu}$ for $\lambda, \mu \in \Omega_\pi$ on $GL(V)$, which in turn identify $GL(V)$ with $GL_n(\mathbb{C})$. 
The restrictions of these coordinate functions to $\pi(G)$ generate a subring $\mathbb{Z}[G]$ of the affine algebra $\mathbb{C}[G]$. 
It can be verified that this subring is, in fact, a \emph{Hopf subalgebra} of $\mathbb{C}[G]$.

Thus, one can define an \emph{affine group scheme} over $\mathbb{Z}$, called the \textbf{Chevalley-Demazure group scheme}\index{Chevalley-Demazure group scheme} of type $\Phi$, by
\[
    G_\pi(\Phi,-): R \longmapsto \operatorname{Hom}(\mathbb{Z}[G], R),
\]
where the image of a ring $R$ under this functor, denoted $G_\pi(\Phi, R)$, is called the \textbf{Chevalley group}\index{Chevalley group} of type $\Phi$ over $R$. Up to isomorphism, $G_\pi(\Phi, R)$ depends only on $\Phi$ and $\Lambda_\pi$, but not on the specific choice of $\pi$. Moreover, by construction, these linear groups $G_\pi(\Phi, R)$ can be realized as subgroups of $GL_n(R)$, where $n = \dim_{\mathbb{C}}(V)$.

Now, let $\alpha \in \Phi$ and $t$ be an independent variable. The homomorphism of $\mathbb{Z}[G]$ on $\mathbb{Z}[t]$, assigning to each coordinate function $x_{\lambda, \mu}$ its value on $x_\alpha(t)$, induces a homomorphism
\[
    G_a(R) := \operatorname{Hom}(\mathbb{Z}[t], R) \longrightarrow G_\pi(\Phi, R) = \operatorname{Hom}(\mathbb{Z}[G], R),
\]
from the additive group $(R, +) \cong G_a(R)$ of the ring $R$ to the Chevalley group $G_\pi(\Phi, R)$. The image of this homomorphism is the \emph{root subgroup}
\[
    \mathfrak{X}_\alpha = \{x_\alpha(t) \mid t \in R\}.
\]
Hence, the \emph{elementary Chevalley group} $E_\pi(\Phi, R)$ is contained in the Chevalley group $G_\pi(\Phi, R)$.

%%%%%%%%%%%%%%%%%%%%%%%%%%%%%%%%%%%%%%%%%%%%%%%%%%%%%%%%%%%%%%

\subsubsection*{The $K_1$ and $K_2$-functors} \index{the $K_1$ and $K_2$-functors}

A central problem in the theory of Chevalley groups is understanding the relationship between $E_\pi(\Phi, R)$ and $G_\pi(\Phi, R)$. As noted earlier, $E_\pi(\Phi, R) \subset G_\pi(\Phi, R)$. A natural question is whether this inclusion is an equality. When $R$ is an algebraically closed field, we have $E_\pi(\Phi, R) = G_\pi(\Phi, R)$. However, this equality does not always hold, even when $R$ is a field. 

If $G$ is simply connected, then for any field $k$, we have $E_{\text{sc}}(\Phi, k) = G_{\text{sc}}(\Phi, k)$. Moreover, this equality also holds for certain other classes of rings, including semilocal rings and Euclidean rings (see Section 7 of \cite{EP&NV}).

\begin{thm}[{\cite{GT}}]\label{E is normal in G}
    \normalfont
    Let $\Phi$ be an irreducible root system of rank $\geq 2$. Then $E_\pi(\Phi, R)$ is a normal subgroup of $G_\pi(\Phi, R)$.
\end{thm}

\begin{rmk}
    If the rank of $\Phi$ is $1$, then the above result is not true in general. See, for example, Section 7 of \cite{EP&NV}.
\end{rmk}

In view of Theorem~\ref{E is normal in G}, if $\Phi$ is an irreducible root system of rank $\geq 2$, then the quotient group 
\[
    K_1(\Phi, R) := G_{sc}(\Phi, R) / E_{sc}(\Phi, R)
\]
is well-defined. The functor $K_{1}(\Phi, -)$ is known as the \emph{$K_1$-functor} of type $\Phi$. 

\medskip

Let $\Phi$ be a root system of rank $\ell$. The \textbf{Steinberg group}\index{Steinberg group} of type $\Phi$ over a ring $R$, denoted by $\text{St}(\Phi, R)$, is the group generated by symbols  
\[
    x'_\alpha(t), \quad t \in R, \, \alpha \in \Phi,
\]  
subject to the relations $(R1)$ and $(R2)$ when $\ell \geq 2$, and the relations $(R1)$ and $(R3)$ when $\ell = 1$.  
The relations $(R1)$, $(R2)$, and $(R3)$ are as defined on page~\pageref{st-relations}, with $x'_\alpha(t)$ replacing $x_\alpha(t)$.

Since the generators $x_\alpha(t)$ of $E_{sc}(\Phi, R)$ also satisfy the relations $(R)$, there is a natural surjective map 
\[
    f(\Phi, R): \text{St}(\Phi, R) \longrightarrow E_{sc}(\Phi, R)
\]
defined by $x'_\alpha(t) \mapsto x_\alpha(t)$. 
We define
\[
    K_{2}(\Phi, R) = \text{ker}(f(\Phi, R)).
\]
The functor $K_{\pi, 2}(\Phi, -)$ is called the \emph{$K_2$-functor} of type $\Phi$.

Assume that $\Phi$ is an irreducible root system of rank $\geq 2$. By definition, we have the following exact sequence:
\[
    1 \longrightarrow K_{2}(\Phi, R) \longrightarrow \text{St}(\Phi, R) \longrightarrow G_{sc}(\Phi, R) \longrightarrow K_{1}(\Phi, R) \longrightarrow 1.
\]

%%%%%%%%%%%%%%%%%%%%%%%%%%%%%%%%%%%%%%%%%%%%%%%%%%%%%%%%%%%%%%

\subsubsection*{Split Maximal Torus}

Let $T$ be a split maximal torus of a connected semisimple Lie group $G$. This defines the subgroup scheme 
\[
T_\pi(\Phi, -): R \longmapsto \operatorname{Hom}(\mathbb{Z}[T], R)
\]
of the Chevalley-Demazure group scheme $G_\pi(\Phi, -)$. 
The coordinate ring $\mathbb{Z}[T]$ is given by $\mathbb{Z}[\lambda_1, \lambda_1^{-1}, \dots, \lambda_\ell, \lambda_\ell^{-1}],$ the algebra of Laurent polynomials corresponding to a $\mathbb{Z}$-basis $\{ \lambda_1, \dots, \lambda_\ell \}$ of the lattice $\Lambda_\pi$. 

If $R$ is a commutative ring with unity, the corresponding group of points 
\[
T_\pi(\Phi, R) = \operatorname{Hom}(\mathbb{Z}[T], R)
\]
is called a \textbf{split maximal torus}\index{subgroups of $G(R)$!split maximal torus} of the Chevalley group $G_\pi(\Phi, R)$. It is well known that 
\[
T_\pi(\Phi, R) \cong \operatorname{Hom}(\Lambda_\pi, R^*).
\]

This identification can be described as follows: Let $\chi \in \operatorname{Hom}(\Lambda_\pi, R^*)$ be a character of $\Lambda_\pi$. Define $h(\chi)$ as an automorphism of $V_\pi(\Phi, R)$ given by 
\[
h(\chi) \cdot v = \chi(\mu) v,
\]
for all $v \in V_{\pi, \mu}(\Phi, R)$ and $\mu \in \Omega_\pi$, where $\Omega_\pi$ is the set of weights corresponding to the representation $\pi$.
Thus, we have 
\[
T_\pi(\Phi, R) = \{ h(\chi) \mid \chi \in \operatorname{Hom}(\Lambda_\pi, R^*) \}.
\]

Note that $H_\pi (\Phi, R)$ is contained in $T_{\pi}(\Phi, R)$. The element $h(\chi) \in H_\pi (\Phi, R) \subset E_\pi(\Phi, R)$ if and only if $\chi \in \text{Hom}(\Lambda_\pi, R^*)$ can be extented to a character $\chi'$ of $\Lambda_{sc}$, that is, $\chi' \in \text{Hom}(\Lambda_{sc}, R^*)$ such that $\chi'|_{\Lambda_{\pi}} = \chi$. Moreover, $h_\alpha(t) = h(\chi_{\alpha, t}) \ (t \in R^*, \alpha \in \Phi),$ where $$ \chi_{\alpha, t}: \lambda \mapsto t^{\langle \lambda, \alpha \rangle} \ (\lambda \in \Lambda_\pi).$$ Therefore $H_\pi (\Phi, R) = E_\pi (\Phi, R) \cap T_\pi(\Phi, R).$

%%%%%%%%%%%%%%%%%%%%%%%%%%%%%%%%%%%%%%%%%%%%%%%%%%%%%%%%%%%%%%

\subsubsection*{Some Important Subgroups}

The element $h(\chi) \in T_\pi(\Phi, R)$ acts on the root group $\mathfrak{X}_{\alpha} = \{x_{\alpha}(\zeta) \mid \zeta \in R\}$ by conjugation as follows:  
\[
    h(\chi) x_\alpha(\zeta) h(\chi)^{-1} = x_\alpha(\chi(\alpha) \zeta).
\]  
Thus, $T_\pi(\Phi, R)$ normalizes $E_\pi(\Phi, R)$. 
Hence the set $G_\pi^0(\Phi, R) := E_\pi(\Phi, R) T_\pi(\Phi, R)$ forms a subgroup of $G_\pi(\Phi, R)$. 
If $R$ is a semilocal ring, then $G_\pi(\Phi, R) = G_\pi^0(\Phi, R)$ (see \cite[Corollary 2.4]{EA2}). 

By a similar argument, the set $B_\pi^0(\Phi, R) := U_\pi(\Phi, R) T_\pi(\Phi, R) = B_\pi(\Phi, R) T_\pi(\Phi, R)$ forms a subgroup of $G_\pi(\Phi, R)$, referred to as the \textbf{standard Borel subgroup}\index{subgroups of $G(R)$!standard Borel subgroup} of $G_\pi(\Phi, R)$. 

Define $N_\pi^0(\Phi, R) := N_\pi(\Phi, R) T_\pi(\Phi, R)$. 
Then $T_\pi(\Phi, R)$ is normal in $N_\pi^0(\Phi, R)$, and the following isomorphism holds:  
\[
    N_\pi^0(\Phi, R) / T_\pi(\Phi, R) \cong N_\pi(\Phi, R) / H_\pi(\Phi, R) \cong W,
\]  
where $W$ is the Weyl group of $\Phi$.

Recall (cf. page~\pageref{subsubsec:closed subsets}) that a subset $S \subset \Phi$ is called \emph{closed} if $\alpha, \beta \in S$ such that $\alpha + \beta \in \Phi$ implies $\alpha + \beta \in S$.  
As before (cf. page~\pageref{U and U-}), for a given closed set $S$ define  
\[
    E_\pi(S, R) := \langle x_\alpha(t) \mid t \in R, \, \alpha \in S \rangle.
\]  
Since $T_\pi(\Phi, R)$ normalizes $E_\pi(S, R)$, the product  
\[
    G_\pi^{0}(\Phi, S, R) := E_\pi(S, R) T_\pi(\Phi, R)
\]  
forms a subgroup of $G_\pi(\Phi, R)$.

As noted earlier (cf. page~\pageref{subsubsec:closed subsets}), any closed set $S$ can be uniquely decomposed into the disjoint union of its \emph{symmetric} (or \emph{reductive}) part, $S^r = \{\alpha \in S \mid -\alpha \in S\}$, and its \emph{special} (or \emph{unipotent}) part, $S^u = \{\alpha \in S \mid -\alpha \not\in S\}$.  
Since $S^u$ forms an ideal in $S$, the subgroup $E_\pi(S^u, R)$ is normal in $E_\pi(S, R)$ and in $G_\pi^0(\Phi, S, R)$. This allows the following Levi decomposition:
\[
    E_\pi(S, R) = E_\pi(S^u, R) \rtimes E_\pi(S^r, R) \quad \text{and} \quad G_\pi^0(\Phi, S, R) = E_\pi(S^u, R) \rtimes G_\pi^0(\Phi, S^r, R).
\]
The subgroup $E_\pi(S^u, R)$ is called the \emph{unipotent radical} of $E_\pi(S, R)$ (resp., $G_\pi(\Phi, S, R)$) and the subgroup $E_\pi(S^r, R)$ (resp., $G_\pi^0(\Phi, S^r, R)$) is called the \emph{Levi subgroup} of $E_\pi(S, R)$ (resp., $G_\pi(\Phi, S, R)$).

A special case arises when $S = P$, a \emph{parabolic} subset of the root system $\Phi$ (i.e., $S$ is a closed subset and satisfies $\Phi = S \cup -S$).  
The subgroup $G_\pi^0(P, R)$ is called the \textbf{standard parabolic subgroup}, and its conjugates are referred to as \textbf{parabolic subgroups}\index{subgroups of $G(R)$!parabolic subgroup}. 
Note that the standard Borel subgroup of $G_\pi(\Phi, R)$ is an example of a parabolic subgroup.

%%%%%%%%%%%%%%%%%%%%%%%%%%%%%%%%%%%%%%%%%%%%%%%%%%%%%%%%%%%%%%

\subsubsection*{Identifications with the Classical Groups}

We now provide a list of some explicit examples of Chevalley groups:

\begin{center}
    \begin{tabular}{c|c|c}
        Type of $\Phi$ & $G_{sc}(\Phi, R)$ & $G_{\text{ad}}(\Phi, R)$ \\
        \hline
        $A_\ell$ & $SL_{\ell + 1}(R)$ & $PSL_{\ell + 1}(R)$ \\
        $B_\ell$ & $Spin_{2\ell + 1}(R)$ & $SO_{2\ell + 1}(R) = PSO_{2\ell + 1}(R)$ \\
        $C_\ell$ & $Sp_{2\ell}(R)$ & $PSp_{2\ell}(R)$ \\
        $D_\ell$ & $Spin_{2\ell}(R)$ & $PSO_{2\ell}(R)$ \\
    \end{tabular}
\end{center}

For a detailed discussion on exceptional groups, see Chapter $2$ of \cite{NV1} or Section $8$ of \cite{EP&NV}. 
We conclude this section with the following remark regarding the notation:

\begin{rmk}
    For convenience and to simplify notation, we often write \( E(R) \) or \( E(\Phi, R) \) instead of \( E_\pi(\Phi, R) \), following the same reasoning for other groups and subgroups defined earlier.
\end{rmk}

%%%%%%%%%%%%%%%%%%%%%%%%%%%%%%%%%%%%%%%%%%%%%%%%%%%%%%%%%%%%%%
%Section - Properties of Chevalley Groups
%%%%%%%%%%%%%%%%%%%%%%%%%%%%%%%%%%%%%%%%%%%%%%%%%%%%%%%%%%%%%%

\section{Properties of Chevalley Groups}

%%%%%%%%%%%%%%%%%%%%%%%%%%%%%%%%%%%%%%%%%%%%%%%%%%%%%%%%%%%%%%

This section presents a comprehensive survey of several well-known results of Chevalley groups. 

%%%%%%%%%%%%%%%%%%%%%%%%%%%%%%%%%%%%%%%%%%%%%%%%%%%%%%%%%%%%%%
%%%%%%%%%%%%%%%%%%%%%%%%%%%%%%%%%%%%%%%%%%%%%%%%%%%%%%%%%%%%%%

\subsection{Chevalley Groups over Fields}

Let $k$ be a field and let $G_\pi(\Phi, k)$ denote a Chevalley group of type $\Phi$ over $k$. As previously noted, if $k$ is algebraically closed, then $G_\pi(\Phi, k) = E_\pi(\Phi, k)$. In general, however, this equality does not always hold. Nevertheless, for any field $k$, we have the relation 
\[
G_\pi(\Phi, k) = E_\pi(\Phi, k) T_\pi(\Phi, k),
\]
where $T_\pi(\Phi, k)$ denotes the corresponding maximal torus.
Additionally, if $G$ is simply-connected, then for any field $k$, we have 
\[
G_{sc}(\Phi, k) = E_{sc}(\Phi, k).
\]

%%%%%%%%%%%%%%%%%%%%%%%%%%%%%%%%%%%%%%%%%%%%%%%%%%%%%%%%%%%%%%

\subsubsection*{Bruhat Decomposition}

A fundamental property of elementary Chevalley groups over a field is that they form a $(B, N)$-pair in the sense of J. Tits \cite{JT}. This structure allows for a precise decomposition of the group as follows:

\begin{thm}[Bruhat decomposition]\index{Bruhat decomposition}
    \normalfont
    Let $G = E_\pi(\Phi, k)$ be an elementary Chevalley group of type $\Phi$ over a field $k$. Let $W$ denote the Weyl group associated with the root system $\Phi$. For each $w \in W$, define 
    \[
    Q_w = \Phi^+ \cap w^{-1}(\Phi^-) \quad \text{and} \quad U_w = \mathfrak{X}_{Q_w}.
    \]
    Furthermore, fix a representative $\eta_w \in N_\pi(\Phi, k)$ corresponding to each $w \in W$. Then the following properties hold:
    \begin{enumerate}[(a)]
        \item The group $G$ admits the decomposition 
        \[
        G = \bigsqcup_{w \in W} B \eta_w B, 
        \]
        as a disjoint union of double cosets of $B$, indexed by the elements of the Weyl group $W$.
        \item Each double coset $B \eta_w B$ can be expressed as 
        \[
        B \eta_w B = B \eta_w U_w.
        \]
        Moreover, every element of $B \eta_w U_w$ (and hence of $G$) has a unique expression of the form $b \cdot \eta_w \cdot u$, where $b \in B$, $\eta_w$ is the chosen representative of $w$, and $u \in U_w$.
    \end{enumerate}
\end{thm}

%%%%%%%%%%%%%%%%%%%%%%%%%%%%%%%%%%%%%%%%%%%%%%%%%%%%%%%%%%%%%%

\subsubsection*{Simplicity of adjoint groups}

\begin{thm}[]
    Let $G = E_{\mathrm{ad}}(\Phi, k)$ be the elementary adjoint Chevalley group of type $\Phi$ over a field $k$. Suppose that $\Phi$ is irreducible. Additionally, if $\lvert k \rvert = 2$, assume that $\Phi$ is not of type $A_1$, $B_2$ or $G_2$. If $\lvert k \rvert = 3$, assume that $\Phi$ is not of type $A_1$. Then the group $G$ is simple.
\end{thm}

The cases excluded in the above theorem are not simple. The following table provides a list of non-trivial normal subgroups for each excluded case:

\begin{center}
    \begin{tabular}{c|c|c|c}
       \textbf{Type of $\Phi$} & $\lvert k \rvert$ & $G$ & \textbf{Normal Subgroup} \\
       \hline
       $A_1$ & $2$ & $PSL_2(\mathbb{F}_2) \cong \mathcal{S}_3$ & $\mathcal{A}_3$ \\
       $B_2$ & $2$ & $PSO_{5} \cong SO_{5} \cong \mathcal{S}_6$ & $\mathcal{A}_6$ \\
       $G_2$ & $2$ & of order $12096 = 2^6 \cdot 3^3 \cdot 7$ & $SU_3(\mathbb{F}_3)$ \\
       $A_1$ & $3$ & $PSL_2(\mathbb{F}_3) \cong \mathcal{S}_4$ & $\mathcal{A}_4$
    \end{tabular}
\end{center}

Here, $\mathcal{S}_n$ denotes the symmetric group on $n$ letters, and $\mathcal{A}_n$ denotes the alternating group on $n$ letters.

%%%%%%%%%%%%%%%%%%%%%%%%%%%%%%%%%%%%%%%%%%%%%%%%%%%%%%%%%%%%%%
%%%%%%%%%%%%%%%%%%%%%%%%%%%%%%%%%%%%%%%%%%%%%%%%%%%%%%%%%%%%%%

\subsection{Center of the Chevalley Groups}

For a group $G$ and a subgroup $H$ of $G$, denote $Z(G)$ as the center of $G$ and $C_{G}(H)$ as the centralizer of $H$ in $G$.

\begin{thm}[{\cite[Main Theorem]{EA&JH}}]\label{thm:EA&JH}
    \normalfont
    Let $R$ be a commutative ring with unity. Assume that the Jacobson radical of $R$ is trivial or the rank of $\Phi$ is at least $2$. Then,
    \[
        C_{G_\pi(\Phi, R)}(E_\pi(\Phi, R)) = Z(G_\pi(\Phi, R)) = \mathrm{Hom}(\Lambda_\pi / \Lambda_r, R^*).
    \] 
    Additionally, if $G$ is of universal or adjoint type, then $Z(G_\pi(\Phi, R)) = Z(E_\pi(\Phi, R))$.
\end{thm}

\begin{cor}
    \normalfont
    Let $\Phi$ and $R$ be as in Theorem~\ref{thm:EA&JH}. If $G$ is of adjoint type, then $Z(G_\pi(\Phi, R))$ is trivial.
\end{cor}

%%%%%%%%%%%%%%%%%%%%%%%%%%%%%%%%%%%%%%%%%%%%%%%%%%%%%%%%%%%%%%
%%%%%%%%%%%%%%%%%%%%%%%%%%%%%%%%%%%%%%%%%%%%%%%%%%%%%%%%%%%%%%

\subsection{Certain Commutator Relations} 

Let $R$ be a commutative ring with unity and $J$ be an ideal of $R$. The natural projection map $R \longrightarrow R/J$ induces a map
\[
    \phi: G_{\pi} (\Phi, R) \longrightarrow G_{\pi} (\Phi, R/J).
\]
Define $G_\pi (\Phi, J) = \ker (\phi)$, called the \textbf{principal congruent}\index{subgroups of $G(R)$!principal congruence subgroup}\label{nomencl:G(J)} subgroup of $G_\pi (\Phi, R)$ of level $J$ and $G_{\pi} (\Phi, R, J) = \phi^{-1} (Z(G_\pi (\Phi, R/J)))$, called the \textbf{full congruence}\index{subgroups of $G(R)$!full congruence subgroup}\label{nomencl:G(R,J)} subgroup of $G_\pi (\Phi, R)$ of level $J$.

Let $E_\pi (\Phi, J)$ be the subgroup of $E_\pi (\Phi, R)$ generated by all $x_\alpha (t) \ (t \in J, \alpha \in \Phi)$ and let $E_\pi (\Phi, R, J)$ be the normal subgroup of $E_\pi (\Phi, R)$ generated by $E_\pi (\Phi, J)$. We call the subgroup $E_\pi (\Phi, R, J)$ of $E_\pi (\Phi, R)$ a \textbf{relative elementary}\index{subgroups of $G(R)$!relative elementary subgroups}\label{nomencl:E(R,J)} subgroup at level $J$.

\begin{thm}[{L. N. Vaserstein \cite{LV}}]\label{thm:commrelations}
    \normalfont
    Let $\Phi$ be an irreducible root system of rank $\geq 2$ and \(R\) a commutative ring with unity. If $J$ is an ideal of $R$, then the following commutator relations hold: 
    \[
        [E_\pi(\Phi, R, J), G_\pi(\Phi, R)] \subset E_\pi(\Phi, R, J) \quad \text{and} \quad [E_\pi(\Phi, R), G_\pi(\Phi, R, J)] \subset E_\pi(\Phi, R, J).
    \]
    Except in the cases where $\Phi = B_2$ or $G_2$ and $R$ has a residue field with two elements, these inclusions are equalities.
\end{thm}  

\begin{cor}
    \normalfont
    Let $\Phi$ and $R$ be as in Theorem~\ref{thm:commrelations}. If $H$ is a subgroup of $G_\pi (\Phi, R)$ such that 
    \[
        E_\pi (\Phi, R, J) \subset H \subset G_\pi (\Phi, R, J)
    \]
    where $J$ is an ideal of $R$. Then $H$ is normal in $G_\pi (\Phi, R)$.
\end{cor}

%%%%%%%%%%%%%%%%%%%%%%%%%%%%%%%%%%%%%%%%%%%%%%%%%%%%%%%%%%%%%%
%%%%%%%%%%%%%%%%%%%%%%%%%%%%%%%%%%%%%%%%%%%%%%%%%%%%%%%%%%%%%%

\subsection{Normal Subgroups of Chevalley Groups}

\begin{thm}[{L. N. Vaserstein \cite{LV}}, E. Abe \cite{EA3}]\label{thm:normalsubgroups}
    \normalfont
    Let \(\Phi\) be an irreducible root system of rank $\geq 2$ and \(R\) a commutative ring with unity. Assume $1/ 2 \in R$ if $\Phi \sim A_\ell, B_\ell, F_4$, and $1/3 \in R$ if $\Phi \sim G_2$. If $H$ is a subgroup of $G_\pi(\Phi, R)$ normalized by $E_\pi(\Phi, R)$, then there exists a unique ideal $J$ of $R$ such that 
    \[
        E_\pi(\Phi, R, J) \subset H \subset G_\pi (\Phi, R, J).
    \]
\end{thm}

\begin{cor}
    Let $\Phi$ and $R$ be as in Theorem~\ref{thm:normalsubgroups}. Then $H$ is a normal subgroup of $E_\pi (\Phi, R)$ if and only if there exists an ideal $J$ such that 
    \[
        E_\pi(\Phi, R, J) \subset H \subset G_\pi (\Phi, R, J) \cap E_\pi (\Phi, R).
    \]
\end{cor}

A subgroup $H$ of a group $G$ is called \textit{characteristic}, if it is mapped into itself under any automorphism of $G$. In particular, any characteristic subgroup is normal. 

\begin{thm}[{\cite[Theorem 5]{LV}}]
    Let $\Phi$ and $R$ be as in Theorem~\ref{thm:normalsubgroups}. Then $E_\pi (\Phi, R)$ is a characteristic subgroup of $G_\pi (\Phi, R)$.
\end{thm}

%%%%%%%%%%%%%%%%%%%%%%%%%%%%%%%%%%%%%%%%%%%%%%%%%%%%%%%%%%%%%%
\subsection{Automorphisms of Chevalley Algebras}
%%%%%%%%%%%%%%%%%%%%%%%%%%%%%%%%%%%%%%%%%%%%%%%%%%%%%%%%%%%%%%

We will investigate the $R$-algebra automorphisms of the Chevalley algebra $\mathcal{L}(\Phi, R)$. 

\vspace{2mm}

\noindent \textbf{Inner Automorphisms:} Let $G_{\text{ad}}(\Phi, R)$ be the elementary Chevalley group of type $\Phi$. By definition, each element of $G_{\text{ad}}(\Phi, R)$ is regarded as an automorphism of the Chevalley algebra $\mathcal{L}(\Phi, R)$. Such automorphisms are called \emph{inner automorphisms}.

\vspace{2mm}

\noindent \textbf{Graph automorphism:}
Let $\Phi$ be an irreducible root system and let $\Delta$ be a fixed simple system of $\Phi$. 
Consider a non-trivial angle-preserving permutation $\rho$ of the simple roots (such a $\rho$ induces an automorphism of the Coxeter graph). 

\begin{enumerate}
    \item If all roots in $\Phi$ are of equal length, then $\rho$ extends to an automorphism of the root system $\Phi$.
    \item If $\Phi$ contains roots of unequal lengths and $p > 1$ is the square of the ratio of the lengths of long to short roots, then $\rho$ must interchange long and short roots. In this case, $\rho$ extends to a permutation of all roots in $\Phi$ that satisfies the following conditions:
    \begin{enumerate}
        \item $\rho$ interchanges long and short roots.
        \item The map $\hat{\rho}$ defined by
        \[
        \alpha \longmapsto \begin{cases}
            \rho(\alpha) & \text{if $\alpha$ is long}, \\
            \frac{1}{p} \cdot \rho(\alpha) & \text{if $\alpha$ is short}
        \end{cases}
        \]
        is an isomorphism of root system $\Phi$.
    \end{enumerate}
\end{enumerate}

The following table summarizes all possible choices for $\rho$:

\begin{center}
    \begin{tabular}{|c|c|c|}
        \hline
        \textbf{Root System} & \textbf{Graph Automorphism $\rho$} & \textbf{Order of $\rho$} \\ 
        \hline
        $A_n \, (n \geq 2)$ & \begin{tikzpicture}[scale=1]
            % Nodes    
            \node[draw,circle,inner sep=0.6mm] at (0,0) {};
            \node[draw,circle,inner sep=0.6mm] at (1,0) {};
            \node[draw,circle,inner sep=0.6mm] at (2,0) {};
            \node[draw,circle,inner sep=0.6mm] at (3,0) {};
            \node at (1.5,0) {$\cdots$};
            % Connections
            \draw (0.08,0)--(0.92,0); 
            \draw (2.08,0)--(2.92,0);
            % Arrows
            %\draw[<->, bend left] (0.08,0.1) to (2.92,0.1);
            %\draw[<->, bend left] (1.08,0.1) to (1.92,0.1);
            \draw[<->] (0.08,0.1) to [out=60,in=120] (2.92,0.1);
            \draw[<->] (1.08,0.1) to [out=60,in=120] (1.92,0.1);
        \end{tikzpicture} & $2$ \\ 
        \hline
        $D_n \, (n \geq 4)$ & \begin{tikzpicture}[scale=1]
            % Nodes
            \node[draw,circle,inner sep=0.6mm] at (0,0) {};
            \node[draw,circle,inner sep=0.6mm] at (1,0) {};
            \node[draw,circle,inner sep=0.6mm] at (2,0) {};
            \node[draw,circle,inner sep=0.6mm] at (3,0) {};
            \node[draw,circle,inner sep=0.6mm] at (4.0,0.75) {};
            \node[draw,circle,inner sep=0.6mm] at (4.0,-0.75) {};
            \node at (1.5,0) {$\cdots$};
            % Connections
            \draw (0.08,0)--(0.92,0);
            \draw (2.08,0)--(2.92,0);
            \draw (3.08,0)--(3.94,0.72);
            \draw (3.08,0)--(3.94,-0.72);
            % Arrows
            \draw[<->] (4.15,0.75) to [out=-45,in=45] (4.15,-0.75);
        \end{tikzpicture} & $2$ \\ 
        \hline
        $D_4$ & \begin{tikzpicture}[scale=1]
            % Nodes
            \node[draw,circle,inner sep=0.6mm] at (-1,0) {};
            \node[draw,circle,inner sep=0.6mm] at (0,0) {};
            \node[draw,circle,inner sep=0.6mm] at (0.5,0.87) {};
            \node[draw,circle,inner sep=0.6mm] at (0.5,-0.87) {};
            % Connections
            \draw (-0.92,0)--(-0.08,0);
            \draw (0.04,0.07)--(0.46,0.8);
            \draw (0.04,-0.07)--(0.46,-0.8);
            % Arrows
            \draw[->] (-0.99,0.1) to [out=84,in=156] (0.4,0.917);
            \draw[->] (0.6,0.8) to [out=-37,in=37] (0.6,-0.8);
            \draw[->] (0.4,-0.917) to [out=-156,in=-84] (-0.99,-0.1);
        \end{tikzpicture} & $3$ \\ 
        \hline
        $E_6$ & \begin{tikzpicture}[scale=1]
            % Nodes
            \node[draw,circle,inner sep=0.6mm] at (0,0) {};
            \node[draw,circle,inner sep=0.6mm] at (1,0) {};
            \node[draw,circle,inner sep=0.6mm] at (2,0) {};
            \node[draw,circle,inner sep=0.6mm] at (3,0) {};
            \node[draw,circle,inner sep=0.6mm] at (4,0) {};
            \node[draw,circle,inner sep=0.6mm] at (2,1) {};
            % Connections
            \draw (0.08,0)--(0.92,0);
            \draw (1.08,0)--(1.92,0);
            \draw (2.08,0)--(2.92,0);
            \draw (3.08,0)--(3.92,0);
            \draw (2,0.08)--(2,0.92);
            % Arrows
            \draw[<->] (0.08,-0.1) to [out=-60,in=-120] (3.92,-0.1);
            \draw[<->] (1.08,-0.1) to [out=-60,in=-120] (2.92,-0.1);
        \end{tikzpicture} & $2$ \\
        \hline
        $B_2$ & \begin{tikzpicture}[scale=1]
            % Nodes    
            \node[draw,circle,inner sep=0.6mm] at (0,0) {};
            \node[draw,circle,inner sep=0.6mm] at (1.5,0) {};
            \node at (0.75,0) {$>$};
            % Connections
            \draw (0.08,0.04)--(1.42,0.04); 
            \draw (0.08,-0.04)--(1.42,-0.04);
            % Arrows
            \draw[<->] (0.08,0.1) to [out=60,in=120] (1.42,0.1);
        \end{tikzpicture} & $2$ \\
        \hline
        $F_4$ & \begin{tikzpicture}[scale=1]
            % Nodes    
            \node[draw,circle,inner sep=0.6mm] at (0,0) {};
            \node[draw,circle,inner sep=0.6mm] at (1.5,0) {};
            \node[draw,circle,inner sep=0.6mm] at (3,0) {};
            \node[draw,circle,inner sep=0.6mm] at (4.5,0) {};
            \node at (2.25,0) {$>$};
            % Connections
            \draw (0.08,0)--(1.42,0); 
            \draw (1.58,-0.04)--(2.92,-0.04);
            \draw (1.58,0.04)--(2.92,0.04); 
            \draw (3.08,0)--(4.42,0);
            % Arrows
            \draw[<->] (1.58,0.1) to [out=60,in=120] (2.92,0.1);
            \draw[<->] (0.08,0.1) to [out=60,in=120] (4.42,0.1);
        \end{tikzpicture} & $2$ \\
        \hline
        $G_2$ & \begin{tikzpicture}[scale=1]
            % Nodes    
            \node[draw,circle,inner sep=0.6mm] at (0,0) {};
            \node[draw,circle,inner sep=0.6mm] at (1.5,0) {};
            \node at (0.75,0) {$>$};
            % Connections
            \draw (0.08,0.06)--(1.42,0.06);
            \draw (0.08,0)--(1.42,0);
            \draw (0.08,-0.06)--(1.42,-0.06);
            % Arrows
            \draw[<->] (0.08,0.1) to [out=60,in=120] (1.42,0.1);
        \end{tikzpicture} & $2$ \\
        \hline
    \end{tabular}
\end{center}

Let $\rho$ be an automorphism of the graph described above (possibly trivial). This automorphism naturally induces an automorphism of the Lie algebra $\mathcal{L}(\Phi, R)$, which is also denoted by $\rho$ (see \cite{AO&EV}).

Let $\rho_1, \dots, \rho_k$ be all distinct automorphisms of $\mathcal{L}(\Phi, R)$ induced by the different symmetries of the given simple system $\Delta$ of $\Phi$ and let $R = R_1 \oplus \cdots \oplus R_k$ be a decomposition of the ring $R$ into a direct sum of ideals. Define an automorphism $\rho$ of the algebra $\mathcal{L}(\Phi, R) = \mathcal{L}(\Phi, R_1) \oplus \cdots \oplus \mathcal{L}(\Phi, R_k)$ by the rule 
\[
\rho(x_1 + \cdots + x_k) = \rho_1(x_1) + \cdots + \rho_k(x_k),
\]
where $x_i \in \mathcal{L}(\Phi, R_i)$. Such an automorphism $\rho$ is called a \emph{graph automorphism} of the algebra $\mathcal{L}(\Phi, R)$.
Clearly, the set of all graph automorphisms forms a group that is isomorphic to the group
\[
\mathcal{D}(\Phi, R) = \left\{ \sum e_i \rho_i \,\middle|\, e_i \in R, \, e_i^2 = e_i, \, e_i e_j = 0 \text{ for } i \neq j, \, \sum e_i = 1 \right\}.
\]

\begin{thm}[{\cite[Theorem 1]{AK}}]
    \normalfont
    Let $R$ be a commutative ring with unity and $\Phi$ be an irreducible root system. Then every $R$-algebra automorphism of the Chevalley algebra $\mathcal{L}(\Phi, R)$ can be uniquely expressed as a composition of an inner automorphism and a graph automorphism. In particular,
    \[
        \text{Aut}_R(\mathcal{L}(\Phi, R)) \cong G_{\text{ad}}(\Phi, R) \rtimes \mathcal{D}(\Phi, R),
    \]
    where $\text{Aut}_R(\mathcal{L}(\Phi, R))$ denote the group of all $R$-algebra automorphisms of the Chevalley algebra $\mathcal{L}(\Phi, R)$. 
\end{thm}

\begin{rmk}
    In Theorem $1$ of \cite{AK}, A. Klyachko states the above result with certain restrictions on $R$. However, the same proof remains valid without these assumptions (see, for instance, \cite{EB12:main}).
\end{rmk}

%%%%%%%%%%%%%%%%%%%%%%%%%%%%%%%%%%%%%%%%%%%%%%%%%%%%%%%%%%%%%%
%%%%%%%%%%%%%%%%%%%%%%%%%%%%%%%%%%%%%%%%%%%%%%%%%%%%%%%%%%%%%%

\subsection{Automorphisms of Chevalley Groups}\label{sec:auto}\index{automorphism of Chevalley group}

We define four special types of automorphisms for the Chevalley group $G_{\pi}(\Phi, R)$. (These types of automorphisms can also be similarly defined for the elementary Chevalley group $E_{\pi}(\Phi, R)$.)

\vspace{2mm}

\noindent \textbf{Inner Automorphisms.}\index{automorphism of Chevalley group!inner} Let $S$ be a ring extension of $R$ and let $g \in N_{G_{\pi}(\Phi, S)}(G_{\pi}(\Phi, R))$. The mapping defined by $x \mapsto gxg^{-1}$ is an automorphism of $G_{\pi}(\Phi, R)$, denoted by $i_g$ and called an \emph{inner automorphism}, induced by the element $g \in G_{\pi}(\Phi, S)$. If $g \in G_{\pi}(\Phi, R)$, then $i_g$ is referred to as a \emph{strictly inner automorphism}.

\vspace{2mm}

\noindent \textbf{Ring Automorphisms.}\index{automorphism of Chevalley group!ring}\label{nomencl:theta} Let $\theta: R \to R$ be a ring automorphism. The mapping $(a_{ij}) \mapsto (\theta(a_{ij}))$ from $G_{\pi}(\Phi, R)$ onto itself defines an automorphism of $G_{\pi}(\Phi, R)$, denoted by $\theta$, and is called a \emph{ring automorphism}. In particular, for $\alpha \in \Phi$ and $t \in R$, any element $x_\alpha(t)$ is mapped to $x_\alpha(\theta(t))$.
If $R$ is a field, the corresponding ring automorphism is referred to as a \emph{field automorphism}.

\vspace{2mm}

\noindent \textbf{Central Automorphisms.}\index{automorphism of Chevalley group!central} Let $Z(G(R))$ denote the center of $G_{\pi}(\Phi, R)$ and let $\tau: G_{\pi}(\Phi, R) \to Z(G(R))$ be a homomorphism of groups. The mapping $x \mapsto \tau(x)x$ is an automorphism of $G_{\pi}(\Phi, R)$, denoted by $\tau$ and called a \emph{central automorphism}. 

Note that, all the central automorphisms of the elementary Chevalley group $E_\pi (\Phi, R)$ are the identity.

\vspace{2mm}

\noindent \textbf{Graph automorphisms.}\index{automorphism of Chevalley group!graph}\label{nomencl:rho} Let $\Phi$ be a root system consisting of roots of equal length and let $\Delta$ be a fixed simple system of $\Phi$. Let $\rho$ be an angle-preserving permutation of the simple roots (as defined in the previous subsection). Then there exists a unique automorphism of $G_\pi(\Phi, R)$, denoted by the same symbol $\rho$, such that for every $\alpha \in \Phi$ and $t \in R$, we have 
\[
\rho(x_\alpha(t)) = x_{\rho(\alpha)}(\epsilon(\alpha)t),
\]
where $\epsilon(\alpha) = \pm 1$ for all $\alpha \in \Phi$, and $\epsilon(\alpha) = 1$ for all $\alpha \in \Delta$.

Now suppose that $\rho_1, \dots, \rho_k$ are all distinct graph automorphisms of the given root system $\Phi$. Assume that the ring $R$ admits a system of orthogonal idempotents: 
\[
    \{\epsilon_1, \dots, \epsilon_k \mid \epsilon_1 + \cdots + \epsilon_k = 1, \, \epsilon_i \epsilon_j = 0 \text{ for } i \neq j\}.
\]
Then the mapping 
\[
    \rho := \epsilon_1 \rho_1 + \cdots + \epsilon_k \rho_k
\]
of the Chevalley group $G_\pi(\Phi, R)$ onto itself is an automorphism, referred to as a \emph{graph automorphism}.

\begin{thm}[{E. I. Bunina \cite{EB24:final, EB12:main}}]
    \normalfont
    Let $R$ be a commutative ring with unity and let $G_{\pi} (\Phi, R)$ (respectively, $E_{\pi}(\Phi, R)$) denote the Chevalley group (resp., elementary Chevalley group) of type $\Phi$ over $R$. Assume that $\Phi$ is an irreducible root system of rank at least $2$. Further, assume that for $\Phi \sim A_2, B_l, C_l$ or $F_4$ we have $1/2 \in R$, for $\Phi \sim G_2$ we have $1/6 \in R$. Then 
    \begin{enumerate}[(a)]
        \item every automorphism of the group $G_\pi (\Phi, R)$ can be expressed as a composition of an inner, a ring, a central and a graph automorphism of $G_\pi (\Phi, R)$. Moreover, if the group is of adjoint type, the inner automorphism in this decomposition is strictly inner, and the central automorphism is trivial.
        \item every automorphism of the group $E_\pi (\Phi, R)$ can be expressed as a composition of an inner, a ring and a graph automorphism of $E_\pi (\Phi, R)$. Furthermore, if the group is of adjoint type, the inner automorphism in this decomposition is strictly inner.
    \end{enumerate}
\end{thm}

\begin{rmk}
    \normalfont 
    If $R$ is a field, then every inner automorphism of the group $G_\pi(\Phi, R)$ (respectively, $E_\pi(\Phi, R)$) can be written as the composition of a strictly inner automorphism and a diagonal automorphism (for the definition of diagonal automorphisms, refer to \cite{RS}).
\end{rmk}

%%%%%%%%%%%%%%%%%%%%%%%%%%%%%%%%%%%%%%%%%%%%%%%%%%%%%%%%%%%%%%
%Section - Minimal Representations
%%%%%%%%%%%%%%%%%%%%%%%%%%%%%%%%%%%%%%%%%%%%%%%%%%%%%%%%%%%%%%

%\section{Minimal Representations}

%%%%%%%%%%%%%%%%%%%%%%%%%%%%%%%%%%%%%%%%%%%%%%%%%%%%%%%%%%%%%%
%%%%%%%%%%%%%%%%%%%%%%%%%%%%%%%%%%%%%%%%%%%%%%%%%%%%%%%%%%%%%%

\chapter{Twisted Chevalley Groups}\label{chapter:TCG}

%%%%%%%%%%%%%%%%%%%%%%%%%%%%%%%%%%%%%%%%%%%%%%%%%%%%%%%%%%%%%%
%%%%%%%%%%%%%%%%%%%%%%%%%%%%%%%%%%%%%%%%%%%%%%%%%%%%%%%%%%%%%%

In this chapter, we explore the concept of twisted Chevalley groups. We begin by defining the twisted root systems associated with an angle-preserving permutation of a fixed simple system. Next, we examine twisted Chevalley algebras and provide the definition of twisted Chevalley groups. We then investigate certain fundamental properties of elementary twisted Chevalley groups. Finally, we present a survey of established results concerning twisted Chevalley groups.

The material of this chapter is largely based on R. Steinberg~\cite{RS}, R. Carter~\cite{RC} and E. Abe~\cite{EA1}. Some of the results presented in Section~\ref{sec:E(R)} are derived from collaborative research with Shripad M. Garge~\cite{SG&DM1}.

%%%%%%%%%%%%%%%%%%%%%%%%%%%%%%%%%%%%%%%%%%%%%%%%%%%%%%%%%%%%%%
\section{Twisted Root Systems}\label{Subsec:TRS}\index{twisted root systems}
%%%%%%%%%%%%%%%%%%%%%%%%%%%%%%%%%%%%%%%%%%%%%%%%%%%%%%%%%%%%%%

Let $E$ be a finite-dimensional real Euclidean vector space and let $\Phi$ be a (crystallographic) root system in $E$. 
Let $\Delta$ and $\Phi^+$ be the simple and positive root systems, respectively. Let $\rho$ be a non-trivial angle preserving permutation of $\Delta$ (such a $\rho$ exists only when $\Phi$ is of type $A_\ell \ (\ell \geq 1), D_\ell \ (\ell \geq 4), E_6, B_2, F_4$ or $G_2$). Note that the possible order of $\rho$ is either $2$ or $3$, with the latter possible only when $\Phi \sim D_4$. We define an isometry $\hat{\rho} \in GL (E)$ as follows: 
\begin{enumerate}
    \item If $\Phi$ has one root length, then define $\hat{\rho}(\alpha)= \rho(\alpha)$ for each $\alpha \in \Delta$.
    \item If $\Phi$ has two root lengths. Then define $\hat{\rho}(\alpha)= \rho(\alpha)/ \sqrt{p}$ for each short root $\alpha \in \Delta$ and $\hat{\rho}(\alpha)= \sqrt{p} \hspace{1mm} \rho(\alpha)$ for each long root $\alpha \in \Delta$, where $p = ||\alpha||^2 / ||\beta||^2$, $\alpha$ is a long root and $\beta$ is a short root. 
\end{enumerate} 

Clearly, the order of $\hat{\rho}$ is the same as that of $\rho$ and $\hat{\rho}$ preserves the sign. 
Note that $\hat{\rho} w_{\alpha} \hat{\rho}^{-1} = w_{\rho(\alpha)}$, hence $\hat{\rho}$ normalizes $W$. 
Define $E_\rho = \{ v \in E \mid \hat{\rho}(v)=v \}$ and $W_\rho = \{ w \in W \mid \hat{\rho}w\hat{\rho}^{-1} =w \}$. 
Let $\hat{\alpha} = 1/o(\rho) \sum_{i=0}^{o(\rho)-1} \hat{\rho}^i (\alpha),$ the average of the elements in the $\hat{\rho}-$orbit of $\alpha$. 
Then $(\beta, \hat{\alpha})=(\beta, \alpha)$ for all $\beta \in E_\rho$. 
Hence the projection of $\alpha$ on $E_\rho$ is $\hat{\alpha}$. 

Note that $W_\rho$ acts faithfully on $E_\rho$. Let $J = J_\alpha \subset \Phi$ be the $\rho-$orbit of $\alpha$ and let $W_{J}$ be the group generated by all $w_\beta \ (\beta \in J_\alpha)$. 
Let $w_{J}$ be the unique element of $W_J$ such that $w_J (P_\alpha) = - P_\alpha$, where $P_\alpha$ is a positive system generated by $J_\alpha$ (such a $w_J$ exists and is of highest length element in $W_J$). 
Then $w_J|_{E_\rho} = w_{\hat{\alpha}}|_{E_\rho}$ and $w_J|_{E_\rho} \in W_{\rho}$. 
In fact, $\{ w_{\hat{\alpha}}|_{E_\rho} \mid \alpha \in \Delta \}$ forms a generating set of $W_\rho$. 
Therefore the group $W_\rho|_{E_\rho}$ is a reflection group. 
Define $\Tilde{\Phi}_\rho = \{ \hat{\alpha} \mid \alpha \in \Phi \}$ and $\Tilde{\Delta}_\rho = \{ \hat{\alpha} \mid \alpha \in \Delta \}$. 
Then $\Tilde{\Phi}_\rho$ is the (possibly non-reduced) root system corresponding to the Weyl group $W_\rho|_{E_\rho}$ and $\Tilde{\Delta}_\rho$ is the corresponding simple system. 
In order to make $\Tilde{\Phi}_\rho$ reduced, we can stick to the set of shortest projections of various directions, and denote it by $\Phi_\rho$.\label{nomencl:Phi_rho} 
Define an equivalence relation $R$ on $\Phi$ by $\alpha \equiv \beta$ iff $\hat{\alpha}$ is a positive multiple of $\hat{\beta}$. 
If $\Phi/R$ denotes the collection of all equivalence classes of this relation, then $\Phi_\rho$ is in one-to-one correspondence with $\Phi / R$ by identifying a root $\hat{\alpha}$ of $\Phi_\rho$ with a class $[\alpha]$ of $\Phi / R$. 
Similarly, there exists a one-to-one correspondence between $\Tilde{\Phi}_\rho$ and $\{ J_\alpha \mid \alpha \in \Phi \}$ by sending a root $\hat{\alpha}$ of $\Tilde{\Phi}_\rho$ to $J_\alpha$. Clearly $-[\alpha] = [-\alpha]$ and $-J_{\alpha} = J_{-\alpha}.$

%%%%%%%%%%%%%%%%%%%%%%%%%%%%%%%%%%%%%%%%%%%%%%%%%%%%%%%%%%%%%%
%%%%%%%%%%%%%%%%%%%%%%%%%%%%%%%%%%%%%%%%%%%%%%%%%%%%%%%%%%%%%%

\subsection{Construction of Twisted Root Systems}

\begin{lemma}[{\cite[page 103]{RS}}]
    \normalfont
    If $\Phi$ is irreducible, then an element of $\Phi/R$ is the positive system of roots of a system of one of the following types:
    \begin{enumerate}[(a)]
        \item $A_1^n, \hspace{1mm} n=1, 2$ or $3$.
        \item $A_2$ (this occurs only if $\Phi$ is of type $A_{2n}$).
        \item $B_2$ (this occurs if $\Phi$ is of type $B_{2}$ or $F_4$).
        \item $G_2$ (this occurs only if $\Phi$ is of type $G_2$).
    \end{enumerate}
\end{lemma}

If a class $[\alpha]$ in $\Phi / R$ is the positive system for a root system of type $X$ (where $X$ is any of the root systems mentioned in the lemma), we write $[\alpha] \sim X$. 
As noted previously, there exists a one-to-one correspondence between $\Phi / R$ and $\Phi_\rho$, so we use $\Phi_\rho$ to represent $\Phi / R$ interchangeably. 
Furthermore, if $\Phi \sim X$ (where $X$ is either $A_\ell, D_\ell, E_6, B_2, F_4$ or $G_2$), we denote $\Phi_\rho \sim {}^n X$, where $n$ is the order of $\rho$.  

Given a root system $\Phi$ and a graph automorphism $\rho$ of $\Phi$, the associated twisted root systems $\Phi_\rho$ are described as follows:

\begin{enumerate}
    \item $\Phi$ is of type $A_{2n-1}$ $(n \geq 1)$ and $o(\rho)=2$. Then $\Phi_\rho$ is of type $C_n$.
    \begin{center}
        \begin{tikzpicture}
            \node at (0,0) {$\alpha_1$};
            \node[draw,circle,inner sep=0.6mm] at (0,-0.4) {};
            \draw (0.08,-0.4)--(1.42,-0.4);
            \node at (1.5,0) {$\alpha_2$};
            \node[draw,circle,inner sep=0.6mm] at (1.5,-0.4) {};
            \draw (1.58,-0.4)--(2.6,-0.4);
            \node at (3,-0.4) {$\cdots$};
            \draw (3.4,-0.4)--(4.42,-0.4);
            \node at (4.5,0) {$\alpha_{n-2}$};
            \node[draw,circle,inner sep=0.6mm] at (4.5,-0.4) {};
            \draw (4.58,-0.4)--(5.92,-0.4);
            \node at (6,0) {$\alpha_{n-1}$};
            \node[draw,circle,inner sep=0.6mm] at (6,-0.4) {};
            
            \draw (6.08,-0.4)--(7.43,-1.34);
            \node at (7.9,-1.4) {$\alpha_n$};
            \node[draw,circle,inner sep=0.6mm] at (7.5,-1.4) {};
            \draw (6.08,-2.4)--(7.43,-1.46);
            
            \node at (6,-2.8) {$\alpha_{n+1}$};
            \node[draw,circle,inner sep=0.6mm] at (6,-2.4) {};
            \draw (4.58,-2.4)--(5.92,-2.4);
            \node at (4.5,-2.8) {$\alpha_{n+2}$};
            \node[draw,circle,inner sep=0.6mm] at (4.5,-2.4) {};
            \draw (3.4,-2.4)--(4.42,-2.4);
            \node at (3,-2.4) {$\cdots$};
            \draw (1.58,-2.4)--(2.6,-2.4);
            \node at (1.5,-2.8) {$\alpha_{2n-2}$};
            \node[draw,circle,inner sep=0.6mm] at (1.5,-2.4) {};
            \draw (0.08,-2.4)--(1.42,-2.4);
            \node at (0,-2.8) {$\alpha_{2n-1}$};
            \node[draw,circle,inner sep=0.6mm] at (0,-2.4) {};
		\end{tikzpicture}
    
        \vspace{5mm}
        
        \begin{tikzpicture}
            \node at (0,0) {$[\alpha_1]$};
            \node[draw,circle,inner sep=0.6mm] at (0,0.4) {};
            \draw (0.08,0.4)--(1.42,0.4);
            \node at (1.5,0) {$[\alpha_2]$};
            \node[draw,circle,inner sep=0.6mm] at (1.5,0.4) {};
            \draw (1.58,0.4)--(2.6,0.4);
            \node at (3,0.4) {$\cdots$};
            \draw (3.4,0.4)--(4.42,0.4);
            \node at (4.5,0) {$[\alpha_{n-2}]$};
            \node[draw,circle,inner sep=0.6mm] at (4.5,0.4) {};
            \draw (4.58,0.4)--(5.92,0.4);
            \node at (6,0) {$[\alpha_{n-1}]$};
            \node[draw,circle,inner sep=0.6mm] at (6,0.4) {};
            \draw (6.08,0.45)--(7.42,0.45);
            \node at (6.75,0.4) {$<$};
            \draw (6.08,0.35)--(7.42,0.35);
            \node at (7.5,0) {$[\alpha_n]$};
            \node[draw,circle,inner sep=0.6mm] at (7.5,0.4) {};
        \end{tikzpicture}
    \end{center}
    
    Note that $[\alpha_1], \dots, [\alpha_{n-1}] \sim A_1 \times A_1$ and $[\alpha_n] \sim A_1$.
    
    \item $\Phi$ is of type $A_{2n}$ $(n \geq 1)$ and $o(\rho)=2$. Then $\Tilde{\Phi}_\rho$ is of type $BC_n$. This is a non-reduced root system. The corresponding reduced root system $\Phi_\rho$ is $B_n$. 
    
    \begin{center}
        \begin{tikzpicture}
            \node at (0,0) {$\alpha_1$};
            \node[draw,circle,inner sep=0.6mm] at (0,-0.4) {};
            \draw (0.08,-0.4)--(1.42,-0.4);
            \node at (1.5,0) {$\alpha_2$};
            \node[draw,circle,inner sep=0.6mm] at (1.5,-0.4) {};
            \draw (1.58,-0.4)--(2.6,-0.4);
            \node at (3,-0.4) {$\cdots$};
            \draw (3.4,-0.4)--(4.42,-0.4);
            \node at (4.5,0) {$\alpha_{n-2}$};
            \node[draw,circle,inner sep=0.6mm] at (4.5,-0.4) {};
            \draw (4.58,-0.4)--(5.92,-0.4);
            \node at (6,0) {$\alpha_{n-1}$};
            \node[draw,circle,inner sep=0.6mm] at (6,-0.4) {};
            \draw (6.08,-0.4)--(7.42,-0.4);
            \node at (7.5,0) {$\alpha_n$};
            \node[draw,circle,inner sep=0.6mm] at (7.5,-0.4) {};
            
            \draw (7.5,-0.48)--(7.5,-1.92);

            \node at (7.5,-2.4) {$\alpha_{n+1}$};
            \node[draw,circle,inner sep=0.6mm] at (7.5,-2.0) {};
            \draw (6.08,-2.0)--(7.42,-2.0);
            \node at (6,-2.4) {$\alpha_{n+2}$};
            \node[draw,circle,inner sep=0.6mm] at (6,-2.0) {};
            \draw (4.58,-2.0)--(5.92,-2.0);
            \node at (4.5,-2.4) {$\alpha_{n+3}$};
            \node[draw,circle,inner sep=0.6mm] at (4.5,-2.0) {};
            \draw (3.4,-2.0)--(4.42,-2.0);
            \node at (3,-2.0) {$\cdots$};
            \draw (1.58,-2.0)--(2.6,-2.0);
            \node at (1.5,-2.4) {$\alpha_{2n-1}$};
            \node[draw,circle,inner sep=0.6mm] at (1.5,-2.0) {};
            \draw (0.08,-2.0)--(1.42,-2.0);
            \node at (0,-2.4) {$\alpha_{2n}$};
            \node[draw,circle,inner sep=0.6mm] at (0,-2.0) {};
		\end{tikzpicture}
    
        \vspace{5mm}
        
        \begin{tikzpicture}
            \node at (0,0) {$[\alpha_1]$};
            \node[draw,circle,inner sep=0.6mm] at (0,0.4) {};
            \draw (0.08,0.4)--(1.42,0.4);
            \node at (1.5,0) {$[\alpha_2]$};
            \node[draw,circle,inner sep=0.6mm] at (1.5,0.4) {};
            \draw (1.58,0.4)--(2.6,0.4);
            \node at (3,0.4) {$\cdots$};
            \draw (3.4,0.4)--(4.42,0.4);
            \node at (4.5,0) {$[\alpha_{n-2}]$};
            \node[draw,circle,inner sep=0.6mm] at (4.5,0.4) {};
            \draw (4.58,0.4)--(5.92,0.4);
            \node at (6,0) {$[\alpha_{n-1}]$};
            \node[draw,circle,inner sep=0.6mm] at (6,0.4) {};
            \draw (6.08,0.45)--(7.42,0.45);
            \node at (6.75,0.4) {$>$};
            \draw (6.08,0.35)--(7.42,0.35);
            \node at (7.5,0) {$[\alpha_n]$};
            \node[draw,circle,inner sep=0.6mm] at (7.5,0.4) {};
        \end{tikzpicture}
    \end{center}
    
    Note that $[\alpha_1], \dots, [\alpha_{n-1}] \sim A_1 \times A_1$ and $[\alpha_n] \sim A_2$.

    \item $\Phi$ is of type $D_n$ $(n \geq 4)$ and $o(\rho)=2$. Then $\Phi_\rho$ is of type $B_{n-1}$.

    \begin{center}
        \begin{tikzpicture}
            \node at (0,0) {$\alpha_1$};
            \node[draw,circle,inner sep=0.6mm] at (0,0.4) {};
            \draw (0.08,0.4)--(1.42,0.4);
            \node at (1.5,0) {$\alpha_2$};
            \node[draw,circle,inner sep=0.6mm] at (1.5,0.4) {};
            \draw (1.58,0.4)--(2.6,0.4);
            \node at (3,0.4) {$\cdots$};
            \draw (3.4,0.4)--(4.42,0.4);
            \node at (4.5,0) {$\alpha_{n-3}$};
            \node[draw,circle,inner sep=0.6mm] at (4.5,0.4) {};
            \draw (4.58,0.4)--(5.92,0.4);
            \node at (6,0) {$\alpha_{n-2}$};
            \node[draw,circle,inner sep=0.6mm] at (6,0.4) {};
            \draw (6.08,0.4)--(7.075,1.12);
            \node at (7.14,1.55) {$\alpha_{n-1}$};
            \node[draw,circle,inner sep=0.6mm] at (7.14,1.15) {};
            \draw (6.08,0.4)--(7.075,-0.32);
            \node at (7.14,-0.75) {$\alpha_n$};
            \node[draw,circle,inner sep=0.6mm] at (7.14,-0.35) {};
        \end{tikzpicture}
        
        \vspace{5mm}
        
        \begin{tikzpicture}
            \node at (0,0) {$[\alpha_1]$};
            \node[draw,circle,inner sep=0.6mm] at (0,0.4) {};
            \draw (0.08,0.4)--(1.42,0.4);
            \node at (1.5,0) {$[\alpha_2]$};
            \node[draw,circle,inner sep=0.6mm] at (1.5,0.4) {};
            \draw (1.58,0.4)--(2.6,0.4);
            \node at (3,0.4) {$\cdots$};
            \draw (3.4,0.4)--(4.42,0.4);
            \node at (4.5,0) {$[\alpha_{n-3}]$};
            \node[draw,circle,inner sep=0.6mm] at (4.5,0.4) {};
            \draw (4.58,0.4)--(5.92,0.4);
            \node at (6,0) {$[\alpha_{n-2}]$};
            \node[draw,circle,inner sep=0.6mm] at (6,0.4) {};
            \draw (6.08,0.45)--(7.42,0.45);
            \node at (6.75,0.4) {$>$};
            \draw (6.08,0.35)--(7.42,0.35);
            \node at (7.5,0) {$[\alpha_{n-1}]$};
            \node[draw,circle,inner sep=0.6mm] at (7.5,0.4) {};
        \end{tikzpicture}
    \end{center}

    Note that $[\alpha_1], \dots, [\alpha_{n-2}] \sim A_1$ and $[\alpha_{n-1}] \sim A_1 \times A_1$.

    \item $\Phi$ is of type $D_4$ and $o(\rho)=3$. Then $\Phi_\rho$ is of type $G_2$.

    \begin{center}
        \begin{tikzpicture}
            \node at (4.5,0) {$\alpha_{1}$};
            \node[draw,circle,inner sep=0.6mm] at (4.5,0.4) {};
            \draw (4.58,0.4)--(5.92,0.4);
            \node at (6,0) {$\alpha_{2}$};
            \node[draw,circle,inner sep=0.6mm] at (6,0.4) {};
            \draw (6.08,0.4)--(7.075,1.12);
            \node at (7.14,1.55) {$\alpha_{3}$};
            \node[draw,circle,inner sep=0.6mm] at (7.14,1.15) {};
            \draw (6.08,0.4)--(7.075,-0.32);
            \node at (7.14,-0.75) {$\alpha_{4}$};
            \node[draw,circle,inner sep=0.6mm] at (7.14,-0.35) {};
        \end{tikzpicture}

        \vspace{5mm}
        
        \begin{tikzpicture}
            \node at (0,0) {$[\alpha_1]$};
            \node[draw,circle,inner sep=0.6mm] at (0,0.4) {};
            \draw (0.08,0.46)--(1.42,0.46);
            \draw (0.08,0.4)--(1.42,0.4);
            \draw (0.08,0.34)--(1.42,0.34);
            \node at (0.75,0.4) {$<$};
            \node at (1.5,0) {$[\alpha_2]$};
            \node[draw,circle,inner sep=0.6mm] at (1.5,0.4) {};
        \end{tikzpicture}
    \end{center}

    Note that $[\alpha_1] \sim A_1^3 = A_1 \times A_1 \times A_1$ and $[\alpha_2] \sim A_1$.

    \item $\Phi$ is of type $E_6$ and $o(\rho)=2$. Then $\Phi_\rho$ is of type $F_4$.

    \begin{center}
        \begin{tikzpicture}
            \node at (0,0) {$\alpha_1$};
            \node[draw,circle,inner sep=0.6mm] at (0,-0.4) {};
            \draw (0.08,-0.4)--(1.42,-0.4);
            \node at (1.5,0) {$\alpha_2$};
            \node[draw,circle,inner sep=0.6mm] at (1.5,-0.4) {};

            \draw (1.58,-0.4)--(2.92,-1.34);
            \node at (3,-1.8) {$\alpha_3$};
            \node[draw,circle,inner sep=0.6mm] at (3,-1.4) {};
            \draw (1.58,-2.4)--(2.92,-1.46);
            \draw (3.08,-1.4)--(4.42,-1.4);
            \node at (4.5,-1.8) {$\alpha_4$};
            \node[draw,circle,inner sep=0.6mm] at (4.5,-1.4) {};
            
            \node at (1.5,-2.8) {$\alpha_{5}$};
            \node[draw,circle,inner sep=0.6mm] at (1.5,-2.4) {};
            \draw (0.08,-2.4)--(1.42,-2.4);
            \node at (0,-2.8) {$\alpha_{6}$};
            \node[draw,circle,inner sep=0.6mm] at (0,-2.4) {};
		\end{tikzpicture}
    
        \vspace{5mm}
        
        \begin{tikzpicture}
            \node at (0,0) {$[\alpha_1]$};
            \node[draw,circle,inner sep=0.6mm] at (0,0.4) {};
            \draw (0.08,0.4)--(1.42,0.4);
            \node at (1.5,0) {$[\alpha_2]$};
            \node[draw,circle,inner sep=0.6mm] at (1.5,0.4) {};
            \draw (1.58,0.45)--(2.92,0.45);
            \node at (2.25,0.4) {$<$};
            \draw (1.58,0.35)--(2.92,0.35);
            \node at (3,0) {$[\alpha_{3}]$};
            \node[draw,circle,inner sep=0.6mm] at (3,0.4) {};
            \draw (3.08,0.4)--(4.42,0.4);
            \node at (4.5,0) {$[\alpha_{4}]$};
            \node[draw,circle,inner sep=0.6mm] at (4.5,0.4) {};
        \end{tikzpicture}
    \end{center}
    
    Note that $[\alpha_1], [\alpha_2] \sim A_1 \times A_1$ and $[\alpha_3], [\alpha_4] \sim A_1$.

    \item $\Phi$ is of type $B_2$ and $o(\rho)=2$. Then $\Phi_\rho$ is of type $A_1$.

    \begin{center}
        \begin{tikzpicture}
                \node at (0,0) {$\alpha_1$};
                \node[draw,circle,inner sep=0.6mm] at (0,0.4) {};
                \draw (0.08,0.45)--(1.42,0.45);
                \draw (0.08,0.35)--(1.42,0.35);
                \node at (0.75,0.4) {$>$};
                \node at (1.5,0) {$\alpha_2$};
                \node[draw,circle,inner sep=0.6mm] at (1.5,0.4) {};
        \end{tikzpicture}
    
        \vspace{5mm}
        
        \begin{tikzpicture}
            \node at (0,0) {$[\alpha_1]$};
            \node[draw,circle,inner sep=0.6mm] at (0,0.4) {};
        \end{tikzpicture}
    \end{center}

    Note that $[\alpha_1] \sim B_2$.

    \item $\Phi$ is of type $G_2$ and $o(\rho)=2$. Then $\Phi_\rho$ is of type $A_1$.
    \begin{center}
        \begin{tikzpicture}
            \node at (0,0) {$\alpha_1$};
            \node[draw,circle,inner sep=0.6mm] at (0,0.4) {};
            \draw (0.08,0.46)--(1.42,0.46);
            \draw (0.08,0.4)--(1.42,0.4);
            \draw (0.08,0.34)--(1.42,0.34);
            \node at (0.75,0.4) {$>$};
            \node at (1.5,0) {$\alpha_2$};
            \node[draw,circle,inner sep=0.6mm] at (1.5,0.4) {};
        \end{tikzpicture}
        
        \vspace{5mm}
        
        \begin{tikzpicture}
            \node at (0,0) {$[\alpha_1]$};
            \node[draw,circle,inner sep=0.6mm] at (0,0.4) {};
        \end{tikzpicture}
    \end{center}
    
    Note that $[\alpha_1] \sim G_2$.

    \item $\Phi$ is of type $F_4$ and $o(\rho)=2$. Then $\Phi_\rho$ is of type $I_2(8)$. In this case, $W_\rho$ is the dihedral group of order $16$.

    \begin{center}
        \begin{tikzpicture}
            \node at (0,0) {$\alpha_1$};
            \node[draw,circle,inner sep=0.6mm] at (0,0.4) {};
            \draw (0.08,0.4)--(1.42,0.4);
            \node at (1.5,0) {$\alpha_2$};
            \node[draw,circle,inner sep=0.6mm] at (1.5,0.4) {};
            \draw (1.58,0.45)--(2.92,0.45);
            \node at (2.25,0.4) {$>$};
            \draw (1.58,0.35)--(2.92,0.35);
            \node at (3,0) {$\alpha_{3}$};
            \node[draw,circle,inner sep=0.6mm] at (3,0.4) {};
            \draw (3.08,0.4)--(4.42,0.4);
            \node at (4.5,0) {$\alpha_{4}$};
            \node[draw,circle,inner sep=0.6mm] at (4.5,0.4) {};
        \end{tikzpicture}

        \vspace{5mm}

        \begin{tikzpicture}
            \node at (0,0) {$[\alpha_1]$};
            \node[draw,circle,inner sep=0.6mm] at (0,0.4) {};
            \draw (0.08,0.4)--(1.42,0.4);
            \node at (0.75,0.6) {$8$};
            \node at (1.5,0) {$[\alpha_2]$};
            \node[draw,circle,inner sep=0.6mm] at (1.5,0.4) {};
        \end{tikzpicture}
    \end{center}
    
    Note that $[\alpha_1] \sim A_1 \times A_1$ and $[\alpha_2] \sim B_2$. (The last graph is the Coxeter graph of $I_2(m),$ where $m=8$.)
\end{enumerate}

In the following table we summaries above description of $\Phi_\rho$ and $\Tilde{\Phi}_\rho$ after the twist:

\begin{center}
    \begin{tabular}{|c|c|c|c|c|}
        \hline
        \multirow{2}{*}{\textbf{Type of $\Phi$}} & \multirow{2}{*}{\textbf{$\Tilde{\Phi}_\rho$}} & \multirow{2}{*}{\textbf{$\Phi_\rho$}} & \multicolumn{2}{c|}{\textbf{Type of Roots}} \\
        \cline{4-5}
        & & & \textbf{Long} & \textbf{Short}  \\
        \hline 
        ${}^2 A_{2n-1} \ (n \geq 2)$ & $C_n$ & $C_n$ & $A_1$ & $A_1^2$ \\
        \hline 
        ${}^2 A_{2n} \ (n \geq 2)$ & $BC_n$ & $B_n$ & $A_1^2$ & $A_2$ \\
        \hline 
        ${}^2 D_{n} \ (n \geq 4)$ & $B_{n-1}$ & $B_{n-1}$ & $A_1$ & $A_1^2$ \\
        \hline
        ${}^3 D_{4}$ & $G_2$ & $G_2$ & $A_1$ & $A_1^3$ \\
        \hline
        ${}^2 E_{6}$ & $F_4$ & $F_4$ & $A_1$ & $A_1^2$ \\
        \hline
        ${}^2 B_{2}$ & $A_1$ & $A_1$ & $B_2$ & $-$ \\
        \hline
        ${}^2 G_{2}$ & $A_1$ & $A_1$ & $G_2$ & $-$ \\
        \hline
        ${}^2 F_{4}$ & $I_2(8)$ & $I_2(8)$ & $A_1^2, B_2$ & $-$ \\
        \hline
    \end{tabular}
\end{center}

%%%%%%%%%%%%%%%%%%%%%%%%%%%%%%%%%%%%%%%%%%%%%%%%%%%%%%%%%%%%%%
%%%%%%%%%%%%%%%%%%%%%%%%%%%%%%%%%%%%%%%%%%%%%%%%%%%%%%%%%%%%%%

\subsection{Types of Pairs of Roots in \texorpdfstring{$\Phi_\rho$}{the twisted root system}}

We aim to categorize pairs of roots $[\alpha]$ and $[\beta] (\neq \pm [\alpha])$ in $\Phi_\rho$ according to their placement within the subsystem generated by them. The positions of $[\alpha]$ and $[\beta]$ can be fully determined by considering their possible sums and differences. Hence we classify them into the following types (see \cite{EA1}):
\begin{enumerate}[(a)]
    \item $[\alpha], [\beta] \in \Phi_\rho$, but $[\alpha]+[\beta],[\alpha]-[\beta] \notin \Phi_\rho$.
        \begin{enumerate}
            \item[(a$_1$)] $[\alpha]+[\beta],[\alpha]-[\beta] \notin \Tilde{\Phi}_\rho$.
            \item[(a$_2$)] $[\alpha]+[\beta],[\alpha]-[\beta] \in \Tilde{\Phi}_\rho$.
                \begin{enumerate}[(i)]
                    \item $[\alpha] \sim A_1, [\beta] \sim A_1, 1/2([\alpha]+[\beta])=[\gamma] \sim A_1^2,$ where $\gamma + \bar{\gamma} = \alpha + \beta \notin \Phi$;
                    \item $[\alpha] \sim A_1^2, [\beta] \sim A_1^2, 1/2([\alpha]+[\beta])=[\gamma] \sim A_2,$ where $\gamma + \bar{\gamma} = \alpha + \beta \text{ or } \alpha + \bar{\beta} \in \Phi$.
                \end{enumerate}
        \end{enumerate}
        
    \item $[\alpha], [\beta] \in \Phi_\rho, [\alpha]+[\beta] \in \Phi_\rho,$ but $[\alpha]-[\beta] \notin \Phi_\rho$.
        \begin{enumerate}[(i)]
            \item $[\alpha] \sim A_1, [\beta] \sim A_1, [\alpha]+[\beta]=[\alpha + \beta] \sim A_1;$
            \item $[\alpha] \sim A_1^2, [\beta] \sim A_1^2, [\alpha]+[\beta] = [\alpha + \beta] \text{ or } [\alpha + \bar{\beta}] \text{ or } [\bar{\alpha} + \beta] \sim A_1^2.$
        \end{enumerate}
        
    \item $[\alpha], [\beta] \in \Phi_\rho, [\alpha]+[\beta] \in \Phi_\rho, [\alpha]-[\beta] \in \Phi_\rho$.
        \begin{enumerate}[(i)]
            \item $[\alpha] \sim A_1^2, [\beta] \sim A_1^2, [\alpha]+[\beta]=[\alpha + \bar{\beta}] \sim A_1;$
            \item $[\alpha] \sim A_2, [\beta] \sim A_2, [\alpha]+[\beta]=[\alpha + \bar{\beta}] \text{ or } [\bar{\alpha} + \beta] \sim A_1^2.$
        \end{enumerate}
        
    \item $[\alpha], [\beta] \in \Phi_\rho, [\alpha]+[\beta] \in \Phi_\rho, [\alpha]+ 2[\beta] \in \Phi_\rho$.
        \begin{enumerate}[(i)]
            \item $[\alpha] \sim A_1, [\beta] \sim A_1^2, [\alpha]+[\beta]=[\alpha + \beta] \sim A_1^2, [\alpha]+2[\beta]=[\alpha + \beta + \bar{\beta}] \sim A_1;$
            \item $[\alpha] \sim A_1^2, [\beta] \sim A_2, [\alpha]+[\beta]= [\alpha + \beta] \text{ or }[\alpha + \bar{\beta}] \text{ or } [\bar{\alpha} + \beta] \sim A_2, [\alpha]+2[\beta]=[\alpha + \beta + \bar{\beta}] \sim A_1^2.$
        \end{enumerate}
        
    \item $[\alpha], [\beta] \in \Phi_\rho, [\alpha]+[\beta] \in \Phi_\rho, [\alpha]+ 2[\beta] \in \Phi_\rho, [\alpha]+ 3[\beta] \in \Phi_\rho, 2[\alpha]+ 3[\beta] \in \Phi_\rho$. 
        \begin{itemize}
            \item $[\alpha] \sim A_1, [\beta] \sim A_1^3, [\alpha]+[\beta]= [\alpha + \beta] \sim A_1^3, [\alpha]+2[\beta]=[\alpha + \beta + \bar{\beta}] \text{ or } [\alpha + \beta + \bar{\bar{\beta}}]\sim A_1^3, [\alpha]+3[\beta]=[\alpha + \beta + \bar{\beta} + \bar{\bar{\beta}}] \sim A_1, 2[\alpha]+3[\beta]=[2\alpha + \beta + \bar{\beta} + \bar{\bar{\beta}}] \sim A_1.$
        \end{itemize}
        
    \item $[\alpha], [\beta] \in \Phi_\rho, [\alpha]+[\beta] \in \Phi_\rho, 2[\alpha]+ [\beta] \in \Phi_\rho,  [\alpha]+ 2[\beta] \in \Phi_\rho, [\alpha]-[\beta] \in \Phi_\rho.$
        \begin{itemize}
            \item $[\alpha] \sim A_1^3, [\beta] \sim A_1^3, [\alpha]+[\beta] = [\alpha + \bar{\beta}] \text{ or } [\alpha + \bar{\bar{\beta}}] \sim A_1^3, 2[\alpha]+[\beta]=[\alpha + \bar{\alpha} + \bar{\bar{\beta}}] \sim A_1, [\alpha]+2[\beta]=[\alpha + \bar{\beta} + \bar{\bar{\beta}}] \sim A_1, [\alpha]-[\beta] = [\alpha - \beta] \sim A_1.$
        \end{itemize}
        
    \item $[\alpha], [\beta] \in \Phi_\rho, [\alpha]+[\beta] \in \Phi_\rho, [\alpha]-[\beta] \in \Phi_\rho, 2[\alpha]- [\beta] \in \Phi_\rho,  [\alpha]- 2[\beta] \in \Phi_\rho.$
        \begin{itemize}
            \item $[\alpha] \sim A_1^3, [\beta] \sim A_1^3, [\alpha]+[\beta] = [\alpha + \bar{\beta}] \text{ or } [\alpha + \bar{\bar{\beta}}] \sim A_1, [\alpha]-[\beta] = [\alpha - \bar{\beta}] \text{ or } [\alpha - \bar{\bar{\beta}}] \sim A_1^3, 2[\alpha]-[\beta]=[\alpha + \bar{\alpha} - \beta] \text{ or } [\alpha + \bar{\alpha} - \bar{\beta}] \sim A_1, [\alpha]-2[\beta]= [\bar{\alpha} - \beta -\bar{\beta}] \text{ or } [\alpha - \beta - \bar{\beta}] \sim A_1.$
        \end{itemize}
\end{enumerate}

\begin{rmk}
    Cases $(e), (f)$ and $(g)$ occur exclusively in the context of ${}^3 D_4$.
\end{rmk}

%%%%%%%%%%%%%%%%%%%%%%%%%%%%%%%%%%%%%%%%%%%%%%%%%%%%%%%%%%%%%%
%%%%%%%%%%%%%%%%%%%%%%%%%%%%%%%%%%%%%%%%%%%%%%%%%%%%%%%%%%%%%%

\subsection{Twisted Weights}\index{twisted weights}

Finally, we examine the action of $\rho$ on the weight lattice $\Lambda_{sc}$. 
Assume that the root system $\Phi$ has one root length. 
Since $\rho$ permutes the simple roots (and thus all roots), its action on the root lattice $\Lambda_r$ is clear.
The fundamental dominant weights $\lambda_1, \dots, \lambda_\ell$ form a $\mathbb{Z}$-basis for the weight lattice $\Lambda_{sc}$. 
The action of $\rho$ on $\Lambda_r$ can be naturally extended to $\Lambda_{sc}$ as follows: for $\lambda_i \in \Lambda_{sc}$, define $\rho(\lambda_i) = \lambda_j$ whenever $\rho(\alpha_i) = \alpha_j$. This induces a $\mathbb{Z}$-linear automorphism of $\Lambda_{sc}$, satisfying $\rho(\Lambda_r) = \Lambda_r$.
Thus, $\rho$ can be viewed as a group automorphism of the fundamental group $\Lambda_{sc} / \Lambda_r$ associated with $\Phi$. 

Let $\Lambda$ denote a sublattice of $\Lambda_{sc}$ containing $\Lambda_r$. 
The quotient $\Lambda / \Lambda_r$ forms a subgroup of $\Lambda_{sc} / \Lambda_r$, which is cyclic except in the case $\Phi = D_{2n}$. 
Therefore, $\rho$ stabilizes this subgroup, i.e., $\rho(\Lambda / \Lambda_r) = \Lambda / \Lambda_r$, and hence $\rho(\Lambda) = \Lambda$.
For the case of $\Phi = D_{2n}$, the fundamental group $\Lambda_{sc} / \Lambda_r$ is isomorphic to $\mathbb{Z}_2 \times \mathbb{Z}_2$. 
Hence there are precisely two proper sublattices $\Lambda_1$ and $\Lambda_2$ of $\Lambda_{sc}$ that contain $\Lambda_r$ as a proper sublattice, satisfying the condition that $\rho (\Lambda_i) \not\subset \Lambda_i$ for $i=1,2$.
Therefore, if $\Lambda_\pi = \Lambda_1$ or $\Lambda_2$, then there is no graph automorphism of $G_\pi (\Phi, R)$ and $E_\pi (\Phi, R)$ even when $1/2 \in R$ (as shown in \cite[page 91]{RS}).

%%%%%%%%%%%%%%%%%%%%%%%%%%%%%%%%%%%%%%%%%%%%%%%%%%%%%%%%%%%%%%
\section{Twisted Chevalley Algebras}\label{sec:TCA}
%%%%%%%%%%%%%%%%%%%%%%%%%%%%%%%%%%%%%%%%%%%%%%%%%%%%%%%%%%%%%%

Let $\mathcal{L} = \mathcal{L}(\Phi, \mathbb{C})$ be a simple Lie algebra over $\mathbb{C}$ with root system $\Phi$. 
Assume that $\Phi$ is of type $A_\ell \ (\ell \geq 2)$, $D_\ell \ (\ell \geq 4)$, or $E_6$. 
Let $\Delta = \{ \alpha_1, \dots, \alpha_\ell \}$ be a fixed simple system.
Consider the Chevalley basis $\{X_\alpha, H_i \mid \alpha \in \Phi, \ i=1, \dots, \ell\}$. 
Let $\rho$ be a non-trivial angle-preserving permutation of $\Phi$, and write $\bar{\alpha} = \rho(\alpha)$ for all $\alpha \in \Phi$. 
This induces an automorphism of $\mathcal{L}$, also denoted by $\rho$, satisfying 
\[
    \rho (H_\alpha) = H_{\bar{\alpha}}, \quad 
    \rho (X_\alpha) = X_{\bar{\alpha}}, \quad 
    \rho (X_{-\alpha}) = X_{-\bar{\alpha}}, 
\]
for all $\alpha \in \Delta$. 
Then we have $\rho (X_\alpha) = \epsilon_\alpha X_{\bar{\alpha}}$, where $\epsilon_\alpha = \pm 1$ for all $\alpha \in \Phi$. 

\begin{lemma}[{\cite[Proposition 3.1]{EA1}}]\label{epsilonalpha}
    \normalfont
    A Chevalley basis of $\mathcal{L}$ can be chosen such that:
    \begin{enumerate}[(a)]
        \item $\epsilon_\alpha = \epsilon_{\bar{\alpha}}$, 
        \item $\epsilon_\alpha = -1$ if $[\alpha] \sim A_2$ and $\alpha = \bar{\alpha}$, 
        \item $\epsilon_\alpha = 1$ otherwise.
    \end{enumerate}
\end{lemma}

Let $R$ be a commutative ring with unity, and let $\theta: R \to R$ be an automorphism of $R$ of order $2$. 
Write $\bar{r} = \theta(r)$ for all $r \in R$. 
Define $R_\theta = \{r \in R \mid r = \bar{r}\}$ and $R_\theta^{-} = \{r \in R \mid r = -\bar{r}\}$. 
If $1/2 \in R$, then $R = R_\theta \oplus R_\theta^{-}$.

Let $\mathcal{L}(\Phi, R)$ denote a Chevalley algebra over $R$ of type $\Phi$ (cf. Section~\ref{sec:CheAlg}). 
The automorphism $\theta$ induces a semi-automorphism of $\mathcal{L}(\Phi, R)$, also denoted by $\theta$, satisfying 
\[
    \theta (r H_i) = \bar{r} H_i \quad (i=1, \dots, \ell), \qquad 
    \theta (r X_\alpha) = \bar{r} X_\alpha \quad (\alpha \in \Phi), 
\]
for all $r \in R$.

Define $\sigma = \rho \circ \theta$. 
Then $\sigma$ is a semi-automorphism of $\mathcal{L}(\Phi, R)$ satisfying:
\[
    \sigma (r H_{\alpha_i}) = \bar{r} H_{\bar{\alpha_i}} \quad (i = 1, \dots, \ell), \qquad 
    \sigma (r X_\alpha) = \epsilon_\alpha \bar{r} X_{\bar{\alpha}} \quad (\alpha \in \Phi),
\]
for all $r \in R$. 
Define 
\[
    \mathcal{L}_{\sigma}(\Phi, R) = \{ X \in \mathcal{L}(\Phi, R) \mid \sigma(X) = X \}.
\]
Then $\mathcal{L}_\sigma(\Phi, R)$ forms a submodule of an $R_\theta$-module $\mathcal{L}(\Phi, R)$. 
This submodule is called the \textbf{twisted Chevalley algebra} \index{twisted Chevalley algebra} over $R$ of type $\Phi$.

%%%%%%%%%%%%%%%%%%%%%%%%%%%%%%%%%%%%%%%%%%%%%%%%%%%%%%%%%%%%%%

\subsection{Basis of \texorpdfstring{$\mathcal{L}_\sigma (\Phi, R)$}{L(R)}}

Suppose there exists an invertible element \( a \in R_\theta^{-} \). Furthermore, assume that \( 2 \) is invertible in \( R \). Let 
\[
\{ X_\alpha, H_i \mid \alpha \in \Phi, \ i=1, \dots, \ell \}
\]
denote the Chevalley basis as described in Lemma~\ref{epsilonalpha}.

Let \( \rho \) be a non-trivial angle-preserving permutation of \( \Phi \) of order \( 2 \), and let \( \Phi_\rho \) be the corresponding twisted Chevalley root system. Define the following elements:

\begin{align*}
    X^{+}_{[\alpha]} &= 
    \begin{cases}
        X_{\alpha}, & \text{if } [\alpha] \sim A_1, \\
        X_{\alpha} + X_{\bar{\alpha}}, & \text{if } [\alpha] \sim A_1^2 \text{ or } A_2;
    \end{cases} \\
    X^{-}_{[\alpha]}(\mathrm{I}) &= a (X_{\alpha} - X_{\bar{\alpha}}) \quad \text{if } [\alpha] \sim A_1^2 \text{ or } A_2; \\
    X^{-}_{[\alpha]}(\mathrm{II}) &= a X_{\alpha + \bar{\alpha}} \quad \text{if } [\alpha] \sim A_2; \\
    H^{+}_{[\alpha]} &= 
    \begin{cases}
        H_{\alpha}, & \text{if } [\alpha] \sim A_1, \\
        H_{\alpha} + H_{\bar{\alpha}}, & \text{if } [\alpha] \sim A_1^2 \text{ or } A_2;
    \end{cases} \\
    H^{-}_{[\alpha]} &= a (H_{\alpha} - H_{\bar{\alpha}}) \quad \text{if } [\alpha] \sim A_1^2 \text{ or } A_2.
\end{align*}

Then, the set
\[
    \{ X^{+}_{[\alpha]}, X^{-}_{[\alpha]}(\mathrm{I}), X^{-}_{[\alpha]}(\mathrm{II}), H^{+}_{[\alpha_i]}, H^{-}_{[\alpha_i]} \mid [\alpha] \in \Phi_\rho, \ [\alpha_i] \in \Delta_\rho \}
\]
forms a basis for the \( R_\theta \)-module \( \mathcal{L}_\sigma(\Phi, R) \). 
Finally, note that 
\[
    \mathcal{L}(\Phi, R) = R \otimes_{R_\theta} \mathcal{L}_\sigma(\Phi, R).
\]
Therefore, the same set also forms a basis for the \( R \)-algebra \( \mathcal{L}(\Phi, R) \).

%%%%%%%%%%%%%%%%%%%%%%%%%%%%%%%%%%%%%%%%%%%%%%%%%%%%%%%%%%%%%%
\section{Twisted Chevalley Groups}\label{Subsec:TCG}
%%%%%%%%%%%%%%%%%%%%%%%%%%%%%%%%%%%%%%%%%%%%%%%%%%%%%%%%%%%%%%

Assume that $\Phi$ is of type $A_n (n \geq 2), D_n (n \geq 4)$ or $E_6$ and let $G_\pi (\Phi, R)$ (resp., $E_\pi (\Phi, R)$) be a Chevalley group (resp., an elementary Chevalley group) over a commutative ring $R$. Let $\sigma$\label{nomencl:sigma} be an automorphism of $G_\pi (\Phi, R)$ which is the product of a graph automorphism $\rho$ and a ring automorphism $\theta$ such that $o(\theta) = o(\rho)$. Denote the corresponding permutation of the roots also by $\rho$. Since $\rho \circ \theta = \theta \circ \rho$, we have $o(\theta) = o(\rho) = o(\sigma)$. Since $E_\pi (\Phi, R)$ is a characteristic subgroup of $G_\pi (\Phi, R)$, $\sigma$ is also an automorphism of $E_\pi (\Phi, R)$. 

Define $G_{\pi, \sigma} (\Phi, R) = \{ g \in G_\pi (\Phi, R) \mid \sigma(g)=g \}$. Clearly, $G_{\pi, \sigma} (\Phi, R)$ is a subgroup of $G_\pi (\Phi, R)$. We call $G_{\pi, \sigma} (\Phi, R)$\label{nomencl:G_sigma(R)} the \textbf{twisted Chevalley group} \index{twisted Chevalley group} over the ring $R$. The notion of the adjoint twisted Chevalley group and the universal (or simply connected) twisted Chevalley group is clear.

Write $E_{\pi, \sigma} (\Phi, R) = E_{\pi} (\Phi, R) \cap G_{\pi, \sigma} (\Phi, R)$. Consider the subgroups $U = U_\pi (\Phi, R)$, $U^{-} = U^{-}_{\pi} (\Phi, R)$, $H = H_\pi (\Phi, R)$, $B = B_\pi (\Phi, R)$ and $N = N_{\pi}(\Phi, R)$ of $E_\pi(\Phi, R)$. 
Then $\sigma$ preserves $U, H, B, U^-$ and $N$. 
Hence we can make sense of $U_\sigma = U_{\pi, \sigma} (\Phi, R)$, $U^{-}_\sigma = U^{-}_{\pi, \sigma} (\Phi, R)$, $H_\sigma = H_{\pi, \sigma} (\Phi, R)$, $B_\sigma = B_{\pi, \sigma} (\Phi, R)$ and $N_\sigma = N_{\pi, \sigma} (\Phi, R)$ (if $A \subset G(R)$ then we define $A_\sigma = A \cap G_\sigma(R)$). Note that $\sigma$ preserves $N/H \cong W$ (as it preserves $N$ and $H$). 
The action thus induced on $W$ is consistent with the permutation $\rho$ of the roots.
Finally, let us define $E'_{\pi, \sigma} (\Phi, R) = \langle U_\sigma, U_\sigma^- \rangle$, a subgroup of $E_{\pi, \sigma} (\Phi, R)$ generated by $U_\sigma$ and $U_\sigma^-$. We call $E'_{\pi, \sigma} (\Phi, R)$\label{nomencl:E'_sigma(R)} the \textbf{elementary twisted Chevalley group} \index{elementary twisted Chevalley group} over the ring $R$. 
Write $H'_\sigma = H'_{\pi, \sigma} (\Phi, R) := H \cap E'_\sigma(R)$, $N'_\sigma = N'_{\pi, \sigma} (\Phi, R) := N \cap E'_\sigma(R)$ and $B'_\sigma = B'_{\pi, \sigma} (\Phi, R) := B \cap E'_\sigma (R)$. Then $B'_\sigma = U_\sigma H'_\sigma$. 

Let $T_\pi (\Phi, R)$ be the standard maximal torus of $G_\pi (\Phi, R)$. Define 
\[
    T_{\pi, \sigma} (\Phi, R) := T_{\pi}(\Phi, R) \cap G_{\pi, \sigma} (\Phi, R)
\]
and called it the \textbf{split maximal torus} \index{subgroups of $G_\sigma(R)$!split maximal torus} of $G_{\pi, \sigma} (\Phi, R)$. 
For a character $\chi \in $ Hom$(\Lambda_\pi, R^*)$, we define its conjugation $\bar{\chi}_\sigma$ with respect to $\sigma$ by $\bar{\chi}_\sigma (\lambda) = \theta (\chi (\rho^{-1}(\lambda)))$ for every $\lambda \in \Lambda_\pi$. 
If $h(\chi) \in T_\pi (\Phi, R)$, then $\sigma (h(\chi)) = h(\bar{\chi}_\sigma)$.
A character $\chi \in $ Hom$(\Lambda_\pi, R^*)$ is called \textbf{self-conjugate with respect to $\sigma$} if $\chi = \bar{\chi}_\sigma$, i.e., $\chi (\rho(\lambda)) = \theta (\chi (\lambda)),$ for every $\lambda \in \Lambda_\pi$. We denote the set of such characters by $\text{Hom}_1 (\Lambda_\pi, R^*) = \{ \chi \in \text{Hom} (\Lambda_\pi, R^*) \mid \chi = \bar{\chi}_\sigma \}$.
Thus we have $T_{\pi, \sigma} (\Phi, R) = \{ h(\chi) \mid \chi \in \text{Hom}_1 (\Lambda_\pi, R^*) \}$. 
Note that, an element $h(\chi) \in H_\sigma \subset T_{\pi, \sigma} (\Phi, R)$ if and only if $\chi$ is a self-conjugate character of $\Lambda_\pi$ (with respect to $\sigma$) that can be extended to a self-conjugate character of $\Lambda_{sc}$.

For the sake of completeness, let us also define $G_{\pi, \sigma}^0 (\Phi, R) = G^0_\pi(\Phi, R) \cap G_{\pi, \sigma} (\Phi, R)$ and $G'_{\pi, \sigma} (\Phi, R) = T_{\pi, \sigma} (\Phi, R) E'_{\pi, \sigma} (\Phi, R)$. 

If $G_\pi (\Phi, R)$ is of type $X$ and $\sigma$ is of order $n$, we say $G_{\pi, \sigma} (\Phi, R)$ is of type ${}^n X.$ We write $G_\pi (\Phi, R) \sim X$ and $G_{\pi, \sigma} (\Phi, R) \sim {}^n X$. We use a similar notation for $E_\pi (\Phi, R)$, $E_{\pi, \sigma}(\Phi, R)$ and $E'_{\pi, \sigma}(\Phi, R)$. 

\begin{rmk}
    For convenience and to simplify notation, we sometimes use the notation $G_\sigma(R)$ or $G_\sigma(\Phi, R)$ in place of the more precise $G_{\pi, \sigma}(\Phi, R)$ when there is no risk of confusion. This convention is similarly extended to other groups described above.
\end{rmk}

%%%%%%%%%%%%%%%%%%%%%%%%%%%%%%%%%%%%%%%%%%%%%%%%%%

\subsection{\texorpdfstring{Identifying Certain Twisted and Untwisted Chevalley \\ Groups}{Identifying Certain Twisted and Untwisted Chevalley Groups}}\label{subsec:Chevalley as Twisted Chevalley}

This subsection will examine isomorphisms between certain Chevalley and twisted Chevalley groups.

Let \(\Phi\) be a root system of type \(A_n \ (n \geq 2)\), \(D_n \ (n \geq 4)\), or \(E_6\), and let \(\Delta\) be the corresponding simple system. Let \(\rho\) be the non-trivial angle-preserving permutation of the simple roots of \(\Phi\). Let \(R\) be a commutative ring with unity. If \(o(\rho) = 2\) (respectively, \(o(\rho) = 3\)), define the ring automorphism \(\theta: R \times R \longrightarrow R \times R\) by \((a,b) \mapsto (b,a)\) (respectively, \(\theta: R \times R \times R \longrightarrow R \times R \times R\) by \((a,b,c) \mapsto (b,c,a)\)).

Next, consider the automorphisms of the group \(G(R \times R)\) (resp., \(G(R \times R \times R)\)) induced by \(\rho\) and \(\theta\), which we shall also denote by the same symbols. Set \(\sigma = \rho \circ \theta\). Now, consider the twisted Chevalley group \(G_\sigma(R \times R)\) (resp., \(G_\sigma(R \times R \times R)\)).

\begin{prop}\label{prop:Chevalley as TChevalley}
    \normalfont
    Retaining the above notations, we establish the following isomorphisms:
    \begin{enumerate}[(a)]
        \item $G_\sigma(R \times R) \cong G(R)$ (respectively, $G_\sigma(R \times R \times R) \cong G(R)$). 
        \item $E'_\sigma(R \times R) \cong E(R)$ (respectively, $E'_\sigma(R \times R \times R) \cong E(R)$).
    \end{enumerate}
\end{prop}

\begin{proof}
    We shall prove the isomorphism \(G_\sigma(R \times R) \cong G(R)\). The remaining assertions can be established in a similar manner and are therefore omitted.

    Let \(x \in G_\sigma(R \times R)\). Since \(G_\sigma(R \times R) \subseteq G(R \times R) \cong G(R) \times G(R)\), there exist elements \(x_1, x_2 \in G(R)\) such that \(x\) corresponds to the pair \((x_1, x_2)\). By definition, \(x\) satisfies \(\sigma(x) = x\), which implies \((\rho \circ \theta)(x) = x\). Therefore, we have \(\theta(x) = \rho^{-1}(x) = \rho(x)\), meaning \( \theta((x_1, x_2)) = \rho((x_1, x_2)) \).

    Note that, $\theta (x) = \theta (x_1, x_2) = (x_2, x_1)$ and the permutation $\rho$ of simple roots induces an automorphism of $G(R)$ such that $\rho ((x_1, x_2)) = (\rho (x_1), \rho(x_2))$.
    Therefore, we have \[ (x_2, x_1) = \theta((x_1, x_2)) = \rho((x_1, x_2)) = (\rho(x_1), \rho(x_2)). \] 
    This implies that $x_2 = \rho(x_1)$, and hence $x = (x_1, \rho(x_1))$. Conversely, for any $x_1 \in G(R)$, the element $x = (x_1, \rho(x_1))$ belongs to $G_\sigma (R \times R)$. Thus, the map $\phi: G(R) \longrightarrow G_\sigma (R \times R)$ defined by $x \mapsto (x, \rho(x))$ establishes the desired isomorphism of groups.
\end{proof}

%%%%%%%%%%%%%%%%%%%%%%%%%%%%%%%%%%%%%%%%%%%%%%%%%%%%%%%%%%%%%%
\section{The Group \texorpdfstring{$E'_{\pi, \sigma} (\Phi, R)$}{E(R)}}\label{sec:E(R)}
%%%%%%%%%%%%%%%%%%%%%%%%%%%%%%%%%%%%%%%%%%%%%%%%%%%%%%%%%%%%%%

Before proceeding, let us establish some notations: denote $\bar{\alpha} = \rho(\alpha), \bar{\bar{\alpha}} = \rho^2(\alpha), \bar{t} = \theta(t)$ and $\bar{\bar{t}} = \theta^2(t)$.
Recall that for \( \sigma \) as defined earlier, the relation $\sigma(x_\alpha(t)) = x_{\bar{\alpha}}(\epsilon_\alpha \bar{t})$ holds for all \( \alpha \in \Phi \). 

From now on, we will always use a fixed Chevalley basis of \( \mathcal{L} \) satisfying the properties described in Lemma~\ref{epsilonalpha}. Under this choice of basis, the image of \( x_\alpha(t) \) under \( \sigma \) is given explicitly by:
\[
    \sigma(x_\alpha(t)) = 
    \begin{cases} 
        x_{\bar{\alpha}}(-\bar{t}), & \text{if } [\alpha] \sim A_2 \text{ and } \alpha = \bar{\alpha}; \\ 
        x_{\bar{\alpha}}(\bar{t}), & \text{otherwise}.
    \end{cases}
\]

%%%%%%%%%%%%%%%%%%%%%%%%%%%%%%%%%%%%%%%%%%%%%%%%%%%%%%%%%%%%%%
%%%%%%%%%%%%%%%%%%%%%%%%%%%%%%%%%%%%%%%%%%%%%%%%%%%%%%%%%%%%%%

\subsection{The Generators of \texorpdfstring{$E'_{\pi, \sigma} (\Phi, R)$}{E(R)}}

We fix a total order on $\Phi$ defined by $\Delta$. 
For a class $[\beta] \in \Phi_\rho$, we choose a representative $\alpha$ such that $\alpha$ is the smallest element in this class. 
Whenever $\alpha$ satisfies this condition, we denote the class $[\beta]$ by $[\alpha]$.
Now we define some special elements of $E'_\sigma (R)$ as follows: 
\begin{enumerate}
    \item If $[\alpha] \sim A_1$ (that is, $[\alpha]=\{ \alpha \}$), then define $x_{[\alpha]}(t) = x_\alpha (t), t \in R_\theta.$ In this case, $x_{[\alpha]}(t)x_{[\alpha]}(u)=x_{[\alpha]}(t+u)$ for every $t,u \in R_\theta$.
    \item If $[\alpha] \sim A_1^2$ (that is, $[\alpha]=\{ \alpha, \bar{\alpha} \}$), then define $x_{[\alpha]}(t) = x_\alpha (t) \ x_{\bar{\alpha}} (\bar{t}), t \in R$. In this case, $x_{[\alpha]}(t)x_{[\alpha]}(u)=x_{[\alpha]}(t+u)$ for every $t,u \in R$. 
    \item If ${[\alpha]} \sim A_1^3$ (that is, $[\alpha]=\{ \alpha, \bar{\alpha}, \bar{\bar{\alpha}} \}$), then define $x_{[\alpha]}(t) = x_\alpha (t) \ x_{\bar{\alpha}} (\bar{t}) \ x_{\bar{\bar{\alpha}}} (\bar{\bar{t}}), \ t \in R.$ 
    In this case, $x_{[\alpha]}(t)x_{[\alpha]}(u)=x_{[\alpha]}(t+u)$ for every $t,u \in R$. 
    \item If ${[\alpha]} \sim A_2$ with $\alpha \neq \bar{\alpha}$ (that is, ${[\alpha]}=\{ \alpha, \bar{\alpha}, \alpha + \bar{\alpha} \}$), then define $$x_{[\alpha]}(t,u) = x_\alpha (t) x_{\bar{\alpha}} (\bar{t}) x_{\alpha + \bar{\alpha}}(N_{\bar{\alpha}, \alpha} u), \hspace{1mm} \text{where } t,u \in R \text{ such that } t \bar{t} = u + \bar{u}.$$ 
    In this case, $x_{[\alpha]}(t,u)x_{[\alpha]}(t',u')=x_{[\alpha]}(t+t', u+u' + \bar{t} t')$ for every $t,u,t',u' \in R$ such that $t \bar{t} = u + \bar{u}$ and $t' \bar{t'} = u' + \bar{u'}$.
\end{enumerate}

Define $\mathcal{A}(R) := \{ (t,u) \mid t, u \in R \text{ such that } t \bar{t} = u + \bar{u} \}$. Note that, for $[\alpha] \sim A_2$ we define $x_{[\alpha]} (t,u)$ only in the case of $(t,u) \in \mathcal{A}(R)$. The product of $x_{[\alpha]}(t,u)$ and $x_{[\alpha]}(t',u')$ suggest the operation on $\mathcal{A}(R)$ as follows: let $(t,u), (t',u') \in \mathcal{A}(R)$, then define an operation $\oplus$ on $\mathcal{A}(R)$ by $(t,u) \oplus (t',u') = (t + t', u+u'+\bar{t}t').$ With this operation $\mathcal{A}(R)$ becomes a group with $(0,0)$ as an identity and $(-t, \bar{u})$ as an inverse of $(t,u)$. From this we can say that $(x_{[\alpha]}(t,u))^{-1} = x_{[\alpha]} (-t, \bar{u})$. Further, we can define an action of the monoid $(R, \times)$ on the set $\mathcal{A}(R)$ by $$r \cdot (t,u) = (rt, r \bar{r} u)$$ for any $r \in R$ and $(t,u) \in \mathcal{A}(R)$.

For $[\alpha] \in \Phi_\rho$, we write 
\begin{align*}
    R_{[\alpha]} = \begin{cases}
        R_\theta & \text{if } [\alpha] \sim A_1, \\
        R & \text{if } [\alpha] \sim A_1^2 \text{ or } A_1^3, \\
        \mathcal{A}(R) & \text{if } [\alpha] \sim A_2.
    \end{cases}
\end{align*}
If $[\alpha] \sim A_2$ then $t \in R_{[\alpha]}$ means that $t$ is a pair $(t_1, t_2)$ such that $(t_1, t_2) \in \mathcal{A}(R)$. Additionally, for $r \in R$ and $t \in R_{[\alpha]}$, the notation $r \cdot t$ means $rt$ if $[\alpha] \sim A_1, A_1^2, A_1^3$, and it means $(rt_1, r \bar{r} t_2)$ if $[\alpha] \sim A_2$.

Define $\mathfrak{X}_{[\alpha]}$ to be the subset of $E'_\sigma(R)$ consists of all $x_{[\alpha]}(t), \ t \in R_{[\alpha]}$. Clearly, $\mathfrak{X}_{[\alpha]}$ is a subgroup of $E_\sigma'(R)$. For a subset $S$ of $\Phi_\rho$, define $\mathfrak{X}_S$ to be the subgroup generated by $\mathfrak{X}_{[\alpha]}, \ [\alpha] \in S$.

\begin{prop}\label{lemma:X_alpha and R_alpha}
    \normalfont
    The subgroup $\mathfrak{X}_{[\alpha]}$ is isomorphic to the additive group $R_{[\alpha]}$.
\end{prop}

The proof of this proposition is straightforward. Next, we present the following result from Steinberg \cite{RS}.

\begin{prop}[{\cite[Lemma 62]{RS}}]\label{lemma 62 of RS}
    \normalfont
    Let $S$ be a closed subset of $\Phi_\rho$, i.e., if $[\alpha], [\beta] \in S$ then $[\alpha] + [\beta] \in S$. Moreover, assume that if $[\alpha] \in S$ then $-[\alpha] \notin S$. Then $\mathfrak{X}_S = \prod_{[\alpha] \in S} \mathfrak{X}_{[\alpha]}$ with the product taken in any fixed order and there is uniqueness of expression on the right. In particular, $U_\sigma = \prod_{[\alpha] > 0} \mathfrak{X}_{[\alpha]}$.
\end{prop}

\begin{cor}
    \normalfont
    The group $E'_\sigma (R)$ is generated by $x_{[\alpha]}(t)$ for all $[\alpha] \in \Phi_\rho$ and $t \in R_{[\alpha]}$.
\end{cor}

The proof of this corollary follows directly from the definition of $E'_\sigma (R)$ and the preceding propositions. 
Finally, we highlight another useful consequence of Proposition~\ref{lemma 62 of RS}, which is also stated in Steinberg \cite{RS}.

\begin{cor}\label{Che-Comm}
    \normalfont
    Let $[\alpha], [\beta] \in \Phi_\rho$ be such that $[\alpha] \neq \pm [\beta]$. Then $[\mathfrak{X}_{[\alpha]}, \mathfrak{X}_{[\beta]}] \subset \mathfrak{X}_S$, where $S = \{ i[\alpha] + j[\beta] \in \Phi_\rho \mid i,j \in \frac{1}{2} \mathbb{Z}_{> 0} \}$.
\end{cor}

\begin{rmk}
    The precise commutator relations will be given in \ref{subsec:CheComm}.
\end{rmk}

%%%%%%%%%%%%%%%%%%%%%%%%%%%%%%%%%%%%%%%%%%%%%%%%%%%%%%%%%%%%%%
%%%%%%%%%%%%%%%%%%%%%%%%%%%%%%%%%%%%%%%%%%%%%%%%%%%%%%%%%%%%%%

\subsection{Chevalley Commutator Formulas}\label{subsec:CheComm} \index{Chevalley commutator formulas}

We now present pivotal formulas that will serve as a cornerstone for the remaining chapters. The numbering in the following formulas reflects the types of root pairs mentioned in the preceding subsection. 

\begin{enumerate}[leftmargin=4em]
    \item[\textbf{(a$_1$)}] $[x_{[\alpha]}(t),x_{[\beta]}(u)]=1$, where $t \in R_{[\alpha]}$ and $u \in R_{[\beta]}$.
    \item[\textbf{(a$_2-$i)}] $[x_{[\alpha]}(t),x_{[\beta]}(u)]=1$, where $t,u \in R_\theta$.
    \item[\textbf{(a$_2-$ii)}] $[x_{[\alpha]}(t),x_{[\beta]}(u)]= x_{[\gamma]}(0, N_{\alpha,\beta} N_{\bar{\gamma},\gamma}(t \bar{u} - \bar{t} u))$, where $1/2([\alpha]+[\beta]) = [\gamma] =\{ \gamma, \bar{\gamma}, \gamma + \bar{\gamma} \}$ and $t,u \in R$.
    \item[\textbf{(b$-$i)}] $[x_{[\alpha]}(t),x_{[\beta]}(u)]= x_{[\alpha]+[\beta]}(N_{\alpha,\beta}tu)$, where $t,u \in R_\theta$.
    \item[\textbf{(b$-$ii)}] $[x_{[\alpha]}(t),x_{[\beta]}(u)]= x_{[\alpha]+[\beta]}(N_{\alpha,\beta}tu)$ or $x_{[\alpha]+[\beta]}(N_{\alpha,\bar{\beta}} t \bar{u})$ or $x_{[\alpha]+[\beta]}(N_{\bar{\alpha},\beta} \bar{t} u)$, where $t,u \in R$.
    \item[\textbf{(c$-$i)}] $[x_{[\alpha]}(t),x_{[\beta]}(u)]= x_{[\alpha]+[\beta]}(N_{\alpha,\bar{\beta}}(t \bar{u} + \bar{t} u))$, where $t,u \in R$.
    \item[\textbf{(c$-$ii)}] $[x_{[\alpha]}(t_1, t_2),x_{[\beta]}(u_1, u_2)]= x_{[\alpha]+[\beta]}(N_{\alpha,\bar{\beta}} t_1 \bar{u_1})$ or $x_{[\alpha]+[\beta]}(N_{\bar{\alpha},\beta} \bar{t_1} u_1)$, where \newline $(t_1, t_2), (u_1, u_2) \in \mathcal{A}(R)$.
    \item[\textbf{(d$-$i)}] $[x_{[\alpha]}(t),x_{[\beta]}(u)]= x_{[\alpha]+[\beta]}(N_{\alpha,\beta}tu) x_{[\alpha]+2[\beta]}(N_{\alpha,\beta} N_{\beta,\alpha + \bar{\beta}} t u \bar{u})$, where $t \in R_\theta, u \in R$.
    \item[\textbf{(d$-$ii)}] 
        \begin{align*} 
                [x_{[\alpha]}(t),x_{[\beta]}(u_1, u_2)] = \begin{cases}
                    x_{[\alpha]+[\beta]}((N_{\alpha,\beta}t u_1, t \bar{t} u_2))x_{[\alpha]+2[\beta]}(N_{\beta,\bar{\beta}} N_{\beta+ \bar{\beta},\alpha} t u_2),  \text{ or }  \\
                    x_{[\alpha]+[\beta]}((N_{\alpha,\bar{\beta}} t \bar{u_1}, t \bar{t} \bar{u_2}))x_{[\alpha]+2[\beta]}(N_{\beta,\bar{\beta}} N_{\beta+ \bar{\beta},\alpha} t \bar{u_2}), \text{ or } \\
                    x_{[\alpha]+[\beta]}((N_{\bar{\alpha},\beta} \bar{t} u_1, t \bar{t} u_2))x_{[\alpha]+2[\beta]}(N_{\beta,\bar{\beta}} N_{\beta+ \bar{\beta},\alpha} t \bar{u_2});
                \end{cases}
        \end{align*}
        where $t \in R$ and $(u_1, u_2) \in \mathcal{A}(R)$.
    \item [\textbf{(e)}] 
        \begin{align*} 
                [x_{[\alpha]}(t),x_{[\beta]}(u)] = \begin{cases}
                    x_{[\alpha]+[\beta]}(N_{\alpha,\beta}tu) x_{[\alpha]+2[\beta]}(N_{\alpha, \bar{\beta}} N_{\beta,\alpha + \bar{\beta}} t u \bar{u}) \\
                    \hspace{5mm} x_{[\alpha]+3[\beta]}(N_{\alpha, \bar{\bar{\beta}}} N_{\beta,\alpha + \bar{\bar{\beta}}} N_{\beta,\alpha + \bar{\beta} + \bar{\bar{\beta}}} t u \bar{u} \bar{\bar{u}}) \\
                    \hspace{10mm} x_{2[\alpha]+3[\beta]}(N_{\beta,\alpha + \bar{\beta}} N_{\alpha + \beta + \bar{\beta}, \alpha + \bar{\bar{\beta}}} t^2 u \bar{u} \bar{\bar{u}}), \text{ or }  \\
                    x_{[\alpha]+[\beta]}(N_{\alpha,\beta}tu) x_{[\alpha]+2[\beta]}(N_{\alpha, \bar{\bar{\beta}}} N_{\beta,\alpha + \bar{\bar{\beta}}} t u \bar{\bar{u}}) \\
                    \hspace{5mm} x_{[\alpha]+3[\beta]}(N_{\alpha, \bar{\bar{\beta}}} N_{\beta,\alpha + \bar{\bar{\beta}}} N_{\beta,\alpha + \bar{\beta} + \bar{\bar{\beta}}} t u \bar{u} \bar{\bar{u}}) \\
                    \hspace{10mm} x_{2[\alpha]+3[\beta]}(N_{\beta,\alpha + \bar{\beta}} N_{\alpha + \beta + \bar{\beta}, \alpha + \bar{\bar{\beta}}} t^2 u \bar{u} \bar{\bar{u}});
                \end{cases}
        \end{align*}
    where $t \in R_\theta, u \in R$.
    \item [\textbf{(f)}] 
        \begin{align*} 
            [x_{[\alpha]}(t),x_{[\beta]}(u)] = \begin{cases}
                x_{[\alpha]+[\beta]}(N_{\alpha,\bar{\beta}} \ t \bar{u} + N_{\bar{\alpha}, \beta} \ \bar{t} u) \\
                \hspace{5mm} x_{2[\alpha]+[\beta]}(N_{\alpha, \bar{\alpha} + \bar{\bar{\beta}}} N_{\bar{\alpha}, \bar{\bar{\beta}}} (t \bar{t} \bar{\bar{u}} + \bar{t} \bar{\bar{t}} u + t \bar{\bar{t}} \bar{u})) \\
                \hspace{10mm} x_{[\alpha]+2[\beta]}(N_{\bar {\beta}, \alpha + \bar{\bar{\beta}}} N_{\alpha, \bar{\bar{\beta}}} (t \bar{u} \bar{\bar{u}} + \bar{t} u \bar{\bar{u}} + \bar{\bar{t}} u \bar{u})), \text{ or }  \\
                x_{[\alpha]+[\beta]}(N_{\alpha,\bar{\bar{\beta}}} t \bar{\bar{u}} + N_{\bar{\bar{\alpha}}, \beta} \bar{\bar{t}} u) \\
                \hspace{5mm} x_{2[\alpha]+[\beta]}(N_{\alpha, \bar{\alpha} + \bar{\bar{\beta}}} N_{\bar{\alpha}, \bar{\bar{\beta}}} (t \bar{t} \bar{\bar{u}} + \bar{t} \bar{\bar{t}} u + t \bar{\bar{t}} \bar{u})) \\
                \hspace{10mm} x_{[\alpha]+2[\beta]}(N_{\bar {\beta}, \alpha + \bar{\bar{\beta}}} N_{\alpha, \bar{\bar{\beta}}} (t \bar{u} \bar{\bar{u}} + \bar{t} u \bar{\bar{u}} + \bar{\bar{t}} u \bar{u}));
                \end{cases}
        \end{align*}
    where $t,u \in R$.
    \item [\textbf{(g)}] $[x_{[\alpha]}(t),x_{[\beta]}(u)] = x_{[\alpha]+[\beta]}(N_{\alpha,\bar{\beta}} (t \bar{u} + \bar{t} \bar{\bar{u}} + \bar{\bar{t}} u))$ or $x_{[\alpha]+[\beta]}(N_{\alpha,\bar{\bar{\beta}}} (t \bar{\bar{u}} + \bar{t} u + \bar{\bar{t}} \bar{u})),$ where $t, u \in R$.
\end{enumerate}

\begin{rmk}
    For the proof of $(a_1)$ to $(d-ii)$, see \cite{EA1}. We will give a proof of part $(e), (f)$ and $(g)$. 
\end{rmk}

\noindent \textit{Proof of $(e):$} Using commutator relations in $G_{\pi}(\Phi,R)$, we can show that
\begin{align*}
    [x_{[\alpha]}(t), x_{[\beta]}(u)] &= \{ x_{\alpha + \beta}(N_{\alpha, \beta} t u) x_{\alpha + \bar{\beta}}(N_{\alpha, \bar{\beta}} t \bar{u}) x_{\alpha + \bar{\bar{\beta}}}(N_{\alpha, \bar{\bar{\beta}}} t \bar{\bar{u}}) \} \{ x_{\alpha + \beta + \bar{\beta}}(N_{\alpha, \bar{\beta}} N_{\beta, \alpha + \bar{\beta}}  t u \bar{u}) \\
    & \hspace{10mm} x_{\alpha + \beta + \bar{\bar{\beta}}}(N_{\alpha, \bar{\bar{\beta}}} N_{\beta, \alpha + \bar{\bar{\beta}}}  t u \bar{\bar{u}}) x_{\alpha + \bar{\beta} + \bar{\bar{\beta}}}(N_{\alpha, \bar{\bar{\beta}}} N_{\bar{\beta}, \alpha + \bar{\bar{\beta}}}  t \bar{u} \bar{\bar{u}}) \} \\
    & \hspace{15mm} \{ x_{\alpha + \beta + \bar{\beta} + \bar{\bar{\beta}}}(N_{\alpha, \bar{\bar{\beta}}} N_{\beta,\alpha + \bar{\bar{\beta}}} N_{\beta, \alpha + \bar{\beta} + \bar{\bar{\beta}}} t u \bar{u} \bar{\bar{u}}) \} \\
    & \hspace{20mm} \{ x_{2\alpha + \beta + \bar{\beta} + \bar{\bar{\beta}}}(N_{\alpha, \bar{\beta}} N_{\alpha, \bar{\bar{\beta}}} N_{\beta,\alpha + \bar{\beta}} N_{\alpha + \beta + \bar{\beta}, \alpha + \bar{\bar{\beta}}} t^2 u \bar{u} \bar{\bar{u}}) \}.
\end{align*} 
From the choice of Chevalley bases (see Lemma \ref{epsilonalpha}), we have $N_{\alpha, \beta} = N_{\alpha, \bar{\beta}} = N_{\alpha, \bar{\bar{\beta}}}$. 
For $X,Y,Z \in \mathcal{L}$, we have Jacobi identity 
$$[X,[Y,Z]] + [Y,[Z,X]] + [Z,[X,Y]] = 0.$$
By taking $X = X_\alpha, Y= X_\beta$ and $Z= X_{\bar{\beta}}$, we get $N_{\beta, \alpha + \bar{\beta}} = N_{\bar{\beta}, \alpha + \beta}$. But then 
\begin{align*}
    N_{\bar{\beta}, \alpha + \bar{\bar{\beta}}} = N_{\beta, \alpha + \bar{\beta}} = N_{\bar{\beta}, \alpha + \beta} = N_{\bar{\bar{\beta}}, \alpha + \bar{\beta}} = N_{\beta, \alpha + \bar{\bar{\beta}}}.
\end{align*}
Now our assertion follows readily. \qed

\vspace{2mm}

\noindent \textit{Proof of $(f):$} Using commutator relations in $G_{\pi}(\Phi,R)$ and the fact that $\alpha + \bar{\beta} = \bar{\alpha} + \beta$, $\alpha + \bar{\bar{\beta}} = \bar{\bar{\alpha}} + \beta$, $\bar{\alpha} + \bar{\bar{\beta}} = \bar{\bar{\alpha}} + \bar{\beta}$, we can show that
\begin{align*}
    [x_{[\alpha]}(t), x_{[\beta]}(u)] &= \{ x_{\alpha + \bar{\beta}} (N_{\alpha, \bar{\beta}} t \bar{u} + N_{\bar{\alpha}, \beta} \bar{t} u) x_{\bar{\alpha} + \bar{\bar{\beta}}} (N_{\bar{\alpha}, \bar{\bar{\beta}}} \bar{t} \bar{\bar{u}} + N_{\bar{\bar{\alpha}}, \bar{\beta}} \bar{\bar{t}} \bar{u}) x_{\bar{\bar{\alpha}} + \beta} (N_{\bar{\bar{\alpha}}, \beta} \bar{\bar{t}} u + N_{\alpha, \bar{\bar{\beta}}} t \bar{\bar{u}}) \} \\
    & \hspace{10mm} \{ x_{\alpha + \bar{\alpha} + \bar{\bar{\beta}}} (N_{\alpha, \bar{\alpha} + \bar{\bar{\beta}}} N_{\bar{\alpha}, \bar{\bar{\beta}}} \ t \bar{t} \bar{\bar{u}} + N_{\bar{\alpha}, \bar{\bar{\alpha}} + \beta} N_{\bar{\bar{\alpha}}, \beta} \ \bar{t} \bar{\bar{t}} u + N_{\alpha, \bar{\bar{\alpha}} + \bar{\beta}} N_{\bar{\bar{\alpha}}, \bar{\beta}} \ t \bar{\bar{t}} \bar{u}) \} \\
    & \hspace{15mm} \{ x_{\alpha + \bar{\beta} + \bar{\bar{\beta}}} (N_{\bar{\beta}, \alpha + \bar{\bar{\beta}}} N_{\alpha, \bar{\bar{\beta}}} \ t \bar{u} \bar{\bar{u}} + N_{\beta, \bar{\alpha} + \bar{\bar{\beta}}} N_{\bar{\alpha}, \bar{\bar{\beta}}} \ \bar{t} u \bar{\bar{u}} + N_{\beta, \bar{\bar{\alpha}} + \bar{\beta}} N_{\bar{\bar{\alpha}}, \bar{\beta}} \ \bar{\bar{t}} u \bar{u}) \}. 
\end{align*}
From the choice of Chevalley basis (see Lemma \ref{epsilonalpha}), we have $N_{\alpha, \bar{\beta}} = N_{\bar{\alpha}, \bar{\bar{\beta}}} = N_{\bar{\bar{\alpha}}, \beta}$ and $N_{\bar{\alpha}, \beta} = N_{\bar{\bar{\alpha}}, \bar{\beta}} = N_{\alpha, \bar{\bar{\beta}}}$. 
For $X,Y,Z \in \mathcal{L}$, we have Jacobi identity 
$$[X,[Y,Z]] + [Y,[Z,X]] + [Z,[X,Y]] = 0.$$
By taking $X = X_{\bar{\alpha}}, Y= X_{\bar{\bar{\alpha}}}$ and $Z= X_{\beta}$, we get $N_{\bar{\alpha}, \bar{\bar{\alpha}} + \beta} N_{\bar{\bar{\alpha}}, \beta} = N_{\bar{\bar{\alpha}}, \bar{\alpha} + \beta} N_{\bar{\alpha}, \beta}$. But then 
\begin{align*}
    N_{\alpha, \bar{\alpha} + \bar{\bar{\beta}}} N_{\bar{\alpha}, \bar{\bar{\beta}}} = N_{\bar{\alpha}, \bar{\bar{\alpha}} + \beta} N_{\bar{\bar{\alpha}}, \beta} = N_{\bar{\bar{\alpha}}, \bar{\alpha} + \beta} N_{\bar{\alpha}, \beta} =
    N_{\alpha, \bar{\bar{\alpha}} + \bar{\beta}} N_{\bar{\bar{\alpha}}, \bar{\beta}}.
\end{align*}
Similarly, by taking $X = X_{\bar{\beta}}, Y= X_{\bar{\bar{\beta}}}, Z= X_{\alpha}$, we get $N_{\bar{\beta}, \alpha + \bar{\bar{\beta}}} N_{\alpha, \bar{\bar{\beta}}} = N_{\bar{\bar{\beta}}, \alpha + \bar{\beta}} N_{\alpha, \bar{\beta}}$. But then 
\begin{align*}
    N_{\beta, \bar{\bar{\alpha}} + \bar{\beta}} N_{\bar{\bar{\alpha}}, \bar{\beta}} = N_{\bar{\beta}, \alpha + \bar{\bar{\beta}}} N_{\alpha, \bar{\bar{\beta}}} = N_{\bar{\bar{\beta}}, \alpha + \bar{\beta}} N_{\alpha, \bar{\beta}} =
    N_{\beta, \bar{\alpha} + \bar{\bar{\beta}}} N_{\bar{\alpha}, \bar{\bar{\beta}}}.
\end{align*}
Now our assertion follows readily. \qed

\vspace{2mm}

\noindent \textit{Proof of $(g):$} Observe that, either $\alpha + \bar{\beta} \in \Phi$ and $\alpha + \bar{\bar{\beta}} \not\in \Phi$ or vice versa. We consider a case where $\alpha + \bar{\beta} \in \Phi$ (the proof of other case is similar and hence omitted). Now using commutator relations in $G_{\pi}(\Phi,R)$ and the fact that $\alpha + \bar{\beta} = \bar{\alpha} + \bar{\bar{\beta}} = \bar{\bar{\alpha}} + \beta$, we can show that 
$$[x_{[\alpha]}(t), x_{[\beta]}(u)] = x_{\alpha+\bar{\beta}}(N_{\alpha,\bar{\beta}} t \bar{u} + N_{\bar{\alpha}, \bar{\bar{\beta}}} \bar{t} \bar{\bar{u}} + N_{\bar{\bar{\alpha}}, \beta} \bar{\bar{t}} u).$$
From the choice of Chevalley basis (see Lemma \ref{epsilonalpha}), we have $N_{\alpha, \bar{\beta}} = N_{\bar{\alpha}, \bar{\bar{\beta}}} = N_{\bar{\bar{\alpha}}, \beta}$ and hence the result follows. \qed

%%%%%%%%%%%%%%%%%%%%%%%%%%%%%%%%%%%%%%%%%%%%%%%%%%%%%%%%%%%%%%
%%%%%%%%%%%%%%%%%%%%%%%%%%%%%%%%%%%%%%%%%%%%%%%%%%%%%%%%%%%%%%

\subsection{The subgroups \texorpdfstring{$N'_\sigma$}{N'} and \texorpdfstring{$H'_\sigma$}{H'}}\label{N and H} 

We now turn our attention to studying the subgroups $N'_{\sigma}$ and $H'_{\sigma}$ of the group $E'_\sigma (R)$. Understanding the structure of $N'_\sigma$ and $H'_\sigma$ is relatively straightforward when $\Phi_\rho \sim {}^2 A_{2n-1} \ (n \geq 2), {}^2 D_n \ (n \geq 4), {}^2 E_6$ or ${}^3 D_4$. However, it becomes more complex when $\Phi_\rho \sim {}^2 A_{2n} \ (n \geq 1)$.

\begin{conv}
    \normalfont
    At this point, we want to establish a convention regarding some notation. The classes $[-\alpha]$ and $-[\alpha]$ denote the same set, but they may differ as ordered sets. If $[\alpha] \sim A_1$, then both notations are identical. If $[\alpha] \sim A_1^2$ or $A_1^3$, then $\alpha'$ represents $[-\alpha]$ where $\alpha' = \min \{ -\alpha, -\bar{\alpha}, -\bar{\bar{\alpha}} \}$. In these cases, as an ordered set, $[-\alpha] = [\alpha'] = \{ \alpha', \bar{\alpha'} \}$ or $\{ \alpha', \bar{\alpha'}, \bar{\bar{\alpha'}} \}$ depending on whether $[\alpha] \sim A_1^2$ or $A_1^3$, respectively. Whence, for $-[\alpha]$, as an ordered set, $-[\alpha] = \{ -\alpha, -\bar{\alpha} \}$ or $\{ -\alpha, -\bar{\alpha}, -\bar{\bar{\alpha}} \}$ depending on whether $[\alpha] \sim A_1^2$ or $A_1^3$, respectively. Finally, if $[\alpha] \sim A_2$, both the notations represent the same ordered set: $\{ - \bar{\alpha}, -\alpha, -\alpha - \bar{\alpha} \}$ if $\alpha < \bar{\alpha}$.
\end{conv}

Write $R^*= \{r \in R \mid \exists s \in R \text{ such that } rs = 1\}$, $R_\theta^* = R_\theta \cap R^*$ and $\mathcal{A}(R)^* := \{ (t,u) \in \mathcal{A}(R) \mid u \in R^* \}.$ For given $[\alpha] \in \Phi_\rho$, we also write 
\[
    R_{[\alpha]}^* = \begin{cases}
        R^*_\theta & \text{if } [\alpha] \sim A_1, \\
        R^* & \text{if } [\alpha] \sim A_1^2 \text{ or } A_1^3, \\
        \mathcal{A}(R)^* & \text{if } [\alpha] \sim A_2;
    \end{cases} \quad \text{and} \quad     R_{[\alpha]}^{\star} = \begin{cases}
        R^*_\theta & \text{if } [\alpha] \sim A_1, \\
        R^* & \text{if } [\alpha] \sim A_1^2, A_1^3 \text{ or } A_2.
    \end{cases}
\]

With these notations established, we proceed to define the following special elements of $N_\sigma$ and $H_\sigma$:
\begin{enumerate}[leftmargin=3.5em]
    \item[$\textbf{(W1)}$] If $[\alpha] \sim A_1$, then define $w_{[\alpha]}(t) := x_{[\alpha]}(t) x_{-[\alpha]}(-t^{-1}) x_{[\alpha]}(t) = w_\alpha (t), t \in R_\theta^*.$
    \item[$\textbf{(W2)}$] If $[\alpha] \sim A_1^2$, then define $w_{[\alpha]}(t) := x_{[\alpha]}(t) x_{-[\alpha]}(-t^{-1}) x_{[\alpha]}(t) = w_{\alpha} (t) w_{\bar{\alpha}} (\bar{t}), t \in R^*.$
    \item[$\textbf{(W3)}$] If $[\alpha] \sim A_1^3$, then define $w_{[\alpha]}(t) := x_{[\alpha]}(t) x_{-[\alpha]}(-t^{-1}) x_{[\alpha]}(t) = w_{\alpha} (t) w_{\bar{\alpha}} (\bar{t}) w_{\bar{\bar{\alpha}}} (\bar{\bar{t}}),$ $t \in R^*.$
    \item[$\textbf{(W4)}$] If $[\alpha] \sim A_2$, then define $w_{[\alpha]}(t,u) := x_{[\alpha]}(t,u) x_{-[\alpha]}( -\bar{u}^{-1} \cdot (t,u)) x_{[\alpha]}( u \bar{u}^{-1} \cdot (t,u)) = x_{[\alpha]}(t,u) x_{-[\alpha]}(- (\bar{u}^{-1}) t, (\bar{u}^{-1})) x_{[\alpha]}(u \bar{u}^{-1} t, u),$ where $(t,u) \in \mathcal{A}(R)^*$. 
    \item[\textbf{(W4$'$)}] If $[\alpha] \sim A_2$ such that $\alpha \neq \bar{\alpha}$, then define $w_{[\alpha]} (t) := w_\alpha (\bar{t}) w_{\bar{\alpha}}(1) w_{\alpha} (t), t \in R^*$. 
    
    \vspace{3mm}
    
    \item[$\textbf{(H1)}$] If $[\alpha] \sim A_1$, then define $h_{[\alpha]}(t) := w_{[\alpha]}(t) w_{[\alpha]}(-1) = h_\alpha (t), \ t \in R^*_{\theta}$.
    \item[$\textbf{(H2)}$] If $[\alpha] \sim A_1^2$, then define $h_{[\alpha]}(t) = w_{[\alpha]}(t) w_{[\alpha]}(-1) = h_\alpha (t) h_{\bar{\alpha}}(\bar{t}), \ t \in R^*$.
    \item [$\textbf{(H3)}$] If $[\alpha] \sim A_1^3$, then define $h_{[\alpha]}(t) = w_{[\alpha]}(t) w_{[\alpha]}(-1) = h_\alpha (t) h_{\bar{\alpha}}(\bar{t}) h_{\bar{\bar{\alpha}}}(\bar{\bar{t}}), \ t \in R^*$.
    \item[$\textbf{(H4)}$] If $[\alpha] \sim A_2$, then define $h_{[\alpha]}((t,u),(t',u')) = w_{[\alpha]}(t,u) w_{[\alpha]}(t',u')$, where \newline $(t,u), (t',u') \in \mathcal{A}(R)^*$. 
    \item[\textbf{(H4$'$)}] If $[\alpha] \sim A_2$ such that $\alpha \neq \bar{\alpha}$, then define $h_{[\alpha]} (t) := h_\alpha (t) h_{\bar{\alpha}}(\bar{t}), t \in R^*$.
\end{enumerate}

\begin{rmk}
    \normalfont
    \begin{enumerate}[(a)]
        \item One can easily verify that the last equality holds in $(W1), (W2), (W3),$ $(H1),$ $(H2)$ and $(H3)$.
        \item Recall that, $\sigma (h_\alpha (t)) = h_{\bar{\alpha}} (\bar{t})$ and $\sigma (w_\alpha (t)) = w_{\bar{\alpha}} (\epsilon_\alpha \bar{t})$, where $\epsilon_\alpha = \pm 1$ (note that $\epsilon_\alpha = -1$ if and only if $[\alpha] \sim A_2$ and $\alpha \neq \bar{\alpha}$ (see Lemma \ref{epsilonalpha})). Hence it is clear that $w_{[\alpha]}(t) \in N'_\sigma \subset N_\sigma$ and $h_{[\alpha]}(t) \in H'_\sigma \subset H_\sigma$, if $[\alpha] \sim A_1, A_1^2, A_1^3.$ Similarly, $w_{[\alpha]}(t,u) \in N'_\sigma \subset N_\sigma$ and $h_{[\alpha]}((t,u),(t',u')) \in H'_\sigma \subset H_\sigma$, if $[\alpha] \sim A_2$  (see Lemma $\ref{lemmaW4'}$ and Lemma $\ref{lemmaRel-of-WandH}$).
        \item If $[\alpha] \sim A_2$ then $w_{[\alpha]}(t) \in N_\sigma$ and $h_{[\alpha]}(t) \in H_\sigma$ (see Lemma $\ref{lemmaW4'}$), but it is not necessary that every $w_{[\alpha]}(t)$ (resp., $h_{[\alpha]} (t)$), $t \in R^*$ contained in $N'_\sigma$ (resp., $H'_\sigma$).
        \item Each $h_{[\alpha]}(t)$ defined in $(H1), (H2), (H3)$ and $(H4')$ is multiplicative as a function of $t$. In particular, $h_{[\alpha]}(t)^{-1} = h_{[\alpha]} (t^{-1})$.
        \item If $[\alpha] \sim A_1, A_1^2$ or $A_1^3$, then $x_{[\alpha]} (t)^{-1} = x_{[\alpha]}(-t)$ and $w_{[\alpha]} (t)^{-1} = w_{[\alpha]}(-t)$.
    \end{enumerate}
\end{rmk}

\begin{lemma}[{\cite[Proposition 2.1]{EA1}}]\label{w^-1}
    If $[\alpha] \sim A_2$ and $(t,u) \in \mathcal{A}(R)^*$ then 
    \[
        w_{[\alpha]}(t,u)^{-1} = w_{[\alpha]} (-t u \bar{u}^{-1}, \bar{u}).
    \]
\end{lemma}

\begin{proof}
     Note that $w_{[\alpha]}(t,u) = x_{[\alpha]}(t,u) x_{-[\alpha]}(- \bar{u}^{-1} t, \bar{u}^{-1}) x_{[\alpha]}(u \bar{u}^{-1} t, u)$. Hence 
    \begin{align*}
        (w_{[\alpha]}(t,u))^{-1} &= [{x_{[\alpha]}(t,u) x_{-[\alpha]}(- (\bar{u}^{-1}) t, (\bar{u}^{-1})) x_{[\alpha]}(u (\bar{u}^{-1}) t, u)}]^{-1} \\
        &= [x_{[\alpha]}(u (\bar{u}^{-1}) t, u)]^{-1} [x_{-[\alpha]}(- (\bar{u}^{-1}) t, (\bar{u}^{-1}))]^{-1} [x_{[\alpha]}(t,u)]^{-1} \\
        &= x_{[\alpha]}(- u (\bar{u}^{-1}) t, \bar{u}) x_{-[\alpha]}((\bar{u}^{-1}) t, (u^{-1})) x_{[\alpha]}(-t, \bar{u}) \\
        &= w_{[\alpha]} (- u (\bar{u}^{-1}) t, \bar{u}).
    \end{align*}
    This proves our lemma.
\end{proof}

% \begin{lemma}[{\cite[Proposition 2.3]{EA1}}]
%     \normalfont
%     If $[\alpha] \sim A_2$ and $(t,u) \in \mathcal{A}(R)^*$ then 
%     \[
%         w_{[\alpha]}(t,u) = w_{-[\alpha]} (-t u^{-2} \bar{u}, \bar{u}^{-1}).
%     \]
% \end{lemma}

\begin{lemma}\label{lemmaW4'}
    \normalfont
    Let $[\alpha] \sim A_2$ such that $\alpha \neq \bar{\alpha}$. Then
    \begin{enumerate}[(a)]
        \item $w_{[\alpha]}(t) \in N_\sigma$ and $w_{[\alpha]} (t) ^{-1} = w_{[\alpha]} (\bar{t})$, for every $t \in R^*$.
        \item $h_{[\alpha]}(t) \in H_\sigma, h_{[\alpha]}(t) = w_{[\alpha]} (\bar{t}) w_{[\alpha]} (1)$ and $h_{[\alpha]}(t)^{-1} = h_{[\alpha]}(t^{-1})$, for every $t \in R^*$.
    \end{enumerate} 
\end{lemma}

\begin{proof}
    Define $E_3 (R)$ to be the subgroup of $SL_3(R)$ generated by 
    \begin{gather*}
        x'_{\alpha}(t) := 1 + t E_{23}, \quad x'_{\bar{\alpha}} (t):= 1 + t E_{12}, \quad x'_{\alpha + \bar{\alpha}} (t):= 1 + t E_{13}, \\
        x'_{-\alpha}(t):= 1 + t E_{32}, \quad x'_{-\bar{\alpha}} (t):= 1 + t E_{21}, \quad x'_{-\alpha - \bar{\alpha}} (t):= 1 + t E_{31}.
    \end{gather*}
    
    Consider a subgroup $K = \langle \mathfrak{X}_{\alpha}, \mathfrak{X}_{-\alpha}, \mathfrak{X}_{\bar{\alpha}}, \mathfrak{X}_{-\bar{\alpha}} \rangle$ of $E_\pi(\Phi, R)$. Then there is a natural surjective homomorphism of groups $$ \psi: E_3 (R) \longrightarrow K$$ given by 
    \[
        x'_{\beta}(t) \longmapsto x_{\beta}(t)
    \]
    for all $\beta \in \{ \pm \alpha, \pm \bar{\alpha}, \pm (\alpha + \bar{\alpha}) \}$ and $t \in R$.
    
    Note that $\sigma = \sigma|_K$ is an automorphism of the subgroup $K$ and there exists a natural automorphism $\sigma'$ of $E_3 (R)$ such that $\sigma \circ \psi = \psi \circ \sigma'$. We have $\psi(E_{3, \sigma'}(R)) = K_{\sigma}(R)$ and $\psi(E'_{3, \sigma'}(R)) = K'_{\sigma} (R)$. Hence it is enough to prove the corresponding results in the group $S(R)$.  
    
    The notation of $w'_{[\alpha]}(t)$ and $h'_{[\alpha]}(t)$ is clear. Note that $w'_{[\alpha]}(t) \in N_{S, \sigma'}(R)$ and $h'_{[\alpha]}(t) \in H_{S, \sigma'}(R),$ hence $w_{[\alpha]}(t) \in N_\sigma$ and $h_{[\alpha]}(t) \in H_\sigma$. Now the proof is immediate from below: $$w'_{[\alpha]} (t) w'_{[\alpha]} (\bar{t}) = \begin{pmatrix}
        0 & 0 & t \\
        0 & -t^{-1}\bar{t} & 0 \\
        \bar{t}^{-1} & 0 & 0
    \end{pmatrix} \begin{pmatrix}
        0 & 0 & \bar{t} \\
        0 & -\bar{t}^{-1}t & 0 \\
        t^{-1} & 0 & 0
    \end{pmatrix} = \begin{pmatrix}
        1 & 0 & 0 \\
        0 & 1 & 0 \\
        0 & 0 & 1
    \end{pmatrix},$$ 
    $$ w'_{[\alpha]} (\bar{t}) w'_{[\alpha]} (1) = \begin{pmatrix}
        0 & 0 & \bar{t} \\
        0 & -\bar{t}^{-1}t & 0 \\
        t^{-1} & 0 & 0
    \end{pmatrix} \begin{pmatrix}
        0 & 0 & 1 \\
        0 & -1 & 0 \\
        1 & 0 & 0
    \end{pmatrix} = \begin{pmatrix}
        \bar{t} & 0 & 0 \\
        0 & t \bar{t}^{-1} & 0 \\
        0 & 0 & t^{-1}
    \end{pmatrix} = h'_{[\alpha]}(t).$$
\end{proof}

\begin{lemma}\label{lemmaRel-of-WandH}
    \normalfont Let $[\alpha] \sim A_2$ such that $\alpha \neq \bar{\alpha}$.
    \begin{enumerate}[(a)]
        \item If $(t, u) \in \mathcal{A}(R)^*$, then $w_{[\alpha]}(t,u) = w_{[\alpha]}(u)$.
        \item If $(t, u), (t', u') \in \mathcal{A}(R)^*$, then $h_{[\alpha]}((t,u), (t',u')) = h_{[\alpha]}(\bar{u} {u'}^{-1}).$
    \end{enumerate}
\end{lemma}

\begin{proof}
    As we pointed out earlier, it is enough to prove the corresponding results in $E_3(R)$. Note that 
    \begin{align*}
        w'_{[\alpha]}(u) &= w'_{\alpha}(\bar{u}) w'_{\bar{\alpha}}(1) w'_{\alpha} (u) \\
        &= \begin{pmatrix}
        1 & 0 & 0 \\
        0 & 0 & \bar{u} \\
        0 & -\bar{u}^{-1} & 0
    \end{pmatrix} \begin{pmatrix}
        0 & 1 & 0 \\
        -1 & 0 & 0 \\
        0 & 0 & 1
    \end{pmatrix} \begin{pmatrix}
        1 & 0 & 0 \\
        0 & 0 & u \\
        0 & -{u}^{-1} & 0
    \end{pmatrix} \\
    &= \begin{pmatrix}
        0 & 0 & u \\
        0 & -u^{-1}\bar{u} & 0 \\
        \bar{u}^{-1} & 0 & 0
    \end{pmatrix} \\
    &= w'_{[\alpha]}(t,u),
    \end{align*} which proves $(a)$. 
    Again note that 
    \begin{align*}
        h'_{[\alpha]}(\bar{u} (u')^{-1}) &= h'_{\alpha}(\bar{u} (u')^{-1}) h'_{\bar{\alpha}}(u (\bar{u'})^{-1}) \\ 
        &= \begin{pmatrix}
        1 & 0 & 0 \\
        0 & \bar{u} (u')^{-1} & 0 \\
        0 & 0 & u' \bar{u}^{-1}
        \end{pmatrix} \begin{pmatrix}
        u (\bar{u'})^{-1} & 0 & 0 \\
        0 & u^{-1} (\bar{u'}) & 0 \\
        0 & 0 & 1
        \end{pmatrix} \\
        &= \begin{pmatrix}
        u(\bar{u'})^{-1} & 0 & 0 \\
        0 & u^{-1}\bar{u} (u')^{-1}\bar{u'}  & 0 \\
        0 & 0 & u' (\bar{u})^{-1}
        \end{pmatrix} \\
        &= h'_{[\alpha]}((t,u),(t',u')),
    \end{align*} which proves $(b)$.
\end{proof}

\begin{defn}
    \normalfont
    Let $[\alpha] \sim A_2$ such that $\alpha \neq \bar{\alpha}$. We define 
    \begin{enumerate}[(a)]
        \item $\mathcal{R}_1 = \{ u \in R^* \mid \exists \ t \in R \text{ such that } (t,u) \in \mathcal{A}(R)^* \}$;
        \item $\mathcal{R}_k = \{ u \in R^* \mid \exists \ u_1, \dots, u_k \in \mathcal{R}_1 \text{ such that } u = u_1 \dots u_k \}$, for given $k \in \mathbb{N}$;
        \item $\mathcal{R}^{(l)} = \{ u \in R^* \mid \exists \ k \in \mathbb{N} \text{ such that } u \in \mathcal{R}_{kl} \} = \cup_{k \in \mathbb{N}} \mathcal{R}_{kl}$, for given $l \in \mathbb{N}$;
        \item $\mathcal{R} := \mathcal{R}^{(1)}, \mathcal{R}':= \mathcal{R}^{(2)}$ and $\mathcal{R}'':=  \cup_{k \in \mathbb{N}} \mathcal{R}_{2k-1}$.
    \end{enumerate}
\end{defn}

\begin{rmk}
    \normalfont
    The following are immediate consequences of the definition, provided $\mathcal{R}_1 \neq \phi$:
    \begin{enumerate}[(a)]
        \item If $u$ is in $\mathcal{R}_k$, then so are $\bar{u}, u^{-1}$ and $a \bar{a} u \ (a \in R^*)$.
        \item For any $l \in \mathbb{N}$, $\mathcal{R}^{(l)}$ is a subgroup of the multiplicative group $R^*$ generated by $\mathcal{R}_l$. In particular, $\mathcal{R}$ and $\mathcal{R}'$ are subgroups of $R^*$.
        \item If $1 \in \mathcal{R}_k$, then $\mathcal{R}_m \subset \mathcal{R}_{k+m},$ for all $m \in \mathbb{N}$. In particular, since $1 \in \mathcal{R}_2$, we have $\mathcal{R}_1 \subset \mathcal{R}_{3} \subset \mathcal{R}_5 \subset \cdots$ and $\mathcal{R}_2 \subset \mathcal{R}_{4} \subset \mathcal{R}_6 \subset \cdots$.
        \item If $R$ is a field then $R^* = \mathcal{R} = \mathcal{R}' = \mathcal{R}_{2k}$, for all $k \in \mathbb{N}$. To see this, it is enough to see that $\mathcal{R}_2 = R^*$ (by part $(c)$). Let $u \in R^*$. If $u = \bar{u}$ then we choose $u_1 \in R^*$ such that $u_1 = - \bar{u_1}$ (such a $u_1$ exists in field) and $u_2 = u (u_1)^{-1}$. Similarly, if $u \neq \bar{u}$ then we can choose $u_1 = (\bar{u} - u)^{-1}$ and $u_2= u (u_1)^{-1}$. In both the cases, we have $u = u_1 u_2$ such that $u_1, u_2 \in \mathcal{R}_1$ as $(0, u_1), (u - \bar{u}, u_2) \in \mathcal{A}(R)^*$. Therefore $u \in \mathcal{R}_2$, as desired.  
    \end{enumerate}
\end{rmk}

\begin{cor}\label{cor:H' subset H}
    \normalfont
    Let $[\alpha] \sim A_2$. If $u \in \mathcal{R}'$ then $h_{[\alpha]}(u) \in H'_\sigma$. Similarly, if $u \in \mathcal{R}''$ then $w_{[\alpha]}(u) \in N'_\sigma$.
\end{cor}

\begin{proof}
    Immediate from Lemma \ref{lemmaRel-of-WandH} and the fact that 
    \[
        w_{[\alpha]}(t_1,u_1) w_{[\alpha]}(t_2,u_2)^{-1} w_{[\alpha]}(t_3, u_3) = w_{[\alpha]}(u_1 u_2^{-1} u_3) \in N'_\sigma.
    \]
    The later part follows from the direct calculation in $E_3(R)$.
\end{proof}

% \begin{rmk}
%     Let $G_\sigma (R)$ be a simply connected twisted Chevalley group. If $\Phi_\rho \sim {}^2 A_{2n-1} \ (n \geq 2), {}^2 D_n \ (n \geq 4), {}^2 E_6$ or ${}^3 D_4$, then we know that $H'_\sigma = H_\sigma$. It is not known whether the similar statement holds in general for the case of $\Phi_\rho \sim {}^2 A_{2n} \ (n \geq 2)$. However, from Corollary \ref{cor:H' subset H}, we can conclude that in this latter case, $H'_\sigma = H_\sigma$ if $\mathcal{R}' = R^*$. \textcolor{red}{\textbf{(Check!)}}
% \end{rmk}

%%%%%%%%%%%%%%%%%%%%%%%%%%%%%%%%%%%%%%%%%%%%%%%%%%%%%%%%%%%%%%
%%%%%%%%%%%%%%%%%%%%%%%%%%%%%%%%%%%%%%%%%%%%%%%%%%%%%%%%%%%%%%

\subsection{The Steinberg Relations} \index{Steinberg relations}

In this subsection, we present some useful relations in the group $E'_\sigma (R)$. 
Recall that $w_\alpha (t) X_\beta w_{\alpha}(t)^{-1} = c t^{-\langle \beta, \alpha \rangle} X_{s_\alpha(\beta)}$, where $c = c(\alpha, \beta) = \pm 1$ is independent of $t, R$ and the representation chosen, and $c(\alpha, \beta)= c(\alpha, -\beta)$ (see Lemma $19(a)$ of \cite{RS}). 
If $\Phi$ is a root system with one root length, then the function $c: \Phi \times \Phi \longrightarrow \{\pm 1 \}$ can be precisely given as follows: 
\[
    c(\alpha, \beta) = c(\alpha, -\beta) = \begin{cases}
        -1 & \text{ if } \alpha = \beta, \\
        1 & \text{ if } \alpha \neq \beta \text{ and } \alpha \pm \beta \not\in \Phi, \\
        N_{\alpha, \beta} & \text{ if } \alpha \neq \beta \text{ and } \alpha + \beta \in \Phi, \alpha - \beta \not\in \Phi.
    \end{cases}
\]

Our objective is to establish relations for twisted Chevalley groups analogous to those in Lemma 20 of \cite{RS} for Chevalley groups. But before we do that, let us consider the function $d: \Phi_\rho \times \Phi \longrightarrow \{ \pm 1 \}$ given by
\[
    d([\alpha], \beta) = \begin{cases}
        c(\alpha, \beta) & \text{ if } [\alpha] \sim A_1, \\
        c(\alpha, \beta) c(\bar{\alpha}, s_\alpha (\beta)) & \text{ if } [\alpha] \sim A_1^2, \\
        c(\alpha, \beta) c(\bar{\alpha}, s_\alpha (\beta)) c(\bar{\bar{\alpha}}, s_{\bar{\alpha}}s_\alpha (\beta)) & \text{ if } [\alpha] \sim A_1^3, \\
        c(\alpha, \beta) c(\bar{\alpha}, s_\alpha (\beta)) c(\alpha, s_{\bar{\alpha}}s_\alpha (\beta)) & \text{ if } [\alpha] \sim A_2.
    \end{cases} 
\]

\begin{lemma}\label{lemma on d}
    \normalfont
    For every $\alpha, \beta \in \Phi$, $d([\alpha], \beta) = d([\alpha], \bar{\beta})$. Moreover if $\beta \in \Phi$ is such that $\beta \neq \bar{\beta}$ and $[\beta] \sim A_2$, then $d([\alpha], \beta + \bar{\beta}) = N_{s_{\alpha}(\bar{\beta}), s_{\alpha}(\beta)} N_{\bar{\beta}, \beta}$. 
\end{lemma}

\begin{proof}
    First, we observe that \( c(\alpha, \beta) = c(\bar{\alpha}, \bar{\beta}) \) since \( N_{\alpha, \beta} = N_{\bar{\alpha}, \bar{\beta}} \).
    Assume \( [\alpha] \sim A_1 \). Then \( c(\alpha, \beta) = c(\alpha, \bar{\beta}) \) and hence \( d([\alpha], \beta) = d([\alpha], \bar{\beta}) \).
    Now, consider \( [\alpha] \sim A_1^2 \). In this case, depending on the possible values of \( \langle \beta, \alpha \rangle \) and \( \langle \bar{\beta}, \alpha \rangle \), we can address several subcases to establish our result. For instance, if \( \alpha \) and \( \beta \) are such that \( \langle \beta, \alpha \rangle = -1 \) and \( \langle \bar{\beta}, \alpha \rangle = 0 \), then 
    \( d([\alpha], \beta) = c(\alpha, \beta) c(\bar{\alpha}, s_\alpha (\beta)) = N_{\alpha, \beta}. \)
    On the other hand, under the same assumption on \( \alpha \) and \( \beta \), we have 
    \( d([\alpha], \bar{\beta}) = c(\alpha, \bar{\beta}) c(\bar{\alpha}, s_\alpha (\bar{\beta})) = c(\alpha, \bar{\beta}) c(\bar{\alpha}, \bar{\beta}) = N_{\bar{\alpha}, \bar{\beta}}. \)
    Therefore, \( d([\alpha], \beta) = d([\alpha], \bar{\beta}) \). One can similarly verify all the other subcases.
    Finally, the cases where \( [\alpha] \sim A_1^3 \) and \([\alpha] \sim A_2 \) can be proved in a similar manner and are thus omitted.

    The second assertion is only valid in the case where \( \Phi_\rho \sim {}^2 A_{2n} \ (n \geq 2) \). First, assume that \( [\alpha] \sim A_2 \). Unless \( [\alpha] = \pm [\beta] \), we have \( \langle \beta, \alpha \rangle = 0 \) and hence \( s_\alpha (\beta) = \beta \) and \( s_\alpha (\bar{\beta}) = \bar{\beta} \). Therefore, 
    \[
        d([\alpha], \beta + \bar{\beta}) = 1 = N_{s_{\alpha}(\bar{\beta}), s_{\alpha}(\beta)} N_{\bar{\beta}, \beta}.
    \]
    Now, if \( [\alpha] = \pm [\beta] \), then since \( d([\alpha], \beta + \bar{\beta}) = d([\alpha], - \beta - \bar{\beta}) \), we can assume that \( \alpha = \beta \) or \( \bar{\beta} \). Without loss of generality, assume \( \alpha = \beta \), then 
    \begin{align*}
        d([\alpha], \beta + \bar{\beta}) &= c(\alpha, \alpha + \bar{\alpha}) c(\bar{\alpha}, s_\alpha (\alpha + \bar{\alpha})) c(\alpha, s_{\bar{\alpha}} s_{\alpha} (\alpha + \bar{\alpha})) \\
        &= c(\alpha, \alpha + \bar{\alpha}) c(\bar{\alpha}, \bar{\alpha}) c(\alpha, -\bar{\alpha}) \\
        &= - c(\alpha, - \alpha - \bar{\alpha}) c(\alpha, \bar{\alpha}) \\
        &= - N_{\alpha, - \alpha - \bar{\alpha}} N_{\alpha, \bar{\alpha}} \\
        &= N_{s_{\alpha}(\bar{\beta}), s_\alpha (\beta)} N_{\bar{\beta}, \beta}.
    \end{align*}
    Now, assume that \( [\alpha] \sim A_1^2 \). In this case, the possible values of \( \langle \beta + \bar{\beta}, \alpha \rangle \) are \(-1, 0\), or \(1\). If \( \langle \beta + \bar{\beta}, \alpha \rangle = 0 \), then \( \langle \beta, \alpha \rangle = 0 = \langle \bar{\beta}, \alpha \rangle \) (since \( [\alpha] \sim A_1^2 \)). Therefore,
    \(
        N_{s_{\alpha}(\bar{\beta}), s_\alpha (\beta)} N_{\bar{\beta}, \beta} = (N_{\bar{\beta}, \beta})^2 = 1.
    \)
    On the other hand,
    \(
        d([\alpha], \beta + \bar{\beta}) = c(\alpha, \beta) c(\bar{\alpha}, s_\alpha (\beta)) = 1.
    \)
    Therefore,
    \(
        d([\alpha], \beta + \bar{\beta}) = N_{s_{\alpha}(\bar{\beta}), s_\alpha (\beta)} N_{\bar{\beta}, \beta}.
    \)
    The cases where \( \langle \beta + \bar{\beta}, \alpha \rangle = -1 \) or \( 1 \) can be proved similarly and are thus omitted.
\end{proof}

A version of the following Proposition is provided in \cite{EA1}. Here, we present more general formulas that cover all cases, unlike those in \cite[4.1 and 4.3]{EA1}.
\begin{prop}\label{prop wxw^{-1}}
    \normalfont 
    For $\alpha, \beta \in \Phi, t \in R^{\star}_{[\alpha]}$ and $u \in R_{[\beta]}$, we have the following relations:
    \[  
        w_{[\alpha]}(t) x_{[\beta]}(u) w_{[\alpha]}(t)^{-1} = 
        \begin{cases}
            x_{s_{[\alpha]}([\beta])}(d([\alpha],\beta') t^{-\langle \beta', \alpha \rangle} \cdot u') & \text{if } [\alpha] \sim A_1, \\
            x_{s_{[\alpha]}([\beta])}(d([\alpha],\beta') t^{-\langle \beta', \alpha \rangle} {(\bar{t})}^{-\langle \beta', \bar{\alpha} \rangle} \cdot u') & \text{if } [\alpha] \sim A_1^2, \\
            x_{s_{[\alpha]}([\beta])}(d([\alpha],\beta') t^{-\langle \beta', \alpha \rangle} {(\bar{t})}^{-\langle \beta', \bar{\alpha} \rangle} {(\bar{\bar{t}})}^{-\langle \beta', \bar{\bar{\alpha}} \rangle} \cdot u') & \text{if } [\alpha] \sim A_1^3, \\
            x_{s_{[\alpha]}([\beta])}(d([\alpha],\beta') t^{-\langle \beta', \alpha \rangle} {(\bar{t})}^{- \langle \beta', \bar{\alpha} \rangle} \cdot u' ) & \text{if } [\alpha] \sim A_2.
        \end{cases} 
    \]

    \noindent Where the values of $\beta'$ and $u'$ depend on the representing element of the class $s_{[\alpha]}([\beta])$. To be precious, if $[\beta] \sim A_1$, then $\beta' = \beta$ and $u' = u$; if $[\beta] \sim A_1^2$, then $\beta' = \beta$ or $\bar{\beta}$ and $u' = u$ or $\bar{u}$, respectively; if $[\beta] \sim A_1^3$, then $\beta' = \beta, \bar{\beta}$ or $\bar{\bar{\beta}}$ and $u' = u, \bar{u}$ or $\bar{\bar{u}}$, respectively; if $[\beta] \sim A_2$, then $\beta' = \beta$ or $\bar{\beta}$ and $u' = u = (u_1, u_2)$ or $(\bar{u_1}, \bar{u_2})$, respectively.
\end{prop}

\begin{proof}
    By simple calculations using Lemma $19(a)$ of \cite{RS}, we have 
    \[
        w_{[\alpha]}(t) x_{\beta}(u) w_{[\alpha]}(t)^{-1} = \begin{cases}
            x_{s_{\alpha}(\beta)}(d([\alpha], \beta) t^{-\langle \beta, \alpha \rangle} u) & \text{ if } [\alpha] \sim A_1, \\
            x_{s_{\alpha}s_{\bar{\alpha}}(\beta)}(d([\alpha], \beta) t^{-\langle \beta, \alpha \rangle} \bar{t}^{-\langle \beta, \bar{\alpha} \rangle} u) & \text{ if } [\alpha] \sim A_1^2, \\
            x_{s_{\alpha}s_{\bar{\alpha}}s_{\bar{\bar{\alpha}}}(\beta)}(d([\alpha], \beta) t^{-\langle \beta, \alpha \rangle} \bar{t}^{-\langle \beta, \bar{\alpha}  \rangle} \bar{\bar{t}}^{-\langle \beta, \bar{\bar{\alpha}}  \rangle} u) & \text{ if } [\alpha] \sim A_1^2, \\
            x_{s_{\alpha + \bar{\alpha}}(\beta)}(d([\alpha], \beta) t^{-\langle \beta, \alpha \rangle} \bar{t}^{ - \langle \beta, \bar{\alpha} \rangle} u) & \text{ if } [\alpha] \sim A_2.
        \end{cases}
    \]
    If $[\beta] \sim A_1$ then our result follows the above equations. Now if $[\beta] \sim A_1^2$, then
    \begin{align*}
        w_{[\alpha]}(t) x_{[\beta]}(u) w_{[\alpha]}(t)^{-1} &= w_{[\alpha]}(t) x_{\beta}(u) x_{\bar{\beta}}(\bar{u}) w_{[\alpha]}(t)^{-1} \\
        &= (w_{[\alpha]}(t) x_{\beta}(u) w_{[\alpha]}(t)^{-1}) (w_{[\alpha]}(t) x_{\bar{\beta}}(\bar{u}) w_{[\alpha]}(t)^{-1}) \\
        &=
        \begin{cases}
            x_{s_{\alpha}(\beta)}(d([\alpha], \beta) t^{-\langle \beta, \alpha \rangle} u) x_{s_{\alpha}(\bar{\beta})}(d([\alpha], \bar{\beta}) t^{-\langle \bar{\beta}, \alpha \rangle} \bar{u}) & \text{ if } [\alpha] \sim A_1, \\
            x_{s_{\alpha}s_{\bar{\alpha}}(\beta)}(d([\alpha], \beta) t^{-\langle \beta, \alpha \rangle} \bar{t}^{-\langle \beta, \bar{\alpha} \rangle} u) \\ 
            \hspace{10mm} x_{s_{\alpha}s_{\bar{\alpha}}(\bar{\beta})}(d([\alpha], \bar{\beta}) t^{-\langle \bar{\beta}, \alpha \rangle} \bar{t}^{-\langle \bar{\beta}, \bar{\alpha} \rangle} \bar{u}) & \text{ if } [\alpha] \sim A_1^2, \\
            x_{s_{\alpha}s_{\bar{\alpha}}s_{\bar{\bar{\alpha}}}(\beta)}(d([\alpha], \beta) t^{-\langle \beta, \alpha \rangle} \bar{t}^{-\langle \beta, \bar{\alpha}  \rangle} \bar{\bar{t}}^{-\langle \beta, \bar{\bar{\alpha}}  \rangle} u) \\ 
            \hspace{10mm} x_{s_{\alpha}s_{\bar{\alpha}}s_{\bar{\bar{\alpha}}}(\bar{\beta})}(d([\alpha], \bar{\beta}) t^{-\langle \bar{\beta}, \alpha \rangle} \bar{t}^{-\langle \bar{\beta}, \bar{\alpha}  \rangle} \bar{\bar{t}}^{-\langle \bar{\beta}, \bar{\bar{\alpha}}  \rangle} \bar{u}) & \text{ if } [\alpha] \sim A_1^2, \\
            x_{s_{\alpha + \bar{\alpha}}(\beta)}(d([\alpha], \beta) t^{-\langle \beta, \alpha \rangle} \bar{t}^{- \langle \beta, \bar{\alpha} \rangle} u) \\
            \hspace{10mm} x_{s_{\alpha + \bar{\alpha}}(\bar{\beta})}(d([\alpha], \bar{\beta}) t^{-\langle \bar{\beta}, \alpha \rangle} \bar{t}^{- \langle \bar{\beta}, \bar{\alpha} \rangle} \bar{u}) & \text{ if } [\alpha] \sim A_2.
        \end{cases} \\
        &= 
        \begin{cases}
            x_{s_{[\alpha]}([\beta])}(d([\alpha],\beta') t^{-\langle \beta', \alpha \rangle} u') & \text{if } [\alpha] \sim A_1, \\
            x_{s_{[\alpha]}([\beta])}(d([\alpha],\beta') t^{-\langle \beta', \alpha \rangle} {(\bar{t})}^{-\langle \beta', \bar{\alpha} \rangle} u') & \text{if } [\alpha] \sim A_1^2, \\
            x_{s_{[\alpha]}([\beta])}(d([\alpha],\beta') t^{-\langle \beta', \alpha \rangle} {(\bar{t})}^{-\langle \beta', \bar{\alpha} \rangle} {(\bar{\bar{t}})}^{-\langle \beta', \bar{\bar{\alpha}} \rangle} u') & \text{if } [\alpha] \sim A_1^3, \\
            x_{s_{[\alpha]}([\beta])}(d([\alpha],\beta') t^{-\langle \beta', \alpha \rangle} {(\bar{t})}^{- \langle \beta', \bar{\alpha} \rangle} u' ) & \text{if } [\alpha] \sim A_2.
        \end{cases}
    \end{align*}
    Where $u'$ is either $u$ or $\bar{u}$, depending on the representative of the class $s_{[\alpha]}([\beta])$. The last equality is follows from Lemma \ref{lemma on d} and the fact that $\langle \alpha, \beta \rangle = \langle \bar{\alpha}, \bar{\beta} \rangle$ for every roots $\alpha, \beta \in \Phi$. The proof for the case when $[\beta] \sim A_1^3$ is similar and will therefore be omitted.
    Now if $[\beta] \sim A_2$, then
    \begin{align*}
        & \hspace{-3mm} w_{[\alpha]}(t) x_{[\beta]}(u) w_{[\alpha]}(t)^{-1} \\
        &= w_{[\alpha]}(t) x_{[\beta]}(u_1,u_2) w_{[\alpha]}(t)^{-1} \\
        &= w_{[\alpha]}(t) x_{\beta}(u_1) x_{\bar{\beta}}(\bar{u_1}) x_{\beta + \bar{\beta}}(N_{\bar{\beta},\beta}u_2) w_{[\alpha]}(t)^{-1} \\
        &= (w_{[\alpha]}(t) x_{\beta}(u_1) w_{[\alpha]}(t)^{-1}) (w_{[\alpha]}(t) x_{\bar{\beta}}(\bar{u_1}) w_{[\alpha]}(t)^{-1}) (w_{[\alpha]}(t) x_{\beta + \bar{\beta}}(N_{\bar{\beta},\beta}u_2) w_{[\alpha]}(t)^{-1}) \\
        &=
        \begin{cases}
            x_{s_{\alpha}(\beta)}(d([\alpha], \beta) t^{-\langle \beta, \alpha \rangle} u_1) x_{s_{\alpha}(\bar{\beta})}(d([\alpha], \bar{\beta}) t^{-\langle \bar{\beta}, \alpha \rangle} \bar{u_1}) \\ 
            \hspace{10mm} x_{s_{\alpha}(\beta + \bar{\beta})}(d([\alpha], \beta + \bar{\beta}) N_{\bar{\beta},\beta} t^{-\langle \beta + \bar{\beta}, \alpha \rangle} u_2) & \text{ if } [\alpha] \sim A_1, \\
            x_{s_{\alpha}s_{\bar{\alpha}}(\beta)}(d([\alpha], \beta) t^{-\langle \beta, \alpha \rangle} \bar{t}^{-\langle \beta, \bar{\alpha} \rangle} u_1) x_{s_{\alpha}s_{\bar{\alpha}}(\bar{\beta})}(d([\alpha], \bar{\beta}) t^{-\langle \bar{\beta}, \alpha \rangle} \bar{t}^{-\langle \bar{\beta}, \bar{\alpha} \rangle} \bar{u_1}) \\ 
            \hspace{10mm} x_{s_{\alpha}s_{\bar{\alpha}}(\beta+ \bar{\beta})}(d([\alpha], \beta + \bar{\beta}) N_{\bar{\beta},\beta} t^{-\langle \beta + \bar{\beta}, \alpha \rangle} \bar{t}^{-\langle \beta + \bar{\beta}, \bar{\alpha} \rangle} u_2) & \text{ if } [\alpha] \sim A_1^2, \\
            x_{s_{\alpha}s_{\bar{\alpha}}s_{\bar{\bar{\alpha}}}(\beta)}(d([\alpha], \beta) t^{-\langle \beta, \alpha \rangle} \bar{t}^{-\langle \beta, \bar{\alpha}  \rangle} \bar{\bar{t}}^{-\langle \beta, \bar{\bar{\alpha}}  \rangle} u_1) \\ 
            \hspace{5mm} x_{s_{\alpha}s_{\bar{\alpha}}s_{\bar{\bar{\alpha}}}(\bar{\beta})}(d([\alpha], \bar{\beta}) t^{-\langle \bar{\beta}, \alpha \rangle} \bar{t}^{-\langle \bar{\beta}, \bar{\alpha}  \rangle} \bar{\bar{t}}^{-\langle \bar{\beta}, \bar{\bar{\alpha}}  \rangle} \bar{u_1}) \\
            \hspace{10mm} x_{s_{\alpha}s_{\bar{\alpha}}s_{\bar{\bar{\alpha}}}(\beta + \bar{\beta})}(d([\alpha], \beta + \bar{\beta}) N_{\bar{\beta},\beta} t^{-\langle \beta + \bar{\beta}, \alpha \rangle} \bar{t}^{-\langle \beta + \bar{\beta}, \bar{\alpha}  \rangle} \bar{\bar{t}}^{-\langle \beta + \bar{\beta}, \bar{\bar{\alpha}}  \rangle} u_2) & \text{ if } [\alpha] \sim A_1^3, \\
            x_{s_{\alpha + \bar{\alpha}}(\beta)}(d([\alpha], \beta) t^{-\langle \beta, \alpha \rangle} \bar{t}^{ - \langle \beta, \bar{\alpha} \rangle} u_1) x_{s_{\alpha + \bar{\alpha}}(\bar{\beta})}(d([\alpha], \bar{\beta}) t^{-\langle \bar{\beta}, \alpha \rangle} \bar{t}^{-\langle \bar{\beta}, \bar{\alpha} \rangle} \bar{u_1}) \\
            \hspace{10mm} x_{s_{\alpha + \bar{\alpha}}(\beta + \bar{\beta})}(d([\alpha], \beta + \bar{\beta}) N_{\bar{\beta},\beta} t^{-\langle \beta + \bar{\beta}, \alpha \rangle} \bar{t}^{-\langle \beta + \bar{\beta}, \bar{\alpha} \rangle} u_2) & \text{ if } [\alpha] \sim A_2.
        \end{cases} \\
        &= 
        \begin{cases}
            x_{s_{[\alpha]}([\beta])}(d([\alpha],\beta') t^{-\langle \beta', \alpha \rangle} \cdot u') & \text{if } [\alpha] \sim A_1, \\
            x_{s_{[\alpha]}([\beta])}(d([\alpha],\beta') t^{-\langle \beta', \alpha \rangle} {(\bar{t})}^{-\langle \beta', \bar{\alpha} \rangle} \cdot u') & \text{if } [\alpha] \sim A_1^2, \\
            x_{s_{[\alpha]}([\beta])}(d([\alpha],\beta') t^{-\langle \beta', \alpha \rangle} {(\bar{t})}^{-\langle \beta', \bar{\alpha} \rangle} {(\bar{\bar{t}})}^{-\langle \beta', \bar{\bar{\alpha}} \rangle} \cdot u') & \text{if } [\alpha] \sim A_1^3, \\
            x_{s_{[\alpha]}([\beta])}(d([\alpha],\beta') t^{-\langle \beta', \alpha \rangle} {(\bar{t})}^{- \langle \beta', \bar{\alpha} \rangle} \cdot u' ) & \text{if } [\alpha] \sim A_2.
        \end{cases}
    \end{align*}
    Where $u'$ is either $(u_1, u_2)$ or $(\bar{u}_1, \bar{u_2})$, depending on the representative of the class $s_{[\alpha]}([\beta])$. The last equality is follows from Lemma \ref{lemma on d} and the fact that $\langle \alpha, \beta \rangle = \langle \bar{\alpha}, \bar{\beta} \rangle$ for every $\alpha, \beta \in \Phi$. 
\end{proof}

\begin{prop}
    \normalfont 
    For $\alpha, \beta \in \Phi, t \in R^{\star}_{[\alpha]}$ and $u \in R^{\star}_{[\beta]}$, we have the following relations:
    \[  
        w_{[\alpha]}(t) w_{[\beta]}(u) w_{[\alpha]}(t)^{-1} = 
        \begin{cases}
            w_{s_{[\alpha]}([\beta])}(d([\alpha],\beta') t^{-\langle \beta', \alpha \rangle} \cdot u') & \text{if } [\alpha] \sim A_1, \\
            w_{s_{[\alpha]}([\beta])}(d([\alpha],\beta') t^{-\langle \beta', \alpha \rangle} {(\bar{t})}^{-\langle \beta', \bar{\alpha} \rangle} \cdot u') & \text{if } [\alpha] \sim A_1^2, \\
            w_{s_{[\alpha]}([\beta])}(d([\alpha],\beta') t^{-\langle \beta', \alpha \rangle} {(\bar{t})}^{-\langle \beta', \bar{\alpha} \rangle} {(\bar{\bar{t}})}^{-\langle \beta', \bar{\bar{\alpha}} \rangle} \cdot u') & \text{if } [\alpha] \sim A_1^3, \\
            w_{s_{[\alpha]}([\beta])}(d([\alpha],\beta') t^{-\langle \beta', \alpha \rangle} {(\bar{t})}^{- \langle \beta', \bar{\alpha} \rangle} \cdot u' ) & \text{if } [\alpha] \sim A_2.
        \end{cases} 
    \]

    \noindent Where $\beta' = \beta, \bar{\beta}$ or $\bar{\bar{\beta}}$ and $u' = u, \bar{u}$ or $\bar{\bar{u}}$, respectively, depending on the representing element of class $s_{[\alpha]}([\beta]).$ 
\end{prop}

\begin{prop}
    \normalfont 
    For $\alpha, \beta \in \Phi, t \in R^{\star}_{[\alpha]}$ and $u \in R^{\star}_{[\beta]}$, we have 
    \[  
        w_{[\alpha]}(t) h_{[\beta]}(u) w_{[\alpha]}(t)^{-1} = h_{s_{[\alpha]}([\beta])} (u).
    \] 
\end{prop}

The proofs of the above two propositions are analogous to the proof of Proposition \ref{prop wxw^{-1}}; therefore, we omit them.
Finally, we conclude this section by stating the following well-known relations.

\begin{prop}[see {\cite[2.4]{KS1}}]\label{prop h(chi)xh(chi)^-1}
    \normalfont
    Let $[\alpha] \in \Phi_\rho$ and $h(\chi) \in T_\sigma (R)$. Then for $u \in R_{[\alpha]}$, we have
    \[
        h(\chi) x_{[\alpha]}(u) h(\chi)^{-1} = x_{[\alpha]} (\chi (\alpha) \cdot u).
    \]
\end{prop}

%%%%%%%%%%%%%%%%%%%%%%%%%%%%%%%%%%%%%%%%%%%%%%%%%%%%%%%%%%%%%%
\section{Properties of Twisted Chevalley Groups}
%%%%%%%%%%%%%%%%%%%%%%%%%%%%%%%%%%%%%%%%%%%%%%%%%%%%%%%%%%%%%%

In this section, we present a survey of several well-known results on twisted Chevalley groups.

%%%%%%%%%%%%%%%%%%%%%%%%%%%%%%%%%%%%%%%%%%%%%%%%%%%%%%%%%%%%%%
%%%%%%%%%%%%%%%%%%%%%%%%%%%%%%%%%%%%%%%%%%%%%%%%%%%%%%%%%%%%%%

\subsection{Twisted Chevalley Groups over Fields}\label{subsec:TCG_over_fields}

Let $G_{\pi, \sigma}(\Phi, k)$ denote a twisted Chevalley group of type $\Phi$ over a field $k$ and let $E'_{\pi, \sigma}(\Phi, k)$ denote its elementary subgroup. It is established that $E_{\pi, \sigma}(\Phi, k)$ coincides with $E'_{\pi, \sigma}(\Phi, k)$. When $k$ is algebraically closed, we also have $G_{\pi, \sigma}(\Phi, k) = E'_{\pi, \sigma}(\Phi, k)$, but this is not guaranteed for general fields.
For any field $k$, the twisted Chevalley group satisfies the relation:
\[
G_{\pi, \sigma}(\Phi, k) = E'_{\pi, \sigma}(\Phi, k) T_{\pi, \sigma}(\Phi, k),
\]
where $T_{\pi, \sigma}(\Phi, k)$ is the maximal torus. Furthermore, if the group $G$ is simply connected, for any field $k$, we have
\[
G_{sc, \sigma}(\Phi, k) = E'_{sc, \sigma}(\Phi, k).
\]

%%%%%%%%%%%%%%%%%%%%%%%%%%%%%%%%%%%%%%%%%%%%%%%%%%%%%%%%%%%%%%

\subsubsection*{Bruhat Decomposition}

A Bruhat decomposition also holds for twisted Chevalley groups over a field.

\begin{thm}[Bruhat decomposition]
    \normalfont
    Let $G = E'_{\pi,\sigma}(\Phi, k)$ be an elementary twisted Chevalley group over a field $k$.
    \begin{enumerate}[(a)]
        \item For each $w \in W_\rho$, the subgroup $U_w := U \cap w^{-1} U^{-} w$ is invariant under $\sigma$. 
        \item For each $w \in W_\rho$, there exists an element $\eta_w \in N'_\sigma$ such that $\eta_w H = w$.
        \item If $\eta_w \ (w \in W_\rho)$ is chosen as in part (b), then $G$ can be expressed as:
        \[
            G = \bigsqcup_{w \in W_\sigma} B_\sigma \eta_w U_{w, \sigma}
        \]
        where $U_{w, \sigma}$ is the set of elements of $U_w$ that remain fixed under the action of $\sigma$. Moreover, the representation on the right is unique.
    \end{enumerate}
\end{thm}

%%%%%%%%%%%%%%%%%%%%%%%%%%%%%%%%%%%%%%%%%%%%%%%%%%%%%%%%%%%%%%

\subsubsection*{Simplicity of adjoint groups}

\begin{thm}[{\cite[Theorem 34]{RS}}]
    \normalfont
    Let $G = E'_{\pi, \sigma} (\Phi, k)$ be an elementary twisted Chevalley group of type $\Phi$ over a field $k$. Assume that $\Phi_\rho \not\sim {}^2A_2$ when $\lvert k \rvert = 4$. Then, $G$ is simple modulo its centre.
\end{thm}

\begin{rmk}
    In this literature, we only consider root systems consisting of roots with the same root length. See \cite{RS} for all other cases.
\end{rmk}

%%%%%%%%%%%%%%%%%%%%%%%%%%%%%%%%%%%%%%%%%%%%%%%%%%%%%%%%%%%%%%

\subsubsection*{Automorphisms}

The definitions of inner automorphisms and field automorphisms are analogous to those in Section~\ref{sec:auto}. 

\begin{thm}[{\cite[Theorem 36]{RS}}]
    \normalfont
    Let $G = E'_{\pi, \sigma} (\Phi, k)$ be an elementary twisted Chevalley group of type $\Phi$ over a field $k$. Assume that $\sigma$ is not the identity. Then every automorphism of $G$ is a product of an inner automorphism and a field automorphism.
\end{thm}

\begin{rmk}
    Every inner automorphism of $G$ can be expressed as the composition of a strictly inner automorphism and a diagonal automorphism (for the definition of diagonal automorphism, see \cite{RS}). Moreover, by definition, field automorphisms commute with the automorphism $\theta$.
\end{rmk}

%%%%%%%%%%%%%%%%%%%%%%%%%%%%%%%%%%%%%%%%%%%%%%%%%%%%%%%%%%%%%%
%%%%%%%%%%%%%%%%%%%%%%%%%%%%%%%%%%%%%%%%%%%%%%%%%%%%%%%%%%%%%%

\subsection{Normality of \texorpdfstring{$E'_{\pi,\sigma} (\Phi, R)$}{E(R)} in \texorpdfstring{$G_{\pi, \sigma} (\Phi, R)$}{G(R)}}\label{sec:SKR}

In \cite{KS2} and \cite{KS3}, K. Suzuki focuses on root system of types $\Phi_\rho \sim {}^2 A_n \ (n \geq 3), {}^2 D_n \ (n \geq 4)$ and ${}^2 E_6$. It is easy to see that the analogous versions of the main theorems in both papers hold for $\Phi_\rho \sim {}^3 D_4$. In this section, we specify the conditions under which the corresponding results are valid and we also state a consequence of both main results which we utilize in the future.

First, consider the following condition on a ring $R$: 
\begin{enumerate}[leftmargin=3.5em]
    \item[\textbf{(A1)}] For any maximal ideal $\m$ of $R$, the natural map $R_{\theta} \longrightarrow (R/(\m \cap \overline{\m}))_{\theta}$ is surjective if $\Phi_\sigma$ is not of type ${}^2A_{2n}$ and the natural map $\mathcal{A}(R) \longrightarrow \mathcal{A}(R/(\m \cap \overline{\m}))$ is surjective and $\mathcal{A}(R)^* \neq \phi$ if $\Phi_\sigma$ is of type ${}^2A_{2n}$.
    \item[\textbf{(A2)}] For any maximal ideal $\m_\theta$ of $R_\theta$, we have $\m_\theta = R_{\theta} \cap \m_\theta R$.
\end{enumerate}

\begin{thm}[see \cite{KS2}]\label{thm:KS2}
    Let $G_\sigma(R)$ and $E'_\sigma(R)$ be as above. Assume that $\Phi_\rho \sim {}^2A_{n} \ (n \geq 3), {}^2D_{n} \ (n \geq 4)$ or ${}^2E_6$, and $R$ satisfies the condition $(A1)$ and $(A2)$ above. Then $E'_\sigma (R)$ is a normal subgroup of $G_\sigma (R)$. 
\end{thm}

Now consider the following conditions:
\begin{enumerate}[leftmargin=3.5em]
    \item[\textbf{(A1$'$)}] For any maximal ideal $m$ of $R$, the natural map $R_{\theta} \longrightarrow (R/(\m \cap \overline{\m} \cap \overline{\overline{\m}}))_{\theta}$ is surjective if $\Phi_\sigma$ is of type ${}^3 D_{4}$. 
\end{enumerate}

\begin{rmk}
    We can state a similar result to Theorem \ref{thm:KS2} for the case of $\Phi_\rho \sim {}^3 D_4$, assuming that $R$ satisfies conditions $(A1')$ and $(A2)$. The proof follows similar lines as those in \cite{KS2}.
\end{rmk}

\begin{lemma}\label{A1&A2}
    \normalfont
    If $2$ (resp., $3$) is invertible in $R$, then $(A1)$ (resp., $(A1')$) and $(A2)$ are satisfied. 
\end{lemma}

\begin{proof}
    Assume that $2$ (resp., $3$) is invertible in $R$. Let $I = \m \cap \overline{\m}$ (resp., $I = \m \cap \overline{\m} \cap \overline{\overline{\m}}$) then $I = \bar{I}$. Let $x \in R$ such that $x + I = \bar{x} + I \implies x-\bar{x} \in I$ (resp., $x - \bar{x}, x - \bar{\bar{x}} \in I$). Set $y = (x+\bar{x})/2 \in R_\theta$ (resp., $y = (x + \bar{x} + \bar{\bar{x}})/3 \in R_\theta$). Then $x - y = (x - \bar{x})/2 \in I$ (resp., $x-y = ((x-\bar{x}) + (x-\bar{\bar{x}}))/3 \in I$), so $x+I = y+I$. Since $y \in R_\theta$, it serves as a pre-image of $x+I$. Therefore, the map $R_\theta \longrightarrow (R/I)_\theta$ is surjective. 

    We now prove that if $2$ is invertible in $R$, then the natural map $\mathcal{A}(R) \longrightarrow \mathcal{A}(R/I)$ is surjective. Let $(x_1 + I, x_2 + I) \in \mathcal{A}(R/I)$. Then 
    \begin{align*}
        (x_1 + I) (\bar{x}_1 + I) = (x_2 + I) + (\bar{x}_2 + I) 
        \implies & x_1 \bar{x}_1 + I = (x_2 + \bar{x}_2) + I \\
        \implies & x_1 \bar{x}_1 - (x_2 + \bar{x}_2) \in I.
    \end{align*}
    Set $y_1 = x_1$ and $y_2 = x_2 + (x_1 \bar{x}_1)/ 2 - (x_2 + \bar{x}_2)/2$. Then $(y_1, y_2) \in \mathcal{A}(R)$, and clearly $(y_1 + I, y_2 + I) = (x_1 + I, x_2 + I)$, showing that $(y_1, y_2) \in \mathcal{A}(R)$ serves as a pre-image of $(x_1 + I, x_2 + I)$. Hence, the natural map $\mathcal{A}(R) \longrightarrow \mathcal{A}(R/I)$ is a surjective.

    We now return to the assumption that $2$ (resp., $3$) is invertible in $R$, and verify condition $(A2)$. It is clear that $\m_\theta \subset R_\theta \cap \m_\theta R$. Let $x \in R_\theta \cap \m_\theta R.$ Then $x = \bar{x}$ and $x = \sum_{i=1}^k m_i x_i$ where $m_i \in \m_\theta$ and $x_i \in R$. Define $y_i = (x_i + \bar{x}_i)/2$ (resp., $y_i = (x_i + \bar{x}_i + \bar{\bar{x}}_i)/3)$ and set $y = \sum_{i=1}^{k} m_i y_i \in \m_\theta$. Then $x - y = \sum_{i=1}^k m_i (x_i - y_i) = \sum_{i=1}^{k} m_i (x_i - \bar{x_i})/2 = (x - \bar{x})/2 = 0$ (resp., $x-y = \sum_{i=1}^{k} m_i ((x_i - \bar{x_i}) + (x_i - \bar{\bar{x_i}}))/3 = ((x - \bar{x}) + (x- \bar{\bar{x}}))/3 = 0$). Therefore $x = y \in \m_\theta$, as desired.
\end{proof}

\begin{cor}\label{cor:KS2}
    Assume that $1/2 \in R$ if $\Phi_\rho \sim {}^2 A_{n} \ (n \geq 3), {}^2 D_n \ (n \geq 4)$ or ${}^2 E_6$, and that $1/3 \in R$ if $\Phi_\rho \sim {}^3 D_4$. Then $E'_\sigma (R)$ is a normal subgroup of $G_\sigma (R)$.
\end{cor}

%%%%%%%%%%%%%%%%%%%%%%%%%%%%%%%%%%%%%%%%%%%%%%%%%%%%%%%%%%%%%%
%%%%%%%%%%%%%%%%%%%%%%%%%%%%%%%%%%%%%%%%%%%%%%%%%%%%%%%%%%%%%%

\subsection{Center of the Twisted Chevalley Groups}\label{sec:Center of TCG}

Recall that in Section \ref{Subsec:TCG}, we defined the set $\text{Hom}_1 (\Lambda_\pi, R^*) = \{ \chi \in \text{Hom} (\Lambda_\pi, R^*) \mid \chi = \bar{\chi}_\sigma \}$. We also define $\text{Hom}_1(\Lambda_\pi / \Lambda_r, R^*)$ as $\{ \chi \in \text{Hom}_1 (\Lambda_\pi, R^*) \mid \chi|_{\Lambda_r} = 1 \}$. The following theorem is a special case of the main theorem of \cite{KS3}.
\begin{thm}\label{thm:KS3}
    \normalfont
    Let $G=G_\sigma (R)$ be a twisted Chevalley group of type $\Phi_\rho \sim {}^2 A_{n} \ (n \geq 3), {}^2 D_n \ (n \geq 4)$ or ${}^2 E_6$ and let $E = E'_\sigma(R)$ be its elementary subgroup. Assume that $1/2 \in R$. Then $Z(G) = C_G (E) = \text{Hom}_1(\Lambda_\pi / \Lambda_r, R^*)$, where $Z(G)$ is a centre of $G$ and $C_G (E)$ is a centralizer of $E$ in $G$. 
\end{thm}

\begin{rmk}
    As before, we can state a similar result to Theorem \ref{thm:KS3} for the case of $\Phi_\rho \sim {}^3D_4$, assuming $1/3 \in R$. The proof follows similar lines as those in \cite{KS3}.
\end{rmk}

\chapter{Normal Subgroups of \texorpdfstring{$G_\sigma(R)$}{G(R)}}\label{chapter:Normal_subgroups}

%%%%%%%%%%%%%%%%%%%%%%%%%%%%%%%%%%%%%%%%%%%%%%
%%%%%%%%%%%%%%%%%%%%%%%%%%%%%%%%%%%%%%%%%%%%%%

The primary objective of this chapter is to classify the subgroups of $G_\sigma (R)$ that are normalized by $E'_\sigma (R)$. As a result, we also characterize all the normal subgroups of $E'_\sigma (R)$. We begin by introducing the concept of congruence subgroups of $G_\sigma (R)$ and examining their key properties. Next, we investigate certain mixed commutator formulas. Finally, we state and prove the main classification theorem.

%%%%%%%%%%%%%%%%%%%%%%%%%%%%%%%%%%%%%%%%%%%%%%
\section{Congruence Subgroups of \texorpdfstring{$G_\sigma(R)$}{G(R)}}\label{sec:E=UHV}
%%%%%%%%%%%%%%%%%%%%%%%%%%%%%%%%%%%%%%%%%%%%%%

Let $R$ be a commutative ring with unity and let $\theta$ be an automorphism of a ring $R$ of order $2$ or $3$. Let $J$ be a $\theta$-invariant ideal of $R$ (that is, $J$ is an ideal of $R$ such that $\theta (J) \subset J$). For $[\alpha] \in \Phi_{\rho}$, define 
\[ 
    J_{[\alpha]} = 
    \begin{cases}
        J_{\theta} & \text{if } [\alpha] \sim A_1, \\
        J & \text{if } [\alpha] \sim A_1^2 \text{ or } A_1^3, \\
        \mathcal{A}(J) & \text{if } [\alpha] \sim A_2; \\
    \end{cases}
\]
where $J_{\theta} = \{ r \in J \mid r = \overline{r} \} = J \cap R_{\theta}$ and $\mathcal{A}(J) = \{ (a,b) \in \mathcal{A}(R) \mid a,b \in J \}.$

The natural projection map $R \longrightarrow R/J$ induces the canonical map 
\[
    \phi: G_{\pi, \sigma}(\Phi, R) \longrightarrow G_{\pi, \sigma}(\Phi, R/J).
\]
Define $G_{\pi, \sigma}(\Phi, J) = \ker(\phi)$ and $G_{\pi, \sigma}(\Phi, R, J) = \phi^{-1}(Z(G_{\pi, \sigma}(\Phi, R/J)))$, where $Z(G)$ denotes the center of a group $G$. 
The subgroup $G_{\pi, \sigma}(\Phi, J)$ of $G_{\pi, \sigma}(\Phi, R)$ is called the \textbf{principal congruence subgroup}\label{nomencl:G_sigma(J)} \index{subgroups of $G_\sigma(R)$!principal congruence subgroup} of level $J$, while $G_{\pi, \sigma}(\Phi, R, J)$ is referred to as the \textbf{full congruence subgroup}\label{nomencl:G_sigma(R,J)} \index{subgroups of $G_\sigma(R)$!full congruence subgroup} of level $J$. 

Let $E'_{\pi, \sigma}(\Phi, J)$ be the subgroup of $E'_{\pi, \sigma}(\Phi, R)$ generated by all elements $x_{[\alpha]}(t)$, where $[\alpha] \in \Phi_{\rho}$ and $t \in J_{[\alpha]}$. Define $E'_{\pi, \sigma}(\Phi, R, J)$ as the normal subgroup of $E'_{\pi, \sigma}(\Phi, R)$ generated by $E'_{\pi, \sigma}(\Phi, J)$.
The subgroup $E_{\pi, \sigma}(\Phi, R, J)$ of $G_{\pi, \sigma}(\Phi, R)$ is referred to as the \textbf{relative elementary subgroup}\label{nomencl:E'_sigma(R,J)} \index{subgroups of $G_\sigma(R)$!relative elementary subgroups} of level $J$.
Note that $E'_{\pi, \sigma}(\Phi, R, J)$ is also a subgroup of $E'_{\pi, \sigma}(\Phi, R) \cap G_{\pi, \sigma}(\Phi, J)$, since the latter is normal in $E'_{\pi, \sigma}(\Phi, R)$ and contains $E'_{\pi, \sigma}(\Phi, J)$. 

\begin{rmk}
    In the absence of ambiguity, we use the shorter notations $G_\sigma(J)$, $G_\sigma(R, J)$, $E'_\sigma(J)$, and $E'_\sigma(R, J)$ instead of $G_{\pi, \sigma}(\Phi, J)$, $G_{\pi, \sigma}(\Phi, R, J)$, $E'_{\pi, \sigma}(\Phi, J)$, and $E'_{\pi, \sigma}(\Phi, R, J)$, respectively.
\end{rmk}

Let $U_\sigma (J)$ (resp., $U^{-}_\sigma (J)$) be the subgroup of $E'_\sigma (R)$ generated by $x_{[\alpha]}(t)$ (resp., $x_{-[\alpha]}(t)$) where $[\alpha] \in \Phi^{+}_\rho$ and $t \in J_{[\alpha]}$. Define $T_\sigma (J) = G_\sigma(J) \cap T_\sigma (R), T_\sigma (R, J) = G_\sigma (R, J) \cap T_\sigma(R), H_\sigma(J) = E_\sigma (J) \cap T_\sigma (R), H_\sigma (R,J) = E_\sigma (R,J) \cap T_\sigma (R), H'_\sigma(J) = E'_\sigma (J) \cap T_\sigma (R)$ and $H'_\sigma (R,J) = E'_\sigma (R,J) \cap T_\sigma (R).$

\begin{lemma}\label{structure of U}
    \normalfont
    Let $J$ be any $\theta$-invariant ideal of $R$. Then each element of $U_\sigma (J)$ is uniquely expressible in the form
    \[
        x_{[\alpha_1]}(t_1) \dots x_{[\alpha_n]}(t_n)
    \]
    where $[\alpha_i] \in \Phi^+_\rho$ and $t_i \in J_{[\alpha_i]}$, the ordering of the roots is arbitrarily chosen and fixed once for all. 
\end{lemma}

\begin{proof}
    The proof is an easy consequence of the Chevalley commutator formula and is therefore omitted.
\end{proof}

\begin{rmk}
    \normalfont
    We can state and prove a similar result for $U^{-}_\sigma (J)$.
\end{rmk}

\begin{rmk}
    Note that $U_\sigma (J) = E'_\sigma (R,J) \cap U_\sigma (R)$ and $U^{-}_\sigma (J) = E'_\sigma (R,J) \cap U^{-}_\sigma (R)$. This can be seen as follows: Clearly, $U_\sigma (J) \subset E'_\sigma (R,J) \cap U_\sigma (R)$. For the reverse inclusion, let $x \in E'_\sigma (R,J) \cap U_\sigma (R)$. Since $E'_\sigma (R,J) \subset G_\sigma (J)$ we have $x \equiv 1$ (mod $J$). From the uniqueness in the above lemma, we conclude that $x \in U_\sigma (J)$. 
\end{rmk}

\begin{lemma}\label{inUHV}
    \normalfont
    Let $J$ be a $\theta$-invariant ideal of $R$ contained in $rad (R)$, the Jacobson radical of $R$. Then for any $[\alpha] \in \Phi_\rho, t \in J_{[\alpha]}, s \in R_{[\alpha]},$ we have 
    $$ x_{[\alpha]}(s) x_{-[\alpha]}(t) x_{[\alpha]}(s)^{-1} = x_{[\alpha]}(a) h x_{-[\alpha]}(b) $$ where $a,b \in J_{[\alpha]}$ and $h \in H'_\sigma (R,J)$.
\end{lemma}

\begin{proof}
    If $[\alpha] \sim A_1, A_1^2$ or $A_1^3$ then for given $t \in J_{[\alpha]}$ and $s \in R_{[\alpha]}$, we have $(1-st) \in R_{[\alpha]}^*$. In this case, we can take $a = -t s^2 (1-st)^{-1}, b= t (1-st)^{-1}$ and $h = h_{[\alpha]}((1-st)^{-1})$. Clearly, $a, b \in J_{[\alpha]}$. But then $h \in E'_\sigma (R,J)$ and hence $h \in H'_\sigma (R,J)$.

    If $[\alpha] \sim A_2$ then for given $t = (t_1, t_2) \in \mathcal{A}(J)$ and $s=(s_1, s_2) \in \mathcal{A}(R)$, we have $1-(s_1 \bar{t_1} - \bar{s_2}t_2) \in R^*.$ Therefore, the equation
    \begin{small}
        \[
            \begin{pmatrix}
                1 & \bar{s}_1 & s_2 \\
                0 & 1 & s_1 \\
                0 & 0 & 1
            \end{pmatrix}
            \begin{pmatrix}
                1 & 0 & 0 \\
                t_1 & 1 & 0 \\
                t_2 & \bar{t}_1 & 1
            \end{pmatrix}
            \begin{pmatrix}
                1 & -\bar{s_1} & \bar{s_2} \\
                0 & 1 & -s_1 \\
                0 & 0 & 1
            \end{pmatrix} = 
            \begin{pmatrix}
                1 & \bar{a_1} & a_2 \\
                0 & 1 & a_1 \\
                0 & 0 & 1
            \end{pmatrix}
            \begin{pmatrix}
                \bar{u} & 0 & 0 \\
                0 & u \bar{u}^{-1} & 0 \\
                0 & 0 & u^{-1}
            \end{pmatrix}
            \begin{pmatrix}
                1 & 0 & 0 \\
                b_1 & 1 & 0 \\
                b_2 & \bar{b_1} & 1
            \end{pmatrix} 
        \]
    \end{small} 
    has a solution for $u, a_1, a_2, b_1, b_2$ and it is given by $u = (1-(\bar{t_1}s_1 - t_2\bar{s_2}))^{-1}, a_1= (t_1\bar{s_2} - \bar{t_1} {s_1}^2 + t_2 s_1 \bar{s_2}) (1-(\bar{t_1}s_1 - t_2\bar{s_2}))^{-1}, a_2 = (t_1 \bar{s_1} \bar{s_2} - \bar{t_1} s_1 s_2 + t_2 s_2 \bar{s_2}) (1-(\bar{t_1}s_1 - t_2\bar{s_2}))^{-1}, b_1= (t_1 - s_1\bar{t_2}) (1-(t_1 \bar{s_1} - \bar{t_2} {s_2}))^{-1}, b_2= t_2 (1-(\bar{t_1}s_1 - t_2\bar{s_2}))^{-1}$. By simple calculation we can see that $(a_1, a_2), (b_1, b_2) \in \mathcal{A}(J)$. But then $$h_{[\alpha]}(u) = x_{[\alpha]}(a)^{-1} x_{[\alpha]}(s) x_{-[\alpha]}(t) x_{[\alpha]}(s)^{-1} x_{-[\alpha]}(b)^{-1} \in E'_\sigma (R, J).$$ Since, $h = h_{[\alpha]}(u) \in T_\sigma (R)$, we have $h = h_{[\alpha]}(u) \in H'_\sigma (R,J)$.
\end{proof}

Let $[\alpha_i]$ be the simple roots of $\Phi_\rho$. We define the height $ht ([\alpha]):= \sum_{i=1}^l m_i$ of a root $[\alpha]= \sum_{i=1}^l m_i [\alpha_i]$ in $\Phi_\rho$. The order of the roots is \textit{regular} if the height $ht ([\alpha])$ is an increasing function of $[\alpha]$. From now on we fix a regular ordering of the roots in $\Phi_\rho$.

\begin{lemma}\label{inUV}
    \normalfont
    Let $J$ be any $\theta$-invariant ideal of $R$. Then for any $[\alpha], [\beta] (\neq [\alpha]) \in \Phi^{+}_\rho$ and $t \in J_{[\alpha]}, s \in R_{[\beta]},$ we have 
    $$ x_{-[\beta]}(s) x_{[\alpha]}(t) x_{-[\beta]}(s)^{-1} = xy$$
    where $x \in U_\sigma (J)$ and $y$ is a product of $x_{-[\gamma]}(u)$'s ($u \in J_{[\gamma]}$) in $U^{-}_\sigma (J)$ such that $-[\gamma] > - [\beta]$.
\end{lemma}

\begin{proof}
    Immediate from the Chevalley commutator formula for $[x_{[\alpha]}(t)^{-1}, x_{-[\beta]}(s)].$
\end{proof}

\begin{prop}\label{lavidecomposition}
    \normalfont
    Let $J$ be a $\theta$-invariant ideal of $R$ contained in $rad (R)$. Then $$E'_\sigma (R,J) = U_\sigma(J) H'_\sigma(R,J) U^{-}_\sigma(J).$$
\end{prop}

\begin{proof}

The proof is similar to that of $2.8$ in \cite{EA2} and $2.1$ in \cite{EA&KS}. However, for the convenience of the reader, we will provide the full proof here.

We write $U = U_\sigma (J), H = H'_\sigma (R, J)$ and $V=U^{-}_\sigma (J)$. Clearly, $UHV \subset E'_\sigma (R, J)$. For converge, it is enough to show that $UHV$ is a normal subgroup of $E'_\sigma (R)$ because then $E'_\sigma (J) \subset UHV$ and hence $E'_\sigma (R, J) \subset UHV$. 

First, we will show that $UHV$ is a subgroup of $E'_\sigma (R)$. In other words, we will show that $g h^{-1} \subset UHV$ for every $g, h \in UHV$. For that, it is enough to show that $g (UHV) \subset UHV$ for any element $g$ of the form $x_{[\beta]}(t) \in U, h_{[\beta]}(t) \in H$ and $x_{-[\beta]}(t) \in U^{-}$. If $g = x_{[\beta]} (t)$, then by Lemma \ref{structure of U}, we have $x_{[\beta]}(t) U \subset U$ and hence $x_{[\beta]}(t) UHV \subset UHV$. Similarly, if $g = h_{[\beta]}(t)$, then from Proposition \ref{prop h(chi)xh(chi)^-1}, we have $h_{[\beta]}(t) U \subset U H$ and hence $h_{[\beta]}(t) UHV \subset UHV$. Finally, if $g= x_{-[\beta]}(t),$ we must show that 
\begin{align}
    x_{-[\beta]}(t) U \subset UHV. \label{eq1}
\end{align}
Because then $x_{-[\beta]}(t) UHV \subset (UHV)HV = UH(VH)V = UH(HV)V = UHV$, the second last equality is follows from Proposition \ref{prop h(chi)xh(chi)^-1}.

To prove (\ref{eq1}), we must show that $x_{-[\beta]}(t)x \in UHV$ for every $x \in U$. Write $$x = x_{[\alpha_1]}(t_1) \dots x_{[\alpha_n]}(t_n) \in U,$$ where each $[\alpha_i] \in \Phi_\rho^+ \ (i=1,\dots,n)$ with $[\alpha_1] > \dots > [\alpha_n]$. Let $m = ht([\beta])$. We will use double induction on the pair $(m, n)$ to prove our result. If $n = 1$ then for any pair $(m,1)$ the result is follows from Lemma \ref{inUHV}, if $[\alpha_1] = [\beta]$; and from Lemma \ref{inUV}, if $[\alpha_1] \neq [\beta]$. Assume that for all $1 \leq k \leq n-1$, the result is true for the pairs $(m,k)$ for all $m \geq 1$. We will show that it is also true for the pair $(m,n)$ for all $m \geq 1$, that is, we will show that $x_{-[\beta]}(t) x_{[\alpha_1]}(t_1) \dots x_{[\alpha_n]}(t_n) \in UHV$. 
If $[\beta] = [\alpha_1]$ then, by Lemma \ref{inUHV}, we have $x_{-[\beta]}(t)x_{[\alpha_1]}(t_1) = x_{[\alpha_1]}(t'_1) h x_{-[\beta]}(t')$, where $h \in H$. 
Thus, by induction hypothesis, 
\begin{align*}
    x_{-[\beta]}(t) x_{[\alpha_1]}(t_1) \dots x_{[\alpha_n]}(t_n) &= x_{[\alpha_1]}(t'_1) h x_{-[\beta]}(t') x_{[\alpha_2]}(t_2) \dots x_{[\alpha_n]}(t_n) \\
    &\in U H (UHV) = U(HU)HV = U(UH)HV = UHV.
\end{align*}
Similarly, if $[\beta] \neq [\alpha_1]$ then, by Lemma \ref{inUV}, we have $x_{-[\beta]}(t)x_{[\alpha_1]}(t_1) = x y x_{-[\beta]}(t)$, where $x \in U$ and $y \in V$ as in Lemma \ref{inUV}. Thus, by induction hypothesis, $$x_{-[\beta]}(t) x_{[\alpha_1]}(t_1) \dots x_{[\alpha_n]}(t_n) = x y x_{-[\beta]}(t) x_{[\alpha_2]}(t_2) \dots x_{[\alpha_n]}(t_n) \in U (UHV) = UHV.$$

Now we will show that $UHV$ is a normal subgroup of $E'_\sigma (R)$. It is suffices to show that $x_{\pm [\alpha]}(t) UHV x_{\pm [\alpha]}(t)^{-1} \subset UHV$ for any root $[\alpha] \in \Phi_\rho^+$ and $t \in R_{[\alpha]}$. Clearly, $x_{[\alpha]}(t) U x_{[\alpha]}(t)^{-1} \subset U$ and $x_{[\alpha]}(t) H x_{[\alpha]}(t)^{-1} \subset UH$ (see Proposition \ref{prop h(chi)xh(chi)^-1}). We will show that $x_{[\alpha]}(t) V x_{[\alpha]}(t)^{-1} \in UHV$. Let $y \in V$ be such that $y = x_{-[\alpha_1]}(t_1) \dots x_{-[\alpha_n]}(t_n)$, where $[\alpha_i] \in \Phi^+_\rho, t_i \in J_{[\alpha_i]} \ (i=1,\dots, n)$. Then 
\begin{align*} 
    x_{[\alpha]}(t) y x_{[\alpha]}(t)^{-1} &= x_{[\alpha]}(t) x_{[\alpha_1]}(t_1) \dots x_{[\alpha_n]}(t_n) x_{[\alpha]}(t)^{-1} \\
    &= (x_{[\alpha]}(t) x_{[\alpha_1]}(t_1) x_{[\alpha]}(t)^{-1}) \dots (x_{[\alpha]}(t) x_{[\alpha_n]}(t_n)) x_{[\alpha]}(t)^{-1}) \in UHV.
\end{align*} 
The containment follows because if $[\alpha_i] = [\alpha]$ then, by Lemma~\ref{inUHV}, $$x_{[\alpha]}(t)x_{[\alpha_i]}(t_i)x_{[\alpha]}(t)^{-1} \in UHV$$ and if $[\alpha_i] \neq [\alpha]$ then, by Lemma \ref{inUV}, $x_{[\alpha]}(t)x_{[\alpha_i]}(t_i)x_{[\alpha]}(t)^{-1} \in UV \subset UHV$, and also the fact that $UHV$ is a subgroup of $G$. Finally, we have
\begin{align*}
    x_{[\alpha]}(t) UHV x_{[\alpha]}(t)^{-1} &= (x_{[\alpha]}(t) U x_{[\alpha]}(t)^{-1}) (x_{[\alpha]}(t) H x_{[\alpha]}(t)^{-1}) (x_{[\alpha]}(t) V x_{[\alpha]}(t)^{-1}) \\
    &\subset (U) (UH) (UHV) = U(HU)HV = U(UH)HV = UHV.
\end{align*}
Similarly, one can show that $x_{-[\alpha]}(t) UHV x_{-[\alpha]}(t)^{-1} \subset UHV$. 
\end{proof}

\begin{prop}\label{P(J) and Q(J)}
    Let $J$ be a $\theta$-invariant ideal of $R$ contained in $rad(R)$.
    Set $P_\sigma(J) = U_\sigma(J) T_\sigma(R) U^{-}_\sigma(R)$ and $Q_\sigma(J) = U_\sigma(R) T_\sigma(R) U^{-}_\sigma(J)$. Then $P_\sigma(J)$ and $Q_\sigma(J)$ are subgroups of $G_\sigma(R)$.
\end{prop}

\begin{proof}
    Note that $E'_\sigma(R,J) = U_\sigma(J) H'_\sigma(R,J) U^{-}_\sigma(J)$ is normalized by $E'_\sigma(R)$ and $T_\sigma(R)$. Set $B_\sigma(R) = U_\sigma(R) T_\sigma(R) = T_\sigma(R) U_{\sigma}(R)$ and $B_\sigma^{-}(R) = U_\sigma^{-}(R) T_\sigma(R) = T_\sigma(R) U_\sigma^{-}(R)$. Clearly, both $B_\sigma (R)$ and $B^{-}_\sigma (R)$ are subgroups of $G_\sigma(R).$ By a similar argument as in Proposition~\ref{lavidecomposition}, we have 
    \[
        P_\sigma(J) = E'_\sigma(R,J)B^{-}_\sigma(R) = B^{-}_\sigma(R) E'_\sigma(R,J) 
    \]
    and
    \[
        Q_\sigma(J) = B_\sigma(R) E'_\sigma(R,J) = E'_\sigma(R,J)B_\sigma(R).
    \]
    Therefore, $P_\sigma(J)$ and $Q_\sigma(J)$ are subgroups of $G_\sigma(R)$.
\end{proof}

For any $\theta$-invariant ideal $J$ of $R$, we have the canonical map $\phi: G_\sigma (R) \longrightarrow G_\sigma (R/J)$ as mentioned above. We now consider the canonical map $\phi': G (R) \longrightarrow G(R/J)$. Clearly, $\phi' |_{G_\sigma(R)} = \phi$. Let $G(J) = \ker(\phi')$ and $G(R,J) = \phi'^{-1} (Z(G(R/J)))$, where $Z(G(R/J))$ is the center of the group $G(R/J)$. Let \(U(J)\) (respectively, \(U^-(J)\)) be the subgroup of \(G(R)\) generated by \(x_{\alpha}(t)\) for \(t \in J\) and \(\alpha \in \Phi^+\) (respectively, \(\alpha \in \Phi^-\)). Define \(T(J) = G(J) \cap T(R)\) and \( T(R,J) = G(R,J) \cap T(R) \).

\begin{lemma}\label{lemma on T(J)}
    \normalfont
    Let $J$ be any $\theta$-invariant ideal of $R$. Then 
    \begin{enumerate}[(a)]
        \item the subgroup $T_\sigma (J)$ of $T_\sigma (R)$ is generated by all $h(\chi)$ such that $\chi = \bar{\chi}_\sigma$ and $\chi (\mu) \equiv 1 $ (mod $J$) for every $\mu \in \Omega_\pi$, where $\Omega_\pi$ is a set of all weights of the representation $\pi$.
        \item the subgroup $T_\sigma (R, J)$ of $T_\sigma (R)$ is generated by all $h(\chi)$ such that $\chi = \bar{\chi}_\sigma$ and $\chi (\alpha) \equiv 1 $ (mod $J$) for every $\alpha \in \Phi$.
    \end{enumerate}
\end{lemma}

\begin{proof}
    To prove our result it is enough to prove that 
    \begin{enumerate}[(a)]
        \item $T_\sigma (J) = T(J) \cap G_\sigma (R)$ and $T(J)$ is a subgroup of $T(R)$ generated by all $h(\chi)$ such that $\chi (\mu) \equiv 1$ (mod $J$) for every $\mu \in \Omega_\pi$.
        \item $T_\sigma (R, J) = T(R, J) \cap G_\sigma (R)$ and $T(R, J)$ is a subgroup of $T(R)$ generated by all $h(\chi)$ such that $\chi (\alpha) \equiv 1$ (mod $J$) for every $\alpha \in \Phi$.
    \end{enumerate}
    
    Since $G_\sigma (J) = G(J) \cap G_\sigma (R)$ and $T_\sigma (R) = T(R) \cap G_\sigma (R)$, the first assertion of the part (a) is clear. The second assertion of part (a) directly follows from the definition of $T(J)$ and the action of $h(\chi)$ on the weight spaces corresponding to the representation $\pi$. 

    Note that the center of $G(R/J)$ is $Z(G(R/J)) = \text{Hom }(\Lambda_\pi / \Lambda_r, (R/J)^*)$ (see \cite{EA&JH}) and the center of $G_\sigma (R/J)$ is $Z(G_\sigma (R/J)) = \text{Hom}_1(\Lambda_\pi / \Lambda_r, (R/J)^*)$ (see Theorem \ref{thm:KS3}). Therefore, $Z(G_\sigma (R/J)) = Z(G(R/J)) \cap G_\sigma (R/J)$. But then $G_\sigma (R, J) = G (R, J) \cap G_\sigma (R)$ and hence the first assertion of part (b) follows. For the second assertion of part (b), let $T' (R, J)$ be the subgroup of $T(R)$ generated by all $h(\chi)$ such that $\chi (\alpha) \equiv 1 $ (mod $J$), for every $\alpha \in \Phi$. We want to show that $T' (R, J) = T (R, J)$. 
    Let $h(\chi) \in T (R, J) = G (R, J) \cap T (R)$. Since $G (R, J)$ is normal subgroup of $G (R)$, we have $[h(\chi), x_{\alpha}(1)] \in G (R, J)$ for all $\alpha \in \Phi$, that is $x_{\alpha}(\chi(\alpha) - 1) \in G (R, J)$ for all $\alpha \in \Phi$. Hence, by the main theorem of \cite{EA&JH}, $\chi(\alpha) \equiv 1$ (mod $J$) for all $\alpha \in \Phi$. Thus, $T (R, J) \subset T'(R, J)$. For the reverse inclusion, let $h(\chi) \in T' (R, J)$. Then $\chi (\alpha) \equiv 1 $ (mod $J$), for every $\alpha \in \Phi$ and hence $\phi' (h(\chi)) \in Z(G (R/J))$ (again by the main theorem of \cite{EA&JH}). That is, $h (\chi) \in G(R,J)$. Thus, we have $T'(R, J) \subset T(R, J)$.
\end{proof}

%%%%%%%%%%%%%%%%%%%%%%%%%%%%%%%%%%%%%%%%%%%%%%%%%%
%%%%%%%%%%%%%%%%%%%%%%%%%%%%%%%%%%%%%%%%%%%%%%%%%%

\begin{prop}\label{G=UTV}
    \normalfont
    Let $J$ be a $\theta$-invariant ideal of $R$ contained in $rad (R)$. Then $$G_\sigma (J) = U_\sigma(J) T_\sigma(J) U^{-}_\sigma(J) \subset G'_\sigma (R).$$
\end{prop}

\begin{proof}
    By $2.3$ of \cite{EA&KS}, we have $G(J) = U(J) T(J) U^-(J)$. Note that $G_\sigma (J) = G(J) \cap G_\sigma (R)$, that is, $G_\sigma(J) = \{ U(J) T(J) U^-(J) \} \cap G_\sigma (R)$. Also, we have $U_\sigma (J) = \{ x \in U(J) \mid \sigma(x) = x \}, U^{-}_\sigma (J) = \{ x \in U^{-}(J) \mid \sigma(x) = x \}$ and $T_\sigma (J) = T_\sigma (R) \cap G_\sigma (J) = T (J) \cap G_\sigma (R)$ (by the proof of Lemma \ref{lemma on T(J)}). Since $U_\sigma (J) \cap U^{-}_\sigma (J) = U_\sigma (J) \cap T_\sigma (J) = U^{-}_\sigma (J)\cap T_\sigma (J) = \{ 1 \}$, we have 
    $$G_\sigma (J) = U_\sigma(J) T_\sigma(J) U^{-}_\sigma(J).$$
    It is also clear that $G_\sigma (J) \subset G'_\sigma(R)$. 
\end{proof}

\begin{prop}\label{G=G'}
    \normalfont
    Let $R$ be a semi-local ring. Then $G_\sigma (R) = G^{0}_\sigma (R) = G'_\sigma (R)$. 
\end{prop}

\begin{proof}
    Let $J = rad(R)$. Since $R$ is semi-local, it has finitely many maximal ideals, say $\mathfrak{m}_1, \dots, \mathfrak{m}_k$. Therefore, by the Chinese remainder theorem, we have $R/J = \prod_{i=1}^k R/\mathfrak{m}_i$. 
    Write $\bar{\mathfrak{m}}_i = \theta (\mathfrak{m}_i)$. Set 
    \[
        J_i = \begin{cases}
            \mathfrak{m}_i & \text{ if } \mathfrak{m}_i = \bar{\mathfrak{m}}_i, \\
            \mathfrak{m}_i \cap \bar{\mathfrak{m}}_i & \text{ if } \mathfrak{m}_i \neq \bar{\mathfrak{m}}_i \text{ and } o(\theta) = 2, \\
            \mathfrak{m}_i \cap \bar{\mathfrak{m}}_i \cap \bar{\bar{\mathfrak{m}}}_i & \text{ if } \mathfrak{m}_i \neq \bar{\mathfrak{m}}_i \text{ and } o(\theta) = 3. 
        \end{cases}
    \]
    By the proof of Proposition $2.2$ of \cite{KS2}, we have $G_\sigma(R/J_i) = G'_\sigma (R/J_i)$ (that proof only addresses the case where $o(\theta) = 2$. However, the proof for the case where $o(\theta) = 3$ follows a similar structure).
    
    Since $R/J = \prod_{i=1}^k R/\mathfrak{m}_i = \prod_{i=1}^l R/J_i$, we have $G_\sigma(R/J) = \prod_{i=1}^l G_\sigma(R/J_i).$ But then $$G_\sigma(R/J) = \prod_{i=1}^l G_\sigma (R/J_i) = \prod_{i=1}^l G'_\sigma (R/J_i) = G'_\sigma (R/J).$$ On the other hand, from Proposition \ref{G=UTV}, $G_\sigma(J) \subset G'_\sigma (R)$. Therefore $G_\sigma(R) \subset G'_\sigma (R)$. Hence $G'_\sigma(R) = G^{0}_\sigma (R) = G_\sigma(R)$, as desired. 
\end{proof}

\begin{cor}\label{G(R,J)=UTV}
    \normalfont
    Let $R$ be a semi-local ring and let $J$ be a $\theta$-invariant ideal of $R$ contained in $rad(R)$. Then 
    $$G_\sigma (R,J) = U_\sigma (J) T_\sigma (R,J) U^{-}_\sigma (J).$$
\end{cor}

\begin{proof}
    Since $G_\sigma (J)$ is normalized by $T_\sigma (R,J)$, we conclude that $G_\sigma (J) T_\sigma (R,J)$ is a subgroup of $G_\sigma (R,J)$. On the other hand, by the above proposition, we have $G_\sigma(R) = G'_\sigma (R) = E'_\sigma (R) T_\sigma(R)$. Let $z \in G_\sigma (R,J) \subset G_\sigma(R)$. Then there exist $x \in E'_\sigma (R)$ and $y \in T_\sigma (R)$ such that $z=xy$. Now consider the canonical map $\phi: G_\sigma(R) \longrightarrow G_\sigma(R/J)$. Then $\phi (z) = \phi (x) \phi(y) \in Z(G_\sigma(R/J))$. Since $\phi(y) \in T_\sigma(R/J)$, we obtain $\phi(x) \in T_\sigma(R/J)$. To be precious, $\phi(x) \in H'_\sigma(R/J)$ as $x \in E'_\sigma (R)$. But then there exist $h \in H'_\sigma (R)$ such that $\phi (h) = \phi (x)$, that is, $xh^{-1} \in G_\sigma (J)$. Hence $hy \in T_\sigma(R,J)$. Write $x'= xh^{-1}$ and $y'= h y$. Then $z= xy = x' y' \in G_\sigma (J) T_\sigma(R,J)$. Therefore, $G_\sigma(R,J) = G_\sigma (J) T_\sigma(R,J) = U_\sigma (J) T_\sigma (R,J) U^{-}_\sigma (J)$, the last equality is due to Proposition \ref{G=UTV}.
\end{proof}

%%%%%%%%%%%%%%%%%%%%%%%%%%%%%%%%%%%%%%%%%%%%%%%%%%
%Section: The Subgroup $E'_\sigma(R,J)$
%%%%%%%%%%%%%%%%%%%%%%%%%%%%%%%%%%%%%%%%%%%%%%%%%%

\section{Commutator Formulae}\label{sec:E(R,J)}

In this section, we present several important commutator formulae. Additionally, we investigate properties of the subgroup $E'_\sigma(R, J)$. Similar properties have been studied by L. N. Vaserstein in \cite{LV} for the case of Chevalley groups. Using his ideas, we will state and prove analogous properties for twisted Chevalley groups. For the remainder of this section, we adopt the following conventions.
 
\begin{conv}
    Assume that $\Phi_\rho$ is irreducible and the rank of $\Phi_\rho > 1$. Any ideal $J$ of $R$ is $\theta$-invariant. If $o(\theta) = 2$ then assume that $1/2 \in R$ and if $o(\theta) = 3$ then assume that $1/2 \in R$ and $1/3 \in R$. 
\end{conv}

%%%%%%%%%%%%%%%%%%%%%%%%%%%%%%%%%%%%%%%%%%%%%%%%%%

\begin{prop}\label{normal}
    For any ideal $J$ of $R$, the subgroup $E'_\sigma (R,J)$ of $G_\sigma (R)$ is normal. In other words,
    \[
        [G_\sigma (R), E'_\sigma (R,J)] \subset E'_\sigma (R,J).
    \]
\end{prop}

%%%%%%%%%%%%%%%%%%%%%%%%%%%%%%%%%%%%%%%%%%%%%%%%%%

\begin{proof}
    First, consider the case where $J = R$. Then, by definition, $E'_{\sigma}(R, J) = E'_\sigma (R)$. The result in this case follows from the fact that $E'_\sigma (R)$ is a normal subgroup of $G_{\sigma}(R)$ (see Corollary~\ref{cor:KS2}). 

Now suppose $J \subsetneq R$. Let $h \in G_\sigma (R)$ and $g \in E'_\sigma (R, J)$. 
We want to prove that $hgh^{-1} \in E'_\sigma (R,J)$. 
We consider the ring $R':= \{(r,s) \in R \times R \mid r-s \in J\}$ and its ideal $J':= \{ (r,0) \in R \times R \mid r \in J \}$.
The automorphism $\theta$ of the ring $R$ can be naturally induced to an automorphism of $R'$, and we denote it by the same letter $\theta$.
Therefore, the group $G_\sigma(R')$ makes sense.
Consider an element $h'':= (h,h)$ of the group $G_\sigma(R') \subset G_\sigma (R) \times G_{\sigma} (R).$ 
Observe that $E'_\sigma (R, J)$ is embedded into the group $E'_\sigma (R', J')$ by $x \mapsto x':=(x,1).$ (This can be seen as follows: There is a natural embedding from $E'_\sigma (J)$ into $E'_\sigma (J')$ given by $x \mapsto (x,1)$. Now any $y \in E'_\sigma(R,J)$ is can be written as a product of the form $\prod_{i=1}^n g_i x_i g_i^{-1}$, where $x_i \in E'_{\sigma}(J)$ and $g_i \in E'_\sigma(R)$. But then $(y,1) = (\prod_{i=1}^n g_i x_i g_i^{-1},1) = \prod_{i=1}^n (g_i,g_i) (x_i,1) (g_i,g_i)^{-1} \in E'_\sigma(R',J')$, as desired.)

Now we claim that $E'_\sigma (R', J') = E'_\sigma (R') \cap G_\sigma (J')$. Clearly, by definition, $E'_\sigma (R', J') \subset E'_\sigma (R') \cap G_\sigma (J')$. For converse, let $$ x = \prod_{i=1}^{n} x_{[\alpha_i]} (t_i) \in E'_\sigma(R') \cap G_\sigma (J')$$ where $t_i \in R'_{[\alpha_i]}$. 
For each $t_i \in R'_{[\alpha_i]}$, choose elements $s_i \in R'_{[\alpha_i]}$ and $u_i \in J'_{[\alpha_i]}$ as follows:
\begin{enumerate}
    \item If $[\alpha_i] \sim A_1$, $A_1^2$, or $A_1^3$ and $t_i = (\alpha_i, \beta_i)$, then set $s_i = (\beta_i, \beta_i)$ and $u_i = (\alpha_i - \beta_i, 0)$. It is clear that $t_i = s_i + u_i$.
    \item If $[\alpha_i] \sim A_2,$ $t_i = (\alpha_i, \beta_i) \in \mathcal{A}(R')$, $\alpha_i = (a_1, a_2) \in R'$ and $\beta_i = (b_1, b_2) \in R'$, then set $s_i = (\gamma_i, \delta_i) \in \mathcal{A}(R')$ and $u_i = (\mu_i, \nu_i)  \in \mathcal{A}(R')$, where $\gamma_i = (a_2,a_2) \in R', \delta_i = (b_2,b_2) \in R', \mu_i = (a_1 - a_2, 0) \in R'$ and $\nu_i = (b_1 - b_2 - \overline{a_2} (a_1 - a_2), 0) \in R'$. Clearly, $t_i = s_i \oplus u_i$.
\end{enumerate}
Set $$y_k = \prod_{i=1}^{k} x_{[\alpha_i]} (s_i) \in E'_\sigma(R')$$ for $0 \leq k \leq n.$ Clearly, $y_0 = 1$ (by the definition). 
We claim that $y_n = 1$. Since $x \in G_\sigma(J')$, we have $x \equiv 1$ (mod $J'$). But then $y_n \equiv 1$ (mod $J'$), that is, $\prod_{i=1}^n x_{[\alpha_i]}(s_i + J') = 1$ in $E'_\sigma (R'/J')$ (the notion of $s_i + J'$ is clear even when $[\alpha] \sim A_2$). Note that there is a natural embedding from $R'/J'$ to $R/J \times R$ which induces an embedding from the group $E'_\sigma (R'/J')$ to the group $E'_{\sigma} (R/J \times R) \cong E'_{\sigma}(R/J) \times E'_{\sigma}(R)$. 
Under this embedding, $$\prod_{i=1}^n (x_{[\alpha_i]}(\beta_i+J),x_{[\alpha_i]}(\beta_i)) = (\prod_{i=1}^n x_{[\alpha_i]}(\beta_i+J), \prod_{i=1}^n x_{[\alpha_i]}(\beta_i)) = (1,1)$$ in $E'_{\sigma}(R/J) \times E'_{\sigma}(R)$. In particular, $\prod_{i=1}^n x_{[\alpha_i]}(\beta_i) = 1$ in $E'_{\sigma}(R)$. Thus
\begin{align*}
    y_n = \prod_{i=1}^n x_{[\alpha_i]}(s_i) = \prod_{i=1}^n x_{[\alpha_i]}(\beta_i, \beta_i) = (\prod_{i=1}^n x_{[\alpha_i]}(\beta_i), \prod_{i=1}^n x_{[\alpha_i]}(\beta_i)) = (1,1) = 1.
\end{align*}
This proves our claim. Finally,
$$x = \prod_{i=1}^{n} x_{[\alpha_i]} (s_i) x_{[\alpha_i]} (u_i) = \prod_{i=1}^{n} y_{i-1}^{-1}y_i x_{[\alpha_i]} (u_i) = y_0^{-1} \Bigg( \prod_{i=1}^{n} y_i x_{[\alpha_i]} (u_i) y_{i}^{-1} \Bigg) y_n \in E'_{\sigma} (R', J'),$$ as desired.

Again by Corollary \ref{cor:KS2}, $E'_\sigma (R')$ is a normal subgroup of $G_\sigma (R')$. So $h'' g' (h'')^{-1} \in E'_\sigma (R'),$ where $g' = (g,1) \in E'_\sigma (R')$. On the other hand, $h'' g' (h'')^{-1} = (hgh^{-1}, 1) \in G_\sigma (J')$. Therefore $h''g'(h'')^{-1} \in E'_\sigma (R') \cap G_\sigma (J') = E'_\sigma (R', J'),$ hence $hgh^{-1} \in E'_\sigma (R, J)$. Thus $E'_\sigma (R, J)$ is normal in $G_\sigma(R)$. 
\end{proof}

We derive the following corollary from the proof of the above Proposition.

\begin{cor}\label{mixcom}
    $[E'_\sigma (R), G_\sigma (J)] \subset E'_\sigma (R,J)$.
\end{cor}

\begin{proof} 
    Take any $h \in E'_{\sigma} (R)$ and $g \in G_\sigma (J)$. Define, as in proof of Proposition~\ref{normal}, $h'' = (h,h) \in E'_{\sigma} (R')$ and $g' = (g,1) \in G_{\sigma} (J')$. Then $[h'',g'] \in E'_\sigma (R') \cap G_\sigma (J') = E'_\sigma (R', J')$ (as $E'_\sigma (R')$ and $G_\sigma (J')$ are normal subgroups of $G_\sigma (R')$). Since $[h'',g'] = ([h,g],1)$, we have $[h,g] \in E'_\sigma(R,J).$ Thus, $[E'_\sigma (R), G_\sigma (J)] \subset E'_\sigma (R,J)$.
\end{proof}

%%%%%%%%%%%%%%%%%%%%%%%%%%%%%%%%%%%%%%%%%%%%%%%%%%

\begin{prop}\label{genofE(R,I)}
    For any ideal $J$ of $R$, the subgroup $E'_\sigma (R,J)$ is generated by elements of the form $x_{[\alpha]}(r)x_{-[\alpha]}(u)x_{[\alpha]}(r)^{-1}$ with $[\alpha] \in \Phi_\rho, r \in R_{[\alpha]}$ and $u \in J_{[\alpha]}$.
\end{prop}

%%%%%%%%%%%%%%%%%%%%%%%%%%%%%%%%%%%%%%%%%%%%%%%%%%

\begin{proof}
    Let $H$ be the subgroup of $E'_\sigma (R,J)$ generated by all $x_{[\alpha]}(r) x_{-[\alpha]}(u) x_{[\alpha]}(r)^{-1}$, where $[\alpha] \in \Phi_\rho, r \in R_{[\alpha]},$ and $u \in J_{[\alpha]}$. We aim to prove that $H = E'_\sigma (R,J)$. Since $E'_\sigma (J) \subset H$, it suffices to show that $H$ is a normal subgroup of $E'_\sigma (R)$. To demonstrate this, we need to verify that 
    $$ g = x_{[\beta]}(s) x_{[\alpha]}(r) x_{-[\alpha]}(u) x_{[\alpha]}(r)^{-1} x_{[\beta]}(s)^{-1} \in H $$ for all $[\alpha], [\beta] \in \Phi_\rho, r \in R_{[\alpha]}, s \in R_{[\beta]}$, and $u \in J_{[\alpha]}$.

\vspace{2mm}

\noindent \textbf{Case A. $[\alpha] \neq \pm [\beta]$.} For $[\gamma], [\delta] (\neq -[\gamma]) \in \Phi_\rho$, we have $$[x_{[\gamma]}(a), x_{[\delta]}(b)] = \prod x_{i [\gamma] + j [\delta]} (f_{i,j}(a,b)),$$ where $f_{i,j}$ is function of $a$ and $b$ with the property that $f_{i,j}(a,b) \in J_{i[\gamma]+j[\delta]}$ if $a \in J_{[\gamma]}$ or $b \in J_{[\delta]}$. Since no convex combination of the roots $-[\alpha], [\beta]$ and $i [\alpha] + j [\beta] \ (i, j \neq 0)$ is $0$, we have 
\begin{align*}
    g &= x_{[\beta]}(s) x_{[\alpha]}(r) x_{-[\alpha]}(u) x_{[\alpha]}(r)^{-1} x_{[\beta]}(s)^{-1} \\
    &= x_{[\alpha]}(r) x_{-[\alpha]}(u) x_{[\alpha]}(r)^{-1} \Big( \prod x_{i[\alpha]+ j[\beta]}(h_{i,j}(s,t,u)) \Big) \in H,
\end{align*}
where $h_{i,j}$ is function of $s, t$ and $u$ such that $h_{i,j}(s,t,u) \in J_{i[\alpha]+j[\beta]}$.

\vspace{2mm}

\noindent \textbf{Case B. $[\alpha] = \pm [\beta]$.} Note that, if $[\alpha] = [\beta]$ then there is nothing to prove. Now assume that $[\alpha] = -[\beta]$. We then have 
$$ g = x_{-[\alpha]}(s) x_{[\alpha]}(r) x_{-[\alpha]}(u) x_{[\alpha]}(r)^{-1} x_{-[\alpha]}(s)^{-1}.$$ 

Since the rank of $\Phi_\rho > 1$, there exists $[\gamma] \in \Phi_\rho$ such that the subroot system $\Phi'$ generated by $[\alpha]$ and $[\gamma]$ is connected of rank $2$. WLOG, we can assume that $[\alpha], [\gamma]$ is base of $\Phi'$. Set $\Phi'_+$ be the set of positive roots of $\Phi'$ with respect to this base, $\Phi'_- = - \Phi'_+, \Phi''_+ = \{ i [\alpha] + j [\gamma] \in \Phi'_+ \mid j>0 \},$ and $\Phi''_- = - \Phi''_+$. Write $U''_+ (J)$ (resp., $U''_- (J)$) for the subgroup of $E'_\sigma (R)$ generated by $x_{[\delta]}(t)$ with $[\delta] \in \Phi''_+$ (resp., $[\delta] \in \Phi''_-$) and $t \in J_{[\delta]}$. Then $U''_+(J)$ and $U''_-(J)$ are subgroups of $H$. 

Now, by Lemma~\ref{lemma:u_1,u_2 exists} (below), for given $u \in J_{[\alpha]}$ we can find $u_1 \in J_{[\alpha] + [\gamma]}$ and $u_2 \in R_{[\gamma]}$ such that $$x_{-[\alpha]}(u) = [x_{-([\alpha] + [\gamma])}(u_1), x_{[\gamma]}(u_2)] h'$$ with $h' \in U''_{-}(J)$. Set \begin{align*}
    g_1 &:= x_{-[\alpha]}(s) x_{[\alpha]}(r) x_{-([\alpha]+[\gamma])}(u_1) x_{[\alpha]}(r)^{-1} x_{-[\alpha]}(s)^{-1} \in U''_-(J), \\
    g_2 &:= x_{-[\alpha]}(s) x_{[\alpha]}(r) x_{[\gamma]}(u_2) x_{[\alpha]}(r)^{-1} x_{-[\alpha]}(s)^{-1} \in U''_+(R), \\
    g_3 &:= x_{-[\alpha]}(s) x_{[\alpha]}(r) h' x_{[\alpha]}(r)^{-1} x_{-[\alpha]}(s)^{-1} \in U''_-(J).
\end{align*}
Then $g = [g_1,g_2] g_3$, which contained in $H$ by Lemma \ref{lemma:H} (below). 
\end{proof}

\begin{lemma}\label{lemma:u_1,u_2 exists}
    For given $u \in J_{[\alpha]}$ we can find $u_1 \in J_{[\alpha] + [\gamma]}$ and $u_2 \in R_{[\gamma]}$ such that $$x_{-[\alpha]}(u) = [x_{-([\alpha] + [\gamma])}(u_1), x_{[\gamma]}(u_2)] h'$$ with $h' \in U''_-(J)$.
\end{lemma}

\begin{proof}
    The Chevalley commutator formula for $[x_{-([\alpha]+[\gamma])}(u_1), x_{[\gamma]}(u_2)]$ suggests that depending on the types of the pair of roots $(-[\alpha] - [\gamma], [\gamma])$ we can choose $u_1$ and $u_2$ as follows:
    \begin{center}
        \begin{tabular}{ccc}
            Type of pair $(-[\alpha] - [\gamma], [\gamma])$ & $u_1$ & $u_2$  \\
            \hline
            $(b-i)$ & $u$ & $\pm1$ \\
            $(b-ii)$ & $u$ or $\bar{u}$ & $\pm 1$ \\
            $(c-i)$ & $u$ & $\pm 1/2$ \\
            $(c-ii)$ & $(u, u \bar{u}/2)$ or $(\bar{u}, u \bar{u}/2)$ & $(\pm 1, 1/2)$ \\
            $(d-i)$ & $u$ & $\pm 1$ \\
            $(d-ii)$ & $(a, b):=u$ or $(\bar{a},b)$ or $(a, \bar{b})$ & $\pm 1$ \\
            $(e)$ & $u$ & $\pm 1$ \\
            $(g)$ & $u$ & $\pm 1/3$
        \end{tabular}
    \end{center}
    Note that each $u_1 \in J_{[\alpha]+[\gamma]}$ and $u_2 \in R_{[\gamma]}$. An immediate observation from Chevalley commutator formula for $[x_{-([\alpha]+[\gamma])}(u_1), x_{[\gamma]}(u_2)]$ is that $h' \in U''_-(J)$.
\end{proof}

\begin{lemma}\label{lemma:H}
    $[U''_-(J),U''_+(R)] \subset H$.
\end{lemma}

\begin{proof}
    Let $h \in U''_- (J)$ and $g \in U''_+ (R)$. Write $$ h = x_{-[\alpha_1]}(u_1) \dots x_{-[\alpha_n]}(u_n),$$ where $[\alpha_i] \in \Phi''_+$ and $0 \neq u_i \in J_{[\alpha_i]}$. We want to show that $[h,g] \in H.$ For that we use induction of $n$. If $n=1$ then $[h,g] = x_{-[\alpha_1]}(u_1) g x_{-[\alpha_1]}(u_1)^{-1} g^{-1}.$ Write $g = x_{[\beta_1]}(v_1) \dots x_{[\beta_m]}(v_m)$ where $[\beta_i] \in \Phi''_+$ and $0 \neq v_i \in R_{[\beta_i]}$. If $[\beta_i] \neq -[\alpha_1]$ for every $i = 1, \dots, m$, then we are done by Chevalley commutator relations. If there is some $i \in \{1, \dots, m\}$ such that $[\beta_i] = [\alpha_1]$ then also we are done by the definition of $H$ and Chevalley commutator relations. 

    Now for general $n$,
    \begin{align*}
        [h,g] &= h g h^{-1} g^{-1} \\
        &= x_{-[\alpha_1]}(u_1) [x_{-[\alpha_2]}(u_2) \dots x_{-[\alpha_n]}(u_n), g] [g, x_{-[\alpha_1]}(u_1)^{-1}] x_{-[\alpha_1]}(u_1)^{-1} \\
        &\in H.
    \end{align*}
    Which proves the lemma.
\end{proof}

%%%%%%%%%%%%%%%%%%%%%%%%%%%%%%%%%%%%%%%%%%%%%%%%%%

\begin{thm}\label{Ch5_mainthm1}
    \normalfont
    For any ideal $J$ of $R$, we have 
    $$E'_\sigma(R, J) = [E'_\sigma (R), E'_\sigma (J)] = [E'_\sigma (R), G_\sigma (R,J)] = [G_\sigma (R), E'_\sigma (R,J)] .$$
\end{thm}

% \begin{thm}[Main Theorem 1]\label{Ch5_mainthm1}
%     Let $\Phi_\rho$ be one of the following types: ${}^2 A_n \ (n \geq 3), {}^2 D_n \ (n \geq 4), {}^2 E_6$, or ${}^3 D_4$. Assume that $1/2 \in R$ and in addition, $1/3 \in R$ if $\Phi_\rho \sim {}^3 D_4$.
%     Let $J$ be a $\theta$-invariant ideal of $R$. Then 
%     \begin{equation*}
%         \begin{split}
%             E'_{\pi,\sigma} (\Phi, R, J) = [E'_{\pi, \sigma} (\Phi, R), E'_{\pi,\sigma}(\Phi, J)] = [E'_{\pi,\sigma} (\Phi, R), G_{\pi,\sigma} (\Phi, R, J)] \\
%             = [G_{\pi,\sigma} (\Phi, R), E'_{\pi,\sigma} (\Phi, R, J)].
%         \end{split}
%     \end{equation*}
% \end{thm}

%%%%%%%%%%%%%%%%%%%%%%%%%%%%%%%%%%%%%%%%%%%%%%%%%%

\begin{proof} 
Note that 
\[
    [E'_\sigma(R), E'_\sigma (J)] \subset [E'_\sigma(R), G_\sigma (R, J)] \quad \text{and} \quad [E'_\sigma(R), E'_\sigma (J)] \subset [G_\sigma(R), E'_\sigma (R,J)].
\]
Also, by Proposition~\ref{normal}, we have $[G_\sigma (R), E'_\sigma (R,J)] \subset E'_\sigma (R,J)$. 
Therefore, to prove our proposition, it is enough to prove the following:
\begin{enumerate}[(i)]
    \item $E'_\sigma(R, J) \subset [E'_\sigma(R), E'_\sigma (J)].$
    \item $[E'_\sigma(R), G_\sigma (R, J)] \subset E'_\sigma(R, J).$
\end{enumerate}

Since $H:= [E'_\sigma(R), E'_\sigma (J)]$ is normal in $E'_\sigma(R)$, to prove (i) it is enough to prove that $x_{[\alpha]}(u) \in H$ for every $[\alpha] \in \Phi_\rho$ and $u \in J_{[\alpha]}$. As in the proof of Proposition \ref{genofE(R,I)}, since the rank of $\Phi_\rho > 1$, there exists $[\beta] \in \Phi_\rho$ such that the subsystem $\Phi'$ generated by $[\alpha]$ and $[\beta]$ is connected of rank $2$. WLOG, we can assume that $[\alpha], [\beta]$ is base of $\Phi'$.

\vspace{2mm}

\noindent \textbf{Case A. $\Phi' \sim A_2$.} In this case, the pair of roots $[\alpha] + [\beta]$ and $-[\beta]$ are of type $(b)$. 

\begin{center}
    \begin{tikzpicture}
        \draw[<->, line width=1pt] (2,0)--(-2,0);
        \draw[<->, line width=1pt] (1,1.73)--(-1,-1.73);
        \draw[<->, line width=1pt] (-1,1.73)--(1,-1.73);
        \node at (2.5,0) {$[\alpha]$};
        \node at (-1.4,2.0) {$[\beta]$};
        \node at (1.4,2.0) {$[\alpha] + [\beta]$};
        \node at (-2.5,0) {$-[\alpha]$};
        \node at (1.4, -2.0) {$-[\beta]$};
        \node at (-1.4, -2.0) {$-[\alpha] - [\beta]$};
    \end{tikzpicture}
\end{center}

\begin{enumerate}[leftmargin=4em]
    \item[$(b-i)$] If $[\alpha] + [\beta] \sim A_1$ and $-[\beta] \sim A_1$, then so is $[\alpha] = ([\alpha] + [\beta]) + (-[\beta])$. In this case, for given $u \in J_{[\alpha]} = J_{\theta}$ we have $$ x_{[\alpha]}(u) = [x_{[\alpha]+[\beta]}(\pm 1), x_{-[\beta]}(u)] \in H.$$
    \item[$(b-ii)$] If $[\alpha] + [\beta] \sim A_1^2$ and $-[\beta] \sim A_1^2$, then so is $[\alpha] = ([\alpha] + [\beta]) + (-[\beta])$. In this case, for given $u \in J_{[\alpha]} = J$ we write $u'= u$ or $\bar{u} \in J$. Then we have $$ x_{[\alpha]}(u) = [x_{[\alpha]+[\beta]}(\pm 1), x_{-[\beta]}(u')] \in H.$$
\end{enumerate}

\vspace{2mm}

\noindent \textbf{Case B. $\Phi' \sim B_2$ and $[\alpha]$ is a long root.} In this case, the pair of roots $[\alpha] + [\beta]$ and $-[\beta]$ are of type $(c)$. 

\begin{center}
    \begin{tikzpicture}
        \draw[<->, line width=1pt] (1,1)--(-1,-1);
        \draw[<->, line width=1pt] (-1,1)--(1,-1);
        \draw[<->, line width=1pt] (2,0)--(-2,0);
        \draw[<->, line width=1pt] (0,2)--(0,-2);
        \node at (2.5,0) {$[\alpha]$};
        \node at (-1.4,1.4) {$[\beta]$};
        \node at (1.6,1.4) {$[\alpha] + [\beta]$};
        \node at (0,2.4) {$[\alpha] + 2 [\beta]$};
        \node at (-2.5,0) {$-[\alpha]$};
        \node at (1.4,-1.4) {$-[\beta]$};
        \node at (-1.6,-1.4) {$-[\alpha] - [\beta]$};
        \node at (0,-2.4) {-$[\alpha] - 2 [\beta]$};
    \end{tikzpicture}
\end{center}

\begin{enumerate}[leftmargin=4em]
    \item[$(c-i)$] If $[\alpha] + [\beta] \sim A_1^2$ and $-[\beta] \sim A_1^2$, then $[\alpha] = ([\alpha] + [\beta]) + (-[\beta]) \sim A_1$. In this case, for given $u \in J_{[\alpha]} = J_{\theta}$ we have $$ x_{[\alpha]}(u) = [x_{[\alpha]+[\beta]}(\pm 1/2), x_{-[\beta]}(u)] \in H.$$
    \item[$(c-ii)$] If $[\alpha] + [\beta] \sim A_2$ and $-[\beta] \sim A_2$, then $[\alpha] = ([\alpha] + [\beta]) + (-[\beta]) \sim A_1^2$. In this case, for given $u \in J_{[\alpha]} = J$ we write $u'= u$ or $\bar{u} \in J$. Then we have $$ x_{[\alpha]}(u) = [x_{[\alpha]+[\beta]}(\pm 1, 1/2), x_{-[\beta]}(u', u \bar{u} / 2)] \in H.$$
\end{enumerate}

\vspace{2mm}

\noindent \textbf{Case C. $\Phi' \sim B_2$ and $[\alpha]$ is a short root.} In this case, the pair of roots $[\alpha] + [\beta]$ and $-[\beta]$ are of type $(d)$ with $[\alpha] + [\beta]$ being the short root.

\begin{center}
    \begin{tikzpicture}
        \draw[<->, line width=1pt] (1.5,1.5)--(-1.5,-1.5);
        \draw[<->, line width=1pt] (-1.5,1.5)--(1.5,-1.5);
        \draw[<->, line width=1pt] (1.5,0)--(-1.5,0);
        \draw[<->, line width=1pt] (0,1.5)--(0,-1.5);
        \node at (2,0) {$[\alpha]$};
        \node at (-1.9,1.9) {$[\beta]$};
        \node at (1.9,1.9) {$2[\alpha] + [\beta]$};
        \node at (0,2) {$[\alpha] + [\beta]$};
        \node at (-2,0) {$-[\alpha]$};
        \node at (1.9,-1.9) {$-[\beta]$};
        \node at (-2.1,-1.9) {$-2[\alpha] - [\beta]$};
        \node at (0,-2) {$-[\alpha] - [\beta]$};
    \end{tikzpicture}
\end{center}

\begin{enumerate}[leftmargin=4em]
    \item[$(d-i)$] If $[\alpha] + [\beta] \sim A_1^2$ and $-[\beta] \sim A_1$, then $[\alpha] = ([\alpha] + [\beta]) + (-[\beta]) \sim A_1^2$ and $2[\alpha] + [\beta] = 2([\alpha] + [\beta]) + (-[\beta]) \sim A_1$. In this case, for given $u \in J_{[\alpha]} = J$ we have 
    \begin{align*}
        x_{[\alpha]}(u) x_{2[\alpha] + [\beta]}(u) &= [x_{-[\beta]}(\pm u), x_{[\alpha]+[\beta]}(\pm 1)] \\
        &= [x_{[\alpha]+[\beta]}(\pm 1), x_{-[\beta]}(\pm u)]^{-1} \in H.
    \end{align*}
    Now observe that $[\alpha] \sim A_1^2, [\alpha] + [\beta] \sim A_1^2$ and $2[\alpha] + [\beta] \sim A_1$. Then by similar argument as in $(c-i)$ above, we can conclude that $x_{2[\alpha]+[\beta]}(u) \in H$. Hence $$ x_{[\alpha]}(u) = (x_{[\alpha]}(u) x_{2[\alpha] + [\beta]}(u)) (x_{2[\alpha] + [\beta]}(u))^{-1} \in H.$$
    \item[$(d-ii)$] If $[\alpha] + [\beta] \sim A_2$ and $-[\beta] \sim A_1^2$, then $[\alpha] = ([\alpha] + [\beta]) + (-[\beta]) \sim A_2$ and $2[\alpha] + [\beta] = 2([\alpha] + [\beta]) + (-[\beta]) \sim A_1^2$. In this case, for given $u = (u_1, u_2) \in J_{[\alpha]} = \mathcal{A}(J) = \mathcal{J}$ we have 
    \begin{align*}
        x_{[\alpha]}(u_1,u_2) x_{2[\alpha] + [\beta]}(\pm u_2) &= [x_{-[\beta]}(\pm 1), x_{[\alpha]+[\beta]}(u_1, u_2)] \\
        &= [x_{[\alpha]+[\beta]}(u_1, u_2), x_{-[\beta]}(\pm 1)]^{-1} \in H.
    \end{align*}
    Now observe that $[\alpha] \sim A_2, [\alpha] + [\beta] \sim A_2$ and $2[\alpha] + [\beta] \sim A_1^2$. Then by similar argument as in $(c-ii)$ above, we can conclude that $x_{2[\alpha]+[\beta]}(\pm u_2) \in H$. Hence $$ x_{[\alpha]}(u_1,u_2) = (x_{[\alpha]}(u_1, u_2) x_{2[\alpha] + [\beta]}(\pm u_2)) (x_{2[\alpha] + [\beta]}(\pm u_2))^{-1} \in H.$$
\end{enumerate}

\vspace{2mm}

\noindent \textbf{Case D. $\Phi' \sim G_2$ and $[\alpha]$ is a long root.} In this case, we consider a subroot system $\Phi''$ of $\Phi'$ generated by roots $[\alpha]$ and $[\alpha] + 3 [\beta]$. Note that $\Phi'' \sim A_2$ and hence, by case 1 (replace $[\beta]$ by $[\alpha] + 3 [\beta]$), we can conclude that $x_{[\alpha]}(u) \in H$.

\begin{center}
    \begin{tikzpicture}
        \draw[<->, line width=1pt] (2.5,0)--(-2.5,0);
        \draw[<->, line width=1pt] (1.25,0.72)--(-1.25,-0.72);
        \draw[<->, line width=1pt] (1.25,2.16)--(-1.25,-2.16);
        \draw[<->, line width=1pt] (0,1.44)--(0,-1.44);
        \draw[<->, line width=1pt] (-1.25,2.16)--(1.25,-2.16);
        \draw[<->, line width=1pt] (-1.25,0.72)--(1.25,-0.72);
        \node at (3,0) {$[\alpha]$};
        \node at (-1.65,0.72) {$[\beta]$};
        \node at (2.05,0.72) {$[\alpha] + [\beta]$};
        \node at (2.05,2.56) {$2[\alpha] + 3 [\beta]$};
        \node at (0,2) {$[\alpha] + 2 [\beta]$};
        \node at (-2.05,2.56) {$[\alpha] + 3 [\beta]$};
        \node at (-3,0) {$-[\alpha]$};
        \node at (1.65,-0.8) {$-[\beta]$};
        \node at (-2.25,-0.8) {$-[\alpha] - [\beta]$};
        \node at (-2.05,-2.56) {$-2[\alpha] - 3 [\beta]$};
        \node at (0,-2) {$-[\alpha] - 2 [\beta]$};
        \node at (2.05,-2.56) {$-[\alpha] - 3 [\beta]$};
    \end{tikzpicture}
\end{center}

\vspace{2mm}

\noindent \textbf{Case E. $\Phi' \sim G_2$ and $[\alpha]$ is a short root.} In this case, the pair of roots $2[\alpha] + [\beta]$ and $-[\alpha] - [\beta]$ are of type $(f)$.

\begin{center}
    \begin{tikzpicture}
        \draw[<->, line width=1pt] (1.5,0)--(-1.5,0);
        \draw[<->, line width=1pt] (2.6,1.5)--(-2.6,-1.5);
        \draw[<->, line width=1pt] (0.75,1.3)--(-0.75,-1.3);
        \draw[<->, line width=1pt] (0,2.6)--(0,-2.6);
        \draw[<->, line width=1pt] (-0.75,1.3)--(0.75,-1.3);
        \draw[<->, line width=1pt] (-2.6,1.5)--(2.6,-1.5);
        \node at (2,0) {$[\alpha]$};
        \node at (-3,1.5) {$[\beta]$};
        \node at (-0.95,1.65) {$[\alpha] + [\beta]$};
        \node at (0.95,1.65) {$2[\alpha] + [\beta]$};
        \node at (3.4,1.7) {$3[\alpha] + [\beta]$};
        \node at (0,3) {$3[\alpha] + 2 [\beta]$};
        \node at (-2,0) {$-[\alpha]$};
        \node at (3,-1.8) {$-[\beta]$};
        \node at (1,-1.65) {$-[\alpha] - [\beta]$};
        \node at (-1.2,-1.65) {$-2[\alpha] - [\beta]$};
        \node at (-3.4,-1.9) {$-3[\alpha] - [\beta]$};
        \node at (0,-3) {$-3[\alpha] - 2 [\beta]$};
    \end{tikzpicture}
\end{center}

Observe that $[\alpha], 2[\alpha] + [\beta], - [\alpha] - [\beta] \sim A_1^3$ and $-[\beta], 3[\alpha] + [\beta] \sim A_1$. For given $u \in J_{[\alpha]} = J$ we write $(u',u'') = (\bar{u}, \bar{\bar{u}})$ or $(\bar{\bar{u}}, \bar{u})$. Then we have 
\begin{gather*}
    x_{[\alpha]}(u) x_{3[\alpha] + [\beta]}(\pm (u^2 + (u')^2 + (u'')^2 - 2 u u' - 2 u' u'' - 2 uu'')/4) x_{-[\beta]}(\pm (u + u' + u'')/2) \\
    = [x_{2[\alpha]+[\beta]}((u + u' - u'')/2), x_{-[\alpha]-[\beta]}(\pm 1)] \in H.
\end{gather*}

Now observe that $- [\beta] \sim A_1$ and $3[\alpha] + [\beta] \sim A_1$. Then by similar argument as in Case A above, we can conclude that $x_{3[\alpha]+[\beta]}(\pm (u^2 + (u')^2 + (u'')^2 - 2 u u' - 2 u' u'' - 2 uu'')/4) \in H$ and $x_{-[\beta]}(\pm (u + u' + u'')^2/4) \in H$. Hence 
\begin{gather*}
    x_{[\alpha]}(u) = (x_{[\alpha]}(u) x_{3[\alpha] + [\beta]}(\pm (u^2 + (u')^2 + (u'')^2 - 2 u u' - 2 u' u'' - 2 uu'')/4) \\ 
    x_{-[\beta]}(\pm (u + u' + u'')/2)) (x_{3[\alpha] + [\beta]}(\pm (u^2 + (u')^2 + (u'')^2 - 2 u u' - 2 u' u'' - 2 uu'')/4) \\
    x_{-[\beta]}(\pm (u + u' + u'')/2))^{-1} \in H.
\end{gather*}

This proves part (i). Now for part (ii), we consider the groups $M:= E'_\sigma (R)$ and $N:= (E'_\sigma (R) \cap G_\sigma(J))/E'_{\sigma}(R, J)$. Observe that the group $M$ is perfect (put $J=R$ in part (i)) and the group $N$ is commutative (by Corollary~\ref{mixcom}). For a fixed $g \in G_\sigma(R,J)$, define a map $\psi_g: M \longrightarrow N$ given by $h \longmapsto [h,g]E'_\sigma(R,J)$. Then $\psi_g$ is a well-defined homomorphism from the perfect group $M$ to a commutative group $N$. Hence $\psi_g$ must be trivial, i.e., $[h,g] \in E'_\sigma (R,J)$ for all $h \in E'_\sigma(R).$ Thus, $[E'_\sigma(R), G_\sigma (R, J)] \subset E'_\sigma(R, J),$ as desired. 
\end{proof}

%%%%%%%%%%%%%%%%%%%%%%%%%%%%%%%%%%%%%%%%%%%%%%%%%%

\begin{cor}\label{cor:normalized}
    \normalfont
    The group $E'_\sigma (R)$ is perfect, that is, $[E'_\sigma (R),E'_\sigma (R)]=E'_\sigma (R)$.
\end{cor}

\begin{proof}
    Immediate by putting $J=R$ in the above proposition.
\end{proof}

%%%%%%%%%%%%%%%%%%%%%%%%%%%%%%%%%%%%%%%%%%%%%%%%%%

\begin{cor}\label{cor:converse}
    \normalfont
    Every subgroup of $G_\sigma (R, J)$ containing $E'_\sigma (R, J)$ is normalized by $E'_\sigma (R)$.
\end{cor}

\begin{proof}
    Let $H$ be a subgroup of $G_\sigma (R, J)$ containing $E'_\sigma (R, J)$. Then $$[E'_\sigma(R), E'_\sigma(J)] \subset [E'_\sigma(R), H] \subset [E'_\sigma(R), G_\sigma (R,J)].$$ By Theorem~\ref{Ch5_mainthm1}, we have $[E'_\sigma (R), H] = E'_\sigma (R,J) \subset H$. Therefore $H$ is normalized by $E'_\sigma (R)$.
\end{proof}

%%%%%%%%%%%%%%%%%%%%%%%%%%%%%%%%%%%%%%%%%%%%%%%%%%

\begin{cor}\label{C=G}
    \normalfont
    Let $C_\sigma (R,J) = \{ x \in G_\sigma (R) \mid [x, E'_\sigma (R)] \subset E'_\sigma (R,J) \}$. Then $$G_\sigma (R,J) = C_\sigma (R,J).$$
\end{cor}

\begin{proof}
    Clearly, by Theorem~\ref{Ch5_mainthm1}, $G_\sigma (R,J) \subset C_\sigma (R,J)$. By definition, we have $$G_\sigma (R,J) = \{ x \in G_\sigma (R) \mid [x, G_\sigma (R)] \subset G_\sigma (J) \}.$$ Let $Z(G)$ be the centre of $G = G_\sigma (R/J)$ and $C_G(E)$ the centralizer of $E = E'_\sigma (R/J)$ in $G$. Then, by Theorem \ref{thm:KS3}, we have $Z(G) = C_G(E)$. But then $$G_\sigma (R,J) = \{x \in G_\sigma (R) \mid [x, E'_\sigma (R)] \subset G_\sigma (J) \}.$$ 
    Since $E'_\sigma (R,J) \subset G_\sigma (J)$, we have $C_\sigma (R,J) \subset G_\sigma (R,J)$.
\end{proof}

%%%%%%%%%%%%%%%%%%%%%%%%%%%%%%%%%%%%%%%%%%%%%%%%%%
%Section: Subgroups of $G_\sigma (R)$ Normalized by $E'_\sigma (R)$
%%%%%%%%%%%%%%%%%%%%%%%%%%%%%%%%%%%%%%%%%%%%%%%%%%

\section{Subgroups of \texorpdfstring{$G_\sigma (R)$}{G(R)} Normalized by \texorpdfstring{$E'_\sigma (R)$}{E(R)}}\label{sec:main sec}

In this section, we state the main theorem of this chapter. Together with Corollary~\ref{cor:converse}, it provides a classification of all subgroups of $G_\sigma(R)$ that are normalized by $E'_\sigma(R)$.

\begin{thm}\label{Ch5_mainthm}
    Let $\Phi_\rho$ be one of the following types: ${}^2 A_n \ (n \geq 3), {}^2 D_n \ (n \geq 4), {}^2 E_6$ or ${}^3 D_4$. Assume that $1/2 \in R$, and in addition, $1/3 \in R$ if $\Phi_\rho \sim {}^3 D_4$. If $H$ is a subgroup of $G_{\pi, \sigma} (\Phi, R)$ normalized by $E'_{\pi, \sigma} (\Phi, R)$, then there exists a unique $\theta$-invariant ideal $J$ of $R$ such that 
    \[
        E'_{\pi, \sigma} (\Phi, R, J) \subset H \subset G_{\pi, \sigma} (\Phi, R, J).
    \]
\end{thm}

In the next section, we present the proof of the above theorem, assuming the validity of certain intermediate results. The rest of the chapter is dedicated to proving these key intermediate results.

\begin{cor}\label{cor:normal_subgroup_of_E}
    Let $R$ and $\Phi_\rho$ be as in Theorem~\ref{Ch5_mainthm}. Then a subgroup $H$ of $E'_{\pi, \sigma}(\Phi, R)$ is normal in $E'_{\pi, \sigma}(\Phi, R)$ if and only if there exists a $\theta$-invariant ideal $J$ of $R$ such that 
    \[
        E'_{\pi, \sigma}(\Phi, R, J) \subset H \subset G_{\pi, \sigma}(\Phi, R, J) \cap E'_{\pi, \sigma}(\Phi, R).
    \]
\end{cor}

\begin{proof}
    The result follows directly from Theorem~\ref{Ch5_mainthm} and Corollary~\ref{cor:converse}.
\end{proof}

%%%%%%%%%%%%%%%%%%%%%%%%%%%%%%%%%%%%%%%%%%%%%%%%%%

As an application of the above theorem, we prove the following result. 

\begin{thm}[P. Gvozdevsky, see Appendix of \cite{SG&DM1}]
    \normalfont
    Let $R$ and $\Phi_\rho$ be as in Theorem~\ref{mainthm1}. 
	Let $H$ be a subgroup of $G_{\pi, \sigma} (\Phi, R)$ containing $E'_{\pi, \sigma} (\Phi, R)$.
    Then $E'_{\pi, \sigma} (\Phi, R)$ can be characterized as the smallest by inclusion among all the subgroups $K\le H$ that satisfy the following properties:
    \begin{enumerate}[(a)]
        \item $K$ is normal and is generated as a normal subgroup by a single element;
        \item $K=[K,K]$;
        \item the centralizer of $K$ in $H$ is abelian.
    \end{enumerate}
\end{thm}

\begin{proof}
    First let us show that $E'_{\pi, \sigma} (\Phi, R)$ satisfies the properties (a)--(c). It satisfies (a) by Proposition~\ref{z to Rz} applied to $z=1$; it satisfies (b) by Corollary~\ref{cor:normalized}; and it satisfies (c) by Theorem~\ref{thm:KS3}.
	
	Now let $K\le H$ be a subgroup that satisfies (a)--(c); we must prove that $E'_{\pi, \sigma} (\Phi, R)\le K$. Since by (a) $K$ is normal in $H$, it follows by Theorem~\ref{mainthm} that there exists a unique $\theta$-invariant ideal $J$ of $R$ such that 
	\[
	E'_\sigma(R, J) \subset K \subset G_\sigma(R, J).
	\]
	
	We claim that $J = R$. By (a) $K$ is generated as a normal subgroup by a single element $g_0$. Now, clearly $J$ is the smallest by inclusion $\theta$-invariant ideal such that $g_0\in G_{\sigma}(R,J)$; hence, $J$ is generated by all the entries of the matrices $\varpi_{ad}(g_0)-e$, $\overline{\varpi_{ad}(g_0)}-e$, and $\overline{\overline{\varpi_{ad}(g_0)}}-e$, where $\varpi_{ad}$ is the adjoint representation of the ambient Chevalley group, and $e$ is the identity matrix. Therefore, the ideal $J$ is finitely generated. Since by (b) the group $K$ is perfect, the uniqueness of the ideal $J$ implies that $J = JJ$. By Nakayama's Lemma, there exists $s \in R$ such that $s \equiv 1 \pmod{J}$ and $sJ = 0$. 
	Thus, $E'_\sigma(sR)$ is contained in the centralizer of $K$ in $H$; hence, (c) implies that $s=0$; hence, we have $J = R$; hence, we have  $E'_\sigma(R) \subset K$.
\end{proof}

\begin{cor}\label{char subgrp}
    \normalfont
	Let $R$ and $\Phi_\rho$ be as in Theorem~\ref{Ch5_mainthm}. 
	Let $H$ be a subgroup of $G_{\pi, \sigma} (\Phi, R)$ containing $E'_{\pi, \sigma} (\Phi, R)$. 
    Then $E'_{\pi, \sigma} (\Phi, R)$ is a characteristic subgroup of $H$. 
    In particular, $E'_{\pi, \sigma} (\Phi, R)$ is a characteristic subgroup of $G_{\pi, \sigma} (\Phi, R)$.
\end{cor}

\begin{proof}
    Let $\phi: H \longrightarrow H$ be an automorphism of the group $H$. By Proposition~\ref{normal}, $E'_\sigma(R)$ is a normal subgroup of $H$, and hence $\phi(E'_\sigma(R))$ is normal in $\phi(H) = H$. Using a similar argument as above, one can show that $\phi(E'_\sigma(R))$ satisfies conditions (a)--(c) of the previous theorem and is a minimal subgroup of $H$ with these properties. Hence $\phi(E'_\sigma(R)) = E'_\sigma(R)$, as desired.
\end{proof}

\begin{rmk}
	The Theorem above implies not only that $E'_\sigma(R)$ is characteristic, but also that any abstract isomorphism between (possibly different) twisted Chevalley groups must preserve elementary subgroup. 
\end{rmk}

%%%%%%%%%%%%%%%%%%%%%%%%%%%%%%%%%%%%%%%%%%%%%%%%%%
%Section: Proof of Theorem~\ref{Ch5_mainthm}
%%%%%%%%%%%%%%%%%%%%%%%%%%%%%%%%%%%%%%%%%%%%%%%%%%

\section{Proof of Theorem~\ref{Ch5_mainthm}}\label{sec:Pf of main thm}

Let $H$ be a subgroup of $G_\sigma (R)$ normalized by $E'_\sigma (R)$. For $[\alpha] \in \Phi_\rho$, we write 
\[
    J_{[\alpha]} (H) = \begin{cases}
        \{ t \in R \mid x_{[\alpha]} (t) \in H \} & \text{if } [\alpha] \sim A_1, A_1^2, A_1^3; \\
        \{ t \in R \mid \exists \ u \in R \text{ with } x_{[\alpha]} (t, u) \in H \text{ or } x_{[\alpha]} (u,t) \in H \} & \text{if } [\alpha] \sim A_2. 
    \end{cases}
\]
Define $J = \displaystyle\bigcup_{[\alpha] \in \Phi_\rho} J_{[\alpha]}(H)$. To demonstrate the main theorem, we will first consider the following two propositions.
%%%%%%%%%%%%%%%%%%%%%%%%%%%%%%%%%%%%%%%%%%%%%%%%%%

\begin{prop}\label{prop:E(R,J) subset of H}
    \normalfont
    Let $J$ be as above. Then
    \begin{enumerate}[(a)]
        \item $J$ is a $\theta$-invariant ideal of $R$.
        \item $E'_\sigma (R, J) \subset H$. 
    \end{enumerate}
\end{prop}

%%%%%%%%%%%%%%%%%%%%%%%%%%%%%%%%%%%%%%%%%%%%%%%%%%

\begin{prop}\label{prop:U(R) cap H subset U(J)}
    \normalfont
    Let $J$ be as above. Then
    \begin{enumerate}[(a)]
        \item $U_\sigma (R) \cap H \subset U_\sigma (J)$. 
        \item $U_\sigma (rad (R)) T_\sigma (R) U^{-}_\sigma (R) \cap H \subset U_\sigma (J) T_\sigma (R, J) U^{-}_\sigma (J)$, where $rad(R)$ is the Jacobson radical of $R$.
    \end{enumerate}
\end{prop}

%%%%%%%%%%%%%%%%%%%%%%%%%%%%%%%%%%%%%%%%%%%%%%%%%%

The proofs of Propositions \ref{prop:E(R,J) subset of H} and \ref{prop:U(R) cap H subset U(J)} can be found in Sections \ref{sec:Pf of prop 1} and \ref{sec:Pf of prop 2}, respectively.
Moving forward, let $\mathfrak{m}$ be a maximal ideal of $R$ and define $\bar{\mathfrak{m}} = \theta(\mathfrak{m})$. Set
\[
    S_{\mathfrak{m}} = \begin{cases} 
        R \setminus \mathfrak{m} & \text{if } \mathfrak{m} = \bar{\mathfrak{m}}, \\
        R \setminus (\mathfrak{m} \cup \bar{\mathfrak{m}}) & \text{if } \mathfrak{m} \neq \bar{\mathfrak{m}} \text{ and } o(\theta) = 2, \\
        R \setminus (\mathfrak{m} \cup \bar{\mathfrak{m}} \cup \bar{\bar{\mathfrak{m}}}) & \text{if } \mathfrak{m} \neq \bar{\mathfrak{m}} \text{ and } o(\theta) = 3.
    \end{cases}
\]
Then $S_\mathfrak{m}$ is a multiplicatively closed subset of $R$ such that $\theta (S_\mathfrak{m}) \subset S_\mathfrak{m}$. Therefore there is a natural automorphism of the ring $S_\mathfrak{m}^{-1}R$ induced by an automorphism $\theta$ of $R$. We denote this automorphism of $S_\mathfrak{m}^{-1}R$ also by $\theta$. Let $\psi_\mathfrak{m}: G_\sigma (R) \longrightarrow G_\sigma (S_\mathfrak{m}^{-1} R)$ be the homomorphism of groups induced by the canonical homomorphism of rings from $R$ to $S_\mathfrak{m}^{-1}R$. We write $R_S$ and $J_S$ for $S_\mathfrak{m}^{-1} R$ and $S_\mathfrak{m}^{-1} J$, respectively. By $S_\theta$, we mean $S_\mathfrak{m} \cap R_\theta$.

%%%%%%%%%%%%%%%%%%%%%%%%%%%%%%%%%%%%%%%%%%%%%%%%%%

\begin{prop}\label{local:H subset G(R,J)}
    $\psi_\mathfrak{m} (H) \subset G_\sigma (R_S, J_S)$. 
\end{prop}

%%%%%%%%%%%%%%%%%%%%%%%%%%%%%%%%%%%%%%%%%%%%%%%%%%

The proof of Proposition~\ref{local:H subset G(R,J)} has been provided in Section \ref{sec:Pf of prop 3}.

%%%%%%%%%%%%%%%%%%%%%%%%%%%%%%%%%%%%%%%%%%%%%%%%%%

\begin{prop}\label{local:[x,g] in E(R,J)}
    For any element $g \in G_\sigma (R_S, J_S),$ there exists an elements $s \in S_\theta$ such that $$[\psi_\mathfrak{m}(x_{[\alpha]}(s \cdot t)), g] \in \psi_\mathfrak{m} (E'_\sigma(R, J))$$ for all $[\alpha] \in \Phi_\rho$ and $t \in  R_{[\alpha]}.$
\end{prop}

\begin{proof}
    Note that $R_S$ is a semi-local ring, and $J_S$ is either contained in $rad(R_S)$ or equal to $R_S$. 
    If $J_S \subset \operatorname{rad}(R_S)$, then by Proposition~\ref{lavidecomposition} and Corollary~\ref{G(R,J)=UTV}, we have $G_\sigma(R_S, J_S) = E'_\sigma(R_S, J_S)\, T_\sigma(R_S, J_S)$.
    If $J_S = R_S$, then by Proposition~\ref{G=G'}, it follows that $G_\sigma(R_S, J_S) = E'_\sigma(R_S, J_S)\, T_\sigma(R_S, J_S)$.
    In both cases, by Proposition \ref{genofE(R,I)}, we conclude that $G_\sigma(R_S, J_S)$ is generated by all $h(\chi) \in T_\sigma (R_S,J_S)$ and all the elements of the form 
    \[
    z_{[\beta]}(u,v) := x_{[\beta]}(u)\, x_{-[\beta]}(v)\, x_{[\beta]}(u)^{-1},
    \]
    where $[\beta] \in \Phi_\rho$, $u \in (R_S)_{[\beta]}$, and $v \in (J_S)_{[\beta]}$. We first prove the lemma for each of these generators.
    
    Let $g = h(\chi) \in T_\sigma (R_S, J_S)$ where $\chi \in \text{Hom}_1(\Lambda_\pi, R_S^{*})$. By Lemma \ref{lemma on T(J)}, we have $\chi(\alpha) \equiv 1$ (mod $J_S$) for every root $\alpha \in \Phi$, that is, $1 - \chi(\alpha) \in J_S$ for every root $\alpha \in \Phi$. Hence, there exists $u_\alpha \in J$ and $s_{\alpha} \in S$ such that $1 - \chi(\alpha) = \displaystyle\frac{u_{\alpha}}{s_{\alpha}}$. Take $s = \displaystyle\prod_{\alpha \in \Phi} s_{\alpha} \bar{s}_\alpha$ or $\displaystyle\prod_{\alpha \in \Phi} s_{\alpha} \bar{s}_\alpha \bar{\bar{s}}_\alpha$ depending on whether $o(\theta) = 2$ or $3$. Clearly, $s \in S_\theta$. 
    If $[\alpha] \sim A_2$ and $t = (t_1, t_2) \in R_{[\alpha]}$, then $t/1$ denotes $(t_1/1, t_2/1)$. Now 
    $$ [\psi_\mathfrak{m}(x_{[\alpha]}(s \cdot t)), g] = [x_{[\alpha]}(s/1 \cdot t/1), h(\chi)] = x_{[\alpha]} (s/1 \cdot t/1) x_{[\alpha]}(\chi(\alpha) s/1 \cdot t/1)^{-1}. $$
    If $[\alpha] \sim A_1, A_1^2$ or $A_1^3$ then 
    $$[\psi_\mathfrak{m}(x_{[\alpha]}(s \cdot t)), g] = x_{[\alpha]}\Big(\frac{st}{1}(1 - \chi(\alpha))\Big) = x_{[\alpha]}\Big(\frac{s t u_\alpha}{s_\alpha}\Big) = \psi_\mathfrak{m}(x_{[\alpha]}(a_{\alpha, t})),$$
    where $a_{\alpha, t} = \Bigg( \displaystyle\prod_{\substack{\beta \in \Phi \\ \beta \neq \alpha}} s_\beta \bar{s}_\beta \Bigg) \bar{s}_\alpha t u_\alpha $ or $\Bigg( \displaystyle\prod_{\substack{\beta \in \Phi \\ \beta \neq \alpha}} s_\beta \bar{s}_\beta \bar{\bar{s}}_\beta \Bigg) \bar{s}_\alpha \bar{\bar{s}}_\alpha t u_\alpha$ depending on whether $o(\theta) = 2$ or $3$. 
    Since $a_{\alpha, t} \in J_{[\alpha]}$, we obtain the required result in this case. Now if $[\alpha] \sim A_2$ then
    \begin{align*}
        [\psi_\mathfrak{m}(x_{[\alpha]}(s \cdot t)), g] &= x_{[\alpha]}\Bigg(\frac{st_1}{1}, \frac{s^2 t_2}{1}\Bigg) x_{[\alpha]} \Bigg(\chi(\alpha) \frac{st_1}{1}, \chi(\alpha) \overline{\chi(\alpha)} \frac{s^2 t_2}{1}\Bigg)^{-1} \\
        &= x_{[\alpha]}\Bigg((1 - \chi(\alpha))\frac{st_1}{1}, (1 - \chi(\alpha)) \frac{s^2t_2}{1} - \chi(\alpha) (1-\overline{\chi(\alpha)}) \frac{s^2 \bar{t}_2}{1} \Bigg) \\
        &= x_{[\alpha]}\Bigg(\frac{s t_1 u_\alpha}{s_\alpha}, \frac{s^2 t_2 u_\alpha}{s_\alpha} - \frac{s^2 \bar{t}_2 (s_\alpha - u_\alpha) \bar{u}_\alpha}{s_\alpha \bar{s}_\alpha} \Bigg) \\
        &= \psi_\mathfrak{m}(x_{[\alpha]}(a_{\alpha, t}, b_{\alpha, t})),
    \end{align*}
    where $a_{\alpha, t} = \Bigg( \displaystyle\prod_{\substack{\beta \in \Phi \\ \beta \neq \alpha}} s_\beta \bar{s}_\beta \Bigg) \bar{s}_\alpha t_1 u_\alpha$ and $b_{\alpha, t} = \Bigg( \displaystyle\prod_{\substack{\beta \in \Phi \\ \beta \neq \alpha}} s_\beta \bar{s}_\beta \Bigg) \bar{s}_\alpha s t_2 u_\alpha - \Bigg( \displaystyle\prod_{\substack{\beta \in \Phi \\ \beta \neq \alpha}} s_\beta \bar{s}_\beta \Bigg) s \bar{t}_2 (s_\alpha - u_\alpha) \bar{u}_\alpha$. Since $(a_{\alpha, t}, b_{\alpha, t}) \in \mathcal{A}(J)$, we obtain the required result in this case.

    \vspace{2mm}
    
    Now let $g = z_{[\beta]}(u, v)$, where $[\beta] \in \Phi_\rho, u \in (R_S)_{[\beta]}$ and $v \in (J_S)_{[\beta]}$. We want to find $s \in S_\theta$ such that 
    $$ [x_{[\alpha]}(s/1 \cdot t/1), z_{[\beta]}(u, v)] \in \psi_\mathfrak{m}(E'_\sigma(R, J)),$$ for all $[\alpha] \in \Phi_\rho$ and $t \in R_{[\alpha]}$.
    If $[\beta] \sim A_1, A_1^2$ or $A_1^3$, then write $u = a/b$ and $v = c/d$, where $a \in R, c \in J$ and $b, d \in S_\mathfrak{m}$. If $[\beta] \sim A_2$, then write $u = (u_1, u_2) = (a_1/b_1, a_2/b_2)$ and $v = (v_1, v_2) = (c_1/d_1, c_2/d_2)$, where $a_1, a_2 \in R,$ $c_1, c_2 \in J$ and $b_1, b_2, d_1, d_2 \in S_\mathfrak{m}$.
    Depending on the type of root $[\beta]$, we can choose the value of $s$ as below:
    \[
        s = \begin{cases}
            (b d)^m & \text{if } [\beta] \sim A_1, \\
            (b \bar{b} d \bar{d})^m & \text{if } [\beta] \sim A_1^2, \\
            (b \bar{b} \bar{\bar{b}} d \bar{d} \bar{\bar{d}})^m & \text{if } [\beta] \sim A_1^3, \\
            (b_1 \bar{b}_1 b_2 \bar{b}_2 d_1 \bar{d}_1 d_2 \bar{d}_2)^m & \text{if } [\beta] \sim A_2, 
        \end{cases}
    \]
    for sufficiently large positive integer $m$. We proceed using a similar method as in Proposition~\ref{genofE(R,I)}. 

    \vspace{2mm}

    \noindent \textbf{Case A. $[\alpha] \neq \pm [\beta]$.} Since no convex combination of the roots $[\alpha], -[\beta]$ and $i [\alpha] + j [\beta] \ (i, j \neq 0)$ is $0$, by using Chevalley commutator formula, we obtain 
    \begin{align*}
        & \hspace{-5mm} [x_{[\alpha]}(s/1 \cdot t/1), z_{[\beta]}(u,v)] \\
        &= x_{[\alpha]}(s/1 \cdot t/1) x_{[\beta]}(u) x_{-[\beta]}(v) x_{[\beta]}(u)^{-1} x_{[\alpha]}(s/1 \cdot t/1)^{-1} x_{[\beta]}(u) x_{-[\beta]}(v) x_{[\beta]}(u)^{-1} \\
        &= \Big( \prod x_{i[\alpha]+ j[\beta]}(c_{i,j}(s, t, u, v)) \Big),
    \end{align*}
    where $c_{i,j} = c_{i,j} (s, t, u, v)$ are functions of $s, t, u$ and $v$ such that $s$ and $v$ (or $\bar{v}, v_1, \bar{v}_1, v_2, \bar{v}_2$; the last four candidates appear only when $v=(v_1, v_2) \in \mathcal{A}(R)$) appear in each term of $c_{i,j}$. But then, for sufficiently large $m$ (see the definition of $s$), the values of $c_{i,j}(s,t,u,v)$ are in $J_{i[\alpha]+j[\beta]}$, as desired.

    \vspace{2mm}

    \noindent \textbf{Case B. $[\alpha] = \pm [\beta]$.} Since the rank of $\Phi_\rho > 1$, there exists $[\gamma] \in \Phi_\rho$ such that $[\beta]$ and $[\gamma]$ is base of connected subsystem $\Phi'$ of $\Phi_\rho$ is of rank $2$. 
    By Lemma \ref{lemma:u_1,u_2 exists}, there exists $a \in (J_S)_{[\beta] + [\gamma]}$ and $b \in (R_S)_{[\gamma]}$ such that
    $$ x_{-[\beta]}(v) = [x_{-[\beta]-[\gamma]}(a), x_{[\gamma]}(b)] h'$$ with $h' = \prod_{i \leq 0, j < 0} x_{i[\beta] + j[\gamma]} (a_{ij}) $. 
    But then 
    \begin{align*}
        z_{[\beta]}(u,v) &= x_{[\beta]}(u) x_{-[\beta]}(v) x_{[\beta]}(u)^{-1} \\ 
        &= x_{[\beta]}(u) \{ [x_{-[\beta]-[\gamma]}(a), x_{[\gamma]}(b)] h' \} x_{[\beta]}(u)^{-1} \\
        &= [x_{[\beta]}(u), x_{-[\beta]-[\gamma]}(a)] x_{-[\beta]-[\gamma]}(a) [x_{[\beta]}(u), x_{[\gamma]}(b)] x_{[\gamma]}(b) [x_{[\beta]}(u), x_{-[\beta]-[\gamma]}(a)^{-1}] \\
        & \hspace{10mm}  x_{-[\beta]-[\gamma]}(a)^{-1} [x_{[\beta]}(u), x_{[\gamma]}(b)^{-1}] x_{[\gamma]}(b)^{-1} [x_{[\beta]}(u), h'] h'.
    \end{align*}
    Now we consider
    \begin{align*}
        [x_{[\alpha]}(s/1 \cdot t/1), z_{[\beta]}(u,v)] &= x_{\pm [\beta]}(s/1 \cdot t/1) z_{[\beta]}(u,v) x_{\pm [\beta]}(s/1 \cdot t/1)^{-1} z_{[\beta]}(u,v)^{-1}.
    \end{align*}
    Using the above expression of $z_{[\beta]}(u,v)$ and Chevalley commutator formulas, we can conclude that $[x_{[\alpha]}(s/1 \cdot t/1), z_{[\beta]}(u,v)]$ can be expressed as a product of elements of the form $x_{[\delta]}(c_{[\delta]})$, where $[\delta] \in \Phi'$ and $c_{[\delta]} \in (R_S)_{[\delta]}$ such that $s$ and $v$ (or $\bar{v}, v_1, \bar{v}_1, v_2, \bar{v}_2$) appear in each term of $c_{[\delta]}$. Consequently, for sufficiently large $m$, the values of $c_{[\delta]}$ are in $J_{[\delta]}$, as desired.

    \vspace{2mm}

    Finally, let $g$ be an arbitrary element of $G_\sigma(R_S, J_S)$. We aim to show that for any given $[\alpha] \in \Phi_\rho$ and $t \in R_{[\alpha]}$, there exists $s \in S_\theta$ such that
    \[
        [\psi_{\mathfrak{m}}(x_{[\alpha]}(s \cdot t)), g] \in \psi_{\mathfrak{m}}(E'_\sigma(R, J)).
    \]
    As noted earlier, $g$ can be written as a product
    \[
        g = x_1 \cdots x_n,
    \]
    where each factor $x_i$ is either of the form $h(\chi)$ for some $\chi \in \mathrm{Hom}_1(\Lambda_\pi, (R_S)^*)$, or $z_{[\beta]}(u, v)$ with $[\beta] \in \Phi_\rho$, $u \in (R_S)_{[\beta]}$ and $v \in (J_S)_{[\beta]}$. 
    We have already established that for each such $x_i$, there exists $s_i \in S_\theta$ such that
    \[
        [\psi_{\mathfrak{m}}(x_{[\alpha]}(s_i \cdot t)), x_i] \in \psi_{\mathfrak{m}}(E'_\sigma(R, J)).
    \]
    Let $s = s_1 \cdots s_n$ and define $w = \psi_{\mathfrak{m}}(x_{[\alpha]}(s \cdot t))$. A straightforward computation yields
    \[
        [w, g] = [w, x_1 \cdots x_n] = \{ [w, x_1] \} \{ {}^{x_1}[w, x_2] \} \{ {}^{x_1 x_2}[w, x_3] \} \cdots \{ {}^{x_1 \cdots x_{n-1}}[w, x_n] \},
    \]
    where ${}^a b = a b a^{-1}$ denotes conjugation.
    Using a similar (though slightly adapted) argument previously applied to each generator $x_i$, we conclude that
    \[
        \{ {}^{x_1 \cdots x_{i-1}}[w, x_i] \} \in \psi_{\mathfrak{m}}(E'_\sigma(R, J)) \quad \text{for each } i.
    \]
    Therefore, the entire commutator $[w, g]$ lies in $\psi_{\mathfrak{m}}(E'_\sigma(R, J))$, as required.
\end{proof}

%%%%%%%%%%%%%%%%%%%%%%%%%%%%%%%%%%%%%%%%%%%%%%%%%%

Before proceeding further, we state a lemma from G. Taddei \cite{GT}. Let $G_\pi(\Phi, R)$ be a Chevalley group over a commutative ring $R$. Consider $R[X]$, the polynomial ring in one variable $X$ with coefficients in $R$. For a maximal ideal $\mathfrak{m}$, let $\psi'_\mathfrak{m}: G_\pi (\Phi, R) \longrightarrow G_\pi (\Phi, R_S)$ denote the natural map induced by the canonical ring homomorphism $R \longrightarrow R_S$.

\begin{lemma}[{\cite[Lemma 3.14]{GT}}]\label{lemma:GT}
    Let $\epsilon (X)$ be an element of $G_\pi(\Phi, R[X])$. Suppose $\psi'_\mathfrak{m} (\epsilon (X)) = 1$ and $\epsilon  (0) = 1$. Then there exists an element $s$ of $S$ such that $\epsilon (sX) = 1$.
\end{lemma}

\begin{rmk}
    \normalfont
    Consider a twisted Chevalley group $G_\sigma (R)$. Note that an automorphism $\theta: R \longrightarrow R$ of order $n$ (where $n=2$ or $3$) can be naturally extended to an automorphism of $R[X]$ of the same order, denoted also by $\theta$. Therefore, we can make sense of the group $G_\sigma (R[X])$.
    Since $\psi'_\mathfrak{m} \mid_{G_\sigma (R)} = \psi_\mathfrak{m}$, we can apply the conclusion of Lemma~\ref{lemma:GT} to the case of twisted Chevalley groups as well. Moreover, in this context, we can also choose $s \in S_\theta$ (see \cite[Lemma $4.7$]{KS2}).
\end{rmk}

%%%%%%%%%%%%%%%%%%%%%%%%%%%%%%%%%%%%%%%%%%%%%%%%%%

\begin{prop}\label{general:[x,g] in E(R,J)}
    For any maximal ideal $\mathfrak{m}$ of $R$ and $g \in H$, there exists $s \in S_\theta$ such that $$[x_{[\alpha]}(s \cdot t), g] \in E'_\sigma (R, J),$$ for all $[\alpha] \in \Phi_\rho$ and $t \in R_{[\alpha]}$.
\end{prop}

\begin{proof}
    The canonical map $R \longrightarrow R_S$ naturally induces the following maps: 
    $$\psi_\mathfrak{m}: G_\sigma (R) \longrightarrow G_\sigma (R_S) \text{ and } \psi''_\mathfrak{m}: G_\sigma (R[X]) \longrightarrow G_\sigma (R_S[X]).$$
    It is clear that $\psi''_\mathfrak{m} \mid_{G_\sigma(R)} = \psi_\mathfrak{m}$.
    Let $g \in H$. Then, by Proposition~\ref{local:H subset G(R,J)}, $\psi_\mathfrak{m} (g) \in G_\sigma (R_S,J_S) \subset G_\sigma (R_S[X], J_S[X])$. By Proposition \ref{local:[x,g] in E(R,J)}, for the case of $\psi''_\mathfrak{m}$, there exists $s' \in S_\theta$ such that
    $$\psi''_\mathfrak{m} ([x_{[\alpha]}(s' X \cdot t), g]) \in \psi''_\mathfrak{m} (E'_\sigma (R[X], J[X])),$$
    for all $[\alpha] \in \Phi_\rho$ and $t \in R_{[\alpha]}$. Moreover, from the proof of Proposition \ref{local:[x,g] in E(R,J)}, we can preciously write 
    $$\psi''_\mathfrak{m} ([x_{[\alpha]}(s' X \cdot t), g]) = \psi''_\mathfrak{m} \Big( \prod_{i=1}^{m} x_{[\alpha_i]} (c_i (X)) \Big),$$
    where $c_i (X) \in (J[X])_{[\alpha_i]}$ such that $c_i (0) = 0$.
    Put $$ \epsilon_{[\alpha]} (X) = [x_{[\alpha]}(s' X \cdot t), g] \Big( \prod_{i=1}^{m} x_{[\alpha_i]} (c_i(X)) \Big)^{-1}.$$
    Then we see that $\epsilon_{[\alpha]} (X)$ satisfies the hypotheses of Lemma~\ref{lemma:GT}, and hence there exists $s'_{[\alpha]} \in S_\theta$ such that $\epsilon_{[\alpha]}(s'_{[\alpha]}X) = 1$. 
    Thus we obtain $$ [x_{[\alpha]}(s' (s'_{[\alpha]}X) \cdot t), g] = \prod_{i=1}^{m} x_{[\alpha_i]} (c_i(s'_{[\alpha]}X)).$$
    Now by taking $X = 1$ and $s_{[\alpha]} = s' s'_{[\alpha]}$, we derive
    $$ [x_{[\alpha]}(s_{[\alpha]} \cdot t), g] \in E'_\sigma (R, J).$$
    Finally, if we set $s = \prod_{[\alpha] \in \Phi_\rho} s_{[\alpha]} \in S_\theta$ then $s$ is the required element.
\end{proof}

%%%%%%%%%%%%%%%%%%%%%%%%%%%%%%%%%%%%%%%%%%%%%%%%%%

Now we are in a position to prove Theorem~\ref{Ch5_mainthm}. 

%%%%%%%%%%%%%%%%%%%%%%%%%%%%%%%%%%%%%%%%%%%%%%%%%%

\vspace{2mm}

\noindent \textit{Proof of Theorem \ref{Ch5_mainthm}.} Let $H$ be a subgroup of $G_\sigma (R)$ normalized by $E'_\sigma (R)$ and let $J$ be as earlier. Since $E'_\sigma (R, J) \subset H$ (see Proposition \ref{prop:E(R,J) subset of H}), it only remains to prove that $H \subset G_\sigma (R,J)$. To demonstrate this, it suffices to show, by Corollary \ref{C=G}, that if $g \in H$ then $[x_{[\alpha]}(t), g] \in E'_\sigma (R,J)$ for every $[\alpha] \in \Phi_\rho$ and $t \in R_{[\alpha]}$.

\vspace{2mm}

\noindent \textbf{Case A. $[\alpha] \sim A_1$:} Define $I_{g, [\alpha]} = \{ s \in R_\theta \mid [x_{[\alpha]} (st), g] \in E'_\sigma (R,J) \text{ for all } t \in R_{\theta} \}$. Then $I_{g,[\alpha]}$ is an ideal of $R_\theta$. To see this, let $s_1, s_2 \in I_{g,[\alpha]}$. Then for all $t \in R_\theta$,
$$ [x_{[\alpha]}((s_1 + s_2)t), g] = (x_{[\alpha]}(s_1 t) [x_{[\alpha]}(s_2 t), g] x_{[\alpha]}(s_1 t)^{-1}) [x_{[\alpha]}(s_2 t), g] \in E'_\sigma(R,J).$$
Therefore, $s_1 + s_2 \in I_{g, [\alpha]}$. By the definition of $I_{g, [\alpha]}$, it is clear that for all $r \in R_\theta$, we have $r s_1 \in I_{g, [\alpha]}$. 
For any maximal ideal $\mathfrak{m}_\theta$ of $R_\theta$, by Proposition \ref{general:[x,g] in E(R,J)} and by Lemma \ref{A1&A2}, there exists $s \in I_{g, [\alpha]}$ such that $s \not \in \mathfrak{m}_\theta$. Thus, we can conclude that $I_{g, [\alpha]} = R_\theta$. But then $1 \in I_{g, [\alpha]}$ and hence $[x_{[\alpha]}(t), g] \in E'_\sigma (R,J)$ for all $t \in R_{\theta}$, as desired.

\vspace{2mm}

\noindent \textbf{Case B. $[\alpha] \sim A_1^2$ or $A_1^3$:} Define $I_{g, [\alpha]} = \{ s \in R \mid [x_{[\alpha]} (st), g] \in E'_\sigma (R,J) \text{ for all } t \in R \}$. Then, by a similar argument as in Case A, we can see that $I_{g,[\alpha]}$ is an ideal of $R$. For any maximal ideal $\mathfrak{m}$ of $R$, by Proposition \ref{general:[x,g] in E(R,J)}, there exists $s \in I_{g, [\alpha]}$ such that $s \not \in \mathfrak{m}$. Thus, we can conclude that $I_{g, [\alpha]} = R$. But then $[x_{[\alpha]}(t), g] \in E'_\sigma (R,J)$ for all $t \in R$, as desired.

\vspace{2mm}

\noindent \textbf{Case C. $[\alpha] \sim A_2$:} Define $A_{g, [\alpha]} = \{ s \in R \mid [x_{[\alpha]} (s t, s \bar{s} t \bar{t}/2), g] \in E'_\sigma (R,J) \text{ for all } t \in R \}$. Let $I_{g, [\alpha]}^{(1)}$ be the ideal of $R$ generated by $A_{g, [\alpha]}$ and $I_{g, [\alpha]}^{(2)} = \{ s \in R_\theta \mid [x_{[\alpha]}(0, st), g] \in E'_\sigma (R, J) \text{ for all } t \in R \text{ such that } t = - \bar{t} \}$. 
We claim that $I_{g, [\alpha]}^{(2)}$ is an ideal of $R_\theta$. To see this, let $s_1, s_2 \in I^{(2)}_{g,[\alpha]}$. Then for all $t \in R$ such that $t = - \bar{t}$, we have 
\begin{align*}
    [x_{[\alpha]}(0, (s_1 + s_2)t), g] &= [x_{[\alpha]}(0, s_1 t) x_{[\alpha]}(0, s_2 t), g] \\
    &= (x_{[\alpha]}(0, s_1 t) [x_{[\alpha]}(0, s_2 t), g] x_{[\alpha]}(0, s_1 t)^{-1}) ([x_{[\alpha]}(0, s_1 t) ,g]) \\
    & \in E'_\sigma (R,J).
\end{align*}
Therefore, $s_1 + s_2 \in I^{(2)}_{g, [\alpha]}$. By the definition of $I^{(2)}_{g, [\alpha]}$, it is clear that for all $r \in R_\theta$, we have $r s \in I^{(2)}_{g, [\alpha]}$. 

For any maximal ideal $\mathfrak{m}$ of $R$, Proposition~\ref{general:[x,g] in E(R,J)} ensures the existence of an element $s \in A_{g, [\alpha]}$ such that $s \notin \mathfrak{m}$. Therefore, we conclude that $I_{g, [\alpha]}^{(1)} = R$.
Next, observe the following:
\begin{quote}
    \textit{For any maximal ideal $\mathfrak{m}$ of $R$ and any $g \in H$, there exists $s \in S_\theta$ such that
    \[
        [x_{[\alpha]}(0, st), g] \in E'_\sigma(R, J)
    \]
    for all $[\alpha] \in \Phi_\rho$ and all $t \in R$ satisfying $t = -\bar{t}$.}
\end{quote}
The proof of the above statement follows the same argument as that of Proposition~\ref{general:[x,g] in E(R,J)} and is therefore omitted. 
Hence, we conclude that $I^{(2)}_{g, [\alpha]} = R_\theta$. 

Let $(t, u) \in \mathcal{A}(R)$. Since $I^{(1)}_{g, [\alpha]} = R$, there exists $s_1, \dots, s_k \in A_{g, [\alpha]}$ and $r_1, \dots, r_k \in R$ such that $r_1 s_1 + \cdots + r_k s_k = 1$. Therefore, $t = r_1 s_1 t + \cdots + r_k s_k t$. Observe that
\begin{align*}
    (t, u) &= (r_1 s_1 t + \cdots + r_k s_k t, u) \\
    &= \left( r_1 s_1 t, \frac{r_1 \bar{r}_1 s_1 \bar{s}_1 t \bar{t}}{2} \right) \oplus \cdots \oplus \left( r_k s_k t, \frac{r_k \bar{r}_k s_k \bar{s}_k t \bar{t}}{2} \right) \oplus (0, c),
\end{align*}
for some $c \in R$ with $c = -\bar{c}$. Set 
\[
    x_i = x_{[\alpha]} \left( r_i s_i t,\; \frac{r_i \bar{r}_i s_i \bar{s}_i t \bar{t}}{2} \right) \quad \text{for each } i = 1, \dots, k, \quad \text{and} \quad x_{k+1} = x_{[\alpha]}(0, c).
\]
Then $x_{[\alpha]}(t, u) = x_1 \dots x_{k+1}$.
Since $s_i \in A_{g, [\alpha]}$, it follows that $x_i \in E'_\sigma (R, J)$ for all $i = 1, \dots, k$. Moreover, as $1 \in I^{(2)}_{g, [\alpha]}$, we have $x_{k+1} \in E'_\sigma (R, J)$.
Therefore,
\begin{align*}
    [x_{[\alpha]}(t,u), g] &= [x_1 \dots x_{k+1}, g] \\
    &= \{ {}^{(x_1 \cdots x_k)}[x_{k+1}, g] \} \{ {}^{(x_1 \cdots x_{k-1})} [x_k, g] \} \cdots \{ {}^{x_1} [x_2, g] \} \{ [x_1, g] \} \in E'_\sigma (R,J),
\end{align*}
as desired. \qed

%%%%%%%%%%%%%%%%%%%%%%%%%%%%%%%%%%%%%%%%%%%%%%%%%%
%Section: Proof of Proposition \ref{prop:E(R,J) subset of H}
%%%%%%%%%%%%%%%%%%%%%%%%%%%%%%%%%%%%%%%%%%%%%%%%%%

\section{Proof of Proposition \ref{prop:E(R,J) subset of H}}\label{sec:Pf of prop 1}

%%%%%%%%%%%%%%%%%%%%%%%%%%%%%%%%%%%%%%%%%%%%%%%%%%

\begin{lemma}\label{A(R)}
    \normalfont
    Let $R$ be a ring with unity, and let $J$ be a $\theta$-invariant ideal of $R$. Assume that $1/2 \in R$. Consider the group $G = (\mathcal{A}(J), \oplus)$. Define the subgroups $H = \{ (0,u) \mid u \in J, \bar{u} = - u \}$ and $K = \langle (r,r \bar{r}/2) \mid r \in J \rangle$ of $G$. Then $H$ is a normal subgroup of $G$; in fact, it is contained in the centre of $G$, and $G = HK$. Moreover, if $J = R$ then $H \subset K$, that is, $G = K$.
\end{lemma}

\begin{proof}
    The first assertion is clear as $(g_1,g_2) (0,u) (g_1,g_2)^{-1} = (0,u)$ for appropriate $g_1$, $g_2$, $u \in J$. For the second assertion, note that for given $(g_1, g_2) \in \mathcal{A}(J)$, we have
    $$(g_1, g_2) = (g_1, g_1 \bar{g_1}/2) (0, (g_2 - \bar{g_2})/2).$$ Hence $G = HK$. Now if $J=R$, then we want to show that $H \subset K$.
    Our work is done if we show that $(0,(g_2 - \bar{g_2})/2)$ is generated by elements of the form $(r, r \bar{r}/2), \ r \in R$, which follows from below: $$ (0,(g_2 - \bar{g_2})/2) = (1+g_2, (1+g_2) (1 + \bar{g_2})/2)^{-1} (1, 1/2) (g_2, g_2 \bar{g_2}/2).$$
\end{proof}

%%%%%%%%%%%%%%%%%%%%%%%%%%%%%%%%%%%%%%%%%%%%%%%%%%%%%%%%%%%%%%%

\begin{lemma}\label{theta Rz}
    \normalfont
    Assume that $1/2 \in R$ if $o(\theta) = 2$ and $1/3 \in R$ if $o(\theta) = 3$. Let $z \in R$. 
    \begin{enumerate}[(a)]
        \item If $z = \bar{z}$, then $(Rz)_\theta = R_\theta z$.
        \item If $o(\theta) = 2$ and $rz + r' \bar{z} \in (Rz + R\bar{z})_\theta$, then there exists $t \in R$ such that $rz + r' \bar{z} = tz + \bar{t} \bar{z}.$
        \item If $o(\theta) = 3$ and $rz + r' \bar{z} + r'' \bar{\bar{z}} \in (Rz + R{\bar{z}} + R{\bar{\bar{z}}})_\theta$, then there exists $t \in R$ such that $rz + r' \bar{z} + r'' \bar{\bar{z}} = tz + \overline{tz} + \overline{\overline{tz}}.$ 
    \end{enumerate}
\end{lemma}

\begin{proof}
    Note that $(a)$ follows from $(b)$ (or $(c)$) and our assumption on $R$. To prove $(b)$, observe that $rz + r' \bar{z} \in (Rz + R\bar{z})_\theta \implies \overline{(rz + r' \bar{z})} = rz + r' \bar{z} \implies (r - \bar{r'}) z = (\bar{r} - r') \bar{z}.$ Set $a:= r-\bar{r'},$ then $\bar{a} = \bar{r} - r'$ and $az = \overline{az}$. Now let $t = r - a/2 = r/2 + \bar{r'}/2$, then $\bar{t} = \bar{r} - \bar{a}/2 = r' + \bar{a}/2$. Then $tz + \overline{tz} = (r - a/2)z + (r' + \bar{a}/2) \bar{z} = rz + r' \bar{z} + (\overline{az} - az)/2 = rz + r' \bar{z},$ as desired. Similarly, to prove $(c)$, observe that $rz + r' \bar{z} + r'' \bar{\bar{z}} \in (Rz + R\bar{z} + R\bar{\bar{z}})_\theta \implies \overline{rz + r' \bar{z} + r'' \bar{\bar{z}}} = rz + r' \bar{z} + r'' \bar{\bar{z}} \implies \bar{r} \bar{z} + \bar{r'} \bar{\bar{z}} + \bar{r''} z = rz + r' \bar{z} + r'' \bar{\bar{z}} \implies (r - \bar{r''}) z + (r' - \bar{r}) \bar{z} + (r'' - \bar{r'}) \bar{\bar{z}} = 0.$ Set $a:= r - \bar{r''}, b:= r' - \bar{r}$ and $c:= r'' - \bar{r'}$, then $\bar{\bar{a}} + \bar{b} + c = 0$ and $az + b \bar{z} + c \bar{\bar{z}} = 0$. Now let $t = r - a/3 + \bar{\bar{b}}/3 = (r + \bar{r''} + \bar{\bar{r'}})/3$, then $\bar{t} = r' - \bar{a}/3 - 2b/3$ and $\bar{\bar{t}} = r'' + 2 \bar{\bar{a}}/3 + \bar{b}/3$. Then 
    \begin{align*}
        tz + \overline{tz} + \overline{\overline{tz}} &= (r - a/3 + \bar{\bar{b}}/3)z + (r' - \bar{a}/3 - 2b/3) \bar{z} + (r'' + 2 \bar{\bar{a}}/3 + \bar{b}/3) \bar{\bar{z}} \\
        &= (rz + \overline{rz} + \overline{\overline{rz}}) - 2(az + b \bar{z} + (-\bar{\bar{a}} - \bar{b})\bar{\bar{z}})/3 - (\bar{a}\bar{z} + \bar{b} \bar{\bar{z}} + (-a - \bar{\bar{b}})z)/3 \\
        &= (rz + \overline{rz} + \overline{\overline{rz}}) - 2(az + b \bar{z} + c \bar{\bar{z}})/3 - (\bar{a}\bar{z} + \bar{b} \bar{\bar{z}} + \bar{c} z)/3 \\
        &= (rz + \overline{rz} + \overline{\overline{rz}}),
    \end{align*}
    as desired.
\end{proof}

%%%%%%%%%%%%%%%%%%%%%%%%%%%%%%%%%%%%%%%%%%%%%%%%%%%%%%%%%%%%%%%

\begin{lemma}\label{A(I+J) = A(I) A(J)}
    \normalfont
    Let $I$ and $J$ be $\theta$-invariant ideals of $R$. Then $\mathcal{A}(I + J) = \mathcal{A}(I) \mathcal{A} (J) = \mathcal{A}(J) \mathcal{A} (I)$.
\end{lemma}

\begin{proof}
    Clearly, $\mathcal{A}(I) \mathcal{A} (J) \subset \mathcal{A}(I + J)$. For converge, let $(r_1, r_2) \in \mathcal{A}(I+J)$. Then $r_1 \bar{r}_1 = r_2 + \bar{r}_2$ and there exists $a_1, a_2 \in I$ and $b_1, b_2 \in J$ such that $r_1 = a_1 + b_1$ and $r_2 = a_2 + b_2$. Therefore,
    \begin{align*}
        & r_1 \bar{r}_1 = r_2 + \bar{r}_2 \\
        \implies & (a_1 + b_1) (\bar{a}_1 + \bar{b}_1) = (a_2 + b_2) + (\bar{a}_2 + \bar{b}_2) \\
        \implies & a_1 \bar{a}_1 + a_1 \bar{b}_1 + b_1 \bar{a}_1 + b_1 \bar{b}_1 = (a_2 + b_2) + (\bar{a}_2 + \bar{b}_2). 
    \end{align*}
    By using this, we can write 
    \begin{align*}
        (r_1, r_2) &= (a_1 + b_1, a_2 + b_2) \\
        &= \Big(a_1, \frac{a_1 \bar{a}_1 + (a_2 - \bar{a}_2)}{2}\Big) \oplus \Big(b_1, \frac{b_1 \bar{b}_1 + (b_2 - \bar{b}_2)}{2}\Big) \oplus \Big(0, \frac{a_1 \bar{b}_1 - \bar{a}_1 b_1}{2}\Big).
    \end{align*}
    Hence $(r_1, r_2) \in \mathcal{A}(I) \mathcal{A}(J)$. Therefore $\mathcal{A}(I+J) = \mathcal{A}(I) \mathcal{A}(J) = \mathcal{A}(J) \mathcal{A}(I)$, later equality is possible because $\mathcal{A}(I) \mathcal{A}(J)$ is a group.
\end{proof}

%%%%%%%%%%%%%%%%%%%%%%%%%%%%%%%%%%%%%%%%%%%%%%%%%%

\begin{prop}\label{z to Rz}
    \normalfont
    Fix a root $[\alpha] \in \Phi_\rho$ and an element $z \in R_{[\alpha]}$. Let $H$ be the normal subgroup of $E'_\sigma (R)$ generated by $x_{[\alpha]} (z)$. Then $H = E'_{\sigma}(R, J)$ where 
    \[
        J = \begin{cases}
            Rz & \text{if } [\alpha] \sim A_1, \\
            Rz + R{\bar{z}} & \text{if } [\alpha] \sim A_1^2, \\
            Rz + R{\bar{z}} + R{\bar{\bar{z}}} & \text{if } [\alpha] \sim A_1^3, \\
            Rz_1 + R{\bar{z_1}} + R (z_2 - \bar{z_2}) & \text{if } [\alpha] \sim A_2.
        \end{cases}
    \]
\end{prop}

\begin{proof}
    Since $x_{[\alpha]}(z) \in E'_\sigma(R, J)$, by Proposition \ref{normal}, we have $H \subset E'_\sigma(R, J)$. For the reverse inclusion, we need to prove that $x_{[\beta]}(t) \in H$ for every $t \in J_{[\beta]}$ and $[\beta] \in \Phi_\rho$. Observe that there exist a sequence of roots $[\alpha_i] \ (i = 1, \dots, n)$ such that $[\alpha_1] = [\alpha], [\alpha_n] = [\beta]$, and for every $i \in \{1, \dots, n-1 \}$, the pair of roots $[\alpha_i], [\alpha_{i+1}]$ contained in some connected subsystem of $\Phi_\rho$ of rank $2$. Now, by applying Lemma \ref{z to Rz in phi'} (below) recursively to the pairs $([\alpha_i], [\alpha_{i+1}])$, we obtain the desired result.
\end{proof}

%%%%%%%%%%%%%%%%%%%%%%%%%%%%%%%%%%%%%%%%%%%%%%%%%%%%%%%%%%%%%%%

\begin{lemma}\label{z to Rz in phi'}
    \normalfont
    Let the notation be as established in Proposition \ref{z to Rz}. Suppose $\Phi'$ is a connected subsystem of $\Phi_\rho$ with rank 2. If $[\gamma] \in \Phi'$ and $s \in R$ such that $x_{[\gamma]}(s) \in H$, then $x_{[\beta]}(t) \in H$ for every $[\beta] \in \Phi'$ and $t \in I_{[\beta]}$, where $I$ is an $\theta$-invariant ideal defined similarly to the ideal $J$ in Proposition \ref{z to Rz} by replacing $z$ with $s$.
\end{lemma}

\begin{proof}
    Let $\mu$ denote the angle between $[\beta]$ and $[\gamma]$. For $r \in R$, let $r'$ denote either $r, \bar{r}$ or $\bar{\bar{r}}$. By $r_1$ and $r_2$, we mean the first and second coordinate of $r = (r_1, r_2) \in \mathcal{A}(R)$.
    Consider the following table outlining the possible types of choices for $\Phi'$:
    \begin{center}
        \begin{tabular}{|c|c|c|c|c|}
            \hline
            \multirow{2}{*}{\textbf{Type}} & \multirow{2}{*}{\textbf{$\Phi_\rho$}} & \multicolumn{2}{c|}{\textbf{Type of Roots}} & \multirow{2}{*}{\textbf{Possible Choices of $\Phi'$}} \\
            \cline{3-4}
             & & \textbf{Long} & \textbf{Short} &  \\
            \hline 
            ${}^2 A_{3}$ & $C_2 (= B_2$) & $A_1$ & $A_1^2$ & $B_2$ (containing long and short roots) \\
            \hline 
            ${}^2 A_{2n-1} \ (n \geq 3)$ & $C_n$ & $A_1$ & $A_1^2$ & $A_2$ (only short), $B_2$ (long and short)\\
            \hline 
            ${}^2 A_{4}$ & $B_2$ & $A_1^2$ & $A_2$ & $B_2$ (containing long and short roots) \\
            \hline 
            ${}^2 A_{2n} \ (n \geq 3)$ & $B_n$ & $A_1^2$ & $A_2$ & $A_2$ (only long), $B_2$ (long and short)\\
            \hline 
            ${}^2 D_{n} \ (n \geq 4)$ & $B_{n-1}$ & $A_1$ & $A_1^2$ & $A_2$ (only long), $B_2$ (long and short) \\
            \hline
            ${}^3 D_{4}$ & $G_2$ & $A_1$ & $A_1^3$ & $A_2$ (only long), $G_2$ (long and short) \\
            \hline
            ${}^2 E_{6}$ & $F_4$ & $A_1$ & $A_1^2$ & $A_2$ (with long), $A_2$ (with short), $B_2$ \\
            \hline
        \end{tabular}
    \end{center}

    \vspace{2mm}

    \noindent \textbf{Case A. $\Phi' \sim A_2$ which contains roots of the type $A_1$.} This case arises only when $\Phi_\rho \sim {}^2 D_{n} \ (n \geq 4), {}^3 D_4$ or ${}^2 E_6$. Let us consider the following subcases:
    \begin{enumerate}[leftmargin=3.5em]
        \item[\textbf{(A1)}] \textbf{$[\gamma] \sim A_1$ and $\mu = \pi/3$.} In this case, the pair $[\gamma]$ and $[\beta] - [\gamma]$ is of type $(b-i)$. Here $[\gamma], [\beta]$ and $[\beta] - [\gamma]$ are of type $A_1$. For $r \in R_\theta$, we have
        $$ [x_{[\gamma]}(s), x_{[\beta] - [\gamma]}(\pm r)] = x_{[\beta]} (rs) \in H.$$
        
        \item[\textbf{(A2)}] \textbf{$[\gamma] \sim A_1$ and $\mu$ is arbitrary.} Observe that, we can find a sequence $[\gamma_1], \dots, [\gamma_m]$ of roots in $\Phi'$ such that $2 \leq m \leq 6, [\gamma_1] = [\gamma], [\gamma_m] = [\beta]$ and the angle between $[\gamma_i]$ and $[\gamma_{i+1}]$ is $\pi / 3$ for every $i = 1, \dots, m-1$. Then, by applying Case (A1) to the pair of roots $([\gamma_{i-1}], [\gamma_{i}])$, we have $x_{[\gamma_i]}(rs) \in H$ for every $r \in R_\theta$ and $i = 2, \dots, m$. In particular, by Lemma \ref{theta Rz}, $x_{[\beta]}(t) \in H$ for every $t \in I_{[\beta]}$, as desired.
    \end{enumerate}
    
    \vspace{2mm}

    \noindent \textbf{Case B. $\Phi' \sim A_2$ which contains roots of the type $A_1^2$.} This case arises only when $\Phi_\rho \sim {}^2 A_{2n-1} \ (n \geq 3) {}^2 A_{2n} \ (n \geq 3)$ or ${}^2 E_6$. Let us consider the following subcases:
    \begin{enumerate}[leftmargin=3.5em]
        \item[\textbf{(B1)}] \textbf{$[\gamma] \sim A_1^2$ and $\mu = \pi/3$.} In this case, the pair $[\gamma]$ and $[\beta] - [\gamma]$ is of the type $(b-ii)$. Here $[\gamma], [\beta]$ and $[\beta] - [\gamma]$ are of type $A_1^2$. For $r \in R$, we have
        $$ [x_{[\gamma]}(s), x_{[\beta] - [\gamma]}(\pm r')] = x_{[\beta]} (rs') \in H.$$

        \item[\textbf{(B2)}] \textbf{$[\gamma] \sim A_1^2$ and $\mu$ is arbitrary.} Observe that, we can find a sequence $[\gamma_1], \dots, [\gamma_m]$ of roots in $\Phi'$ such that $2 \leq m \leq 6, [\gamma_1] = [\gamma], [\gamma_m] = [\beta]$ and the angle between $[\gamma_i]$ and $[\gamma_{i+1}]$ is $\pi / 3$ for every $i = 1, \dots, m-1$. Then, by applying Case (B1) to the pair of roots $([\gamma_{i-1}],[\gamma_i])$, we have $x_{[\gamma_i]}(rs') \in H$ for every $r \in R$ and $i = 2, \dots, m$. In particular, $x_{[\beta]}(rs') \in H$ for every $r \in R$.
        
        Note that, we can find a subsystem $\Phi''$ of $\Phi_\rho$ of type $B_2$ such that $[\beta] \in \Phi''$. Since $x_{[\beta]}(rs') \in H$, by Case C (if $\Phi_\rho \sim {}^2 A_{2n-1}$ or ${}^2 E_6$) and by Case D (if $\Phi_\rho \sim {}^2 A_{2n}$) (see below), we have $x_{[\beta]}(\bar{r}\bar{s'}) \in H$ for every $r \in R$. In particular, $x_{[\beta]}(rs)$ and $x_{[\beta]}(r \bar{s}) \in H$ for every $r \in R$. Therefore, $x_{[\beta]}(t) \in H$ for every $t \in I_{[\beta]}$, as desired.
    \end{enumerate}

    \vspace{2mm}
    
    \noindent \textbf{Case C. $\Phi' \sim B_2$ which contains roots of the type $A_1$ and $A_1^2$.} This case arises only when $\Phi_\rho \sim {}^2 A_{2n-1} \ (n \geq 2), {}^2 D_{n} \ (n \geq 4)$ or ${}^2 E_6$. Let us consider the following subcases: 
    \begin{enumerate}[leftmargin=3.5em]
        \item[\textbf{(C1)}] \textbf{$[\gamma] \sim A_1$ and $\mu = \pi/4$.} In this case, the pair $[\gamma]$ and $[\beta] - [\gamma]$ is of the type $(d-i)$. Here $[\gamma], 2[\beta] - [\gamma] \sim A_1$ and $[\beta], [\beta] - [\gamma] \sim A_1^2$. For $r \in R$, we have
        $$ [x_{[\gamma]}(s), x_{[\beta] - [\gamma]}(\pm r/2)] = x_{[\beta]} (rs/2) x_{2[\beta] - [\gamma]} (\pm r \bar{r} s/4) \in H.$$
        Now put $-r$ instead of $r$, we get $x_{[\beta]} (-rs/2) x_{2[\beta] - [\gamma]} (\pm r \bar{r} s/4) \in H$. But then
        $$ x_{[\beta]}(rs) = \{x_{[\beta]} (rs/2) x_{2[\beta] - [\gamma]} (\pm r \bar{r} s/4) \} \{ x_{[\beta]} (-rs/2) x_{2[\beta] - [\gamma]} (\pm r \bar{r} s/4) \}^{-1} \in H.$$
        
        \item[\textbf{(C2)}] \textbf{$[\gamma] \sim A_1^2$ and $\mu = \pi/4$.} In this case, the pair $[\gamma]$ and $[\beta] - [\gamma]$ is of the type $(c-i)$. Here $[\gamma], [\beta] - [\gamma] \sim A_1^2$ and $[\beta] \sim A_1$. For $r \in R$, we have
        $$ [x_{[\gamma]}(s), x_{[\beta] - [\gamma]}(\pm r)] = x_{[\beta]} (r \bar{s} + \bar{r} s) \in H.$$

        \item[\textbf{(C3)}] \textbf{$[\gamma] \sim A_1^2$ and $\mu = \pi/2$.} In this case, the pair $[\gamma]$ and $[\beta] - [\gamma]$ is of the type $(d-i)$. Here $[\gamma], [\beta] \sim A_1^2$ and $[\beta] + [\gamma], [\beta] - [\gamma] \sim A_1$. For $r \in R_\theta$, we have
        $$[x_{[\beta] - [\gamma]} (\pm r), x_{[\gamma]}(s)] = x_{[\beta]} (rs) x_{[\gamma] + [\beta]} (\pm rs\bar{s}) \in H.$$
        Now, by Case (C2), we know that $x_{[\gamma] + [\beta]} (\pm rs\bar{s}) \in H$. Hence 
        $$ x_{[\beta]}(rs) = \{x_{[\beta]} (rs) x_{[\gamma] + [\beta]} (\pm rs\bar{s})\} \{x_{[\gamma] + [\beta]} (\pm rs\bar{s})\}^{-1} \in H.$$
        In particular, $x_{[\beta]}(s) \in H$. By Case (C1) and Case (C2), we have $x_{[\beta]} (s + \bar{s}) \in H$. But then $x_{[\beta]}(\bar{s}) \in H$. Now, by switching roles of ${[\beta]}$ and $[\gamma]$, we get $x_{[\gamma]}(\bar{s}) \in H$. If we replace $s$ by $\bar{s}$ in above argument, then we get $x_{[\beta]}(r\bar{s}) \in H$, for every $r \in R_\theta$. 
        
        \item[\textbf{(C4)}] \textbf{$[\gamma]$ and $\mu$ are arbitrary.} Observe that, we can find a sequence $[\gamma_1], \dots, [\gamma_m]$ of roots in $\Phi'$ such that $2 \leq m \leq 8, [\gamma_1] = [\gamma], [\gamma_m] = [\beta]$ and the angle between $[\gamma_i]$ and $[\gamma_{i+1}]$ is $\pi / 4$ for every $i = 1, \dots, m-1$. 
        \begin{enumerate}[1.]
            \item If $[\gamma] \sim A_1$ and $[\beta] \sim A_1$ then $m$ must be odd. By applying Case (C1) to pairs $([\gamma_{i-1}], [\gamma_{i}])$, Case (C2) to pairs $([\gamma_i], [\gamma_{i+1}])$ for $i = 2, 4, \dots, m-1$, inductively, we get $x_{[\beta]}((r + \bar{r}) s) \in H$ for every $r \in R$. Now by Lemma \ref{theta Rz}, we have $x_{[\beta]}(t) \in H$ for every $t \in I_{[\beta]}$.
            \item If $[\gamma] \sim A_1$ and $[\beta] \sim A_1^2$ then $m$ must be even. Note that $[\gamma_{m-1}]$ is of type $A_1$. Hence, by the above case, we have $x_{[\gamma_{m-1}]}(s) \in H$. Now by applying Case (C1) to the pair $([\gamma_{m-1}], [\gamma_m])$, we conclude that $x_{[\beta]}(t) \in H$ for every $t \in I_{[\beta]}$.
            \item If $[\gamma] \sim A_1^2$ and $[\beta] \sim A_1$ then $m$ must be even. By applying Case (C3) to pairs $([\gamma_i], [\gamma_{i+2}])$ for $i = 1, 3, \dots, m-3$, inductively, we get $x_{[\gamma_{m-1}]}(s) \in H$. Now by applying Case (C2) to the pair $([\gamma_{m-1}],[\gamma_m])$, we get $x_{[\beta]}(r \bar{s} + \bar{r} s) \in H$ for every $r \in R$. Finally by Lemma \ref{theta Rz}, we have $x_{[\beta]}(t) \in H$ for all $t \in I_{[\beta]}$.
            \item If $[\gamma] \sim A_1^2$ and $[\beta] \sim A_1^2$ then $m$ must be odd. By applying Case (C2) to pairs $([\gamma_{i-1}], [\gamma_{i}])$, Case (C1) to pairs $([\gamma_i], [\gamma_{i+1}])$ for $i = 2, 4, \dots, m-1$, inductively, we get $x_{[\beta]}(r_1(r_2 \bar{s} + \bar{r_2} s)) \in H$ for every $r_1, r_2 \in R$. Now by applying Case (C3) to pairs $([\gamma_{i}, \gamma_{i+2}])$ for $i=1,3,\dots, m-2,$ inductively, we get $x_{[\beta]}(r_3 s)$ and $x_{[\beta]}(r_4 \bar{s}) \in H$ for every $r_3, r_4 \in R_\theta$. 
            But then for every $r_5 \in R$, we have $$x_{[\beta]}(r_5^2 s) = x_{[\beta]}(r_5(r_5s + \bar{r_5}\bar{s})) \{x_{[\beta]}(r_5 \bar{r_5} \bar{s})\}^{-1} \in H.$$ Similarly, we have $x_{[\beta]}(r_6^2 \bar{s}) \in H$ for every $r_6 \in R$. 
            Finally, for given $r \in R$ we have 
            $$ x_{[\beta]}(rs) = x_{[\beta]}\Big(\Big(\frac{1+r}{2}\Big)^2s\Big) \Big\{x_{[\beta]}\Big(\Big(\frac{1-r}{2}\Big)^2s\Big) \Big\}^{-1} \in H.$$
            Similarly, we have $x_{[\beta]}(r \bar{s}) \in H$ for every $r \in R$. By Lemma \ref{theta Rz}, we conclude that $x_{[\beta]}(t) \in H$ for all $t \in I_{[\beta]}.$
        \end{enumerate}
    \end{enumerate}

    \vspace{2mm}

    \noindent \textbf{Case D. $\Phi' \sim B_2$ which contains roots of the type $A_1^2$ and $A_2$.} This case arises only when $\Phi_\rho \sim {}^2 A_{2n} \ (n \geq 2)$. Let us consider the following subcases: 
    \begin{enumerate}[leftmargin=3.5em]
        \item[\textbf{(D1)}] \textbf{$[\gamma] \sim A_1^2$ and $\mu = \pi /4$.} First observe that the pair $[\gamma]$ and $[\beta] - [\gamma]$ is of the type $(d-ii)$. Here $[\gamma], 2[\beta] - [\gamma] \sim A_1^2$ and $[\beta], [\beta] - [\gamma] \sim A_2$. In this case, for $r = (r_1, r_2) \in \mathcal{A}(R)$ we have 
        \begin{equation*}\label{eqn_a}
            [x_{[\gamma]}(s), x_{[\beta] - [\gamma]}(\pm r_1'/2, r_2'/4)] = x_{[\beta]} (r_1 s'/2 , r_2 s \bar{s}/4) x_{2[\beta] - [\gamma]} (\pm r_2' s/4) \in H.
        \end{equation*}
        By putting $-r_1$ instead of $r_1$, we get $x_{[\beta]} (-r_1 s'/2, r_2 s \bar{s}/4) x_{2[\beta] - [\gamma]} (\pm r_2' s/4) \in H$. But then
        \begin{equation*}\label{eqn_b}
            \begin{split}
                x_{[\beta]}(r_1 s', \frac{r_1 \bar{r_1}}{2}s \bar{s}) &= x_{[\beta]}(r_1 s', \frac{r_2 + \bar{r_2}}{2}s \bar{s}) \\
                &= \{x_{[\beta]} (r_1 s'/2 , r_2 s \bar{s}/4) x_{2[\beta] - [\gamma]} (\pm r_2 s/4)\} \\ 
                & \hspace{10mm} \{x_{[\beta]} (-r_1 s'/2, r_2 s \bar{s}/4) x_{2[\beta] - [\gamma]} (\pm r_2 s/4)\}^{-1} \in H.
            \end{split}
        \end{equation*}
        Now observe that the pair $[\gamma]$ and $2[\beta] - [\gamma]$ is of type $(a_{2}-ii)$. Here $[\gamma], [\delta]:=2[\beta] - [\gamma] \sim A_1^2$ and $[\beta] = 1/2([\gamma] + [\delta]) \sim A_2$. In this case, for $r \in R$ we have 
        \begin{equation*}\label{eqn_c}
            [x_{[\gamma]}(s), x_{[\delta]}(\pm r)] = x_{[\beta]}(0, \bar{r}s - r \bar{s}) \in H.
        \end{equation*}
        We now claim that for every $(r_3 s', r_4 s') \in \mathcal{A}(Rs')$, the element $x_{[\beta]}(r_3 s', r_4 s') \in H$. As in proof of Lemma \ref{A(R)}, we consider the following decomposition of $(r_3 s', r_4 s'):$ 
        $$(r_3 s', r_4 s') = (r_3 s', r_3 \bar{r_3} s \bar{s}/2) \oplus (0, (r_4 s' - \overline{r_4 s'})/2).$$
        Therefore, we have 
        \begin{equation*}\label{eqn_d}
            x_{[\beta]}(r_3 s', r_4 s') = x_{[\beta]}(r_3 s', r_3 \bar{r_3} s \bar{s}/2) x_{[\beta]}(0, (r_4 s' - \overline{r_4 s'})/2) \in H.
        \end{equation*}
        Finally, by Proposition \ref{prop wxw^{-1}} (or by Proposition $4.1$ of \cite{EA1}), we have 
        \begin{equation}\label{eq_wx=xw}
            w_{2[\beta] - [\gamma]}(1) x_{[\gamma]} (s) {w_{2[\beta] - [\gamma]}(1)}^{-1} = x_{[\gamma]}(-\bar{s}).
        \end{equation}
        Therefore, $x_{[\gamma]}(-\bar{s}) \in H$ and hence $x_{[\gamma]}(\bar{s}) \in H.$ Replacing $s$ by $\bar{s}$ in above, we get $x_{[\beta]}(r_3 \bar{s'}, r_4 \bar{s'}) \in H$ for every $(r_3 \bar{s'}, r_4 \bar{s'}) \in \mathcal{A}(R\bar{s'})$. In particular, we have $x_{[\beta]}(r_3 s, r_4 s) \in H$ (resp., $x_{[\beta]}(r_3 \bar{s}, r_4 \bar{s}) \in H$) for every $(r_3 s, r_4 s) \in \mathcal{A}(Rs)$ (resp., $(r_3 \bar{s}, r_4 \bar{s}) \in \mathcal{A}(R\bar{s})$).
        
        \item[\textbf{(D2)}] \textbf{$[\gamma] \sim A_2$ and $\mu = \pi /4$.} Note that, the pair $[\gamma]$ and $[\beta] - [\gamma]$ is of the type $(c-ii)$. Here $[\gamma], [\beta]-[\gamma] \sim A_2$ and $[\beta] \sim A_1^2$. For $r \in R$, we have
        \begin{equation}\label{eq_d2_1}
            [x_{[\gamma]}(s_1, s_2), x_{[\beta] - [\gamma]}(\pm r', r \bar{r}/2)] = x_{[\beta]} (r s_1') \in H.
        \end{equation}
        By (\ref{eq_wx=xw}), we also have $x_{[\beta]} (\bar{r} \bar{s_1'}) \in H$. In particular, for every $r \in R$, we have $x_{[\beta]}(rs_1) \in H$ and $x_{[\beta]}(r \bar{s_1}) \in H$.

        Observe that, the pair $[\gamma]$ and $[\beta] - 2[\gamma]$ is of the type $(d-ii)$. Here $[\beta] - [\gamma], [\gamma] \sim A_2$ and $[\beta], [\beta] - 2 [\gamma] \sim A_1^2$. For $r \in R$ we have 
        \begin{equation}\label{eq_d2_2}
            [x_{[\beta] - 2[\gamma]}(\pm r'), x_{[\gamma]}(s_1, s_2)] = x_{[\beta] - [\gamma]} (rs_1', r \bar{r} s_2') x_{[\beta]} (\pm r's_2'') \in H. 
        \end{equation}
        Note that, $(x_{[\gamma]}(s_1, s_2))^{-1} = x_{[\gamma]}(-s_1, \bar{s_2}) \in H$. By above, we have $x_{[\beta]}(s_1) \in H$. By applying case (D1) to the pair $([\beta], [\gamma])$, we have $x_{[\gamma]}(2s_1, 2s_1\bar{s_1}) \in H$. But then 
        $$x_{[\gamma]}(s_1, \bar{s_2}) = x_{[\gamma]}(-s_1, \bar{s_2}) x_{[\gamma]}(2s_1, 2s_1\bar{s_1}) \in H.$$

        Put $(s_1, \bar{s_2})$ instead of $(s_1, s_2)$ and $-r$ instead of $r$ in (\ref{eq_d2_2}), we get $$x_{[\beta] - [\gamma]} (- r s_1', r \bar{r} \bar{s_2'}) x_{[\beta]} (\pm (-r') \bar{s_2''}) \in H.$$ Since the elements $x_{[\beta] - [\gamma]} (\cdot)$ and $x_{[\beta]} (\cdot)$ are commutes with each other, we have 
        \begin{align*}
            x_{[\beta]}(\pm r'(s_2'' - \bar{s_2''})) &=  \{x_{[\beta] - [\gamma]} (r s_1', r \bar{r} s_2') x_{[\beta]} (\pm r' s_2'')\} \\
            & \hspace{15mm} \{ x_{[\beta] - [\gamma]} (- r s_1', r \bar{r} \bar{s_2'}) x_{[\beta]} (\pm (-r') \bar{s_2''}) \} \in H.
        \end{align*}
        In particular, we have $x_{[\beta]} (\frac{r}{2} (s_2 - \bar{s_2})) \in H$ for every $r \in R$. Again by (\ref{eq_d2_1}), we have $x_{[\beta]}(\frac{r}{2} (s_1 \bar{s_1})) = x_{[\beta]} (\frac{r}{2} (s_2 + \bar{s_2})) \in H$ for every $r \in R$. But then $$ x_{[\beta]}(r s_2) = x_{[\beta]} (\frac{r}{2} (s_2 + \bar{s_2})) x_{[\beta]}(\frac{r}{2} (s_2 - \bar{s_2})) \in H,$$
        for every $r \in R$. Similarly, we have $x_{[\beta]}(r \bar{s_2}) \in H$ for every $r \in R$.
        
        \item[\textbf{(D3)}] \textbf{$[\gamma]$ and $\mu$ are arbitrary.} Observe that, we can find a sequence $[\gamma_1], \dots, [\gamma_m]$ of roots in $\Phi'$ such that $2 \leq m \leq 8, [\gamma_1] = [\gamma], [\gamma_m] = [\beta]$ and the angle between $[\gamma_i]$ and $[\gamma_{i+1}]$ is $\pi / 4$ for every $i = 1, \dots, m-1$.
        \begin{enumerate}[1.]
            \item If $[\gamma] \sim A_1^2$ and $[\beta] \sim A_1^2$ then $m$ must be odd. By applying Case (D1) to pairs $([\alpha_i], [\alpha_{i+1}])$ and Case (D2) to pairs $([\alpha_{i+1}], [\alpha_{i+2}])$ for every $i = 1, 3, \dots, m-2$, recursively, we get $x_{[\beta]}(rs) \in H$ and $x_{[\beta]}(r\bar{s}) \in H$ for every $r \in R$. Therefore, we have $x_{[\beta]}(t) \in H$ for every $t \in I_{[\beta]}$.

            \item If $[\gamma] \sim A_1^2$ and $[\beta] \sim A_2$ then $m$ must be even. Note that $[\gamma_{m-1}] \sim A_1^2$, by the above case, $x_{[\gamma_{m-1}]}(s) \in H$ and $x_{[\gamma_{m-1}]}(\bar{s}) \in H$. By applying case (D1) to the pair $([\alpha_{m-1}], [\alpha_m])$, we get $x_{[\beta]}(r_1s, r_2s) \in H$ and $x_{[\beta]}(r_1 \bar{s}, r_2 \bar{s}) \in H$ for every $(r_1s,r_2s) \in \mathcal{A}(Rs)$. Finally, by the proof of Lemma \ref{A(I+J) = A(I) A(J)}, we have $x_{[\beta]}(t) \in H$ for every $t \in I_{[\beta]}$.

            \item If $[\gamma] \sim A_2$ and $[\beta] \sim A_1^2$ then $m$ must be even. Note that $[\gamma_{2}] \sim A_1^2$, by applying case (D2) to the pair $([\gamma_1], [\gamma_2])$, we get $x_{[\gamma_2]}(r s_1)$, $x_{[\gamma_2]}(r \bar{s_1})$, $x_{[\gamma_2]}(r s_2)$, $x_{[\gamma_2]}(r \bar{s_2}) \in H$ for every $r \in R$. Now by Case 1 above, we have $x_{[\beta]}(r s_1),$ $x_{[\beta]}(r s_2),$ $x_{[\beta]}(r \bar{s_1}),$ $x_{[\beta]}(r \bar{s_2}) \in H$ for every $r \in R$.
            Therefore, we have $x_{[\beta]}(t) \in H$ for every $t \in I_{[\beta]}$.

            \item If $[\gamma] \sim A_2$ and $[\beta] \sim A_2$ then $m$ must be odd. Note that $[\gamma_{m-1}] \sim A_1^2$, by the case 3 above, we have $x_{[\gamma_{m-1}]}(s_1), x_{[\gamma_{m-1}]}(s_2), x_{[\gamma_{m-1}]}(\bar{s_1}), x_{[\gamma_{m-1}]}(\bar{s_2}) \in H$. By applying case (D1) to the pair $([\gamma_{m-1}], [\gamma_{m}])$, we have $x_{[\beta]}(r_1 s_1, r_2 s_1)$, $x_{[\beta]}(r_1 \bar{s_1}, r_2 \bar{s_1}) \in H$ for every $(r_1 s_1, r_2 s_1) \in \mathcal{A}(Rs_1)$ and $x_{[\beta]}(r_3 s_2, r_4 s_2)$, $x_{[\beta]}(r_3 \bar{s_2}, r_4 \bar{s_2}) \in H$ for every $(r_3 s_2, r_4 s_2) \in \mathcal{A}(Rs_2)$. Therefore, by the proof of Lemma \ref{A(I+J) = A(I) A(J)}, we have $x_{[\beta]}(t) \in H$ for every $t \in I_{[\beta]}$.
        \end{enumerate}
    \end{enumerate}

    \vspace{2mm}

    \noindent \textbf{Case E. $\Phi' \sim G_2$ which contains roots of the type $A_1$ and $A_1^3$.} This case arises only when $\Phi_\rho \sim {}^3 D_{4}$. Let us consider the following subcases:
    \begin{enumerate}[leftmargin=3.5em]
        \item[\textbf{(E1)}] \textbf{$[\gamma] \sim A_1$ and $\mu = \pi/6$.} In this case, the pair $[\gamma]$ and $[\beta] - [\gamma]$ is of the type $(e)$. Here $[\gamma], 3[\beta] - [\gamma], 3[\beta] - 2[\gamma] \sim A_1$ and $[\beta], [\beta] - [\gamma], 2[\beta] - [\gamma] \sim A_1^3$. For $r \in R$, we have
        \begin{gather*}
            [x_{[\gamma]}(s), x_{[\beta] - [\gamma]}(\pm r/2)] = x_{[\beta]} (rs/2) x_{2[\beta] - [\gamma]} (\pm r r' s /4) \\
            x_{3[\beta] - 2[\gamma]} (\pm r \bar{r} \bar{\bar{r}} s/8) x_{3[\beta] - [\gamma]} (\pm r \bar{r} \bar{\bar{r}} s^2/8) \in H.
        \end{gather*}
        Note that, long roots in $G_2$ form a subsystem of type $A_2$. Hence, by Case A, we have $x_{3[\beta] - 2[\gamma]} (\pm r \bar{r} \bar{\bar{r}} s/8) \in H$ and $x_{3[\beta] - [\gamma]} (\pm r \bar{r} \bar{\bar{r}} s^2/8) \in H$. But then 
        \begin{gather*}
             \{ x_{[\beta]} (rs/2) x_{2[\beta] - [\gamma]} (\pm r r' s/4) x_{3[\beta] - 2[\gamma]} (\pm r \bar{r} \bar{\bar{r}} s/8) x_{3[\beta] - [\gamma]} (\pm r \bar{r} \bar{\bar{r}} s^2/8) \} \\
             \{x_{3[\beta] - 2[\gamma]} (\pm r \bar{r} \bar{\bar{r}} s/8) x_{3[\beta] - [\gamma]} (\pm r \bar{r} \bar{\bar{r}} s^2/8)\}^{-1} \\
             = x_{[\beta]} (rs/2) x_{2[\beta] - [\gamma]} (\pm r r' s/4) \in H.
        \end{gather*}
        Now put $-r$ instead of $r$, we get $x_{[\beta]} (-rs/2) x_{2[\beta] - [\gamma]} (\pm r r' s/4) \in H$. Finally, 
        $$ x_{[\beta]} (rs) = \{ x_{[\beta]} (rs/2) x_{2[\beta] - [\gamma]} (\pm r r' s/4) \} \{ x_{[\beta]} (-rs/2) x_{2[\beta] - [\gamma]} (\pm r r' s/4) \}^{-1} \in H.$$

        \item[\textbf{(E2)}] \textbf{$[\gamma] \sim A_1^3$ and $\mu = \pi/6$.} In this case, the pair $[\gamma]$ and $[\beta] - [\gamma]$ is of type $(g)$. Here $[\gamma] \sim A_1^3, [\beta] - [\gamma] \sim A_1^3$ and $[\beta] \sim A_1$. For $r \in R$, we have
        $$ [x_{[\gamma]}(s), x_{[\beta] - [\gamma]}(\pm r')] = x_{[\beta]} (rs + \bar{r} \bar{s} + \bar{\bar{r}} \bar{\bar{s}}) \in H.$$

        \item[\textbf{(E3)}] \textbf{$[\gamma] \sim A_1^3$ and $\mu = \pi/3$.} In this case, the pair $[\gamma]$ and $[\beta] - [\gamma]$ is of type $(f)$. Here $[\gamma] \sim A_1^3, [\beta] - [\gamma] \sim A_1^3, [\beta] \sim A_1^3, 2[\beta] - [\gamma] \sim A_1$ and $[\gamma] + [\beta] \sim A_1$. For $r \in R$, we have
        \begin{gather*}
            [x_{[\gamma]}(s), x_{[\beta] - [\gamma]}(\pm r/2)] = x_{[\beta]} ((r' s \pm r s')/2) x_{[\gamma] + [\beta]} (\pm(r \bar{s} \bar{\bar{s}} + \bar{r} s \bar{\bar{s}} + \bar{\bar{r}} s \bar{s})/2) \\
            x_{2[\beta] - [\gamma]} (\pm(r \bar{r} \bar{\bar{s}} + \bar{r} \bar{\bar{r}} s + r \bar{\bar{r}} \bar{s})/2) \in H.
        \end{gather*}
        Here $(s',r') = (\bar{s}, \bar{r})$ or $(\bar{\bar{s}}, \bar{\bar{r}})$.
        Note that $x_{[\gamma]}(-s) = \{x_{[\gamma]}(s)\}^{-1} \in H$. Hence if we replace $r$ (resp., $s$) by $-r$ (resp., $-s$), then 
        \begin{gather*}
            x_{[\beta]} ((r' s \pm r s')/2) x_{[\gamma] + [\beta]} (\pm(-r \bar{s} \bar{\bar{s}} - \bar{r} s \bar{\bar{s}} - \bar{\bar{r}} s \bar{s})/2) \\
            x_{2[\beta] - [\gamma]} (\pm(-r \bar{r} \bar{\bar{s}} - \bar{r} \bar{\bar{r}} s - r \bar{\bar{r}} \bar{s})/2) \in H.
        \end{gather*}
        Since $x_{[\beta]}(\cdot), x_{[\gamma] + [\beta]}(\cdot)$ and $x_{2[\beta] - [\gamma]}(\cdot)$ commutes with each other, we have
        \begin{gather*}
            x_{[\beta]} (r' s \pm r s') = \{ x_{[\beta]} ((r' s \pm r s')/2) x_{[\gamma] + [\beta]} (\pm(r \bar{s} \bar{\bar{s}} + \bar{r} s \bar{\bar{s}} + \bar{\bar{r}} s \bar{s})/2) \\ 
            x_{2[\beta] - [\gamma]} (\pm(r \bar{r} \bar{\bar{s}} + \bar{r} \bar{\bar{r}} s + r \bar{\bar{r}} \bar{s})/2) \} \{ x_{2[\beta] - [\gamma]} (\pm(-r \bar{r} \bar{\bar{s}} - \bar{r} \bar{\bar{r}} s - r \bar{\bar{r}} \bar{s})/2) \\
            x_{[\gamma] + [\beta]} (\pm(-r \bar{s} \bar{\bar{s}} - \bar{r} s \bar{\bar{s}} - \bar{\bar{r}} s \bar{s})/2) x_{[\beta]} ((r' s \pm r s')/2)\} \in H.
        \end{gather*}

        \item[\textbf{(E4)}] \textbf{$[\gamma]$ and $\mu$ are arbitrary.} Observe that, we can find a sequence $[\gamma_1], \dots, [\gamma_m]$ of roots in $\Phi'$ such that $2 \leq m \leq 12, [\gamma_1] = [\gamma], [\gamma_m] = [\beta]$ and the angle between $[\gamma_i]$ and $[\gamma_{i+1}]$ is $\pi / 6$ for every $i = 1, \dots, m-1$. 
        \begin{enumerate}[1.]
            \item If $[\gamma] \sim A_1$ and $[\beta] \sim A_1$ then $m$ must be odd. By applying Case (E1) to pairs $([\gamma_i], [\gamma_{i+1}])$ and Case (E2) to pairs $([\gamma_{i+1}], [\gamma_{i+2}])$ for $i=1, 3, \dots, m-2$, recursively, we get, $x_{[\beta]} ((r +\bar{r} + \bar{\bar{r}})s) \in H$ for every $r \in R$. Now by Lemma \ref{theta Rz}, we have $x_{[\beta]}(t) \in H,$ for every $t \in I_{[\beta]}$.

            \item If $[\gamma] \sim A_1$ and $[\beta] \sim A_1^3$ then $m$ must be even. Since $[\gamma_{m-1}] \sim A_1$, by the above Case, we have  $x_{[\gamma_{m-1}]}(s) \in H$. Now by applying Case (E1) to the pair $([\gamma_{m-1}], [\gamma_{m}])$, we get $x_{[\beta]} (rs) \in H$ for every $r \in R$. Hence we have $x_{[\beta]}(t) \in H,$ for every $t \in I_{[\beta]} = I$.
            
            \item If $[\gamma] \sim A_1^3$ and $[\beta] \sim A_1$ then $m$ must be even. By apply Case (E2) to pairs $([\gamma_i], [\gamma_{i+1}])$ and Case (E1) to pairs $([\gamma_{i+1}], [\gamma_{i+2}])$ for $i=1, 3, \dots, m-3$, recursively, and finally Case (E2) to the pair $([\gamma_{m-1}],[\gamma_m])$, we get $x_{[\beta]} (r s + \bar{r} \bar{s} + \bar{\bar{r}} \bar{\bar{s}}) \in H$ for every $r \in R$. Now by Lemma \ref{theta Rz}, we have $x_{[\beta]}(t) \in H,$ for every $t \in I_{[\beta]}$.
            
            \item If $[\gamma] \sim A_1^3$ and $[\beta] \sim A_1^3$ then $m$ must be odd and $m \geq 3$. We first show that $x_{[\gamma_3]}(t) \in H$ for every $t \in I$. Note that $[\gamma_3] \sim A_1^3$. By applying Case (E2) to the pair $([\gamma_{1}], [\gamma_{2}])$, we get $x_{[\gamma_2]} (r s + \bar{r} \bar{s} + \bar{\bar{r}} \bar{\bar{s}}) \in H$ for every $r \in R$. Next, by applying Case (E1) on the pair $([\gamma_2], [\gamma_3])$, we conclude that $x_{[\gamma_3]} (r s + \bar{r} \bar{s} + \bar{\bar{r}} \bar{\bar{s}}) \in H$ for every $r \in R$. Finally, by applying Case (E3) on the pair $([\gamma_1], [\gamma_3])$, we obtain $x_{[\gamma_3]}(r s \pm r' s') \in H$ for every $r \in R$ (where $a'$ denotes $\bar{a}$ or $\bar{\bar{a}}$). 
            
            Suppose $x_{[\gamma_3]}(r s + r' s') \in H$. Then $x_{[\gamma_3]}(r'' s'') \in H$ for every $r \in R$ (where $a'' = \bar{a}$ if $a' = \bar{\bar{a}}$ and vice-versa). By reversing the roles of $[\gamma_1]$ and $[\gamma_3]$ together with roles of $s$ and $s''$, we get $x_{[\gamma_1]}(r s + \bar{r} \bar{s} + \bar{\bar{r}} \bar{\bar{s}}) \in H$ and $x_{[\gamma_1]}(r's') \in H$. Since $x_{[\gamma_1]}(s) \in H$, it follows that $x_{[\gamma_1]}(\bar{s})$ and $x_{[\gamma_1]}(\bar{\bar{s}})$ are also in $H$. Applying the same process again with $\bar{s}$ (resp., $\bar{\bar{s}}$), we obtain $x_{[\gamma_3]}(r \bar{s}'') \in H$ (resp., $x_{[\gamma_3]}(r \bar{\bar{s}}'') \in H$) for every $r \in R$. In particular, we have $x_{[\gamma_3]}(t) \in H$ for every $t \in I_{[\gamma_3]}$. 
            
            Now, suppose $x_{[\gamma_3]}(rs - r's') \in H$. Then $x_{[\gamma_3]}(2rs + r'' s'') \in H$ for every $r \in R$. By reversing the roles of $[\gamma_1]$ and $[\gamma_3]$ together with roles of $s$ and $2s + s''$, we obtain $x_{[\gamma_1]}((2r+r')s + (2\bar{r}+ \bar{r'}) \bar{s} + (2 \bar{\bar{r}} + \bar{\bar{r'}}) \bar{\bar{s}}) \in H$ and $x_{[\gamma_1]}(4s + 4s'' + s') \in H$. Since the map $r \mapsto 2r + r'$ from $R$ to itself is surjective, we have $x_{[\gamma_1]}(rs + \bar{r}\bar{s} + \bar{\bar{r}} \bar{\bar{s}}) \in H$ for every $r \in R$. Consequently, $x_{[\gamma_1]} (3s') \in H$, and hence $x_{[\gamma_1]}(\bar{s})$ and $x_{[\gamma_1]}(\bar{\bar{s}}) \in H$ (as $1/3 \in R$). Applying the same process again with $\bar{s}$ and $\bar{\bar{s}}$, we get $x_{[\gamma_3]}(3rs) = x_{[\gamma_3]}(rs + \bar{r} \bar{s} + \bar{\bar{r}} \bar{\bar{s}}) x_{[\gamma_3]} (rs - \bar{r} \bar{s}) x_{[\gamma_3]} (rs - \bar{\bar{r}} \bar{\bar{s}}) \in H$ for every $r \in R$. Thus, we get $x_{[\gamma_3]}(rs) \in H$ for every $r \in R$. Similarly, we can show that $x_{[\gamma_3]}(r \bar{s}) \in H$ and $x_{[\gamma_3]}(r \bar{\bar{s}}) \in H$ for every $r \in R$. In particular, we have $x_{[\gamma_3]}(t) \in H$ for every $t \in I_{[\gamma_3]}$, as required. 
            
            Now if $m=3$, then we are done. If not, we repeat this process for the pair $([\gamma_{i}], [\gamma_{i+2}])$ for every $i = 3, \dots, m-2$ to obtain desired result.
        \end{enumerate}
    \end{enumerate}
    This completes the proof of our lemma.
\end{proof}

%%%%%%%%%%%%%%%%%%%%%%%%%%%%%%%%%%%%%%%%%%%%%%%%%%

\vspace{2mm}

\noindent \textit{Proof of Proposition \ref{prop:E(R,J) subset of H}.} Let $J$ be as in the hypothesis of Proposition \ref{prop:E(R,J) subset of H}. Let $t, u \in J$. Then there exists $[\alpha], [\beta] \in \Phi_\rho$ such that $t \in J_{[\alpha]}(H)$ and $u \in J_{[\beta]}(H)$. Let $[\gamma] \in \Phi_\rho$ be such that it is either of type $A_1^2$ or $A_1^3$. By Proposition \ref{z to Rz}, for every $r \in R$ we have $x_{[\gamma]}(rt), x_{[\gamma]}(\bar{t}), x_{[\gamma]}(u) \in H$ and hence $x_{[\gamma]} (t + u) \in H$. Therefore, if $t,u \in J$ and $r \in R$ then we have $t+u, rt, \bar{t} \in J$. Thus $J$ is a $\theta$-invariant ideal of $R$. Now for the second assertion, it follows from Proposition \ref{z to Rz} that $E'_\sigma (J) \subset H$. Since $H$ is normalized by $E'_\sigma (R)$, we conclude that $E'_\sigma (R,J) \subset H$, as desired. \qed

%%%%%%%%%%%%%%%%%%%%%%%%%%%%%%%%%%%%%%%%%%%%%%%%%%
%Section: Proof of Proposition \ref{prop:U(R) cap H subset U(J)}
%%%%%%%%%%%%%%%%%%%%%%%%%%%%%%%%%%%%%%%%%%%%%%%%%%

\section{Proof of Proposition \ref{prop:U(R) cap H subset U(J)}} \label{sec:Pf of prop 2}

%%%%%%%%%%%%%%%%%%%%%%%%%%%%%%%%%%%%%%%%%%%%%%%%%%

Let $\Phi_\rho$ be an irreducible root system. We fix a simple system $\Delta_\rho = \{ [\alpha_1], \dots, [\alpha_l] \}$ of $\Phi_\rho$. Recall that, for a root $[\alpha] = \sum_{i=1}^l m_{i} [\alpha_i] \in \Phi_\rho$, we defined $ht([\alpha])= \sum_{i=1}^l m_i$. We say a root $[\beta]$ is \textit{highest} if the height of $[\beta]$ is maximal, i.e., $ht([\beta]) = \max \{ ht([\alpha]) \mid [\alpha] \in \Phi_{\rho} \}$. Note that there is a unique highest root in an irreducible root system and it is a long positive root. 
Therefore we sometimes call it \textit{highest long root}. Similarly, we say $[\gamma]$ is a \textit{highest short root} if $ht([\gamma]) = \max \{ ht([\alpha]) \mid [\alpha] \in \Phi_{\rho} \text{ and } [\alpha] \text{ is short root} \}$. There is a unique highest short root in an irreducible root system.

\begin{lemma}\label{lemma:U cap H}
    \normalfont
    Let $x: = \prod_{[\alpha] \in \Phi_\rho^{+}} x_{[\alpha]}(t_{[\alpha]}) \in U(R) \cap H$ with $t_{[\alpha]} \in R_{[\alpha]}$ (the product is taken over disjoint roots in any fixed order). Then $x_{[\alpha]}(t_{[\alpha]}) \in H$ for all $[\alpha] \in \Phi_\rho^{+}$.
\end{lemma}

\begin{proof}
    The proof closely resembles that of Lemma 3.1 and Proposition 1 in Section 3 of \cite{EA3}. However, the calculations presented here are distinct from those in \cite{EA3}. For the convenience of the reader, we provide the full proof below.

    \vspace{2mm}
    
    \noindent \textbf{Case A. $\Phi_\rho \sim {}^2 A_3$:} In this case, after the twist, $\Phi_\rho$ becomes a root system of type $B_2$. Let $[\alpha]$ and $[\beta]$ be the simple roots, with $[\alpha]$ being the long root. We first claim that 
    \begin{equation}\label{eq_a}
        \textit{if } x = x_{[\alpha] + [\beta]}(t) x_{[\alpha]+ 2 [\beta]}(u) \in H, \textit{ then } x_{[\alpha] + [\beta]}(t), x_{[\alpha]+ 2 [\beta]}(u) \in H. 
    \end{equation}
    For any $r \in R_{[\alpha]} = R_\theta$, we have
    $$ H \ni [x_{-[\alpha]}(r), x] = x_{[\beta]} (\pm rt) x_{[\alpha] + 2 [\beta]} (\pm r t \bar{t}).$$
    Since $x \in H$ then so is $x^{-1} = x_{[\alpha] + [\beta]}(-t) x_{[\alpha]+ 2 [\beta]}(-u)$. Therefore we can replacing $t$ and $u$ by $-t$ and $-u$, respectively, and we get $x_{[\beta]} (\pm r(-t)) x_{[\alpha] + 2 [\beta]} (\pm r t \bar{t}) \in H$. But then 
    $$x_{[\beta]}(\pm 2rt) = \{ x_{[\beta]} (\pm rt) x_{[\alpha] + 2 [\beta]} (\pm r t \bar{t}) \} \{ x_{[\beta]} (\pm r(-t)) x_{[\alpha] + 2 [\beta]} (\pm r t \bar{t}) \}^{-1} \in H.$$
    Put $r = \pm 1/2$, we get $x_{[\beta]}(t) \in H$. By Proposition \ref{z to Rz}, we get $x_{[\alpha] + [\beta]}(t) \in H$ and hence $x_{[\alpha]+ 2 [\beta]}(u) \in H$. This proves (\ref{eq_a}).
    Now let $x = x_{[\beta]}(t) x_{[\alpha]}(u) x_{[\alpha] + [\beta]} (v) x_{[\alpha] + 2[\beta]}(w) \in H$. Then
    $$ [x_{[\alpha]}(1), x] = x_{[\alpha] + [\beta]} (\pm t) x_{[\alpha]+2[\beta]} (\pm t \bar{t}) \in H.$$ 
    By (\ref{eq_a}), we have $x_{[\alpha] + [\beta]} (\pm t) \in H$. Again by Proposition \ref{z to Rz}, we get $x_{[\beta]} (t) \in H$. Consequently, $x_1:= x_{[\alpha]}(u) x_{[\alpha] + [\beta]} (v) x_{[\alpha] + 2[\beta]}(w) \in H$. Now,
    $$ [x_{[\beta]}(1), x_1] = x_{[\alpha] + [\beta]} (\pm u) x_{[\alpha]+2[\beta]} (\pm u \pm (v + \bar{v})) \in H.$$
    Again, by (\ref{eq_a}), $x_{[\alpha] + [\beta]} (\pm u) \in H$, and hence $x_{[\alpha]} (u) \in H$ (by Proposition \ref{z to Rz}). But then $$x_{[\alpha] + [\beta]} (v) x_{[\alpha] + 2[\beta]}(w) \in H.$$ Finally, by (\ref{eq_a}), $x_{[\alpha] + [\beta]} (v) \in H$ and $ x_{[\alpha] + 2[\beta]}(w) \in H$, as desired.
    
    \vspace{2mm}
    
    \noindent \textbf{Case B. $\Phi_\rho \sim {}^2 A_4$:} In this case, after the twist, $\Phi_\rho$ becomes a root system of type $B_2$. Let $[\alpha]$ and $[\beta]$ be the simple roots, with $[\alpha]$ being the long root. We first claim that 
    \begin{equation}\label{eq_b}
        \textit{if } x = x_{[\alpha] + [\beta]}(t) x_{[\alpha]+ 2 [\beta]}(u) \in H, \textit{ then } x_{[\alpha] + [\beta]}(t), x_{[\alpha]+ 2 [\beta]}(u) \in H. 
    \end{equation}
    For any $r = (r_1, r_2) \in R_{[\alpha]} = \mathcal{A}(R)$, we have
    \begin{equation}\label{eq_c}
        H \ni [x_{[\beta]}(r), x] = x_{[\alpha] + 2 [\beta]} (\pm r'_1 t'_1),
    \end{equation}
    where $r'_1$ denotes either $r_1$ or $\bar{r_1}$, similar for $t'_1$.
    Now $$ H \ni [x_{-[\beta]}(r), x] = x_{[\alpha]}(\pm r'_1 t'_1) x_{[\alpha]}(\pm r'_2 u) x_{[\alpha] + [\beta]}(\pm r'_1 u', r'_2 u \bar{u}).$$
    By (\ref{eq_c}) and Proposition \ref{z to Rz}, $x_{[\alpha]}(\pm r'_1 t'_1) \in H$. But then $$x_{[\alpha]}(\pm r'_2 u) x_{[\alpha] + [\beta]}(\pm r'_1 u', r'_2 u \bar{u}) \in H.$$ 
    Now if we put $(r_1, r_2) = (1,1/2)$, then we get $x_{[\alpha]}(\pm u/2) x_{[\alpha] + [\beta]} (\pm u', u \bar{u}/2) \in H$ and if we put $(r_1, r_2) = (-1, 1/2)$, then we get $x_{[\alpha]}(\pm u/2) x_{[\alpha] + [\beta]} (\pm (-u'), u \bar{u}/2) \in H$. But then
    $$x_{[\alpha]}(\pm u) = \{ x_{[\alpha]}(\pm u/2) x_{[\alpha] + [\beta]} (\pm u', u \bar{u}/2) \} \{ x_{[\alpha]}(\pm u/2) x_{[\alpha] + [\beta]} (\pm (-u'), u \bar{u}/2) \} \in H.$$
    Again by Proposition \ref{z to Rz}, we have $x_{[\alpha] + 2[\beta]} (u) \in H$ and hence $x_{[\alpha]+ [\beta]}(t) \in H$. This proves (\ref{eq_b}).
    
    Now let $x = x_{[\beta]}(s) x_{[\alpha]}(t) x_{[\alpha] + [\beta]} (u) x_{[\alpha] + 2[\beta]}(v) \in H$. Then
    $$ [x_{[\alpha]}(1), x] = x_{[\alpha] + [\beta]} (\pm s'_1, s'_2) x_{[\alpha] + [\beta]} (0, \pm (v - \bar{v})) x_{[\alpha]+2[\beta]} (\pm s_2) \in H,$$ where $s = (s_1, s_2) \in \mathcal{A}(R)$ and $s_1' = s_1$ or $\bar{s_1}$, similar for $s'_2$. 
    By (\ref{eq_b}), we have 
    $$ x_{[\alpha] + [\beta]} (\pm s'_1, s'_2 \pm (v - \bar{v})) \in H \text{ and } x_{[\alpha]+2[\beta]} (\pm s_2) \in H.$$ 
    Again by Proposition \ref{z to Rz}, we get $x_{[\beta]} (s)$. But then $x_1:= x_{[\alpha]}(t) x_{[\alpha] + [\beta]} (u) x_{[\alpha] + 2[\beta]}(v) \in H$. 
    Now, $$ [x_{[\beta]}(1, 1/2), x_1] = x_{[\alpha] + [\beta]} (\pm t', t \bar{t}/2 \pm (t \bar{u_1} - \bar{t}u_1)) x_{[\alpha]+2[\beta]} (\pm t \pm u'_1 ) \in H.$$
    Again by (\ref{eq_b}), $x_{[\alpha] + [\beta]} (\pm t', t \bar{t}/2 \pm (t \bar{u_1} - \bar{t}u_1)) \in H$ and hence $x_{[\alpha]} (t) \in H$ (by Proposition \ref{z to Rz}). But then $x_{[\alpha] + [\beta]} (u) x_{[\alpha] + 2[\beta]}(v) \in H$ and again by (\ref{eq_b}), $x_{[\alpha] + [\beta]} (u) \in H$ and $ x_{[\alpha] + 2[\beta]}(v) \in H$, as desired.
    
    \vspace{2mm}
    
    \noindent \textbf{Case C. $\Phi_\rho \sim {}^3 D_4$:} In this case, after the twist, $\Phi_\rho$ becomes a root system of type $G_2$. Let $[\alpha]$ and $[\beta]$ be the simple roots, with $[\alpha]$ being the long root. 
    We first claim that if $$x = x_{[\alpha] + 2[\beta]} (t) x_{[\alpha] + 3[\beta]} (u) x_{2[\alpha] + 3 [\beta]}(v) \in H,$$ then $x_{[\alpha] + 2[\beta]} (t), x_{[\alpha] + 3[\beta]} (u), x_{2[\alpha] + 3 [\beta]}(v) \in H.$
    Note that $[x_{[\alpha]}(1), x] = x_{2[\alpha] + 3[\beta]} (\pm u) \in H$ and $[x_{-[\alpha]}(1), x] = x_{[\alpha] + 3[\beta]} (\pm v) \in H$. But then, by Proposition \ref{z to Rz}, we have $x_{2[\alpha] + 3[\beta]} (v) \in H$, $x_{[\alpha] + 3[\beta]} (u) \in H$ and hence $x_{[\alpha] + 2[\beta]}(t) \in H$, which proves the claim. 
    We next claim that if $$x = x_{[\alpha] + [\beta]}(t) x_{[\alpha] + 2[\beta]}(u) x_{[\alpha] + 3[\beta]}(v) x_{2[\alpha] + 3[\beta]}(w) \in H,$$ then $x_{[\alpha] + [\beta]}(t), x_{[\alpha] + 2[\beta]}(u), x_{[\alpha] + 3[\beta]}(v), x_{2[\alpha] + 3[\beta]}(w) \in H$. For any $r \in R_\theta$, we have 
    $$y(r) = [x_{-[\alpha]}(r), x] = x_{[\beta]} (\pm rt) x_{[\alpha] + 2[\beta]}(\pm r t t') x_{[\alpha] + 3 [\beta]}(\pm r^2 t \bar{t} \bar{\bar{t}} \pm r w) x_{2[\alpha] + 3 [\beta]} (\pm r t \bar{t} \bar{\bar{t}}) \in H,$$ where $t' = \bar{t}$ or $\bar{\bar{t}}$. For any $s \in R_\theta$, we have
    \begin{align*}
        y(r,s) &= [x_{[\alpha]}(s), y(r)] = x_{[\alpha] + [\beta]} (\pm srt) x_{[\alpha] + 2[\beta]}(\pm s r^2 t t') x_{[\alpha]+3[\beta]}(\pm s r^3 t \bar{t} \bar{\bar{t}}) \\
        & \hspace{35mm} x_{2[\alpha] + 3[\beta]}(\pm s^2 r^3 t \bar{t} \bar{\bar{t}} \pm s r^2 t \bar{t} \bar{\bar{t}} \pm srw) \in H.
    \end{align*}
    Let $x_1:= y(r,s)^{-1} y (1,rs) = x_{[\alpha]+2[\beta]}(\pm s (r^2 - r)tt') x_{[\alpha]+3[\beta]}(v') x_{2[\alpha]+3[\beta]}(u') \in H$. By above we have $x_{[\alpha] + 2[\beta]}(\pm s (r^2 - r)tt') \in H$. Put $r = -1$ and $s=1/2$, we have $x_{[\alpha]+2[\beta]}(tt') \in H$. 
    But then, by Proposition \ref{z to Rz}, we have $x_{[\alpha]+2[\beta]}(\pm t t'), x_{[\alpha]+3[\beta]}(\pm t\bar{t} \bar{\bar{t}}), x_{2[\alpha] + 3[\beta]} (\pm t \bar{t} \bar{\bar{t}}) \in H$. Hence,
    $$y(1) x_{2[\alpha] + 3[\beta]} (\pm t \bar{t} \bar{\bar{t}})^{-1} x_{[\alpha]+3[\beta]}(\pm t\bar{t} \bar{\bar{t}})^{-1} x_{[\alpha]+2[\beta]}(\pm t t')^{-1} = x_{[\beta]}(\pm t) x_{[\alpha] + 3 [\beta]}(\pm w) \in H.$$
    Further, $[x_{-2[\alpha]-3[\beta]}(1), x_{[\beta]}(\pm t) x_{[\alpha] + 3[\beta]}(\pm w)] = x_{-[\alpha]} (\pm w) \in H$. By Proposition \ref{z to Rz}, we have $x_{[\alpha] + 3[\beta]}(\pm w) \in H$ and hence $x_{[\beta]}(\pm t) \in H$. Again by Proposition \ref{z to Rz}, $x_{[\alpha] + [\beta]}(t) \in H$ and hence $x_{[\alpha] + 2[\beta]}(u) x_{[\alpha] + 3[\beta]}(v) x_{2[\alpha] + 3[\beta]}(w) \in H$. By above claim, we have $x_{[\alpha] + 2[\beta]}(u)$, $x_{[\alpha] + 3[\beta]}(v)$, $x_{2[\alpha] + 3[\beta]}(w) \in H$, as desired. 
    
    Finally, let $x = x_{[\beta]}(t_1) x_{[\alpha]}(t_2) x_{[\alpha]+[\beta]}(t_3) x_{[\alpha]+2[\beta]}(t_4) x_{[\alpha] + 3[\beta]}(t_5) x_{2[\alpha]+3[\beta]} (t_6) \in H$. Note that 
    $$[x_{[\alpha]}(1), x] = x_{[\alpha] + [\beta]} (\pm t_1) x_{[\alpha] + 2[\beta]}(\pm t_1 t'_1) x_{[\alpha]+3[\beta]} (\pm t_1 \bar{t}_1 \bar{\bar{t}}_1) x_{2[\alpha]+3[\beta]}(\pm t_1 \bar{t}_1 \bar{\bar{t}}_1) \in H,$$
    where $t'_1 = \bar{t}_1$ or $\bar{\bar{t}}_1$. By above claim we have $x_{[\alpha]+[\beta]}(\pm t_1) \in H$, hence $x_{[\beta]}(t_1) \in H$ (by Proposition \ref{z to Rz}). Therefore, we have $$x_1:= x_{[\alpha]}(t_2) x_{[\alpha]+[\beta]}(t_3) x_{[\alpha]+2[\beta]}(t_4) x_{[\alpha] + 3[\beta]}(t_5) x_{2[\alpha]+3[\beta]} (t_6) \in H.$$
    Note that $$ [x_{[\beta]}(1), x_1] = x_{[\alpha] + [\beta]}(\pm t_2) x_{[\alpha]+2[\beta]}(s_4) x_{[\alpha] + 3[\beta]}(s_5) x_{2[\alpha]+3[\beta]} (s_6) \in H,$$ for some $s_4, s_5, s_6 \in R$. Again by above claim we have $x_{[\alpha] + [\beta]} (\pm t_2) \in H$ and hence $x_{[\alpha]}(t_2) \in H$. Thus $x_{[\alpha]+[\beta]}(t_3) x_{[\alpha]+2[\beta]}(t_4) x_{[\alpha] + 3[\beta]}(t_5) x_{2[\alpha]+3[\beta]} (t_6) \in H$. But again by above claim we have $x_{[\alpha]+[\beta]}(t_3), x_{[\alpha]+2[\beta]}(t_4), x_{[\alpha] + 3[\beta]}(t_5), x_{2[\alpha]+3[\beta]} (t_6) \in H,$ as desired.
    
    \vspace{2mm}
    
    \noindent \textbf{Case D. The rank of $\Phi_\rho > 2$:} Let $[\beta]$ be the highest long root in $\Phi_\rho$ and $[\beta']$ be the highest short root in $\Phi_\rho$. For $x = \prod_{[\alpha] \in \Phi^{+}_\rho} x_{[\alpha]} (t_{[\alpha]})$ (product is taken over some fixed order on the roots), we set $\Phi(x) = \{ [\alpha] \in \Phi^{+}_\rho \mid t_{[\alpha]} \neq 0 \}$. We use induction on $n$ to prove the following statement.

    \vspace{2mm}

    \noindent \textbf{($P_n$):} If $\Phi (x)$ only contains the roots $[\beta], [\beta']$ or $[\alpha]$ with $ht ([\alpha]) \geq ht([\beta]) - n + 1$. Then the conclusion of the lemma holds.

    \vspace{2mm}

    \noindent \textit{Proof of $(P_1)$}: We will show that if $x= x_{[\beta]}(t) x_{[\beta']}(t') \in H,$ then all factors of $x$ is contained in $H$. The subsystem generated by $[\beta]$ and $[\beta']$ is of the type ${}^2 A_{3}$ if $\Phi_\rho \sim {}^2 A_{2n-1} \ (n \geq 3), {}^2 D_{n} \ (n \geq 4)$ or ${}^2 E_6$ and is of type ${}^2 A_4$ if $\Phi_\rho \sim {}^2 A_{2n} \ (n \geq 3)$. Thus we are done by Case $A$ and Case $B$, above.

    \vspace{2mm}

    \noindent \textit{Proof of $(P_{n}) \implies (P_{n+1})$}: Assume that $(P_{n})$ holds, that is, assume that if $\Phi(x)$ only contains the roots $[\alpha]$ with $ht([\alpha]) \geq ht ([\beta]) - n + 1$, $[\beta]$ or $[\beta']$, then all factors of $x$ are contained in $H$. To prove $(P_{n+1})$, let $x \in H$ be such that $\Phi(x)$ only contains the roots  $[\alpha]$ with $ht([\alpha]) \geq ht ([\beta]) - n$, $[\beta]$ or $[\beta']$. It is enough to show that if $[\delta] \in \Phi (x)$ be such that $ht([\delta]) = ht ([\beta]) - n$ and $[\delta] \neq [\beta], [\beta']$ then $x_{[\delta]}(t_{[\delta]}) \in H$. Note that there exists a simple root $[\alpha_i] \in \Delta_\rho$ such that $[\delta] + [\alpha_i] \in \Phi_\rho$ and $[\delta] - [\alpha_i] \not \in \Phi_\rho$ (see 3.6 of \cite{EA2}). Let $\Phi'$ be the subsystem generated by $[\alpha_i]$ and $[\delta]$.
    \begin{enumerate}
        \item Suppose $\Phi'$ is of type $A_2$. In this case, the pair $[\alpha_i]$ and $[\delta]$ is either of the type $(b-i)$ or of the type $(b-ii)$. Take
        \begin{small}
            \[
            H \ni [x_{[\alpha_i]}(1), x] = \begin{cases}
                x_{[\delta] + [\alpha_i]} (\pm t_{[\delta]}) x' & \text{if } [\alpha_i], [\delta] \text{ is of type } (b-i), \\
                x_{[\delta] + [\alpha_i]} (\pm t_{[\delta]}) x' \text{ or } x_{[\delta] + [\alpha_i]} (\pm \bar{t}_{[\delta]}) x' & \text{if } [\alpha_i], [\delta] \text{ is of type } (b-ii);
            \end{cases}
            \]
        \end{small}where $x'$ is a product of elements $x_{[\alpha]}(t_{[\alpha]})$ with $[\alpha] \neq [\delta] + [\alpha_i]$ and $ht ([\alpha]) > ht ([\delta])$. Hence, by $(P_n)$, we have $x_{[\delta]+[\alpha_i]} (\pm t_{[\delta]})$ or $x_{[\delta]+[\alpha_i]} (\pm \bar{t}_{[\delta]}) \in H$. But then, by Proposition \ref{z to Rz}, $x_{[\delta]} (t_{[\delta]}) \in H$.

        \item Suppose $\Phi'$ is of type $B_2$ and $[\delta]$ is a short root. In this case, the pair $[\alpha_i]$ and $[\delta]$ is of the type $(d-i)$ if $\Phi_\rho \sim {}^2 A_{2n-1}, {}^2 D_{n+1}$ or ${}^2 E_6$ and of the type $(d-ii)$ if $\Phi_\rho \sim {}^2A_{2n}$ (with $[\alpha_i]$ being the long root). Take 
        \[
            H \ni [x_{[\alpha_i]}(1), x] = \begin{cases}
                x_{[\delta] + [\alpha_i]} (t_{[\delta]}) x' & \text{if } \Phi_\rho \sim {}^2 A_{2n-1}, {}^2 D_{n+1} \text{ or } {}^2 E_6, \\
                x_{[\delta] + [\alpha_i]} (s_{[\delta]}) x' & \text{if } \Phi_\rho \sim {}^2 A_{2n};
            \end{cases}
        \]
        where $x'$ is a product of elements $x_{[\alpha]}(t_{[\alpha]})$ with $[\alpha] \neq [\delta] + [\alpha_i]$, $ht ([\alpha]) > ht ([\delta])$ and $s_{[\delta]} = (\pm t_1, t_2)$ or $(\pm \bar{t_1}, \bar{t_2})$ if $t_{[\delta]}= (t_1, t_2).$ Hence, by $(P_n)$, we have $x_{[\delta]+[\alpha_i]} (t_{[\delta]})$ or $x_{[\delta]+[\alpha_i]} (s_{[\delta]}) \in H$. But then, by Proposition \ref{z to Rz}, $x_{[\delta]} (t_{[\delta]}) \in H$. 

        \item Suppose $\Phi'$ is of type $B_2$ and $[\delta_i]$ is a long root.  In this case, the pair $[\delta]$ and $[\alpha_i]$ is of the type $(d-i)$ if $\Phi_\rho \sim {}^2 A_{2n-1}, {}^2 D_{n+1}$ or ${}^2 E_6$ and of the type $(d-ii)$ if $\Phi_\rho \sim {}^2A_{2n}$. Take
        \[
            H \ni \begin{cases}
                [x_{[\alpha_i]}(1), x] = x_{[\delta] + [\alpha_i]} (\pm t_{[\delta]}) x' & \text{if } {}^2 A_{2n-1}, {}^2 D_{n+1} \text{ or } \Phi_\rho \sim {}^2 E_6, \\
                [x_{[\alpha_i]}(1, 1/2), x] = x_{[\delta] + [\alpha_i]} (\pm t'_{[\delta]}, t_{[\delta]} \bar{t}_{[\delta]}/2) x' & \text{if } \Phi_\rho \sim {}^2 A_{2n};
            \end{cases}
        \]
        where $x'$ is a product of elements $x_{[\alpha]}(t_{[\alpha]})$ with $[\alpha] \neq [\delta] + [\alpha_i]$, $ht ([\alpha]) > ht ([\delta])$ and $t'_{[\delta]} = t_{[\delta]}$ or $\bar{t}_{[\delta]}$. Hence, by $(P_n)$, we have $x_{[\delta]+[\alpha_i]} (\pm t_{[\delta]})$ or $x_{[\delta] + [\alpha_i]} (\pm t'_{[\delta]}, t_{[\delta]} \bar{t}_{[\delta]}/2) \in H$. But then, by Proposition \ref{z to Rz}, $x_{[\delta]} (t_{[\delta]}) \in H$.
    \end{enumerate}
    This proves the lemma.
\end{proof}

%%%%%%%%%%%%%%%%%%%%%%%%%%%%%%%%%%%%%%%%%%%%%%%%%%

We labelled the simple roots $[\alpha_1], [\alpha_2], \dots, [\alpha_l]$ from one end of the Dynkin diagram to the other end such that 
\[
    [\alpha_1] \sim \begin{cases}
        A_1^2 & \text{if } \Phi_\rho \sim {}^2A_{n} \ (n \geq 3); \\
        A_1 & \text{if } \Phi_\rho \sim {}^2D_{n} \ (n \geq 4), {}^3D_{4} \text{ or } {}^2E_{6}. \\
        
    \end{cases} 
\]
Let $[\beta]$ be the highest root in $\Phi_\rho$. Note that there is a unique simple root $[\gamma] \in \Delta_\rho$ such that $\langle [\beta], [\gamma] \rangle \neq 0.$ 
The following table give us the precious values of $[\beta]$ and $[\gamma]$:
\begin{center}
    \begin{tabular}{c|c|c}
        \textbf{Type of $\Phi_\rho$} & \textbf{$[\beta]$} & \textbf{$[\gamma]$} \\
        \hline
        ${}^2 A_{2n-1} \ (n \geq 2)$ & $2 [\alpha_1] + 2 [\alpha_2] + \dots + 2 [\alpha_{n-1}] + [\alpha_n]$ & $[\alpha_1]$ \\
        ${}^2 A_{2n} \ (n \geq 2)$ & $[\alpha_1] + 2[\alpha_2] + 2 [\alpha_3] + \dots + 2[\alpha_{n}]$ & $[\alpha_2]$ \\
        ${}^2 D_{n} \ (n \geq 4)$ & $[\alpha_1] + 2[\alpha_2] + 2 [\alpha_3] + \dots + 2[\alpha_{n-1}]$ & $[\alpha_2]$ \\
        ${}^3 D_{4}$ & $2[\alpha_1] + 3[\alpha_2]$ & $[\alpha_1]$ \\
        ${}^2 E_{6}$ & $2 [\alpha_1] + 3 [\alpha_2] + 4 [\alpha_3] + 2 [\alpha_4]$ & $[\alpha_1]$ \\
    \end{tabular}
\end{center}

%%%%%%%%%%%%%%%%%%%%%%%%%%%%%%%%%%%%%%%%%%%%%%%%%%

\begin{lemma}\label{lemma:UTV cap H}
    \normalfont
    Let $[\beta]$ and $[\gamma]$ be as above. If $z = x_{[\gamma]}(t) xhy \in U_\sigma (R) T_\sigma (R) U^{-}_\sigma (R) \cap H,$ where $x_{[\gamma]}(t) x \in U_\sigma (R),$ $x$ is a product of elements $x_{[\alpha]}(t_{[\alpha]})$ with $[\alpha] \neq [\gamma], [\alpha] \in \Phi^{+}_\rho$ and $t_{[\alpha]} \in R_{[\alpha]}, h \in T_\sigma (R)$ and $y \in U^{-}_\sigma (R)$. Then $x_{[\gamma]}(t) \in H$.
\end{lemma}

\begin{proof}
    We write ${}^a b$ for the conjugate $a b a^{-1}$. 
    
    \vspace{2mm}
    
    \noindent \textbf{Case A. The rank of $\Phi_\rho > 2$:} 
    Let 
    \begin{small}
        \begin{align*}
            H \ni z_1 &= [x_{-[\gamma]}(1), z] \\
            &= [x_{-[\gamma]} (1), x_{[\gamma]}(t)] \{ {}^{x_{[\gamma]}(t)} [x_{-[\gamma]}(1), x] \} \{ {}^{x_{[\gamma]}(t) x} [x_{-[\gamma]}(1), h] \} \{ {}^{x_{[\gamma]}(t) x h} [x_{-[\gamma]}(1), y] \} \\
            &= {}^{x_{[\gamma]}(t) x} \{ \{ {}^{x^{-1} x_{[\gamma]}(t)^{-1}} [x_{-[\gamma]} (1), x_{[\gamma]}(t)] \} \{ {}^{x^{-1}} [x_{-[\gamma]}(1), x] \} \{ [x_{-[\gamma]}(1), h] \} \{ {}^{h} [x_{-[\gamma]}(1), y] \} \}.
        \end{align*}
    \end{small}
    Note that, ${}^{x^{-1} x_{[\gamma]}(t)^{-1}} [x_{-[\gamma]} (1), x_{[\gamma]}(t)] = {}^{x^{-1}} [x_{[\gamma]} (t)^{-1}, x_{-[\gamma]}(1)] = [x_{[\gamma]} (t)^{-1}, x_{-[\gamma]}(1)] x'$ with $x' \in U_{\sigma}(R),$ $[x_{-[\gamma]}(1),h] = x_{-[\gamma]}(a)$ for some $a \in R_{[\gamma]}$ and $x^{-1} [x_{-[\gamma]} (1), x] = [x^{-1}, x_{-[\gamma]} (1)] \in U_\sigma (R)$. Set $u_1 = [x_{[\gamma]}(t)^{-1}, x_{-[\gamma]} (1)], x_1= x' [x^{-1}, x_{-[\gamma]} (1)]$ and $y_1 = x_{-[\gamma]}(a) \{ {}^{h} [x_{-[\gamma]}(1), y] \}$. 
    Thus we have $z_1 = {}^{x_{[\gamma]}(t) x}(u_1 x_1 y_1)$ and since $z_1 \in H$ then so is $u_1 x_1 y_1$. 
    Observe that $x_1 \in U_\sigma (R)$ is a product of elements $x_{[\alpha]}(t_{[\alpha]})$ with $[\alpha] \neq [\gamma]$, and $y_1 \in U^{-}_\sigma (R)$ is a product of elements $x_{-[\alpha]}(s_{[\alpha]})$ with $m_{[\gamma]}([\alpha]) \geq 1$. 
    In this case, there exists a root $[\delta] \in \Delta_\rho$ such that $\{[\gamma], [\delta]\}$ is a base of a subsystem of $\Phi_\rho$ of type $A_2$ (note that $[\gamma] \sim A_1$ or $A_1^2$ then so is $[\delta]$). Now let
    \begin{align*}
        H \ni z_2 &= [x_{[\delta]}(1), u_1 x_1 y_1] \\
        &= \{ [x_{[\delta]}(1), u_1] \} \{ {}^{u_1} [x_{[\delta]}(1), x_1] \} \{ {}^{u_1 x_1} [x_{[\delta]}(1), y_1] \} \\
        &= {}^{u_1 x_1} \{ \{ {}^{x_1^{-1} u_1^{-1}} [x_{[\delta]}(1), u_1] \} \{ {}^{x_1^{-1}} [x_{[\delta]}(1), x_1] \} \{ [x_{[\delta]}(1), y_1] \} \}.
    \end{align*}
    Note that, 
    \begin{align*}
        {}^{x_1^{-1} u_1^{-1}} [x_{[\delta]}(1), u_1] &= {}^{x_1^{-1}} [u_1^{-1}, x_{[\delta]}(1)] \\
        &= x_1^{-1} (u_1^{-1} x_{[\delta]}(1) u_1 x_{[\delta]}(1)^{-1}) x_1 \\
        &= {}^{x_1^{-1} x_{-[\gamma]}(1)} \{ x_{[\delta]} (\pm t') x_{[\gamma] + [\delta]} (\pm t'^{2}) \},
    \end{align*}
    where $t' = t$ or $\bar{t}$. Observe that, there exists $x'_1 \in U_\sigma (R)$ such that $x_1^{-1} x_{-[\gamma]}(1) = x_{-[\gamma]}(1) x'_1$. Then,
    \begin{align*}
        z_2 &= {}^{u_1 x_1} \{ \{ {}^{x_1^{-1} u_1^{-1}} [x_{[\delta]}(1), u_1] \} \{ {}^{x_1^{-1}} [x_{[\delta]}(1), x_1] \} \{ [x_{[\delta]}(1), y_1] \} \} \\
        &= {}^{u_1 x_1 x_{-[\gamma]}(1)} \{ \{ {}^{x'_1} (x_{[\delta]}(\pm t') x_{[\gamma]+[\delta]}(\pm t'^2)) \} \{ {}^{x_{-[\gamma]}(1)^{-1} x_1^{-1}} [x_{[\delta]}(1), x_1] \} \{ {}^{x_{-[\gamma]}(1)^{-1}} [x_{[\delta]}(1), y_1] \} \} \\
        &= {}^{u_1 x_1 x_{-[\gamma]}(1)} \{ \{ x_{[\delta]}(\pm t') x_{[\gamma]+[\delta]}(\pm t'^2) x''_1 \} \{ {}^{x_{-[\gamma]}(1)^{-1} x_1^{-1}} [x_{[\delta]}(1), x_1] \} \{ {}^{x_{-[\gamma]}(1)^{-1}} [x_{[\delta]}(1), y_1] \} \},
    \end{align*}
    where $x''_1 \in U_\sigma (R)$. Set $x_2 = x''_1 \{ {}^{x_{-[\gamma]}(1)^{-1} x_1^{-1}} [x_{[\delta]}(1), x_1] \}$ and $y_2 = {}^{x_{-[\gamma]}(1)^{-1}} [x_{[\delta]}(1), y_1]$. Note that $x_2 \in U_\sigma (R)$ is a product of elements $x_{[\alpha]}(t_{[\alpha]})$ with $[\alpha] \in \Phi^{+}_\rho, [\alpha] \neq [\gamma], [\delta], [\gamma]+[\delta]$ and $y_2  \in U^{-}_\sigma (R)$ is a product of elements $x_{-[\alpha]}(s_{[\alpha]})$ with $[\alpha] \in \Phi^{+}_\rho$ and $m_{[\gamma]} ([\alpha]) \geq 1$. But then
    $$z_2 = [x_{[\delta]}(1), u_1 x_1 y_1] = {}^{u_1 x_1 x_{-[\gamma]}(1)} \{ x_{[\delta]}(\pm t') x_{[\gamma]+[\delta]}(\pm t'^2) x_2 y_2 \}.$$
    Since $z_2 \in H$ then so is $x_{[\delta]}(\pm t') x_{[\gamma]+[\delta]}(\pm t'^2) x_2 y_2$. 
    Let
    \begin{small}
        \begin{align*}
            H \ni z_3 &= [x_{-[\gamma]}(1), x_{[\delta]}(\pm t') x_{[\gamma]+[\delta]}(\pm t'^2) x_2 y_2] \\
            &= [x_{-[\gamma]}(1), x_{[\delta]}(\pm t')] \{ {}^{x_{[\delta]}(\pm t')} [x_{-[\gamma]}(1), x_{[\gamma]+[\delta]}(\pm t'^2)] \} \{ {}^{x_{[\delta]}(\pm t') x_{[\gamma]+[\delta]}(\pm t'^2)} [x_{-[\gamma]}(1), x_2] \} \\
            & \hspace{20mm} \{ {}^{x_{[\delta]}(\pm t') x_{[\gamma]+[\delta]}(\pm t'^2) x_2} [x_{-[\gamma]}(1), y_2] \} 
        \end{align*}
    \end{small}
    Note that $[x_{-[\gamma]}(1), x_{[\delta]}(\pm t')]=1$ and $y_3 := [x_{-[\gamma]}(1), y_2] = 1$. The formal assertion is clear. To see the latter assertion, observe that the highest root $[\beta]$ is the only root with $m_{[\gamma]}([\beta]) = 2,$ therefore $x_{-[\beta]}(u)$ is the only possible factor of $y_3$. 
    Assume $y_3 = x_{-[\beta]}(u) \neq 1,$ then $-[\beta] = -[\alpha] - [\gamma]$ for some $[\alpha]$ such that $x_{-[\alpha]}(v)$ is factor of $y_2$. Hence, $-[\alpha] = -[\alpha_1] + [\delta]$ or $(-[\alpha_1]+[\delta]) -[\alpha_2]$ for some $[\alpha_1], [\alpha_2]$ with $m_{[\gamma]}([\alpha_i]) \geq 1 \ (i=1,2).$ But it is impossible for the first case, as $[\alpha_1] = [\beta] - [\gamma] + [\delta]$ is never a root, and for the second case, $2 \geq m_{[\gamma]}([\alpha_1] - [\delta] + [\alpha_2]) = m_{[\gamma]} ([\beta] - [\gamma]) = 1$. Which proves the claim. Finally, we have
    \begin{align*}
        z_3 &= \{ {}^{x_{[\delta]}(\pm t')} [x_{-[\gamma]}(1), x_{[\gamma]+[\delta]}(\pm t'^2)] \} \{ {}^{x_{[\delta]}(\pm t') x_{[\gamma]+[\delta]}(\pm t'^2)} [x_{-[\gamma]}(1), x_2] \} \\
        &= {}^{x_{[\delta]}(\pm t') x_{[\gamma]+[\delta]}(\pm t'^2)} \{ {}^{ x_{[\gamma]+[\delta]}(\pm t'^2)^{-1}} [x_{-[\gamma]}(1), x_{[\gamma]+[\delta]}(\pm t'^2)] \} \{ [x_{-[\gamma]}(1), x_2] \} \\
        &= {}^{x_{[\delta]}(\pm t') x_{[\gamma]+[\delta]}(\pm t'^2)} \{ [x_{[\gamma]+[\delta]}(\pm t'^2)^{-1}, x_{-[\gamma]}(1)]  [x_{-[\gamma]}(1), x_2] \} \\
        &= {}^{x_{[\delta]}(\pm t') x_{[\gamma]+[\delta]}(\pm t'^2)} \{ x_{[\delta]} (\pm t'^2) x_3 \},
    \end{align*}
    where $x_3 = [x_{-[\gamma]}(1), x_2] \in U_\sigma (R)$. Since $z_3 \in H$, then so is $x_{[\delta]}(\pm t'^2) x_3$. Therefore, by Lemma \ref{lemma:U cap H}, we have $x_{[\delta]} (\pm t'^2) \in H$. But then, by Proposition \ref{z to Rz}, $x_{[\gamma] + [\delta]} (\pm t'^2) \in H$. Hence $x_{[\gamma] + [\delta]} (\pm t'^2)^{-1} \{ x_{[\delta]}(\pm t') x_{[\gamma]+[\delta]}(\pm t'^2) x_2 y_2 \} = x_{[\delta]}(\pm t') x_2 y_2 \in H$. Now let 
    \begin{align*}
        H \ni z_4 &= [x_{-[\gamma] - [\delta]} (1), x_{[\delta]}(\pm t') x_2 y_2] \\
        &= [x_{-[\gamma] - [\delta]} (1), x_{[\delta]}(\pm t')] \{ {}^{x_{[\delta]}(\pm t')}[x_{-[\gamma] - [\delta]} (1), x_2] \} \{ {}^{x_{[\delta]}(\pm t') x_2} [x_{-[\gamma] - [\delta]} (1), y_2] \} \\
        &= x_{-[\gamma]}(\pm t') x_4 y_4,
    \end{align*}
    where $x_4 = {}^{x_{[\delta]}(\pm t')}[x_{-[\gamma] - [\delta]} (1), x_2]$ and $y_4 = {}^{x_{[\delta]}(\pm t') x_2} [x_{-[\gamma] - [\delta]} (1), y_2]$. Since $x_2$ does not have factors of the form $x_{[\alpha]}(t_{[\alpha]})$ with $[\alpha] = [\gamma], [\delta], [\gamma]+[\delta]$, we conclude that $x_4 \in U_\sigma(R)$ and it does not have a factor of the form $x_{[\delta]} (s_{[\delta]})$. Now we claim that $[x_{-[\gamma] - [\delta]} (1), y_2] = 1$ or $x_{-[\beta]}(s)$ for some $s \in R_{[\beta]}$. The latter case is possible only when $\Phi_\rho \sim {}^2 A_{2n+1}$. To see this, observe that the highest root $[\beta]$ is the only root with $m_{[\gamma]}([\beta]) = 2$, therefore $x_{-[\beta]}(u)$ is the only possible factor of $[x_{-[\gamma] - [\delta]} (1), y_2]$. Assume $\Phi_\rho \not\sim {}^{2} A_{2n+1}$ and $[x_{-[\gamma] - [\delta]} (1), y_2] = x_{-[\beta]}(s) \neq 1$, then $-[\beta] = -[\alpha] - [\gamma] - [\delta]$ for some root $[\alpha]$ such that $x_{-[\alpha]}(v)$ is a factor of $y_2$. Hence, $-[\alpha] = - [\alpha_1] + [\delta]$ or $-[\alpha] = (-[\alpha_1] + [\delta]) - [\alpha_2]$ for some $[\alpha_1], [\alpha_2]$ with $m_{[\gamma]}([\alpha_i]) \geq 1 (i=1,2).$ But it is impossible for the first case, as in this case $[\alpha_1] = [\beta] - 2[\gamma]$ is never a root, and for the second case, $2 \geq m_{[\gamma]}([\alpha_1] - [\delta] + [\alpha_2]) = m_{[\gamma]}([\beta] - [\gamma] - [\delta]) = 1$. Which proves the claim. 
    
    Suppose $\Phi_\rho \not\sim {}^2 A_{2n+1}$ then $z_4 = x_{-[\gamma]}(\pm t') x_4$. In this case, let
    \begin{align*}
        H \ni z_5 &= [x_{[\beta]}(1), x_{-[\gamma]} (\pm t') x_4] \\
        &= [x_{[\beta]}(1), x_{-[\gamma]}(\pm t')] \{ {}^{x_{-[\gamma]}(\pm t')} [x_{[\beta]}(1), x_4] \} \\
        &= x_{[\beta] - [\gamma]}(\pm t').
    \end{align*}
    Thus, by Proposition \ref{z to Rz}, we have $x_{[\gamma]}(t) \in H$. 
    
    Now suppose $\Phi_\rho \sim {}^2 A_{2n+1}$, then $ [x_{-[\gamma] - [\delta]}(1), y_2] = x_{-[\beta]}(s)$ for some $s \in R_{[\beta]}$. We can rewrite the expression of $z_4$ as follows:
    \begin{align*}
        z_4 &= [x_{-[\gamma] - [\delta]} (1), x_{[\delta]}(\pm t')] \{ {}^{x_{[\delta]}(\pm t')}[x_{-[\gamma] - [\delta]} (1), x_2] \} \{ {}^{x_{[\delta]}(\pm t') x_2} [x_{-[\gamma] - [\delta]} (1), y_2] \} \\
        &= {}^{x_{[\delta]}(\pm t') x_2} \{ \{ {}^{x_2^{-1} x_{[\delta]}(\pm t')^{-1}} [x_{-[\gamma] - [\delta]} (1), x_{[\delta]}(\pm t')] \} \{ {}^{x_2^{-1}} [x_{-[\gamma] - [\delta]} (1), x_2] \} \{ x_{-[\beta]}(s) \} \} \\
        &= {}^{x_{[\delta]}(\pm t') x_2} \{ \{ {}^{x_2^{-1}} [x_{[\delta]}(\pm t')^{-1}, x_{-[\gamma] - [\delta]} (1)] \} \{ [x_2^{-1}, x_{-[\gamma] - [\delta]} (1)] \} \{ x_{-[\beta]}(s) \} \} \\
        &= {}^{x_{[\delta]}(\pm t') x_2} \{ \{ {}^{x_2^{-1}} x_{-[\gamma]}(\pm t') \} \{ [x_2^{-1}, x_{-[\gamma] - [\delta]} (1)] \} \{ x_{-[\beta]}(s) \} \}
    \end{align*}
    Let $x'_2$ be such that ${}^{x_2^{-1}} x_{-[\gamma]}(\pm t') = x_{-[\gamma]}(\pm t') x'_2$. Clearly, $x'_2 \in U_\sigma (R)$. We set $x'_4 = x'_2 [x_2^{-1}, x_{-[\gamma] - [\delta]}(1)]$. Note that $x'_4 \in U_\sigma (R)$ such that it can not contain a factor of type $x_{[\delta]}(s_{[\delta]})$. But then 
    $$ z_4 = {}^{x_{[\delta]}(\pm t') x_2} \{ x_{-[\gamma]} (\pm t') x'_4 x_{-[\beta]}(s) \}.$$
    Since $z_4 \in H$, so is $x_{-[\gamma]} (\pm t')^{-1} x'_4 x_{-[\beta]}(s)$. Now, let
    \begin{align*}
        H \ni z'_5 &= [x_{-[\delta]}(1), x_{-[\gamma]} (\pm t') x'_4 x_{-[\beta]}(s)] \\
        &= [x_{-[\delta]}(1), x_{-[\gamma]} (\pm t')] \{ {}^{x_{-[\gamma]} (\pm t')} [x_{-[\delta]}(1), x'_4] \} \{ {}^{x_{-[\gamma]} (\pm t') x'_4} [x_{-[\delta]}(1), x_{-[\beta]}(s)] \} \\
        &= {}^{x_{-[\gamma]}(\pm t')} \{ \{{}^{x_{-[\gamma]}(\pm t')^{-1}} [x_{-[\delta]}(1), x_{-[\gamma]} (\pm t')] \} \{ [x_{-[\delta]}(1), x'_4] \} \} \\
        &= {}^{x_{-[\gamma]}(\pm t')} \{ \{[x_{-[\gamma]}(\pm t')^{-1}, x_{-[\delta]}(1)] \} \{ [x_{-[\delta]}(1), x'_4] \} \} \\
        &= {}^{x_{-[\gamma]}(\pm t')} \{ \{ x_{-[\gamma] - [\delta]} (\pm t') \} \{ [x_{-[\delta]}(1), x'_4] \} \} \\
        & = {}^{x_{-[\gamma]}(\pm t')} \{ x_{-[\gamma] - [\delta]} (\pm t') x_5 \},
    \end{align*}
    where $x_5 = [x_{-[\delta]}(1), x'_4] \in U_\sigma (R)$. Note that $x_{-[\gamma] - [\delta]} (\pm t') x_5 \in H$. Finally, let
    \begin{align*}
        H \ni z_6 &= [x_{[\beta]}(1), x_{-[\gamma] - [\delta]} (\pm t') x_5] \\
        &= [x_{[\beta]}(1), x_{-[\gamma] - [\delta]} (\pm t')] \{ {}^{x_{-[\gamma] - [\delta]} (\pm t')} [x_{[\beta]}(1), x_5] \} \\
        &= [x_{[\beta]}(1), x_{-[\gamma] - [\delta]} (\pm t')] \\
        &= x_{[\beta] - [\gamma] - [\delta]} (\pm t') x_{[\beta] - 2 [\gamma] - 2 [\delta]} (\pm t \bar{t})
    \end{align*}
    Now, by Lemma \ref{lemma:U cap H}, we have $x_{[\beta] - [\gamma] - [\delta]} (\pm t') \in H$ and hence, by Proposition \ref{z to Rz}, $x_{[\gamma]} (t) \in H$.

    \vspace{2mm}
    
    \noindent \textbf{Case B. $\Phi_\rho \sim {}^2 A_3$:} 
    Let $[\delta] \in \Delta_\rho$ be such that $\{[\gamma], [\delta]\}$ forms a base of $\Phi_\rho$. Note that $[\gamma]$ is a short root. The idea of the proof is the same as in Case A. We leave the details to the reader.
    Let $z = x_{[\gamma]}(t) x h y = x_{[\gamma]}(t) x h x_{-[\gamma]-[\delta]}(s_1) x_{-[\gamma]}(s_2) x_{-2[\gamma]-[\delta]}(s_3) x_{-[\delta]}(s_4) \in H$.
    We first show that $x_{-[\delta]}(s_4) \in H$. Write $x' = x^{-1} x_{[\gamma]}(t)^{-1}$ and $y' = \{ x_{-[\gamma]-[\delta]}(s_1) x_{-[\gamma]}(s_2)$ $x_{-2[\gamma]-[\delta]}(s_3) \}^{-1}$ and consider 
    $$ H \ni z_1 = [x_{[\delta]}(1), z^{-1}] = {}^{x_{-[\delta]}(-s_4)y'} \{ u_1 y_1 x_1 \}, $$
    where $u_1 = [x_{-[\delta]}(s_4), x_{[\delta]}(1)], y_1 = x_{-[\gamma]}(s'_1) x_{-[\gamma]-[\delta]}(s'_2) x_{-2[\gamma]-[\delta]}(s'_3) \in U^{-}_\sigma (R)$ and $x_1 = x_{[\delta]}(t'_1) x_{[\gamma]+[\delta]} (t'_2) x_{2[\gamma] + [\delta]} (t'_3) \in U_\sigma (R)$. Now consider
    $$H \ni z_2 = [x_{-[\gamma]}(1), u_1 y_1 x_1] = {}^{u_1 y_1 x_{[\delta]}(1)} \{ x_{-[\gamma]}(\pm s_4) y_2 x_2 \},$$
    where $y_2 = x_{-[\gamma]-[\delta]}(s''_1) x_{-2[\gamma]-[\delta]}(s''_2) \in U^{-}_\sigma (R)$ and $x_2 = x_{[\delta]}(t''_1) x_{[\gamma]+[\delta]} (t''_2) \in U_\sigma (R)$. Now let $z'_2 = x_{-[\gamma]}(\pm s_4) y_2 x_2$ and consider
    $$ H \ni z_3 = [x_{[\delta]}(1), z'_2] = x_{-[\gamma]}(\pm s''_1) x_{-2[\gamma] - [\delta]}(s'''_2).$$
    But then $$w_{2[\gamma]+[\delta]}(1) \{ x_{-[\gamma]}(\pm s''_1) x_{-2[\gamma] - [\delta]}(s'''_2) \}w_{2[\gamma]+[\delta]}(1)^{-1} = x_{[\gamma] + [\delta]}(\pm s''_1) x_{2[\gamma] + [\delta]}(\pm s'''_2) \in H.$$ 
    By Lemma \ref{lemma:U cap H}, we have $x_{[\gamma] + [\delta]}(\pm s''_1) \in H$ and hence, by Proposition \ref{z to Rz}, we have $x_{-[\gamma]-[\delta]}(s''_1) \in H$. Now $x_{-[\gamma]-[\delta]}(s''_1)^{-1} (z'_2) = x_{-[\gamma]}(\pm s_4) x_{-2[\gamma]-[\delta]}(s''_2) x_{[\delta]}(t''_1) x_{[\gamma]+[\delta]} (t''_2) \in H$. But then 
    \begin{gather*}
        w_{2[\gamma]+[\delta]}(1) w_{[\gamma] + [\delta]}(1) \{ x_{-[\gamma]}(\pm s_4) x_{-2[\gamma]-[\delta]}(s''_2) x_{[\delta]}(t''_1) x_{[\gamma]+[\delta]} (t''_2) \} w_{[\gamma] + [\delta]}(1)^{-1} w_{2[\gamma]+[\delta]}(1)^{-1} \\
        = x_{[\gamma] + [\delta]}(\pm s_4) x_{[\delta]}(s''_2) x_{2[\gamma] + [\delta]}(t''_1) x_{[\gamma]} (t''_2) \in H.
    \end{gather*}
    By Lemma \ref{lemma:U cap H}, we have $x_{[\gamma] + [\delta]}(\pm s_4) \in H$ and hence, by Proposition \ref{z to Rz}, we have $x_{-[\delta]}(s_4) \in H$.
    Finally, $z' = z x_{-[\delta]}(s_4)^{-1} \in H$. Now consider, 
    $$ H \ni z_4 = [x_{[\delta]}(1), z'] = {}^{x_{[\gamma]}(t) x} \{ x_{[\gamma]+[\delta]} (\pm t) x_4 y_4 \}, $$
    where $x_4 =  x_{[\delta]}(t_{41}) x_{2[\gamma]+[\delta]}(t_{42})$ and $y_4 = x_{-[\gamma]}(s_{41}) x_{-2[\gamma]-[\delta]}(s_{42})$. Finally, we claim that $x_{-2[\gamma]-[\delta]}(s_{42}) \in H$. To see this consider
    $$ H \ni z_5 = [x_{[\gamma]+[\delta]}(1), x_{[\gamma]+[\delta]} (\pm t) x_4 y_4] = {}^{x_{[\gamma]+[\delta]} (\pm t) x_4 x_{-[\gamma]}(s_{41})} \{ x_{-[\gamma]}(\pm s_{42}) x_{[\delta]}(s_{52}) \}.$$
    But then $w_{[\gamma]}(1) \{ x_{-[\gamma]}(\pm s_{42}) x_{[\delta]}(s_{52}) \} w_{[\gamma]}(1)^{-1} = x_{[\gamma]}(\pm s_{42}) x_{2[\gamma] + [\delta]}(\pm s_{52}) \in H.$ By Lemma \ref{lemma:U cap H}, we have $x_{[\gamma]}(\pm s_{42}) \in H$, and hence, by Proposition \ref{z to Rz}, we have $x_{-2[\gamma]-[\delta]}(s_{42}) \in H$. Finally, we have $z'' = x_{[\gamma]+[\delta]} (\pm t) x_4 y_4 x_{[\gamma]}(\pm s_{42})^{-1} \in H$. But then 
    $$w_{[\gamma]}(1) z'' w_{[\gamma]}(1)^{-1} = x_{[\gamma]+[\delta]} (\pm t) x_{2[\gamma] + [\delta]} (\pm t_{41}) x_{[\delta]}(\pm t_{42}) x_{[\gamma]}(\pm s_{41}) \in H.$$ By Lemma \ref{lemma:U cap H}, $x_{[\gamma] + [\delta]} (t) \in H$ and hence, by Proposition \ref{z to Rz}, $x_{[\gamma]}(t) \in H$.
    
    \vspace{2mm}

    \noindent \textbf{Case C. $\Phi_\rho \sim {}^2 A_4$:} 
    Let $[\delta] \in \Delta_\rho$ be such that $\{[\gamma], [\delta]\}$ forms a base of $\Phi_\rho$. Note that $[\gamma]$ is a long root. The idea of the proof is the same as in Case A. We leave the details to the reader.
    We first consider 
    $$ z_1 = [x_{-[\gamma]}(1), z] = {}^{x_{[\gamma]}(t) x}(u_1 x_1 y_1),$$
    where $u_1 = [x_{[\gamma]}(t)^{-1}, x_{-[\gamma]}(1)],$ $x_1 = x_{[\delta]}(t_1) x_{[\gamma]+[\delta]} (t_2) x_{[\gamma] + 2[\delta]} (t_3)$ and 
    $y_1 = x_{-[\gamma]}(s_1)$ $x_{-[\gamma] - [\delta]}(s_2)$ $x_{-[\gamma] - 2[\delta]} (s_3)$.
    We than consider  
    \begin{align*}
        H \ni z_2 &= [x_{[\delta]}(1, 1/2), u_1 x_1 y_1] \\
        &= {}^{u_1 x_1 x_{-[\gamma]}(1)} \{ x_{[\delta]} (\pm t, (t \bar{t} \pm (t - \bar{t}))/2) x_2 y_2 \},
    \end{align*}
    where $x_2 = x_{[\gamma] + [\delta]}(t'_1) x_{[\gamma] + 2[\delta]} (t'_2)$ and $y_2 = x_{-[\gamma]}(s'_1) x_{-[\gamma] - [\delta]}(s'_2)$. Set $u_2 = x_{[\delta]} (\pm t, (t \bar{t} \pm (t - \bar{t}))/2)$ and $z'_2 = u_2 x_2 y_2$.
    Our next goal is to show that $x_{-[\gamma] - [\delta]}(s'_2) \in H$. Write $s'_2 = (a,b) \in R_{-[\gamma] - [\delta]}$. Now consider
    \begin{align*}
        H \ni z_3 &= [x_{[\delta]}(1, 1/2), z'_2] = {}^{u_2 x_2} \{ x_{[\gamma] + 2 [\delta]} (t''_1) x_{-[\gamma]} (\pm a') \}, 
    \end{align*}
    where $a'= a$ or $\bar{a}$ and $t''_1 \in R_{[\gamma] + 2 [\delta]}$. Since $z_3 \in H$ we have $x_{[\gamma] + 2 [\delta]} (t''_1) x_{-[\gamma]} (\pm a') \in H$. But then $w_{[\gamma]}(1) \{x_{[\gamma] + 2 [\delta]} (t''_1) x_{-[\gamma]} (\pm a')\} w_{[\gamma]}(1)^{-1} = x_{[\gamma] + 2 [\delta]} (t''_1) x_{[\gamma]} (\pm a') \in H$ (see Proposition \ref{prop wxw^{-1}}). By Lemma \ref{lemma:U cap H}, we have $x_{[\gamma]}(\pm a') \in H$. Next, we consider
    \begin{align*}
        H \ni z_4 &= [x_{[\gamma] + 2[\delta]}(1), z'_2] = {}^{u_2 x_2} \{ x_{[\delta]} (t''_2) x_{-[\gamma]} (\pm b') \}, 
    \end{align*}
    where $b'= b$ or $\bar{b}$ and $t''_2 \in R_{[\delta]}$. Since $z_3 \in H$ we have $x_{[\delta]} (t''_2) x_{-[\gamma]} (\pm b') \in H$. But then $w_{[\gamma]}(1) \{x_{[\delta]} (t''_2) x_{-[\gamma]} (\pm b')\} w_{[\gamma]}(1)^{-1} = x_{[\gamma] + [\delta]} (\pm t''_2) x_{[\gamma]} (\pm b') \in H$ (see Proposition \ref{prop wxw^{-1}}). By Lemma \ref{lemma:U cap H}, we have $x_{[\gamma]}(\pm b') \in H$. Since $x_{[\gamma]}(\pm a') \in H$ and $x_{[\gamma]}(\pm b') \in H$, by Proposition \ref{z to Rz}, we have $x_{-[\gamma]-[\delta]}(a,b) \in H$. But then 
    $$ z_2 x_{-[\gamma] - [\delta]}(a,b)^{-1} = x_{[\delta]}(\pm t, t \bar{t} \pm (t - \bar{t}))/2) x_{[\gamma] + [\delta]}(t'_1) x_{[\gamma] + 2[\delta]} (t'_2) x_{-[\gamma]}(s'_1) \in H.$$ 
    Finally, 
    \begin{gather*}
        w_{[\gamma]}(1) \{ x_{[\delta]}(\pm t, t \bar{t} \pm (t - \bar{t}))/2) x_{[\gamma] + [\delta]}(t'_1) x_{[\gamma] + 2[\delta]} (t'_2) x_{-[\gamma]}(s'_1) \} w_{[\gamma]}(1)^{-1} \\
        = x_{[\gamma] + [\delta]}(\pm t', t \bar{t} \pm (t - \bar{t}))/2) x_{[\delta]}(t'_{11}) x_{[\gamma] + 2[\delta]} (t'_{22}) x_{[\gamma]}(s'_{11}) \in H,
    \end{gather*}
    where $t' = t$ or $\bar{t}$, and $t'_{11} \in R_{[\delta]}, t'_{22} \in R_{[\gamma] + 2[\delta]}, s'_{11} \in R_{[\gamma]}$. Finally, by Lemma \ref{lemma:U cap H}, we have $x_{[\gamma] + [\delta]}(\pm t', t \bar{t} \pm (t - \bar{t}))/2) \in H$, and by Proposition \ref{z to Rz}, we have $x_{[\gamma]}(t) \in H$, as desired.
    
    \vspace{2mm}

    \noindent \textbf{Case D. $\Phi_\rho \sim {}^3 D_4$:} Let $[\delta] \in \Delta_\rho$ be such that $\{[\gamma], [\delta]\}$ forms a base of $\Phi_\rho$. Note that $[\gamma]$ is a long root. The idea of the proof is the same as in Case A. We leave the details to the reader.
    We first consider 
    $$ z_1 = [x_{-[\gamma]}(1), z] = {}^{x_{[\gamma]}(t) x}(u_1 x_1 y_1),$$
    where $u_1 = [x_{[\gamma]}(t)^{-1}, x_{-[\gamma]}(1)], x_1 = x_{[\delta]}(t_1) x_{[\gamma]+[\delta]} (t_2) x_{[\gamma] + 2[\delta]} (t_3) x_{[\gamma] + 3[\delta]} (t_4) x_{2[\gamma] + 3[\delta]} (t_5)$ and 
    $y_1 = x_{-[\gamma]}(s_1) x_{-[\gamma] - [\delta]}(s_2) x_{-[\gamma] - 2[\delta]} (s_3) x_{-[\gamma] - 3[\delta]} (s_4) x_{-2[\gamma] - 3[\delta]} (s_5)$.
    We than consider 
    \begin{align*}
        H \ni z_2 = [x_{[\delta]}(1), u_1 x_1 y_1] = {}^{u_1 x_1 x_{-[\gamma]}(1)} \{ x_{[\delta]} (\pm t) x_2 y_2 \},
    \end{align*}
    where $x_2 = x_{[\gamma] + [\delta]}(t'_1) x_{[\gamma] + 2[\delta]} (t'_2) x_{[\gamma] + 3[\delta]}(t'_3) x_{2[\gamma] + 3[\delta]} (t'_4)$ and $y_2 = x_{-[\gamma]}(s'_1)$ $x_{-[\gamma] - [\delta]}(s'_2)$ $x_{-[\gamma]-2[\delta]}(s'_3) x_{-2[\gamma] - 3[\delta]}(s'_4)$. We set $z'_2 = x_{[\delta]}(\pm t) x_2 y_2$.
    Our next goal is to show that $x_{-[\gamma] -2[\delta]}(s'_3) \in H$. Now consider 
    \begin{align*}
        H \ni z_3 &= [x_{[\gamma] + 3 [\delta]}(1), z'_2] \\
        &= {}^{x_{[\delta]}(\pm t) x_2 x_{-[\gamma]}(s'_1) x_{-[\gamma]-[\delta]}(s'_2) x_{-2[\gamma] - 3 [\delta]}(s'_4) } \{ x_{[\delta]} (\pm s'_3) x_{-[\gamma]} (\pm s'_4 \pm s'_3 \bar{s'_3} \bar{\bar{s'_3}}) \\
        & \hspace{20mm} x_{-[\gamma] - [\delta]} (\pm s'_3 \bar{s'_3}) x_{-2[\gamma] - 3[\beta]} (\pm s'_3 \bar{s'_3} \bar{\bar{s'_3}}) \}. 
    \end{align*}
    Since $z_3 \in H$ we have 
    $$x_{[\delta]} (\pm s'_3) x_{-[\gamma]} (\pm s'_4 \pm s'_3 \bar{s'_3} \bar{\bar{s'_3}}) x_{-[\gamma] - [\delta]} (\pm s'_3 \bar{s'_3}) x_{-2[\gamma] - 3[\beta]} (\pm s'_3 \bar{s'_3} \bar{\bar{s'_3}}) \in H.$$ 
    But then 
    \begin{small}
        \begin{gather*}
            w_{2[\gamma]+3[\delta]}(1) \{ x_{[\delta]} (\pm s'_3) x_{-[\gamma]} (\pm s'_4 \pm s'_3 \bar{s'_3} \bar{\bar{s'_3}}) x_{-[\gamma] - [\delta]} (\pm s'_3 \bar{s'_3}) x_{-2[\gamma] - 3[\beta]} (\pm s'_3 \bar{s'_3} \bar{\bar{s'_3}}) \} w_{2[\gamma]+3[\delta]}(1)^{-1} \\
            = x_{[\delta]} (\pm s'_3) x_{[\gamma] + 3[\delta]} (\pm s'_4 \pm s'_3 \bar{s'_3} \bar{\bar{s'_3}}) x_{[\gamma] + 2[\delta]} (\pm s'_3 \bar{s'_3}) x_{2[\gamma] + 3[\beta]} (\pm s'_3 \bar{s'_3} \bar{\bar{s'_3}}) \in H.
        \end{gather*}
    \end{small}
    By Lemma \ref{lemma:U cap H}, we have $x_{[\delta]}(\pm s'_3), x_{[\gamma] + 3[\delta]} (\pm s'_4 \pm s'_3 \bar{s'_3} \bar{\bar{s'_3}}) \in H$ and hence, by Proposition \ref{z to Rz}, $x_{-[\gamma] - 2[\delta]} (s'_3), x_{[\gamma] + 3[\delta]}(\pm s'_3 \bar{s'_3} \bar{\bar{s'_3}}) \in H$. But then $x_{[\gamma] + 3[\delta]} (s'_4) \in H$, and again by Proposition \ref{z to Rz}, $x_{-2[\gamma]-3[\delta]}(s'_4) \in H$. Therefore
    \begin{align*}
        z''_2 &= z'_2 x_{-2[\gamma]-3[\delta]}(s'_4)^{-1} x_{-[\gamma]-2[\delta]}(s'_3)^{-1} \\
        &= x_{[\delta]}(\pm t) x_{[\gamma] + [\delta]}(t'_1) x_{[\gamma] + 2[\delta]} (t'_2) x_{[\gamma] + 3[\delta]}(t'_3) x_{2[\gamma] + 3[\delta]} (t'_4) x_{-[\gamma]}(s'_1) x_{-[\gamma] - [\delta]}(s'_2) \in H.
    \end{align*}
    Now replacing $z'_2$ by $w_{2[\gamma]+3[\delta]}(1) z'_2 w_{2[\gamma]+3[\delta]}(1)^{-1}$ in above process we can conclude that $x_{[\gamma]+[\delta]}(t'_1) \in H$. Therefore, we have $$z'''_2 = x_{[\gamma]+[\delta]}(t'_1)^{-1} z''_2 = x_{[\delta]}(\pm t) x_{[\gamma]+2[\delta]}(t''_2) x_{[\gamma] + 3[\delta]}(t''_3) x_{2[\gamma] + 3[\delta]} (t''_4) x_{-[\gamma]}(s'_1) x_{-[\gamma] - [\delta]}(s'_2) \in H.$$ 
    Finally, 
    \begin{align*}
        z_4 &= w_{[\gamma]}(1) w_{[\gamma]+[\delta]}(1) \{ z'''_2 \} w_{[\gamma]+[\delta]}(1)^{-1} w_{[\gamma]}(1)^{-1} \\ 
        &= x_{[\gamma]+2[\delta]}(\pm t) x_{[\gamma] + [\delta]}(\pm t''_2) x_{2[\gamma] + 3[\delta]}(t''_3) x_{[\gamma]} (t''_4) x_{[\gamma]+3[\delta]}(s'_1) x_{[\delta]}(s'_2) \in H,
    \end{align*}
    Thus, by Lemma \ref{lemma:U cap H}, we have $x_{[\gamma] + 2[\delta]}(\pm t) \in H$, and by Proposition \ref{z to Rz}, we have $x_{[\gamma]}(t) \in H$, as desired.
\end{proof}

\begin{rmk}
    The idea of the above proof is motivated by that of the Lemma in 3.4 of \cite{KS3}. However, that proof contains an error, which we rectify here. Specifically, the value of $s$ on page 11 is incorrect; the correct value should be $s = \pm a_1 a t^2$ instead of $s = \pm a_1 t + a_1 a t^2$. This leads to a mistake in the last paragraph of the proof of Case 1, where $x_{\beta}(\pm a_1 a t^2) \in H$ should appear instead of $x_{\beta}(\pm a_1 t \pm a_1 a t^2) \in H$. However, the corrected statement does not yield the required result.
\end{rmk}

%%%%%%%%%%%%%%%%%%%%%%%%%%%%%%%%%%%%%%%%%%%%%%%%%%

\vspace{2mm}

\noindent \textit{Proof of Proposition \ref{prop:U(R) cap H subset U(J)}.} The first part is clear from Lemma \ref{lemma:U cap H}. For the second part, let $xhy \in U_\sigma (rad(R)) T_\sigma (R) U^{-}_\sigma (R) \cap H$, where $x \in U_\sigma (rad(R)), h \in T_\sigma(R)$ and $y \in U^{-}_\sigma (R)$. 
% For $x= \prod_{[\alpha]>0} x_{[\alpha]}(t_{[\alpha]}) \in U_\sigma (R)$ (resp., $y = \prod_{[\alpha]>0} x_{-[\alpha]}(t_{[\alpha]}) \in U^{-}_\sigma (R)$), we define $\Phi(x)$ (resp., $\Phi(y)$) to be the set $\{ [\alpha] \mid t_{[\alpha]} \neq 0 \}$ (the product is taken over disjoint roots in any fixed order). We fix $[\beta], [\gamma], [\delta],$ etc. as in Lemma \ref{lemma:UTV cap H}.

First assume that $x = 1 = y$. Then $z = h = h(\chi) \in T_\sigma (R)$ for some $\chi$. Since $h(\chi) \in H$, for any root $\alpha \in \Phi$ we have $[x_{[\alpha]}(1), h(\chi)] = x_{[\alpha]}(1 - \chi(\alpha))$ is an element of $H$. Thus, $1 - \chi(\alpha) \in J$, that is, $\chi(\alpha) \equiv 1$ (mod $J$) for every $\alpha \in \Phi$. Therefore, by Lemma \ref{lemma on T(J)}, we have $h(\chi) \in T_\sigma (R, J)$. 

We now assume that $x \neq 1$ or $y \neq 1$. Observe that, if we prove that every factor of $x$ and $y$ is in $H$ then we are done. Furthermore, by conjugating with elements of the Weyl group, applying Lemma~\ref{inUHV}, and using the Chevalley commutator formulas, it suffices to prove the following: 
\textit{for a fixed short root (respectively, a fixed long root) \( [\alpha] \), any factor of the form \( x_{[\alpha]}(t) \) appearing in the expression of an element \( x' \) or \( y' \), where \( z' = x' h' y' \in U_\sigma(R)\, T_\sigma(R)\, U^{-}_\sigma(R) \cap H \), is contained in \( H \)}. 

By the Lemma \ref{lemma:UTV cap H} and the above observation, if $x_{[\alpha]}(t_{[\alpha]})$ is a factor of $x$ or $y$ with $[\alpha]$ being the same type as $[\gamma]$ then $x_{[\alpha]}(t_{[\alpha]}) \in H$. Now it remains to show that $x_{[\alpha]}(t_{[\alpha]}) \in H$ for a root $[\alpha]$ not of the type $[\gamma]$. Note that we can assume that both $x$ and $y$ contains only factors $x_{[\alpha]}(t_{[\alpha]})$ where $[\alpha]$ is not of the type $[\gamma]$. 

Suppose $\Phi_\rho \not \sim {}^2 A_{2n+1}$, then $[\gamma]$ is a long root and hence it is of the same type as the highest root $[\beta]$. In particular, by our assumption, $x$ and $y$ does not contains a factor $x_{[\beta]}(t_{[\beta]})$ and $x_{-[\beta]}(t_{[\beta]})$, respectively. Let $x_{[\alpha_1]}(t)$ be a factor of $x$ or $y$. Then $[\alpha_1]$ is of different type then $[\gamma]$. Choose a root $[\alpha_2] \in \Phi_\rho^{+}$ such that $-[\alpha_2]$ and $[\beta]$ generates the subsystem of type $B_2$. Since $[\alpha_1]$ and $[\alpha_2]$ is of the same type, we can conjugate $z$ by an element of the Weyl group in such a way that $x_{-[\alpha_2]}(t)$ is a factor of that new element, say $z_1$. Now let $z_2 = [x_{[\beta]}(1), z_1]$. Then $z_2 \in U_\sigma(R) \cap H$ and it contains a factor $x_{[\beta] - [\alpha_2]} (t')$, where $t' = t$ or $\bar{t}$ if $[\alpha_2] \not \sim A_2$; and $t' = (t_1, t_2), (t_1, \bar{t}_2)$ or $(\bar{t}_1, t_2)$ if $[\alpha_2] \sim A_2$ and $t=(t_1, t_2)$. By Lemma \ref{lemma:U cap H}, we have $x_{[\beta] - [\alpha_2]} (t') \in H$ and hence, by Proposition \ref{z to Rz}, $x_{[\alpha_1]}(t) \in H$. 

Now suppose $\Phi_\rho \sim {}^2 A_{2n+1}$, then $[\gamma]$ is a short root. Let $x_{[\alpha_1]}(t)$ be a factor of $x$ or $y$ with $[\alpha_1]$ being a long root. Choose a root $[\alpha_2] \in \Phi_\rho$ such that $[\alpha_1]$ and $[\alpha_2]$ form a subsystem of type $B_2$. Let $z_1 = [x_{[\alpha_2]}(1), z]$. Then $z_1 \in H$ contains a factor $x_{[\alpha_1] + [\alpha_2]}(t)$ with $[\alpha_1] + [\alpha_2]$ being a short root, that is, it is of the same type as $[\gamma]$. Hence from above observation, $x_{[\alpha_1] + [\alpha_2]}(t) \in H$ and hence, by Proposition \ref{z to Rz}, $x_{[\alpha]}(t) \in H$, as desired.
\qed

%%%%%%%%%%%%%%%%%%%%%%%%%%%%%%%%%%%%%%%%%%%%%%%%%%
%Section: Proof of Propositions \ref{local:H subset G(R,J)}
%%%%%%%%%%%%%%%%%%%%%%%%%%%%%%%%%%%%%%%%%%%%%%%%%%

\section{Proof of Proposition \ref{local:H subset G(R,J)}}\label{sec:Pf of prop 3}

%%%%%%%%%%%%%%%%%%%%%%%%%%%%%%%%%%%%%%%%%%%%%%%%%%

Let the notations be as in Section \ref{sec:Pf of main thm}. Define \( k_\mathfrak{m} = R / \mathfrak{m} \cong R_{\mathfrak{m}} / \mathfrak{m} R_{\mathfrak{m}} \), and similarly for \( k_{\bar{\mathfrak{m}}} \) and \( k_{\bar{\bar{\mathfrak{m}}}} \). Clearly, \( R_{\mathfrak{m}} \cong R_{\bar{\mathfrak{m}}} \cong R_{\bar{\bar{\mathfrak{m}}}} \) and \( k_\mathfrak{m} \cong k_{\bar{\mathfrak{m}}} \cong k_{\bar{\bar{\mathfrak{m}}}} \). For $S=S_\mathfrak{m}$, we set 
\begin{small}
    \[
        I_S = \begin{cases}
            \mathfrak{m} & \text{if } \mathfrak{m} = \bar{\mathfrak{m}}, \\
            \mathfrak{m} \cap \bar{\mathfrak{m}} & \text{if } \mathfrak{m} \neq \bar{\mathfrak{m}} \text{ and } o(\theta) = 2, \\
            \mathfrak{m} \cap \bar{\mathfrak{m}} \cap \bar{\bar{\mathfrak{m}}} & \text{if } \mathfrak{m} \neq \bar{\mathfrak{m}} \text{ and } o(\theta) = 3;
        \end{cases} 
        \text{ and }  
        k_S = \begin{cases}
            k_\mathfrak{m} & \text{if } \mathfrak{m} = \bar{\mathfrak{m}}, \\
            k_\mathfrak{m} \times k_{\bar{\mathfrak{m}}} & \text{if } \mathfrak{m} \neq \bar{\mathfrak{m}} \text{ and } o(\theta) = 2, \\
            k_\mathfrak{m} \times k_{\bar{\mathfrak{m}}} \times k_{\bar{\bar{\mathfrak{m}}}} & \text{if } \mathfrak{m} \neq \bar{\mathfrak{m}} \text{ and } o(\theta) = 3.
        \end{cases}
    \]
\end{small}
Then $R/I_S \cong R_S / (I_S R_S) \cong k_S$.
Note that the ring automorphism $\theta: R \longrightarrow R$ also induces an automorphism of $k_S$, which is also denoted by $\theta$.

%%%%%%%%%%%%%%%%%%%%%%%%%%%%%%%%%%%%%%%%%%%%%%%%%%

\begin{prop}\label{prop:U(RS) cap psi(H) subset U(JS)}
    \normalfont
    Let the notations be as in Section \ref{sec:Pf of main thm}.
    \begin{enumerate}[(a)]
        \item $U_\sigma (R_S) \cap \psi_\mathfrak{m}(H) \subset U_\sigma (J_S)$. 
        \item $U_\sigma (I_S R_S) T_\sigma (R_S) U^{-}_\sigma (R_S) \cap \psi_\mathfrak{m} (H) \subset U_\sigma (J_S) T_\sigma (R_S, J_S) U^{-}_\sigma (J_S).$
    \end{enumerate}
\end{prop}

At first glance, the above proposition may appear to be an immediate consequence of Proposition \ref{prop:U(R) cap H subset U(J)}, but this is not the case. The crucial point is that the subgroup $\psi_\mathfrak{m}(H)$ of $G_\sigma(R_S)$ may not be normalized by $E'_\sigma(R_S)$. However, once we prove the following lemma, the proof of the above lemma will be similar to the proof of Proposition \ref{prop:U(R) cap H subset U(J)} and hence is omitted.

\begin{lemma}
    \normalfont
    Let the notations be as in Section \ref{sec:Pf of main thm}. Let $z \in H$. 
    \begin{enumerate}[(a)]
        \item If $\psi_\mathfrak{m} (z) = \prod_{[\alpha] \in \Phi_\rho^{+}} x_{[\alpha]}(t_{[\alpha]}) \in U_\sigma (R_S)$ with $t_{[\alpha]} \in (R_S)_{[\alpha]}$, then for each $[\alpha] \in \Phi^{+}_\rho$ there exists $s_{[\alpha]} \in S_\theta$ such that $x_{[\alpha]}(s_{[\alpha]} \cdot t_{[\alpha]}) \in H$.
        \item Let $[\beta]$ and $[\gamma]$ be as in Lemma \ref{lemma:UTV cap H}. If $\psi_\mathfrak{m}(z) = x_{[\gamma]}(t) xhy \in U_\sigma (R_S) T_\sigma (R_S) U^{-}_\sigma (R_S),$ where $x_{[\gamma]}(t) x \in U_\sigma (R_S),$ $x$ is a product of elements $x_{[\alpha]}(t_{[\alpha]})$ with $[\alpha] \neq [\gamma], [\alpha] \in \Phi^{+}_\rho$ and $t_{[\alpha]} \in (R_S)_{[\alpha]}, h \in T_\sigma (R_S)$ and $y \in U^{-}_\sigma (R_S)$. Then there exists $s \in S_\theta$ such that $x_{[\gamma]}(s \cdot t) \in H$.
    \end{enumerate}
\end{lemma}

\begin{proof}
    The proofs of parts (a) and (b) follow a similar approach to Lemmas \ref{lemma:U cap H} and \ref{lemma:UTV cap H}, respectively, but also incorporate the method used in proving Proposition \ref{general:[x,g] in E(R,J)}. 
    Let $\psi''_\mathfrak{m}$ be as in the proof of Proposition \ref{general:[x,g] in E(R,J)}.
    
    \vspace{2mm}
    
    \noindent (a) Let $\psi_\mathfrak{m}(z) = \displaystyle\prod_{[\alpha] \in \Phi^+} x_{[\alpha]}(t_{[\alpha]}) \in U_\sigma(R_S) \subset U_\sigma(R_S[X])$, with $t_{[\alpha]} \in (R_S)_{[\alpha]} \subset (R_S[X])_{[\alpha]}$. We will first prove the result in the case where $\Phi_\rho \sim {}^2 A_3$. The other cases can be proven similarly using the same techniques (cf. Lemma \ref{lemma:U cap H}) and are therefore omitted.

    As in Lemma \ref{lemma:U cap H}, let $[\alpha]$ and $[\beta]$ be the simple roots, with $[\alpha]$ being the long root. We first claim that
    \begin{equation}\label{eq_aa}
        \begin{split}
            \text{if } \psi_\mathfrak{m}(z) = x_{[\alpha] + [\beta]}(t) x_{[\alpha] + 2[\beta]}(u) \in \psi_\mathfrak{m}(H), \text{ then } \\
            x_{[\alpha] + [\beta]}(s_1 \cdot t), x_{[\alpha] + 2[\beta]}(s_2 \cdot u) \in H \text{ for some } s_1, s_2 \in S_\theta.
        \end{split}
    \end{equation}
    Since $\psi''_\mathfrak{m} \mid_{G_\sigma(R_S)} = \psi_\mathfrak{m}$, we have $\psi_\mathfrak{m}(z) = \psi''_\mathfrak{m}(z) \in \psi''_\mathfrak{m}(H)$. For any $r \in (R_S[X])_{[\alpha]} = (R_S[X])_\theta$, we have
    \[ [x_{-[\alpha]}(r), \psi''_\mathfrak{m}(z)] = x_{[\beta]}(\pm rt) x_{[\alpha] + 2[\beta]}(\pm rt \bar{t}). \]
    Write $t = a/b$ where $a \in R$ and $b \in S$. Let $r = b \bar{b} r' X$ with $r' \in S_\theta$. Then
    \[ \psi''_\mathfrak{m}([x_{-[\alpha]}(b \bar{b} r' X), z]) = \psi''_\mathfrak{m}(x_{[\beta]} (\pm b \bar{b} r' t X) x_{[\alpha] + 2[\beta]} (\pm b \bar{b} r' t \bar{t} X)). \]
    Now let 
    \[ \epsilon(X) = [x_{-[\alpha]}(b \bar{b} r' X), z](x_{[\beta]} (\pm b \bar{b} r' t X) x_{[\alpha] + 2[\beta]} (\pm b \bar{b} r' t \bar{t} X))^{-1}. \]
    Then $\epsilon(X)$ satisfies the hypothesis of Lemma \ref{lemma:GT}, and hence there exists $s' \in S_\theta$ such that $\epsilon(s' X) = 1$. Thus we have
    \[ [x_{-[\alpha]}(b \bar{b} s' r' X), z] = x_{[\beta]} (\pm b \bar{b} s' r' t X) x_{[\alpha] + 2[\beta]} (\pm b \bar{b} s' r' t \bar{t} X). \]
    Setting $X = 1$, we get 
    \[ [x_{-[\alpha]}(b \bar{b} s' r'), z] = x_{[\beta]} (\pm b \bar{b} s' r' t) x_{[\alpha] + 2[\beta]} (\pm b \bar{b} s' r' t \bar{t}) \in H. \]
    Since $\psi''_\mathfrak{m}(z) \in \psi''_\mathfrak{m}(H)$, then so is $(\psi''_\mathfrak{m}(z))^{-1} = x_{[\alpha] + [\beta]}(-t) x_{[\alpha] + 2[\beta]}(-u)$. Therefore, we can replace $t$ and $u$ by $-t$ and $-u$ respectively, and we obtain 
    \[ x_{[\beta]} (\pm b \bar{b} s' r' (-t)) x_{[\alpha] + 2[\beta]} (\pm b \bar{b} s' r' t \bar{t}) \in H. \] 
    Thus, 
    \begin{equation*}
        \begin{split}
            x_{[\beta]}(\pm 2 b \bar{b} s' r' t) = \left\{ x_{[\beta]} (\pm b \bar{b} s' r' t) x_{[\alpha] + 2[\beta]} (\pm b \bar{b} s' r' t \bar{t}) \right\} \\
            \left\{ x_{[\beta]} (\pm b \bar{b} s' r' (-t)) x_{[\alpha] + 2[\beta]} (\pm b \bar{b} s' r' t \bar{t}) \right\}^{-1} \in H.
        \end{split}
    \end{equation*}
    Set $r' = \pm 1/2$ and let $s_1 = b \bar{b} s' \in S_\theta$, then $x_{[\beta]}(s_1 t) \in H$. By Proposition \ref{z to Rz}, we get $x_{[\alpha] + [\beta]}(s_1 t) \in H$. 
    Now, in $G_\sigma(R_S)$, for any $r_1 \in R$, we have 
    \[ [x_{-[\beta]}(r_1), \psi_\mathfrak{m}(z)] = x_{[\alpha]}(\pm (r_1 \bar{t} + \bar{r}_1 t)) x_{[\alpha] + [\beta]}(\pm r_1 u) x_{[\alpha]}(\pm r_1 \bar{r}_1 u) \in \psi_\mathfrak{m}(H). \]
    Put $r_1 = s_1$, then by the proof of Lemma \ref{z to Rz in phi'}, we have $x_{[\alpha]}(\pm (s_1 \bar{t} + \bar{s}_1 t)) \in \psi_\mathfrak{m}(H)$, and hence
    \begin{equation*}
        \begin{split}
            x_{[\alpha]}(\pm (s_1 \bar{t} + \bar{s}_1 t))^{-1} (x_{[\alpha]}(\pm (s_1 \bar{t} + \bar{s}_1 t)) x_{[\alpha] + [\beta]}(\pm s_1 u) x_{[\alpha]}(\pm s_1 \bar{s}_1 u)) \\
            = x_{[\alpha] + [\beta]}(\pm s_1 u) x_{[\alpha]}(\pm s_1 \bar{s}_1 u) \in \psi_\mathfrak{m}(H).
        \end{split}
    \end{equation*}
    Thus, 
    \begin{equation*}
        w_{[\beta]}(1) x_{[\alpha] + [\beta]}(\pm s_1 u) x_{[\alpha]}(\pm s_1 \bar{s}_1 u) w_{[\beta]}(1)^{-1} = x_{[\alpha] + [\beta]}(\pm s_1 u) x_{[\alpha] + 2[\beta]}(\pm s_1 \bar{s}_1 u) \in \psi_\mathfrak{m}(H).
    \end{equation*}
    Finally, by the above argument, we have $x_{[\alpha] + [\beta]}(s_2 u) \in H$ for some $s_2 \in S_\theta$. Therefore, by Proposition \ref{z to Rz}, we have $x_{[\alpha] + 2[\beta]}(s_2 u) \in H$. This proves (\ref{eq_aa}).

    Now let $x = x_{[\beta]}(t) x_{[\alpha]}(u) x_{[\alpha] + [\beta]}(v) x_{[\alpha] + 2[\beta]}(w) \in H$. By a similar argument as above and as in Lemma \ref{lemma:U cap H}, we can prove the desired result.

    \vspace{2mm}
    
    \noindent (b) This part can also be proven using the same reasoning as in the proofs of Lemma \ref{lemma:UTV cap H} and the methods of G. Taddei \cite{GT} (cf. Lemma \ref{lemma:GT}), similar to the approach taken for part (a). Therefore, we omit the detailed proof.
\end{proof}

\begin{rmk}
    To prove Proposition \ref{prop:U(RS) cap psi(H) subset U(JS)}, we must use not only the method of proof of Proposition \ref{prop:U(R) cap H subset U(J)}, but also the lemma by G. Taddei \cite[Lemma 3.14]{GT}, as we did in the proof of the previous lemma.
\end{rmk}

% \noindent \textbf{Outlines of the proof of part (b):}
%
% \noindent \textbf{Step 1:} Let $z \in H$ and $\psi_\mathfrak{m}(z) = xhy \in U_\sigma (I_S R_S) T_\sigma (R_S) U^{-}_\sigma (R_S)$. \\
% \noindent \textbf{Step 2:} We can assume that $x \neq 1$ or $y \neq 1$. To see this we suppose $\psi_\mathfrak{m}(z) = h$. Since $h = h(\chi) \not\in G_\sigma(R_S, J_S)$, by Lemma \ref{lemma on T(J)}, there exists $\alpha \in \Phi$ such that $\chi(\alpha) - 1 \not\in J_S$. But then there exists an element $t \in (R_S)_{[\alpha]}$ such that $[x_{[\alpha]}(t), h] = x_{[\alpha]}((\chi(\alpha) - 1) \cdot t) \not\in U_\sigma(J_S)$.
% Furthermore, there exists an element $s' \in S_\theta$ such that $t_1 := s' \cdot t \in R_{[\alpha]}$ and $t_2 := (s'(\chi(\alpha) - 1)) \cdot t \in R_{[\alpha]}$. Therefore, $\psi_\mathfrak{m}([x_{[\alpha]}(t_1), z]) = \psi_\mathfrak{m}(x_{[\alpha]}(t_2))$.
% By a similar argument as in Proposition \ref{general:[x,g] in E(R,J)}, we obtain an element $s \in S_\theta$ such that $[x_{[\alpha]}(s \cdot t_1), z] = x_{[\alpha]}(s \cdot t_2)$. Since $z \in H$, we have $x_{[\alpha]}(s \cdot t_2) \in H$ and hence $s \cdot t_2 \in J_{[\alpha]}$.
% Consequently, $\psi_\mathfrak{m}(x_{[\alpha]}(s \cdot t_2)) \in U_\sigma(J_S)$, which implies $x_{[\alpha]}((\chi(\alpha) - 1) \cdot t) \in U_\sigma(J_S)$, leading to a contradiction. Thus, we can assume that $x \neq 1$ or $y \neq 1$. \\
% \noindent \textbf{Step 3:} Next part follows similar lines to the proof of Proposition \ref{prop:U(R) cap H subset U(J)}.

%%%%%%%%%%%%%%%%%%%%%%%%%%%%%%%%%%%%%%%%%%%%%%%%%%

\vspace{2mm}

\noindent \textit{Proof of Proposition \ref{local:H subset G(R,J)}.} If $J_S = R_S$, the proof is complete. Therefore, we assume $J_S \neq R_S$. Since $J_S$ is a $\theta$-invariant proper ideal of $R_S$, we have $J_S \subset \operatorname{rad}(R_S) = I_S R_S$. By Corollary \ref{G(R,J)=UTV}, it follows that 
$$
G_\sigma(R_S, J_S) = U_\sigma(J_S) T_\sigma(R_S, J_S) U^{-}_\sigma(J_S).
$$ 
Assume, for the sake of contradiction, that $\psi_\mathfrak{m}(H) \not\subset G_\sigma(R_S, J_S)$. Under this assumption, we will show that there exists an element $z \in H$ such that $\psi_\mathfrak{m}(z) \not\in G_\sigma(R_S, J_S)$ and $\psi_\mathfrak{m}(z) \in U_\sigma(I_S R_S) T_\sigma(R_S) U^{-}_\sigma(R_S)$. This, however, leads to a contradiction with Proposition \ref{prop:U(RS) cap psi(H) subset U(JS)}.

Let $\pi: G_\sigma (R_S) \longrightarrow G_\sigma (k_S)$ be the canonical homomorphism. 
Suppose $\pi (\psi_\mathfrak{m} (H))$ is central. Then $\psi_\mathfrak{m} (H) \subset G_\sigma (R_S, I_S R_S)$ and hence, by Corollary~\ref{G(R,J)=UTV}, we are done. 
Now assume that $\pi(\psi_\mathfrak{m}(H))$ is non-central. Since $\pi \circ \psi_\mathfrak{m}$ is surjective on elementary subgroups, the subgroup $\pi(\psi_\mathfrak{m}(H))$ of $G_\sigma (k_S)$ is normalized by $E'_\sigma (k_S)$. We claim that $E'_\sigma (k_S) \subset \pi(\psi_\mathfrak{m}(H))$. Assuming the claim to be true for the moment, let us proceed to prove the rest of the result. 
For a given $x_{[\alpha]}(t + I_S) \in U^{-}_\sigma (k_S) \ (t \not \in I_S)$, by our claim, there exists $z \in H$ such that 
$$ \pi (\psi_\mathfrak{m} (z)) = x_{[\alpha]}(t + I_S) = \pi (\psi_\mathfrak{m} (x_{[\alpha]} (t))).$$
Therefore, $\psi_\mathfrak{m} (z) (\psi_\mathfrak{m} (x_{[\alpha]} (t)))^{-1} \in \ker{\pi} = G_\sigma (I_S R_S) = U_\sigma(I_S R_S) T_\sigma (I_S R_S) U^{-}_\sigma (I_S R_S)$, the last equality is due to Proposition \ref{G=UTV}. 
But then $$\psi_\mathfrak{m} (z) \in U_\sigma(I_S R_S) T_\sigma (I_S R_S) U^{-}_\sigma (R_S) \subset U_\sigma(I_S R_S) T_\sigma (R_S) U^{-}_\sigma (R_S),$$ as desired.

Now it only remains to prove the claim. To do this, we observe that $E'_\sigma(k_S) \cap \pi(\psi_\mathfrak{m}(H))$ is a normal subgroup of $E'_\sigma(k_S)$. However, the group $E'_\sigma(k_S)$ is simple over its center. 
To see this, assume first that $\mathfrak{m} = \bar{\mathfrak{m}}$. In this case, $k_S = k_\mathfrak{m}$ is a field, and our result follows from \cite[Theorem 34]{RS}. Now, consider the case where $\mathfrak{m} \neq \bar{\mathfrak{m}}$. By Proposition~\ref{prop:Chevalley as TChevalley}, we can deduce that $E'_\sigma(k_S)$ is isomorphic to $E_\pi(\Phi, k_\mathfrak{m})$. Therefore, applying \cite[Theorem 5]{RS}, we conclude that $E'_\sigma(k_S)$ is simple over its center.
Thus, we can conclude that either $E'_\sigma(k_S) \cap \pi(\psi_\mathfrak{m}(H)) \subset Z(G_\sigma(k_S))$ or $E'_\sigma(k_S) \cap \pi(\psi_\mathfrak{m}(H)) = E'_\sigma(k_S)$. Assume that $E'_\sigma(k_S) \cap \pi(\psi_\mathfrak{m}(H)) \subset Z(G_\sigma(k_S))$. Then, by Proposition \ref{G=G'}, we have $\pi(\psi_\mathfrak{m}(H)) \subset T_\sigma(k_S)$. Since $\pi(\psi_\mathfrak{m}(H))$ is non-central, there exists $h(\chi) \in \pi(\psi_\mathfrak{m}(H))$ with $\chi(\alpha) \neq 1$ for some $\alpha \in \Phi$. But then $[h(\chi), x_{[\alpha]}(1)] = x_{[\alpha]}(\chi(\alpha)-1) \in \pi(\psi_\mathfrak{m}(H))$, which is a contradiction to our assumption. Therefore, we must have $E'_\sigma(k_S) \cap \pi(\psi_\mathfrak{m}(H)) = E'_\sigma(k_S)$, that is, $E'_\sigma(k_S) \subset \pi(\psi_\mathfrak{m}(H))$. \qed

%%%%%%%%%%%%%%%%%%%%%%%%%%%%%%%%%%%%%%%%%%%%%%%%%%
%%%%%%%%%%%%%%%%%%%%%%%%%%%%%%%%%%%%%%%%%%%%%%%%%%

%%%%%%%%%%%%%%%%%%%%%%%%%%%%%%%%%%%%%%%%%%%%%%

\chapter{Normalizer of Twisted Chevalley Groups in a Larger Groups}\label{chapter:Normalizer_of_TCG}

%%%%%%%%%%%%%%%%%%%%%%%%%%%%%%%%%%%%%%%%%%%%%%

Let $R$ be a subring of a ring $S$. This chapter aims to describe the normalizers of $G_\sigma(R)$ and $E'_\sigma(R)$ in the larger group $G_\sigma(S)$. Throughout this chapter, we are going to assume that $\Phi_\rho \sim {}^2 A_n \ (n \geq 3), {}^2 D_n \ (n \geq 4)$ or ${}^2 E_6$.

%%%%%%%%%%%%%%%%%%%%%%%%%%%%%%%%%%%%%%%%%%%%%%
%Section: The Tangent Algebra
%%%%%%%%%%%%%%%%%%%%%%%%%%%%%%%%%%%%%%%%%%%%%%

\section{The Tangent Algebra} \index{the tangent algebra}

%%%%%%%%%%%%%%%%%%%%%%%%%%%%%%%%%%%%%%%%%%%%%%

In this section, we demonstrate how the Lie algebra associated with elementary Chevalley groups can also be reconstructed from the group itself. To achieve this, we introduce the concept of the tangent algebra, following the approach outlined by A. A. Klyachko in \cite{AK}.

%%%%%%%%%%%%%%%%%%%%%%%%%%%%%%%%%%%%%%%%%%%%%%

\subsection{The case of Chevalley Groups}

Let $\mathcal{L}$ be a standard simple Lie algebra of type $\Phi$ over $\mathbb{C}$ (i.e., it can be thought of as a subalgebra of $M_n(\mathbb{C})$). 
Let $R$ be a commutative ring with unity. 
Consider Chevalley algebra $\mathcal{L}(\Phi, R)$ and corresponding adjoint elementary Chevalley group $E (R) = E_{\text{ad}}(\Phi, R)$. 

Let $R[t]$ be the polynomial ring over $R$ with indeterminate $t$. We call elements of $E(R[t])$ \emph{curves} on $E$ defined over $R$ and refer to $E(R[t])$ as the group of these curves. 
Let $f(t) \in R[t]$ be a polynomial. We define a map $\text{REP}_f: E(R[t]) \to E(R[t])$ that maps a curve $g(t) \in E(R[t])$ to the curve $g(f(t))$. This map is a group endomorphism of $E(R[t])$. We call $g(f(t))$ the \emph{reparameterization} of the curve $g(t)$ by the polynomial $f(t)$. 

We now define a tangent space 
\[
    T(E (R)) = \{ X \in M_n(R) \mid 1 + t X + t^2 Y \in E (R[t]) \text{ for some } Y \in M_n(R[t]) \}.
\]
This set is an $R [E(R)]$-module, i.e., it is closed under the following operations:
\begin{enumerate}[(i)]
    \item Addition: $(1 + t X + o(t))(1+ t Y + o(t)) = 1 + t(X+Y) + o(t)$; 
    \item The scalar multiplication by $r \in R$: $\text{REP}_{rt}(1 + tX + o(t)) = 1 + r t X + o(t)$;
    \item The action of the group $E (R)$: $g(1 + t X + o(t))g^{-1} = (1+t gXg^{-1} + o(t))$ (henceforth, we define $g \circ X := g X g^{-1}$).
\end{enumerate}
Here, $o(t)$ denotes a matrix of the form $t^2 Z$, where $Z \in M_n(R[t])$.

\begin{prop}[{\cite[Proposition 6]{AK}}]\label{prop:tangentspace of E(R)}
    $\mathcal{L}(\Phi, R) = T(E(R))$.
\end{prop}

\begin{proof}
    $(\subseteq):$ Consider a Chevalley basis $\{X_\alpha, H_{\alpha_i} \mid \alpha \in \Phi \text{ and } \alpha_i \in \Delta \}$ of $\mathcal{L}$. Since $T(E(R))$ is an $R[E(R)]$-module, to prove $\mathcal{L}(\Phi, R) \subset T(E(R))$, it suffices to show that each basis element is contained in $T(E(R))$. 
    Clearly, $X_\alpha \in T(E(R))$ because $x_{\alpha}(t) \in E (R[t])$. The remaining basis vectors $H_{\alpha_i} \in T(E(R))$ because 
    \[
        H_{\alpha_i} = x_{\alpha_i} (1) \circ X_{-\alpha_i} + X_{\alpha_i} - X_{-\alpha_i}.
    \]
    
    \noindent $(\supseteq):$ 
    Let $X \in T(E(R))$. This means that $1 + tX + o(t) \in E(R[t])$, and therefore it can be expressed as a product of elementary generators:
    \begin{equation}\label{eq_6.1.1}
        1 + tX + o(t) = \prod_{j=1}^{m} x_{\alpha_j}(r_j t^{k_j}).
    \end{equation}
    Clearly, we can assume that $k_j \in \{0, 1\}$. 
    Let $\mathcal{J}'$ be the subset of $\mathcal{J} = \{1, \dots, m\}$ consisting of all $j$ such that $k_j = 0$. 
    Now, for all $j \in \mathcal{J} \setminus \mathcal{J}'$, define 
    \[
        g_j = \prod_{i \in \mathcal{J}_j} x_{\alpha_i}(r_i) \in E(R),
    \]
    where $\mathcal{J}_i = \mathcal{J}' \cap \{1, \dots, j\}$.
    By putting $t = 0$ in equation~\eqref{eq_6.1.1}, we obtain 
    \[
        g_m = \prod_{j \in \mathcal{J}'} x_{\alpha_j}(r_j) = 1.
    \]
    Using this, we can rewrite equation~\eqref{eq_6.1.1} as 
    \[
        1 + tX + o(t) = \prod_{j \in \mathcal{J} \setminus \mathcal{J}'} g_j x_{\alpha_j}(r_j t) g_j^{-1}.
    \]
    By comparing both sides, we conclude that $X = \displaystyle\sum_{j \in \mathcal{J} \setminus \mathcal{J}'} g_j \circ (r_j X_{\alpha_j}) \in \mathcal{L}(\Phi, R)$, as desired.
\end{proof}

%%%%%%%%%%%%%%%%%%%%%%%%%%%%%%%%%%%%%%%%%%%%%%

\subsection{The case of Twisted Chevalley Groups}

We now replicate the argument in the context of twisted Chevalley groups. 
Let $R$ be a commutative ring with unity. 
As before, assume there exists an automorphism $\theta$ of $R$ of order two. 
This automorphism can be extended to the polynomial ring $R[t]$ with the variable $t$ as follows: $t \mapsto t$ and $r \mapsto \theta(r)$ for every $r \in R$. Consequently, the group $E'_{\text{ad}, \sigma}(\Phi, R[t])$ is well-defined and contains $E'_{\text{ad}, \sigma}(\Phi, R)$ as a subgroup.

Let $\mathcal{L}$ be a standard simple Lie algebra of type $\Phi$ over $\mathbb{C}$. Consider the twisted Chevalley algebra $\mathcal{L}_\sigma(\Phi, R)$ and the corresponding adjoint elementary twisted Chevalley group $E'_\sigma(R) = E'_{\text{ad}, \sigma}(\Phi, R)$. 
As before, for $f(t) \in R_\theta[t]$, we define a map 
\[
    \text{REP}_f: E'_\sigma(R[t]) \longrightarrow E'_\sigma(R[t])
\]
that sends $g(t) \mapsto f(g(t))$. This map is then an endomorphism of the group $E'_\sigma(R)$.
Define a tangent space
$$T(E'_\sigma (R)) = \{ X \in M_n(R) \mid 1 + t X + t^2 Y \in E'_\sigma (R[t]) \text{ for some } Y \in M_n(R[t]) \}.$$
This set is an $R_\theta [E'_\sigma (R)]$-module, i.e., it is closed with respect to the following operations: 
\begin{enumerate}[(i)]
    \item Addition: $(1 + t X + o(t))(1+ t Y + o(t)) = 1 + t(X+Y) + o(t)$; 
    \item The scalar multiplication by $r \in R_\theta$: $\text{REP}_{rt}(1 + tX + o(t)) = 1 + r t X + o(t)$;
    \item The action of the group $E'_\sigma (R)$: $g(1 + t X + o(t))g^{-1} = (1+t gXg^{-1} + o(t))$ (henceforth, we define $g \circ X := g X g^{-1}$).
\end{enumerate}

Observe that the map $\sigma: E_{\text{ad}}(\Phi, R) \longrightarrow E_{\text{ad}}(\Phi, R)$ can be extended to a $R_\theta$-linear map from $M_n(R)$ onto itself, also denoted by the same symbol $\sigma$. 

\begin{prop}\label{prop:tangentspace of E'(R)}
    \normalfont
    $\mathcal{L}_\sigma (\Phi, R) = T(E'_\sigma (R)).$
\end{prop}

\begin{proof} 
    \noindent $(\subseteq):$ Consider the $R_\theta$-basis of $\mathcal{L}_\sigma(\Phi, R)$ as described in Section~\ref{sec:TCA}. To establish the result, it suffices to show that each basis element lies in $T(E'_\sigma(R))$. We proceed as follows:
    
    \vspace{2mm}

    \noindent \textbf{Case $[\alpha] \sim A_1$ or $A_1^2$.}
    \begin{enumerate}[(a)]
        \item $X_{[\alpha]}^{+} \in T(E'_\sigma (R))$ because $x_{[\alpha]}(t) \in E'_\sigma (\Phi, R_\theta[t])$.
        \item $X_{[\alpha]}^{-} (\mathrm{I}) \in T(E'_\sigma (R))$ because $x_{[\alpha]}(a t) \in E'_\sigma (\Phi, R_\theta[t])$ (this is only in the case of $[\alpha] \sim A_1^2$).
        \item $H_{[\alpha]}^{+} \in T(E'_\sigma (R))$ because $H_{[\alpha]}^{+} = x_{[\alpha]} (1) \circ X_{-[\alpha]}^{+} + X_{[\alpha]}^{+} - X_{-[\alpha]}^{+}$.
        \item $H_{[\alpha]}^{-} \in T(E'_\sigma (R))$ because $H_{[\alpha]}^{-} = x_{[\alpha]} (1) \circ X_{-[\alpha]}^{-} (\mathrm{I}) + X_{[\alpha]}^{-}(\mathrm{I}) - X_{-[\alpha]}^{-}(\mathrm{I})$ (this is only in the case of $[\alpha] \sim A_1^2$).
    \end{enumerate}

    \noindent \textbf{Case $[\alpha] \sim A_2$.}
    \begin{enumerate}[(a)]
        \item $X_{[\alpha]}^{+} \in T(E'_\sigma (R))$ because $x_{[\alpha]}(t, t^2/2) \in E'_\sigma (\Phi, R_\theta[t])$.
        \item $X_{[\alpha]}^{-} (\mathrm{I}) \in T(E'_\sigma (R))$ because $x_{[\alpha]}(at, (at)^2/2) \in E'_\sigma (\Phi, R_\theta[t])$.
        \item $X_{[\alpha]}^{-} (\mathrm{II}) \in T(E'_\sigma (R))$ because $x_{[\alpha]}(0, at) \in E'_\sigma (\Phi, R_\theta[t])$.
        \item $H_{[\alpha]}^{+} \in T(E'_\sigma (R))$ because $H_{[\alpha]}^{+} = x_{[\alpha]} (1, 1/2) \circ X_{-[\alpha]}^{+} + \frac{1}{2} X_{[\alpha]}^{+} - X_{-[\alpha]}^{+}$.
        \item $H_{[\alpha]}^{-} \in T(E'_\sigma (R))$ because $H_{[\alpha]}^{-} = x_{[\alpha]} (1, 1/2) \circ X_{-[\alpha]}^{-} + \frac{3}{2} X_{[\alpha]}^{-} (\mathrm{I}) - X^{-}_{[\alpha]}(\mathrm{II}) - X^{-}_{-[\alpha]}(\mathrm{I})$.
    \end{enumerate}

    \noindent $(\supseteq):$ 
    Let $X \in T(E'_\sigma (R))$, i.e., $1 + t X + o(t) \in E'_\sigma (R[t])$. Write it as a product of elementary generators:
    \begin{equation}\label{eq_6.2.1}
        1 + t X + o(t) = \prod_{j=1}^m x_{[\alpha_j]}(f_j(t)),
    \end{equation}
    where $f_j (t) = r_j \cdot t^{k_j}$ if $[\alpha_j] \sim A_1$ or $A_1^2$; $(r_j^{(1)} t^{k_j^{(1)}}, r_j^{(2)} t^{k_j^{(2)}})$ if $[\alpha] \sim A_2$.
    Clearly, we can assume that $k_j, k_{j}^{(1)}, k_{j}^{(2)} \in \{ 0, 1 \}$.
    Let $\mathcal{J}'$ be the subset of $\mathcal{J} = \{1, \dots, m\}$ consisting of all $j$ such that $k_j = 0$ or $k_j^{(1)} = k_j^{(2)} = 0$.
    Now, for all $j \in \mathcal{J} \setminus \mathcal{J}'$, define
    \[
        g_j = \prod_{i \in \mathcal{J}_j} x_{[\alpha_i]}(f_i (t)) = \prod_{i \in \mathcal{J}_j} x_{[\alpha_i]}(f_i (0)) \in E'_\sigma (R),
    \]
    where $\mathcal{J}_i = \mathcal{J}' \cap \{ 1, \dots, j\}$.
    By putting $t=0$ in equation (\ref{eq_6.2.1}), we obtain
    \[
        g_m = \prod_{j \in \mathcal{J}'} x_{[\alpha_j]}(f_j (0)) = 1.
    \]
    Using this, we can rewrite equation (\ref{eq_6.2.1}) as
    \[
        1 + tX + o(t) = \prod_{j \in \mathcal{J} \setminus \mathcal{J}'} g_j x_{[\alpha_j]}(f_j(t)) g_j^{-1}.
    \]
    By comparing both sides, we conclude that $X = \displaystyle\sum_{j \in \mathcal{J} \setminus \mathcal{J}'} g_j \circ (r_j X^{\pm}_{[\alpha_j]}) \in \mathcal{L}_\sigma (\Phi, R)$, as desired.
\end{proof}

%%%%%%%%%%%%%%%%%%%%%%%%%%%%%%%%%%%%%%%%%%%%%%

\section{Key Lemma}\label{sec:key lemma}

We now introduce a lemma that plays a crucial role in establishing the main result of this chapter.

\begin{lemma}
    The set $E'_{\pi, \sigma} (\Phi, R)$ generates $M_n (R)$ as an $R$-algebra, where $n$ is the dimension of the corresponding representation $\pi$ of Lie algebra $\mathcal{L}$. 
\end{lemma}

\begin{proof}
        Let $\mathcal{M}$ be the $R$-subalgebra of $M_n (R)$ generated by the set $E'_\sigma (R)$. To prove the lemma we must show that $\mathcal{M} = M_n (R)$.
        We now assume that $\pi$ is a minimal representation (the definition of minimal representation can be found in Section 3 of Chapter 3 in \cite{NV1}). Therefore $\pi (X_\alpha)^3 = 0$ for all $\alpha \in \Phi$ (note that we only consider the case where $\Phi \sim A_n, D_n$, or $E_6$). In particular, we have 
        \[
            x_{\alpha}(t) = 1 + t \, \pi(X_\alpha) + \frac{t^2 \, \pi (X_\alpha)^2}{2!}.
        \]

        We now claim that the elements $\pi (X_\alpha) \in \mathcal{M}$ for every $\alpha \in \Phi$.
        To see this, let $\alpha \in \Phi$ and let $[\alpha]$ denote the corresponding element in $\Phi_\rho$. 
        Just for our convenience, we use the notation $X_{\alpha}$ to denote the linear map $\pi(X_{\alpha})$.
    
        Suppose $[\alpha] \sim A_1$, then 
        \[
            x_{[\alpha]}(t) = x_{\alpha}(t) = 1 + t \, X_\alpha + \frac{t^2 \, (X_\alpha)^2}{2!}.
        \] 
        Then we can easily check that 
        \[
            X_\alpha = \frac{(x_{[\alpha]}(1) - x_{[\alpha]}(-1))}{2}.
        \] 
    
        Now suppose $[\alpha] \sim A_1^2$, then 
        \begin{equation*}
            \begin{split}
                x_{[\alpha]}(t) = x_{\alpha} (t) x_{\bar{\alpha}}(\bar{t}) = 1 + t \, X_\alpha + \bar{t} \, X_{\bar{\alpha}} + \frac{t^2 \, (X_\alpha)^2}{2} + \frac{(\bar{t})^2 \, (X_{\bar{\alpha}})^2}{2} + t \bar{t} \, (X_\alpha \, X_{\bar{\alpha}}) \\
                \frac{t (\bar{t})^2 \, (X_\alpha \, (X_{\bar{\alpha}})^2)}{2} + \frac{(t^2) \bar{t} \, ((X_\alpha)^2 \, X_{\bar{\alpha}})}{2} + \frac{(t \bar{t})^2 \, ((X_\alpha)^2 \, (X_{\bar{\alpha}})^2)}{4}.
            \end{split}
        \end{equation*}
        A simple calculation on roots shows that 
        \[
            X_\alpha \, (X_{\bar{\alpha}})^2 
            = (X_\alpha)^2 \, X_{\bar{\alpha}} 
            = (X_\alpha)^2 \, (X_{\bar{\alpha}})^2
            = 0.
        \]
        Therefore, 
        \begin{equation*}
            x_{[\alpha]}(t) = 1 + 
            \Big( t \, X_\alpha + \bar{t} \, X_{\bar{\alpha}} \Big) + 
            \frac{1}{2} \Big( t \, X_\alpha + \bar{t} \, X_{\bar{\alpha}} \Big)^2.
        \end{equation*}
        Then we have 
        \begin{align*}
            X_{\alpha} + X_{\bar{\alpha}} &= \frac{1}{2} (x_{[\alpha]}(1) - x_{[\alpha]}(-1)), \text{ and} \\
            X_{\alpha} - X_{\bar{\alpha}} &= \frac{1}{2a} (x_{[\alpha]}(a) - x_{[\alpha]}(-a)).
        \end{align*}
        Hence 
        \begin{align*}
            X_{\alpha} &= \frac{1}{4} \Big( \big(x_{[\alpha]}(1) - x_{[\alpha]}(-1) \big) + a^{-1} \big(x_{[\alpha]}(a) - x_{[\alpha]}(-a) \big) \Big), \text{ and} \\
            X_{\bar{\alpha}} &= \frac{1}{4} \Big( \big(x_{[\alpha]}(1) - x_{[\alpha]}(-1) \big) - a^{-1} \big(x_{[\alpha]}(a) - x_{[\alpha]}(-a) \big) \Big).
        \end{align*}

        Finally suppose $[\alpha] \sim A_2$, then 
        \begin{align*}
            x_{[\alpha]}(t,u) &= x_{\alpha} (t) x_{\bar{\alpha}}(\bar{t}) x_{\alpha + \bar{\alpha}} (N_{\bar{\alpha}, \alpha} u) \\
            &= 1 + \Big( t \, X_\alpha + \bar{t} \, X_{\bar{\alpha}} + N_{\bar{\alpha}, \alpha} u \, X_{\alpha + \bar{\alpha}} \Big) 
            + \Big( \frac{t^2}{2} \, (X_\alpha)^2 + \frac{(\bar{t})^2}{2} \, (X_{\bar{\alpha}})^2 + \frac{u^2}{2} \, (X_{\alpha + \bar{\alpha}})^2 \\
            & \hspace{5mm} + t \bar{t} \, (X_\alpha \, X_{\bar{\alpha}}) + N_{\bar{\alpha},\alpha} t u \, (X_\alpha \, X_{\alpha + \bar{\alpha}}) + N_{\bar{\alpha},\alpha} \bar{t} u \, (X_{\bar{\alpha}} \, X_{\alpha + \bar{\alpha}}) \Big) \\
            & \hspace{5mm} + \Big( N_{\bar{\alpha}, \alpha} t \bar{t} u \, (X_\alpha \, X_{\bar{\alpha}} \, X_{\alpha+ \bar{\alpha}}) + \frac{t^2 \bar{t}}{2} \, ((X_\alpha)^2 \, X_{\bar{\alpha}}) + \frac{t (\bar{t})^2}{2} \, (X_\alpha \, (X_{\bar{\alpha}})^2) \Big).
        \end{align*}
        Note that all the other terms in the expansion of the product $x_{\alpha} (t) x_{\bar{\alpha}}(\bar{t}) x_{\alpha + \bar{\alpha}} (N_{\bar{\alpha}, \alpha} u)$ which does not appear in the above expression are $0$ (for instance $X_\alpha \, (X_{\alpha + \bar{\alpha}})^2 = 0, (X_\alpha)^2 \, X_{\alpha + \bar{\alpha}} = 0,$ etc.). 
        First note that, 
        \[
            X_{\alpha + \bar{\alpha}} = \frac{N_{\bar{\alpha},\alpha}}{2a} \Big( X_{[\alpha]}(0,a) - x_{[\alpha]}(0, -a) \Big).
        \]
        Now define
        \begin{align*}
            \mathcal{F}_{[\alpha]}(t, u) &= x_{[\alpha]}(t,u) - x_{[\alpha]}(-t, u) \\
            &= 2(t X_\alpha + \bar{t} X_{\bar{\alpha}}) + 2 u (t X_\alpha X_{\alpha + \bar{\alpha}} + \bar{t} X_{\bar{\alpha}} X_{\alpha + \bar{\alpha}}) + t \bar{t} (t (X_\alpha)^2 X_{\bar{\alpha}} + \bar{t} X_{\alpha} (X_{\bar{\alpha}})^2).
        \end{align*}
        Then 
        \begin{align*}
            X_{\alpha} + X_{\bar{\alpha}} &= \frac{1}{12} \big( 8 \mathcal{F}(1, 1/2) - \mathcal{F}(2, 2) \big), \text{ and } \\
            X_{\alpha} - X_{\bar{\alpha}} &= \frac{1}{12 a} \big( 8 \mathcal{F}(a, -\frac{a^2}{2}) - \mathcal{F}(2a, -2a^2) \big).
        \end{align*}
        Hence
        \begin{align*}
            X_{\alpha} &= \frac{1}{24} \Big( \big( 8 \mathcal{F}(1, 1/2) - \mathcal{F}(2, 2) \big) + a^{-1} \big( 8 \mathcal{F}(a, -\frac{a^2}{2}) - \mathcal{F}(2a, -2a^2) \big) \Big), \text{ and } \\
            X_{\bar{\alpha}} &= \frac{1}{24} \Big( \big( 8 \mathcal{F}(1, 1/2) - \mathcal{F}(2, 2) \big) - a^{-1} \big( 8 \mathcal{F}(a, -\frac{a^2}{2}) - \mathcal{F}(2a, -2a^2) \big) \Big).
        \end{align*}
    
        Therefore, $\pi (X_{\alpha}) \in \mathcal{M}$ for every $\alpha \in \Phi$. 
        In particular, $\pi (\mathcal{L}(R)) \subset \mathcal{M}$. 
        Hence, $E_\pi(\Phi, R) \subset \mathcal{M}$. 
        By Lemma $2$ of \cite{EB24:final}, the set $E(R)$ generates $M_n(R)$ as an $R$-algebra. Therefore, $\mathcal{M} = M_n(R)$. 
        This completes the proof of the lemma.
\end{proof}

%%%%%%%%%%%%%%%%%%%%%%%%%%%%%%%%%%%%%%%%%%%%%%
%Section - Normalizer of $G_{\pi,\sigma}(\Phi, R)$ and $E'_{\pi,\sigma}(\Phi, R)$ in $G_{\pi,\sigma}(\Phi, S)$
%%%%%%%%%%%%%%%%%%%%%%%%%%%%%%%%%%%%%%%%%%%%%%

\section{Normalizer of \texorpdfstring{$G_{\pi,\sigma}(\Phi, R)$}{G(R)} and \texorpdfstring{$E'_{\pi,\sigma}(\Phi, R)$}{E(R)} in \texorpdfstring{$G_{\pi,\sigma}(\Phi, S)$}{G(S)}}\label{sec:normalizer_of_G_and_E}

The primary goal of this section is to establish the following theorem:

\begin{thm}\label{thm:normalizer_of_G_and_E}
    \normalfont
    Let $R$ be a commutative ring with unity. 
    Let $G_{\pi, \sigma} (\Phi, R)$ be a twisted Chevalley group of type ${}^2 A_n \ (n \geq 3), {}^2 D_n \ (n \geq 4)$ or ${}^2 E_6$, and let $E'_{\pi, \sigma} (\Phi, R)$ be its elementary subgroup. 
    Assume that $1/2 \in R$ and, if $\Phi_\rho \sim {}^2 A_{2n}$, additionally assume that $1/3 \in R$. 
    Furthermore, assume that there exists an invertible element $a \in R$ such that $\theta(a) = - a$.  
    If $S$ is a ring extension of $R$, then we have
    \[
        N_{G_{\pi,\sigma}(\Phi, S)}(G_{\pi, \sigma}(\Phi, R)) = N_{G_{\pi,\sigma} (\Phi, S)} (E'_{\pi, \sigma} (\Phi, R)).
    \]
    Moreover, if $G$ is of adjoint type, we have
    \[
        N_{G_{\text{ad},\sigma}(\Phi, S)}(G_{\text{ad}, \sigma}(\Phi, R)) = N_{G_{\text{ad},\sigma} (\Phi, S)} (E'_{\text{ad}, \sigma} (\Phi, R)) = G_{\text{ad}, \sigma} (\Phi, R).
    \]
\end{thm}

\begin{proof}
    We begin by proving that $N_{G_{\pi, \sigma} (\Phi, S)}(G_{\pi, \sigma} (\Phi, R)) = N_{G_{\pi, \sigma} (\Phi, S)}(E'_{\pi, \sigma} (\Phi, R))$. 
    Since $E'_{\pi, \sigma} (\Phi, R)$ is a characteristic subgroup of $G_{\pi,\sigma} (\Phi, R)$ (see Corollary~\ref{char subgrp}), it follows that 
    $$ N_{G_{\pi, \sigma} (\Phi, S)}(G_{\pi, \sigma} (\Phi, R)) \subset N_{G_{\pi,\sigma} (\Phi, S)}(E'_{\pi, \sigma} (\Phi, R)). $$ 
    
    For the reverse inclusion, let $g \in N_{G_{\pi,\sigma} (\Phi, S)} (E'_{\pi, \sigma} (\Phi, R))$. 
    Thus, $g E'_{\pi, \sigma} (\Phi, R) g^{-1} = E'_{\pi, \sigma} (\Phi, R).$
    By the above lemma, we have $g M_n (R) g^{-1} = M_n (R)$. 
    This implies that $g GL_n (R) g^{-1} = GL_n (R)$. 
    In particular, $g G_{\pi, \sigma} (\Phi, R) g^{-1} \subset GL_n (R)$.
    On the other hand, $g G_{\pi, \sigma} (\Phi, R) g^{-1} \subset G_{\pi, \sigma} (\Phi, S)$ as $g \in G_{\pi, \sigma} (\Phi, S)$.
    Therefore, we have 
    \[
        g G_{\pi, \sigma}(\Phi, R) g^{-1} \subset GL_n(R) \cap G_{\pi, \sigma} (\Phi, S) = G_{\pi, \sigma} (\Phi, R).
    \]
    Thus, $g \in N_{G_{\pi, \sigma} (\Phi, S)}(G_{\pi, \sigma} (\Phi, R))$, which implies that $$N_{G_{\pi, \sigma} (\Phi, S)}(E'_{\pi, \sigma} (\Phi, R)) \subset N_{G_{\pi, \sigma} (\Phi, S)}(G_{\pi, \sigma} (\Phi, R)),$$ as desired.  

    Now, it remains to prove that $N_{G_{\text{ad}, \sigma} (\Phi, S)} (E'_{\text{ad}, \sigma} (\Phi, R)) = G_{\text{ad}, \sigma} (\Phi, R).$ 
    Clearly, $$G_{\text{ad}, \sigma} (\Phi, R) \subset N_{G_{\text{ad}, \sigma} (\Phi, S)} (G_{\text{ad}, \sigma} (\Phi, R)) = N_{G_{\text{ad}, \sigma} (\Phi, S)} (E'_{\text{ad}, \sigma} (\Phi, R)).$$ 
    To establish the reverse inclusion, let $g \in N_{G_{\text{ad}, \sigma} (\Phi, S)} (E'_{\text{ad}, \sigma} (\Phi, R))$. 
    Then, by definition, $$g E'_{\text{ad}, \sigma} (\Phi, R) g^{-1} = E'_{\text{ad}, \sigma} (\Phi, R).$$ 
    Using Proposition~\ref{prop:tangentspace of E'(R)}, we conclude that $$g (\text{ad}(\mathcal{L}_\sigma (\Phi, R))) g^{-1} = \text{ad}(\mathcal{L}_\sigma (\Phi, R)).$$
    By ``complexification", we get $$g (\text{ad}(\mathcal{L} (\Phi, R))) g^{-1} = \text{ad}(\mathcal{L} (\Phi, R)).$$
    That is, $i_g$ is an automorphism of the Lie algebra $\text{ad}(\mathcal{L}(\Phi, R))$. 
    Thus, $i_g$ can be decomposed as $i_g = i_{g'} \circ \delta$, where $g' \in G_{\text{ad}}(\Phi, R)$ and $\delta$ is a graph automorphism (see Theorem 1 of \cite{AK} or Lemma 3 of \cite{EB12:main}). 
    Consequently, $i_g$ is an automorphism of $E_{\text{ad}} (\Phi, R)$ that admits the same decomposition.
    Since the intersection between the set of all graph automorphisms and the set of all inner automorphisms of $E_{\text{ad}} (\Phi, R)$ is trivial (see \cite{EA4}), we deduce that $\delta = \operatorname{id}$.
    It follows that $g \cdot {g'}^{-1} = \lambda \cdot 1_n \in G_{\text{ad}}(\Phi, S)$ for some $\lambda \in S$. 
    Since there are no scalar matrices in the group of adjoint type (as the centre is trivial, see \cite{EA&JH}), we conclude that $\lambda = 1$. 
    Consequently, $g = g' \in G_{\text{ad}, \sigma}(\Phi, S) \cap G_{\text{ad}}(\Phi, R) = G_{\text{ad}, \sigma}(\Phi, R)$.
    Therefore, $N_{G_{\text{ad}, \sigma} (\Phi, S)} (E'_{\text{ad}, \sigma} (\Phi, R)) \subset G_{\text{ad}, \sigma} (\Phi, R)$, as desired.
    This completes the proof of Theorem~\ref{thm:normalizer_of_G_and_E}.
\end{proof}

%%%%%%%%%%%%%%%%%%%%%%%%%%%%%%%%%%%%%%%%%%%%%%%%%% 

%\setcounter{secnumdepth}{-2} 

%\include{conclusion and future work} 

%%%%%%%%%%%%%%%%%%%%%%%%%%%%%%%%%%%%%%%%%%%%%%%%%% 

%\include{appendix} 

%%%%%%%%%%%%%%%%%%%%%%%%%%%%%%%%%%%%%%%%%%%%%%%%%% 

\printindex

%%%%%%%%%%%%%%%%%%%%%%%%%%%%%%%%%%%%%%%%%%%%%%%%%% 


\begin{thebibliography}{AAAA}
%\bibliographystyle{alpha}
%\thispagestyle{empty}
\addcontentsline{toc}{chapter}{Bibliography}

%%%%%%%%%%%%%%%%%%%%%%%%%%%%%%%%%%%%%%%%%%%%%%%%%%%%%%%%%%%%%%

\bibitem{EA1} Eiichi Abe, \emph{Coverings of twisted Chevalley groups over commutative rings}, Sci. Rep. Tokyo Kyoiku Daigaku Sect. A, Vol. 13 (1977), 194--218.

\bibitem{EA2} Eiichi Abe, \emph{Chevalley groups over local rings}, Tohoku Math. J. (2), Vol. 21 (1969), 474--494.

\bibitem{EA3} Eiichi Abe, \emph{Chevalley groups over commutative rings}, Radical theory, Proc. 1988, Sendai Conf., Vol. 83 (1989), 1--23.

\bibitem{EA4} Eiichi Abe, \emph{Chevalley groups over commutative rings. {N}ormal subgroups and automorphisms}, Second {I}nternational {C}onference on {A}lgebra ({B}arnaul, 1991), Contemp. Math., Vol. 184 (1995), 13--23.

\bibitem{EA5} Eiichi Abe, \emph{Hopf algebras}, Cambridge University Press, Cambridge-New York (1980).

\bibitem{EA&JH} Eiichi Abe and James F. Hurley, \emph{Centers of Chevalley groups over commutative rings}, Comm. Algebra, Vol. 16 (1988), 57--74.

\bibitem{EA&KS} Eiichi Abe and Kazuo Suzuki, \emph{On normal subgroups of Chevalley groups over commutative rings}, Tohoku Math. J. (2), Vol. 28 (1976), 185--198.

\bibitem{ABak1} Anthony Bak, \emph{On modules with quadratic forms}, Algebraic {$K$}-{T}heory and its {G}eometric {A}pplications ({C}onf., {H}ull, 1969), Vol. 108 (1969), 55--66.

\bibitem{ABak2} Anthony Bak, \emph{Subgroups of the general linear group normalized by relative elementary groups}, Algebraic {$K$}-theory, {P}art {II} ({O}berwolfach, 1980), Vol. 967 (1982), 1--22.

\bibitem{AB&RP} Anthony Bak and Raimund Preusser, \emph{The {E}-normal structure of odd dimensional unitary groups}, J. Pure Appl. Algebra, Vol. 222 (2018), No. 9, 2823--2880.

\bibitem{AB&NV} Anthony Bak and Nikolai A. Vavilov, \emph{Normality for elementary subgroup functors}, Math. Proc. Cambridge Philos. Soc., Vol. 118 (1995), 35--47.

\bibitem{HB1} Hyman Bass, \emph{$K$-theory and stable algebra}, Inst. Hautes \'Etudes Sci. Publ. Math., Vol. 22 (1964), 5--60.

\bibitem{HB2} Hyman Bass, \emph{Algebraic {$K$}-theory}, W. A. Benjamin, Inc., New York-Amsterdam (1968).

\bibitem{AB1} Armand Borel, \emph{Linear algebraic groups}, Graduate Texts in Mathematics, Springer-Verlag, New York (1991).

\bibitem{AB2} Armand Borel, \emph{Properties and linear representations of {C}hevalley groups}, Lecture Notes in Math., Vol. 131 (1970), 1--55.

\bibitem{ZB&NV} Zenon I. Borevich and N. A. Vavilov, \emph{The distribution of subgroups in the general linear group over a commutative ring} Proc. Steklov. Inst. Math, N. 3 (1985), 27--46.

\bibitem{EB24:final} Elena I. Bunina, \emph{Automorphisms of Chevalley groups over commutative rings}, Comm. Algebra, Vol. 52 (2024), 2313--2327.

\bibitem{EB12:main} Elena I. Bunina, \emph{Automorphisms of Chevalley groups of different types over commutative rings}, J. Algebra, Vol. 355 (2012), 154--170.

\bibitem{EB07:first} Elena I. Bunina, \emph{Automorphisms of Chevalley groups of certain types over local rings}, Uspekhi Mat. Nauk, Vol. 62 (2007), 143--144.

\bibitem{EB08:b2andg2} Elena I. Bunina, \emph{Automorphisms of Chevalley groups of types $B_2$ and $G_2$ over local rings}, Journal of Mathematical Sciences, Vol. 155, No. 6 (2008), 795--814.

\bibitem{EB09:alwithhalfele} Elena I. Bunina, \emph{Automorphisms of elementary adjoint Chevalley groups of types $A_l, D_l, E_l$ over local rings with $1/2$}, Algebra and Logic, Vol. 48, No. 4 (2009), 250--267.

\bibitem{EB10:alwithhalf} Elena I. Bunina, \emph{Automorphisms of Chevalley groups of types $A_l, D_l,$ or $E_l$ over local rings with $1/2$}, Journal of Mathematical Sciences, Vol. 167, No. 6 (2010), 749--766.

\bibitem{EB10:alwithouthalf} Elena I. Bunina, \emph{Automorphisms of Chevalley groups of types $A_l, D_l, E_l$ over local rings without $1/2$}, Journal of Mathematical Sciences, Vol. 169, No. 5 (2010), 589--613.

\bibitem{EB10:blwithhalf} Elena. I. Bunina, \emph{Automorphisms of Chevalley groups of type $B_l$ over local rings with $1/2$}, Journal of Mathematical Sciences, Vol. 169, No. 5 (2010), 557--588.

\bibitem{EB10:f4withhalf} Elena I. Bunina, \emph{Automorphisms of Chevalley groups of type $F_4$ over local rings with $1/2$}, Journal of Algebra, Vol. 323, No. 8 (2010), 2270--2289.

\bibitem{EB&PV14:g2withouthlfnor} Elena I. Bunina and P. A. Verëvkin, \emph{Normalizers of Chevalley groups of type $G_2$ over local rings without $1/2$}, Journal of Mathematical Sciences, Vol. 201, No. 4 (2014), 446--449.

\bibitem{EB&PV14:g2withouthalf} Elena I. Bunina and P. A. Verëvkin, \emph{Automorphisms of Chevalley groups of type $G_2$ over local rings without $1/2$}, Journal of Mathematical Sciences, Vol. 197, No. 4 (2014), 479--491.

\bibitem{EB&MV24:g2with1/3} Elena I. Bunina and M. A. Vladykina, \emph{Automorphisms of Chevalley groups of type $G_2$ over a commutative ring $R$ with $1/3$ generated by the invertible elements and $2R$}, Journal of Mathematical Sciences, Vol. 284, No. 4 (2014), 431--441.

\bibitem{RC} Roger W. Carter, \emph{Simple Groups of Lie Type}, 2nd edition, Wiley, London (1989).

\bibitem{SG&DM1} Shripad M. Garge and Deep H. Makadiya, \emph{On Normal Subgroups of Twisted Chevalley Groups over Commutative Rings}, Accepted in Journal of Algebra (2025).

\bibitem{SG&DM2} Shripad M. Garge and Deep H. Makadiya, \emph{Normalizer of Twisted Chevalley Groups over Commutative Rings}, \href{https://doi.org/10.48550/arXiv.2503.16894}{arXiv:2503.16894} (2025).

\bibitem{SG&DM3} Shripad M. Garge and Deep H. Makadiya, \emph{Triangular and Unitriangular Factorization of Twisted Chevalley Groups}, \href{https://doi.org/10.48550/arXiv.2505.20224}{arXiv.2505.20224} (2025). 


\bibitem{RH&NV} Roozbeh Hazrat and Nikolai A. Vavilov, \emph{{$K_1$} of Chevalley groups are nilpotent}, J. Pure Appl. Algebra, Vol. 179 (2003), 99--116. 

\bibitem{JH} James E. Humphreys, \emph{Introduction to Lie Algebras and Representation Theory}, Graduate Texts in Mathematics, Springer-Verlag, New York-Berlin (1972).

\bibitem{JH_LAG} James E. Humphreys, \emph{Linear Algebraic Groups}, Graduate Texts in Mathematics, Springer-Verlag, New York-Heidelberg (1975).

\bibitem{JH_RC} James E. Humphreys, \emph{Reflection groups and Coxeter groups}, Cambridge University Press, Cambridge (1990).

\bibitem{AK} Anton A. Klyachko, \emph{Automorphisms and isomorphisms of {C}hevalley groups and algebras}, J. Algebra, Vol. 324 (2010), 2608--2619.

\bibitem{VK} V. I. Kope\u{i}ko, \emph{The stabilization of symplectic groups over a polynomial ring}, Math. U.S.S.R., Sbornik, Vol. 34 (1978), 655--669.

\bibitem{RL} Ramji Lal, Algebra 4--Lie Algebras, Chevalley Groups, and their Representations, Infosys Science Foundation Series, Springer, Singapore (2021).

\bibitem{AO&EV} A. L. Onishchik and \`E. B. Vinberg, \emph{Lie groups and algebraic groups}, Springer Series in Soviet Mathematics, Springer-Verlag, Berlin (1990).

\bibitem{VP&AS} V. A. Petrov and A. K. Stavrova, \emph{Elementary subgroups in isotropic reductive groups}, St. Petersburg Math. J., Vol. 20 (2009), no. 4, 625--644.

\bibitem{EP&NV} Eugene Plotkin and Nikolai A. Vavilov, \emph{Chevalley groups over commutative rings. \Romannum{1}. {E}lementary calculations}, Acta Appl. Math., Vol. 45 (1996), 73--113. 

\bibitem{EP1} Eugene Plotkin, \emph{Surjective stabilization of the {$K_1$}-functor for some exceptional {C}hevalley groups}, Zap. Nauchn. Sem. Leningrad. Otdel. Mat. Inst. Steklov., Vol. 198 (1991), 65--88. 

\bibitem{EP2} Eugene Plotkin, \emph{Surjective stability of the {$K_1$}-functor for {C}hevalley groups of normal and twisted types}, Russian Math. Surveys, Vol. 44 (1989), 239--240. 

\bibitem{RP1} Raimund Preusser, \emph{The {E}-normal structure of {P}etrov's odd unitary groups over commutative rings}, Communications in Algebra, Vol. 48 (2020), No. 3, 1114--1131.

\bibitem{RP2} Raimund Preusser, \emph{Sandwich classification for {${\rm GL}_n(R)$}, {${\rm O}_{2n}(R)$} and {${\rm U}_{2n}(R,\Lambda)$} revisited}, Journal of Group Theory, Vol. 21 (2018), No. 1, 21--44.

\bibitem{RP3} Raimund Preusser, \emph{Sandwich classification for {$O _{2n+1}(R)$} and {$U_{2n+1}(R,\Delta)$} revisited}, Journal of Group Theory, Vol. 21 (2018), No. 4, 539--571.

\bibitem{RP4} Raimund Preusser, \emph{Structure of hyperbolic unitary groups {II}: {C}lassification of {E}-normal subgroups}, Algebra Colloquium, Vol. 24 (2017), No. 2, 195--232.

\bibitem{TS} Tonny A. Springer, \emph{Linear algebraic groups}, 2nd edition, Birkh\"auser Boston, Inc., Boston, MA (1998).

\bibitem{AS&AS} Anastasia Stavrova and Alexei Stepanov, \emph{Normal structure of isotropic reductive groups over rings}, Journal of Algebra, Vol. 656 (2024), 486--515.

\bibitem{MS} Michael R. Stein, \emph{Stability theorems for {$K\sb{1}$}, {$K\sb{2}$}\ and related functors modeled on {C}hevalley groups}, Japan. J. Math., Vol. 4 (1978), 77--108.

\bibitem{RS} Robert Steinberg, \emph{Lectures on Chevalley Groups}, Yale University Press (1968).

\bibitem{RSTCG} Robert Steinberg, \emph{Variations on a Theme of Chevalley}, Pacific J. Math., Vol. 9 (1959), 875--891.

\bibitem{AS1} Andrei A. Suslin, \emph{On the structure of the general linear group over polynomial rings}, Soviet Math. Izv., Vol. 41 (1977), 503--516.

\bibitem{AS&VK} Andrei A. Suslin and V. I. Kope\u iko, \emph{Quadratic modules and the orthogonal group over polynomial rings}, Zap. Nau\v cn. Sem. Leningrad. Otdel. Mat. Inst. Steklov. (LOMI), Vol. 71 (1977), 216--250, 287.

\bibitem{KS1} Kazuo Suzuki, \emph{On Normal Subgroups of Twisted Chevalley Groups over Local Rings}, Sci. Rep. Tokyo Kyoiku Daigaku Sect. A, Vol. 13 (1977), 238--249.

\bibitem{KS2} Kazuo Suzuki, \emph{Normality of the Elementary Subgroups of Twisted Chevalley Groups over Commutative Rings}, J. Algebra, Vol. 175 (1995), 526--536.

\bibitem{KS3} Kazuo Suzuki, \emph{Centers of Twisted Chevalley Groups over Commutative Rings}, Kumamoto J. Math., Vol. 6 (1993), 1--9.

\bibitem{GT} Giovanni Taddei, \emph{Normalit\'{e} des groupes \'{e}l\'{e}mentaires dans les groupes de {C}hevalley sur un anneau}, Contemp. Math. (2), Vol. 55 (1986), 693--710.

\bibitem{GT2} Giovanni Taddei, \emph{Invariance du sous-groupe symplectique \'el\'ementaire dans le groupe symplectique sur un anneau}, C. R. Acad. Sci. Paris S\'er. I Math., Vol. 295 (1982), 313--329.

\bibitem{JT} Jacques Tits, \emph{Algebraic and abstract simple groups}, Ann. of Math. (2), Vol. 80 (1964), 313–-329.

\bibitem{LV} Leonid N. Vaserstein, \emph{On normal subgroups of Chevalley groups over commutative rings}, Tohoku Math. J. (2), Vol. 38 (1986), 219--230.

\bibitem{LV2} Leonid N. Vaserstein, \emph{On the normal subgroups of {${\rm GL}\sb{n}$}\ over a ring}, Algebraic {$K$}-theory, {E}vanston 1980 ({P}roc. {C}onf., {N}orthwestern {U}niv., {E}vanston, {I}ll., 1980), Vol. 854 (1981), 456--465.

\bibitem{LV3} Leonid N. Vaserstein, \emph{Normal subgroups of orthogonal groups over commutative rings}, Amer. J. Math., Vol. 110 (1988), 955--973.

\bibitem{LV4} Leonid N. Vaserstein, \emph{Normal subgroups of symplectic groups over rings}, $K$-Theory, Vol. 2 (1989), 647--673.

\bibitem{NV1} Nikolai A. Vavilov, \emph{Structure of Chevalley groups over commutative rings}, Nonassociative algebras and related topics (Hiroshima, 1990), World Sci. Publ., River Edge, NJ (1991), 219--335. 

\bibitem{NV2} Nikolai A. Vavilov, \emph{A third look at weight diagrams}, Rend. Sem. Mat. Univ. Padova, Vol. 104 (2000), 201--250. 

\bibitem{NV3} Nikolai A. Vavilov, \emph{Intermediate subgroups in Chevalley groups}, London Math. Soc. Lecture Note Ser., Vol. 207 (1995), 233--280. 

\bibitem{NV4} Nikolai A. Vavilov, \emph{On subgroups of the split classical groups}, Trudy Mat. Inst. Steklov, Vol. 183 (1990), 29--41. 

\bibitem{NV5} Nikolai A. Vavilov, \emph{The structure of split classical groups over a commutative ring}, Dokl. Akad. Nauk SSSR, Vol. 299 (1988), 1300--1303.

\bibitem{NV6} Nikolai A. Vavilov, \emph{On subgroups of the split classical groups}, Trudy Mat. Inst. Steklov, Vol. 183 (1990), 29--41.

\bibitem{NV7} Nikolai A. Vavilov, \emph{On the problem of normality of the elementary subgroup in a Chevalley group}, Algebraic and discrete systems, Ivanovo Univ. (1988), 7--25.

%%%%%%%%%%%%%%%%%%%%%%%%%%%%%%%%%%%%%%%%%%%%%%%%%%%%%%%%%%%%%%

\end{thebibliography}
\end{document}